\documentclass[a4paper]{article}

\usepackage[english,francais]{babel}
\usepackage{lmodern}
\usepackage[T1]{fontenc}
\rmfamily
\DeclareFontShape{T1}{lmr}{b}{sc}{<->ssub*cmr/bx/sc}{}
\DeclareFontShape{T1}{lmr}{bx}{sc}{<->ssub*cmr/bx/sc}{}

\usepackage[pass]{geometry}

\usepackage[all]{xy}
\usepackage{tabularx}
\usepackage{tikz}
\usepackage{calc}
\usepackage{graphicx}
\usepackage{fancyhdr}
\usepackage[babel=true]{csquotes}

\usepackage{amssymb,amsmath,amsthm,amscd}
\usepackage{mathdots}
\usepackage{mathrsfs}
\usepackage[retainorgcmds]{IEEEtrantools}

\usepackage[noprefix]{nomencl}
\input{Nomencl/nomencl.tex}
\makenomenclature

\usepackage{appendix}
\usepackage[nottoc,notlof,notlot,notbib]{tocbibind}
\newenvironment{dedicace}{%
  \newpage\thispagestyle{empty}
  \hfill\begin{minipage}{100mm}\begin{flushright}\it}{%
  \end{flushright}\end{minipage}\vfill}

\newtheorem{introthm}{Théorème}
\newtheorem{introconj}[introthm]{Conjecture}
\newtheorem{introdefi}[introthm]{Définition}
\newtheorem{introprop}[introthm]{Proposition}
\newtheorem{introcor}[introthm]{Corollaire}
\newtheorem{introquest}[introthm]{Question}
\newtheorem{intropb}[introthm]{Problème}
\newtheorem{introprinc}[introthm]{Principe}

\newtheorem{defi}{Definition}[section]
\newtheorem{rem}[defi]{Remark}
\newtheorem{lem}[defi]{Lemma}
\newtheorem{prop}[defi]{Proposition}
\newtheorem{defi-prop}[defi]{Definition-proposition}

\newtheorem{thm}[defi]{Theorem}

\newtheorem*{conj}{Conjecture}

\newtheorem{claimn}{Claim}

\newtheorem{definition}{Définition}[section]
\newtheorem{exemple}[definition]{Exemple}
\newtheorem{remarque}[definition]{Remarque}
\newtheorem{lemme}[definition]{Lemme}
\newtheorem{proposition}[definition]{Proposition}
\newtheorem{defiprop}[definition]{Définition-Proposition}
\newtheorem{corollaire}[definition]{Corollaire}
\newtheorem{theoreme}[definition]{Théorème}

\newcommand{\ZZ}{\mathbb Z}

\newcommand{\glie}{\mathfrak g}
\newcommand{\Lglie}{{{}^L \! \! \; \mathfrak g}}

\newcommand{\Lpsi}{{{}^L \! \! \; \psi_i}}

\newcommand{\hlie}{\mathfrak h}
\newcommand{\hlied}{\mathfrak h^\ast}
\newcommand{\hliedZ}{\mathfrak h^\ast_\ZZ}

\newcommand{\Lhlie}{{{}^L \! \; \! \mathfrak h}}

\newcommand{\Lhlied}{{{}^L \! \; \! \mathfrak h^\ast}}
\newcommand{\LhliedZ}{{{}^L \! \; \! \mathfrak h^\ast_\ZZ}}

\newcommand{\LVR}{{{}^L \! \: \! \VR}}

\newcommand{\LV}{{{}^L \! \! \; V}}
\newcommand{\LVq}{{{}^L \! \! \; V_{\! \! \; q}}}
\newcommand{\LXpm}{{{}^L \! X^\pm_i}}

\newcommand{\slt}{\mathfrak{sl}_2}

\newcommand{\Uq}{{U_{\! \! \: q}}}
\newcommand{\Ut}{{U_{\! \! \; t}}}
\newcommand{\Uqt}{{U_{\! \! \: q,t}}}

\newcommand{\Lsglie}{{{}^L \! \! \; \hat \glie}}

\newcommand{\NN}{\mathbb N}
\newcommand{\NNI}{\mathbb N^{\ast}}
\newcommand{\QQ}{\mathbb Q}
\newcommand{\KK}{\Bbbk}

\newcommand{\Ugeneric}[2]{{U_{#1}(#2)}}
\newcommand{\Uc}[1][\glie]{\Ugeneric{\! \! \; \KK}{#1}}
\newcommand{\Uh}[1][\glie,\psi]{\Ugeneric{\! \! \: h}{#1}}

\newcommand{\tglie}{\tilde \glie}

\newcommand{\Utgeneric}[2]{\Ugeneric{#1}{#2}}
\newcommand{\Utc}[1][\tglie]{\Utgeneric{\! \! \; \KK}{#1}}
\newcommand{\Uth}[1][\tglie,\psi]{\Utgeneric{\! \! \: h}{#1}}

\newcommand{\vpitgeneric}[2]{\varpi_{#2}(#1)}
\newcommand{\vpitc}[1][\glie]{\vpitgeneric{#1}{\KK}}
\newcommand{\vpith}[1][\glie,\psi]{\vpitgeneric{#1}{h}}

\newcommand{\Xdat}{\mathfrak X}
\newcommand{\Xdom}{{X_{\! \! \: \mathrm{dom}}}}

\newcommand{\Aone}{\mathfrak A_1}

\newcommand{\Unpgeneric}[2]{{U^{> 0}_{\! \! \; #1}(#2)}}
\newcommand{\Unpc}[1][\glie]{\Unpgeneric{\KK}{#1}}

\newcommand{\Unngeneric}[2]{{U^{< 0}_{\! \! \; #1}(#2)}}
\newcommand{\Unnc}[1][\glie]{\Unngeneric{\KK}{#1}}

\newcommand{\UCageneric}[2]{{U^0_{\! \! \; #1}(#2)}}
\newcommand{\UCac}[1][\glie]{\UCageneric{\KK}{#1}}

\newcommand{\Deltageneric}[2]{{\Delta^i_{#2}(#1)}}
\newcommand{\Deltac}[1][\glie]{\Deltageneric{#1}{\! \: \KK}}
\newcommand{\Deltah}[1][\glie,\psi]{\Deltageneric{#1}{h}}

\newcommand{\wt}{\mathrm{wt}}

\newcommand{\Ocatgeneric}[2]{{\mathcal O_{#1}#2}}
\newcommand{\Ocatc}[1][(\glie)]{\Ocatgeneric{\KK}{#1}}

\newcommand{\Chgeneric}[3]{{\mathrm{Ch}_{#1} (#2 #3)}}
\newcommand{\Ch}[2][\glie,]{\Chgeneric{}{#1}{#2}}
\newcommand{\Chc}[1]{\Chgeneric{\: \! \KK}{#1}{}}

\newcommand{\Ointgeneric}[2]{{\mathcal O^{\! \: \mathrm{int}}_{#1}#2}}
\newcommand{\Ointc}[1][(\glie)]{\Ointgeneric{\KK}{#1}}
\newcommand{\Ointh}[1][(\glie,\psi)]{\Ointgeneric{\! \! \: h}{#1}}

\newcommand{\Lgeneric}[2]{{L_{#2}(#1)}}
\newcommand{\Lc}[1]{\Lgeneric{#1}{\KK}}
\newcommand{\Lh}[1]{\Lgeneric{#1}{h}}

\newcommand{\Cinth}[1]{{\mathcal C^{\mathrm{int}}_{\! \! \; h} (#1)}}

\newcommand{\cch}{[| h |]}
\newcommand{\Kh}{\KK \cch}

\newcommand{\Endgeneric}[1]{{\mathrm{End}_{#1}}}
\newcommand{\Endc}{\Endgeneric{\: \! \KK}}
\newcommand{\Endh}{\Endgeneric{\: \! \Kh}}
\newcommand{\EndR}{\Endgeneric{\! \! \; R}}

\newcommand{\hzero}[1]{#1_{\vert h=0}}

\newcommand{\Cat}{\! \mathrm{\it{Cat}}}
\newcommand{\Fun}{\mathrm{\it{Fun}}}
\newcommand{\FF}{\mathbb F}
\newcommand{\Shape}{\mathscr S}

\newcommand{\Alggeneric}[1]{{\mathscr A_{#1}}}
\newcommand{\Algc}{\Alggeneric{\KK}}

\newcommand{\AlgR}{\Alggeneric{\! \! \: R}}

\newcommand{\XX}{\mathbb D}
\renewcommand{\AA}{\mathbb A}

\newcommand{\pAlggeneric}[2]{{\dot{\mathscr A}_{#1}(#2)}}

\newcommand{\pAlgh}[1][\Xdat,\psi]{\pAlggeneric{h}{#1}}
\newcommand{\pAlgR}[1][\Xdat,\psi]{\pAlggeneric{\! \! \: R}{#1}}

\newcommand{\sltF}{\mathfrak{sl}_{2,F}}
\newcommand{\FUgeneric}[2]{{U_{#1}(#2)}}
\newcommand{\FUc}[1][{\glie_{\! \! \; F}}]{\FUgeneric{\! \! \; \KK}{#1}}
\newcommand{\FUh}[1][{\glie_{\! \! \; F}}]{\FUgeneric{\! \! \: h}{#1}}

\newcommand{\BB}{\mathbb B}
\newcommand{\BBv}{B}
\newcommand{\BBe}[1][]{E^{#1}}

\newcommand{\pigeneric}[2]{{\pi_{#2}(#1)}}
\newcommand{\pic}[1][\glie]{\pigeneric{#1}{\KK}}
\newcommand{\pih}[1][\glie,\psi]{\pigeneric{#1}{\! \! \; h}}

\newcommand{\Ikergeneric}[2]{I_{#2}(#1)}

\newcommand{\Ikerh}[1][\glie,\psi]{\Ikergeneric{#1}{h}}

\newcommand{\psicl}{\psi_{\mathrm{cl}}}
\newcommand{\psiclI}{{\psi^I_{\! \! \: \mathrm{cl}}}}
\newcommand{\psiqI}{{\psi^I_{\! \! \; q}}}

\newcommand{\glieF}{{\glie_{\! \! \; F}}}
\newcommand{\FDeltageneric}[2]{{\Delta^i_{#2}(#1)}}
\newcommand{\FDeltac}[1][\glieF]{\FDeltageneric{#1}{\! \: \KK}}
\newcommand{\FDeltah}[1][\glieF]{\FDeltageneric{#1}{h}}

\newcommand{\II}{\mathbf I}
\newcommand{\FUUgeneric}[2]{{\mathbb U_{#1}(#2)}}
\newcommand{\FUUc}[1][\glieF]{\FUUgeneric{\! \! \; \KK}{#1}}

\newcommand{\UUgeneric}[2]{{\mathbb U_{#1}(#2)}}
\newcommand{\UUc}[1][\glie]{\UUgeneric{\! \! \; \KK}{#1}}
\newcommand{\UUh}[1][\glie,\psi]{\UUgeneric{\! \! \: h}{#1}}

\newcommand{\pitgeneric}[2]{\tilde \pi_{#2}(#1)}

\newcommand{\pith}[1][\glie,\psi]{\pitgeneric{#1}{h}}

\newcommand{\Pigeneric}[2]{{\Pi_{#2}(#1)}}
\newcommand{\Pic}[1][\glie]{\Pigeneric{#1}{\KK}}
\newcommand{\Pih}[1][\glie,\psi]{\Pigeneric{#1}{\! \! \; h}}

\newcommand{\Lph}[1]{{L'_{\! \! \; h}(#1)}}
\newcommand{\Uph}[1]{{U'_{\! \! \; h}(#1)}}

\newcommand{\psiq}{{\psi_{\! \! \; q}}}

\newcommand{\Tinf}[1][\Kh]{\Theta(#1)}

\newcommand{\Sinf}[1][\Kh]{\mathscr E(#1)}
\newcommand{\Sinfh}{\mathscr E_h}

\newcommand{\GQEeqgeneric}[3]{{\mathscr E^{[#2]}_{#3}(#1)}}

\newcommand{\GQEeqh}[2][\psi_1,\psi_2]{\GQEeqgeneric{#1}{#2}{h}}

\newcommand{\Coloh}[1][]{{\mathrm{Col} ^{#1}_h \! \! \; (\BB)}}
\newcommand{\Colo}[2][]{{\mathrm{Col} ^{#1}_{\! \! \: #2} \! \! \; (\BB)}}

\newcommand{\vsp}{\vspace*{0.1cm}}

\newcommand{\dnh}{\psi}

\newcommand{\CF}[1][\! \! \; R]{{\mathcal C_{\! \! \: R}(\glieF)}}
\newcommand{\Cint}[2][\! R]{{\mathcal C^{\mathrm{int}}_{#1}(#2)}}
\newcommand{\Cinthat}[2][\! R]{{\widehat{\mathcal C}^{\mathrm{int}}_{#1}(#2)}}

\newcommand{\Uhat}[2][\! R]{{\widehat U_{#1}(#2)}}

\newcommand{\UF}[1][\! R]{{U_{#1}(\glieF)}}
\newcommand{\UFslt}[1][\! R]{{U_{#1}(\sltF)}}
\newcommand{\DeltaF}[1][\! \! \; R]{\FDeltageneric{\glieF}{#1}}
\newcommand{\pihat}[2][\! \! \: R]{{\widehat \pi_{#1}(#2)}}

\newcommand{\VR}{{V_{\! \! \: R}}}

\newcommand{\Vh}{{V_{\! \! \; h}}}
\newcommand{\Vhp}{{V'_{\! \! \: h}}}

\newcommand{\Deltahat}[1][\! \! \; R]{{\widehat \Delta^i_{#1}(\glie,\psi)}}
\newcommand{\UR}[2][\! R]{{U_{#1} (#2)}}

\newcommand{\Oint}[2][\! \! \! \; R]{{\mathcal O^{\! \: \mathrm{int}}_{#1}(#2)}}
\newcommand{\AR}{{A_{\! \! \; R}}}

\newcommand{\xiV}{{{}^\xi \! \! \; V}}

\newcommand{\xiVR}{{{}^\xi \! \! \; \VR}}

\newcommand{\gliepF}{{\glie'_{\! \! \: F}}}
\newcommand{\LglieF}{{\Lglie_{\! \! \; F}}}
\newcommand{\UFp}[1][\! R]{{U_{#1}(\gliepF)}}
\newcommand{\UFL}[1][\! R]{{U_{#1}(\LglieF)}}
\newcommand{\xipsi}{{{}^\xi \! \! \; \psi_i(d)}}
\newcommand{\DeltaFp}[1][\! \! \; R]{\FDeltageneric{\gliepF}{#1}}
\newcommand{\Vq}{{V_{\! \! \; q}}}
\newcommand{\Vqt}{{V_{\! \! \; q,t}}}

\newcommand{\Vqpm}{{V_{\! \! \: q^{\pm 1}}}}
\newcommand{\LVqpm}{{{}^L \! \! \; V_{\! \! \; q^{\pm 1}}}}

\newcommand{\Ointq}[1][\glie]{{\mathcal O^{\mathrm{int}}_{\! q}(#1)}}
\newcommand{\LC}{{{}^L \! \! \; C}}
\newcommand{\LXdat}{{{}^L \! \! \; \Xdat}}
\newcommand{\LX}{{{}^L \! X}}

\newcommand{\Lchi}{{{}^L \! \! \; \chi}}
\newcommand{\Lalpha}{{{}^L \! \! \: \alpha_i}}
\newcommand{\Lcalpha}{{{}^L \! \! \: \check \alpha_i}}
\newcommand{\Llambda}{{{}^L \! \lambda}}

\newcommand{\Lq}{{L_{q}}}

\newcommand{\omegaF}[1][\! \! \; R]{{\omega_{#1}(\glieF)}}
\newcommand{\omegahat}[1][\! \! \; R]{{\widehat \omega_{#1}(\glie,\psi)}}

\newcommand{\eqhat}[3][\! \! \: R]{{\widehat{\mathscr E}^{\! \; [#3]}_{#1}(#2)}}

\newcommand{\Lrep}[2][\! \! \: R]{{L_{#1}(#2)}}
\newcommand{\Uck}[2][\! R]{{\widehat U_{#1}(#2)}}
\newcommand{\UUck}[2][\! R]{{\widehat{\mathbb U}_{#1}(#2)}}

\newcommand{\UB}[2][\! R]{{U_{#1}(#2)}}

\newcommand{\FUU}[2][\! R]{{\mathbb U_{#1}(#2)}}

\newcommand{\CC}{\mathbb C}

\newcommand{\Aq}{\mathcal{A}}

\newcommand{\Urh}[1]{U_{#1}(\mathfrak{sl}_2)}
\newcommand{\Uhhp}{{U_{\! \! \: h,h'}(\mathfrak{sl}_2,g)}}
\newcommand{\Ughp}{{U_{\! gh'}(\slt)}}

\newcommand{\cmod}{\mathcal{C}(\mathfrak{sl}_2)}
\newcommand{\cmodhh}{\mathcal{C}^{\: \! h,h'}(\mathfrak{sl}_2,g)}
\newcommand{\cmodh}{\mathcal{C}^{\: \! h}(\mathfrak{sl}_2)}

\newcommand{\grot}{\text{Rep} (\mathfrak{sl}_2)}
\newcommand{\groth}{\text{Rep}^{h} (\mathfrak{sl}_2)}

\newcommand{\grothh}{\text{Rep}^{h,h'} (\mathfrak{sl}_2,g)}

\newcommand{\xxn}[4]{\left( Q_{#4} \, T^{\: \! {\{ H_{#1}^{#2} #3 \}}_{\! \! \: Q_{#4}}} \right)^{H_{#1}^{#2} #3} - \ \left( Q_{#4} \, T^{\: \! {\{ H_{#1}^{#2} #3 \}}_{\! \! \: Q_{#4}}} \right)^{-H_{#1}^{#2}}}
\newcommand{\xxd}[4]{Q_{#4} \, T^{\: \! {\{ H_{#1}^{#2} #3 \}}_{\! \! \: Q_{#4}}} \ - \ Q_{#4}^{-1} \, T^{\: \! - {\{ H_{#1}^{#2} #3 \}}_{\! \! \: Q_{#4}}}}
\newcommand{\xxf}[4]{\frac{\xxn{#1}{#2}{#3}{#4}}{\xxd{#1}{#2}{#3}{#4}}}

\newcommand{\xxp}[2]{\prod_{e = 1, -1} \xxf{e}{#1}{#2}{}}

\title{Algèbres de Kac-Moody colorées, algèbres enveloppantes quantiques généralisées et interpolation de Langlands}
\author{Alexandre Bouayad}

\begin{document}
\parindent = 0pt

\thispagestyle{empty}
\newgeometry{left=2cm,right=2cm}

\begin{center}
  {\Large \bf \textsc{Université Paris VII - Denis Diderot}} \\
  \vspace*{0.5cm}

  {\Large \bf \'Ecole doctorale Paris Centre} \\

  \vspace*{\fill}

  {\Huge \textsc{Thèse de doctorat}} \\ [3ex]
  {\LARGE {Discipline : Mathématiques}} \\
  \vspace*{1.5cm}

  {\large présentée par} \\
  \vspace*{0.5cm}
  {\LARGE {\bf Alexandre \textsc{Bouayad}}} \\
  \vspace*{1.5cm}

  \hrule
  \vspace*{0.3cm}
  {\LARGE {\bf Algèbres enveloppantes quantiques généralisées,}} \\
  \vspace*{0.25cm}
  {\LARGE {\bf Algèbres de Kac-Moody colorées et}} \\
  \vspace*{0.25cm}
  {\LARGE {\bf Interpolation de Langlands}} \\
  \vspace*{0.3cm}
  \hrule
  \vspace*{1cm}

  {\large \bf Generalised Quantum Enveloping Algebras, Coloured Kac-Moody Algebras and Langlands Interpolation \\ }
  \vspace*{1cm}

  {\large Dirigée par David \textsc HERNANDEZ}

  \vspace*{2cm}

  {\large Soutenue le 8 juillet 2013 devant le jury composé de :}

  \vspace*{0.5cm}

  {\Large
    \begin{tabular}{lll}
      M. David \textsc{Hernandez}     & Université Paris Denis Diderot & Directeur   \\
      M. Kevin \textsc{McGerty}       & University of Oxford           & Rapporteur  \\
      M. Marc \textsc{Rosso}          & Université Paris Denis Diderot & Examinateur \\
      M. Olivier \textsc{Schiffmann}  & Université Paris-Sud 11        & Rapporteur  \\
      Mme. Michela \textsc{Varagnolo} & Université de Cergy Pontoise   & Examinateur
    \end{tabular}
  }

\end{center}
\vspace*{\fill}

\clearpage
\restoregeometry
\pagenumbering{roman} \setcounter{page}{1}
\thispagestyle{empty}
\vspace*{\fill}

\begin{tabular}{ccl}
  \includegraphics[scale=0.45]{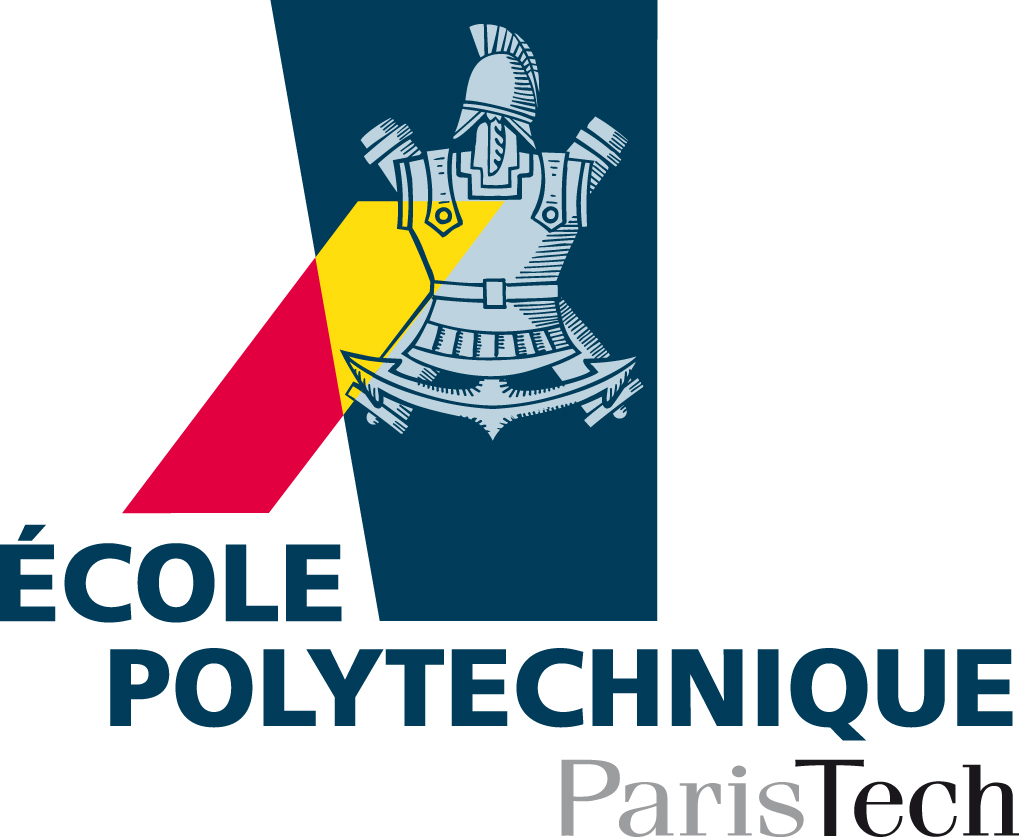}       & \hspace*{0.15cm} &
  \begin{minipage}[t]{(\textwidth)/2-0.15cm}
    Centre de Mathématiques \\
    Laurent Schwartz \\ \\
    Ecole Polytechnique \\
    91 128 Palaiseau cedex
  \end{minipage}
  \\ \\ \\ \\ \\ \\ \\

  \includegraphics{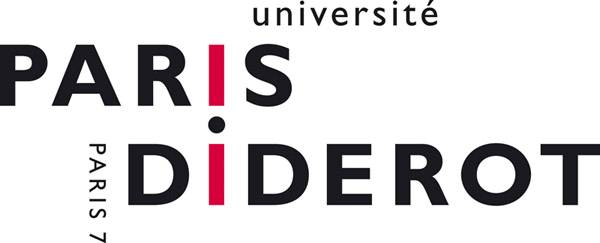}                  &                  &
  \begin{minipage}[t]{(\textwidth)/2-0.15cm}
    Institut de Mathématiques \\
    de Jussieu - Paris Rive Gauche \\ \\
    5 rue Thomas Mann \\
    75 205 Paris cedex 13
  \end{minipage}
  \\ \\ \\ \\ \\ \\ \\

  \includegraphics[scale=0.6]{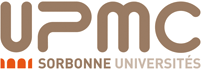} &                  &
  \begin{minipage}[t]{(\textwidth)/2-0.15cm}
    \'Ecole doctorale Paris Centre \\ \\
    4 place Jussieu \\
    75 252 Paris cedex 05
  \end{minipage}
\end{tabular}

\vspace*{\fill}

\clearpage
\thispagestyle{empty}
\begin{dedicace}
  À mes parents.
\end{dedicace}

\newpage
\thispagestyle{empty}
\mbox{}
\newpage
\clearpage
\pagenumbering{roman} \setcounter{page}{1}
\part*{Remerciements}
\righthyphenmin=62
\lefthyphenmin=62

Mes premiers remerciements sont naturellement destinés à mon directeur de thèse David Hernandez. Je lui suis très reconnaissant de m'avoir proposé un sujet d'étude dont les motivations sont aussi variées que profondes. Durant la difficile et parfois sombre traversée, sur laquelle ce manuscrit rescapé tentera d'apporter quelques lumières, travailler sous la direction de David Hernandez se révéla être une excellente formation à la recherche. Son expérience, ses conseils, sa disponibilité et son attention ont été une aide précieuse. \\

Kevin McGerty et Olivier Schiffmann ont accepté d'être les rapporteurs de cette thèse. Je souhaite les remercier vivement pour le temps qu'ils ont généreusement consacré à cette tâche. Je les remercie également pour leur enthousiasme et je suis heureux qu'ils aient tous deux accepté de se rendre à Paris à l'occasion de la soutenance. J'aimerais de nouveau remercier ici Kevin McGerty pour son invitation à l'université d'Oxford, ainsi que pour les nombreuses et passionnantes discussions que nous avons partagées alors. \\

Je suis très honoré que Marc Rosso et Michela Varagnolo fassent partie de mon jury de thèse. Je les remercie d'avoir accepté si volontiers l'invitation. Je remercie par ailleurs Marc Rosso pour ses encouragements toujours bienveillants, mais aussi pour son aide si déterminante dont j'ai eu la chance de bénéficier à plusieurs moments importants de ma thèse. Je remercie Michela Varagnolo pour son invitation à venir présenter mes recherches dans le GDR ``Théorie de Lie Algébrique et Géométrique''. \\

Je suis très reconnaissant à Dennis Gaitsgory, qui malgré une activité déjà débordante, trouva le temps de patiemment m'écouter. Je le remercie tout autant de m'avoir fait découvrir avec un enthousiasme si communicatif certains des sujets qui le passionnent. Je le remercie enfin pour son accueil lors de mon séjour à l'université Harvard, et pour son soutien qu'il m'apporta à de si nombreuses reprises. \\

Je remercie Edward Frenkel pour les discussions passionnantes que j'ai eues l'immense chance de partager avec lui. Je le remercie également pour son soutien qui m'a été précieux. Je remercie Vyjayanthi Chari pour son invitation en Inde et son accueil chaleureux en Californie. Je remercie Pavel Etingof, Kobi Kremnitzer, Peter Littelmann, Hiraku Nakajima et Raphaël Rouquier pour leur générosité, le soutien et l'attention qu'ils m'ont accordés. J'ai par ailleurs une pensée émue pour Jean-Louis Loday dont je fis la connaissance lors d'une école d'été à Pékin. Je débutais alors ma thèse, ses encouragements et sa joie des mathématiques ont eu sur moi une influence certaine. \\

Des spécialistes avec qui j'ai eu la chance de discuter pendant la préparation de cette thèse m'ont aussi beaucoup apporté, je pense notamment à Benjamin Enriquez, Jacob Greenstein, Bernhard Keller, Bruno Valette et Éric Vasserot. À cet égard, je tiens à remercier tout particulièrement Olivier Mathieu. \\

Mes remerciements vont au Centre de Mathématiques de Laurent Schwartz à l'École Polytechnique, ainsi qu'à l'Institut de Mathématiques de Jussieu, qui, durant la préparation de cette thèse, me proposèrent un cadre de travail idéal. À ce propos, je remercie tout autant le département de mathématiques de l'université Harvard, ainsi que la Fondation Sciences Mathématiques de Paris. \\

Je veux avoir ici une pensée toute particulière pour nombre de (post)doctorants que j'ai eu la chance de rencontrer durant ces dernières années. Ils ont ensemble fourni la preuve constante d'un fait qui m'apporta tant de réconfort : il est vraisemblablement possible d'être autant sympathique que talentueux. Parmi tous ces héros que j'ai eu l'immense privilège de côtoyer, certains ont accepté de partager avec moi davantage de leur temps : le petit briquet dont le calme olympien en désarma (énerva) plus d'un, Loulou qui fit de son anticonformisme une raison de vivre (ou de rire), Drago\c{s} dont la soif de logique se parait des couleurs sanguines d'illustres ancêtres, Paloma qui fit de l'indécision un art, Elena qui remarqua très tôt mes aptitudes en anglais, Jacquie Chan dont la gloire n'est plus à présenter, Daniel qui reçut à maintes reprises le prix du mathématicien le plus cool, Lukas qui ne me haït point, Johan dont le rire était aussi fréquent que les fautes d'orthographe, Yohan le séducteur séduit, Claire dont le secret est encore à découvrir, Sarah qui court si vite. \\

Je remercie de tout mon coeur Marie et Baihui, dont la douceur si généreuse fit sur moi des miracles. \\

Je remercie tous mes vieux amis qui ont ces dernières années, de nouveau, observer inlassablement le mouvement tumultueux de mes humeurs, réjouis par mon optimisme conquérant et concernés par mes doutes abyssaux. \\

J'aimerais ici remercier mes parents. C'est sans nul doute à eux que je dois le plus. Plutôt que d'énoncer de nouveaux remerciements, je préfère maintenant brandir cette vérité que seul les initiés sont capables d'étayer : \emph{il n'existe pas parents plus grandioses que maman et papa Bouayad}. Enfin, je remercie mon frère Aurélien, pour qui j'ai une immense admiration. Je n'oublie bien sûr pas Maurice.

\clearpage
\vspace*{\fill}
\part*{Résumé}
\righthyphenmin=62
\lefthyphenmin=62
Nous proposons dans cette thèse un nouveau processus de déformation des algèbres de Kac-Moody (K-M) et de leurs représentations. La direction de déformation est indiquée par une collection de nombres, appelée \emph{coloriage}. Les nombres naturels mènent par exemple aux algèbres classiques, tandis que les nombres quantiques mènent aux algèbres quantiques. \\

Nous établissons dans un premier temps des conditions nécessaires et suffisantes sur les coloriages, de telle sorte que le processus dépend polynômialement d'un paramètre formel et fournit les \emph{algèbres enveloppantes quantiques généralisées (algèbres GQE)}. Nous levons par la suite les restrictions et montrons que le processus existe toujours via les \emph{algèbres de Kac-Moody colorées}. \\

Nous formulons la \emph{conjecture GQE}, qui prévoit que toute représentation dans la catégorie $\mathcal O^{\mathrm{int}}$ d'une algèbre de K-M peut être déformée en une représentation d'une algèbre GQE associée. Nous donnons divers exemples pour lesquels la conjecture est vérifiée et réalisons une première étape vers sa résolution en prouvant que les algèbres de K-M sans relations de Serre peuvent être déformées en des algèbres GQE sans relations de Serre. \\

Admettant la conjecture GQE, nous établissons un résultat analogue dans le cas des algèbres de K-M colorées, nous prouvons que les théories des représentations déformées sont parallèles à la théorie classique, nous explicitons une présentation de Serre déformée pour les algèbres GQE, nous prouvons que ces dernières sont les représentants d'une classe naturelle de déformations formelles des algèbres de K-M et sont $\hlie$-triviales en type fini. En guise d'application, nous expliquons en termes d'interpolation les dualités de Langlands classique et quantique entre représentations d'algèbres de Lie et nous proposons une nouvelle approche afin de résoudre une conjecture de Frenkel-Hernandez. \\

En général, nous prouvons qu'il est possible d'interpoler les représentations de deux algèbres \hbox{de K-M} colorées isogéniques par les représentations d'une troisième. En observant que la conjecture GQE est vérifiée dans le cas des algèbres quantiques standards, nous donnons une nouvelle preuve de la dualité de Langlands classique mentionnée précédemment (les premières preuves sont dues à Littelmann et McGerty).



\vspace*{\fill}

\clearpage
\vspace*{\fill}
\part*{Abstract}
\begin{otherlanguage}{english}
  \righthyphenmin=62
\lefthyphenmin=62
We propose in this thesis a new deformation process of Kac-Moody (K-M) algebras and their representations. The direction of deformation is given by a collection of numbers, called a \emph{colouring}. The natural numbers lead for example to the classical algebras, while the quantum numbers lead to the associated quantum algebras. \\

We first establish sufficient and necessary conditions on colourings to allow the process depend polynomially on a formal parameter and to provide the \emph{generalised quantum enveloping (GQE) algebras}. We then lift the restrictions and show that the process still exists via the \emph{coloured Kac-Moody algebras}. \\

We formulate the \emph{GQE conjecture} which predicts that every representation in the category $\mathcal O^{\mathrm{int}}$ of a K-M algebra can be deformed into a representation of an associated GQE algebra. We give various evidences for this conjecture and make a first step towards its resolution by proving that Kac-Moody algebras without Serre relations can be deformed into GQE algebras without Serre relations. \\

In case the conjecture holds, we establish an analog result for coloured K-M algebras, we prove that the deformed representation theories are parallel to the classical one, we explicit a deformed Serre presentation for GQE algebras, we prove that the latter are the representatives of a natural class of formal deformations of K-M algebras and are $\hlie$-trivial in finite type. As an application, we explain in terms of interpolation both classical and quantum Langlands dualities between representations of Lie algebras, and we propose a new approach which aims at proving a conjecture of Frenkel-Hernandez. \\

In general, we prove that representations of two isogenic coloured K-M algebras can be interpolated by representations of a third one. Observing that standard quantum algebras satisfy the GQE conjecture, we give a new proof of the previously mentioned classical Langlands duality (the first proofs are due to Littelmann and McGerty).

\end{otherlanguage}
\vspace*{\fill}


\clearpage
\renewcommand{\contentsname}{Table des matières -- Contents}
\tableofcontents

\newpage
\thispagestyle{empty}
\mbox{}
\newpage
\clearpage
\pagenumbering{arabic} \setcounter{page}{1}
\addcontentsline{toc}{part}{Introduction}
\part*{Introduction}
\righthyphenmin=62
\lefthyphenmin=62

Cette thèse est divisée en trois grande partie. La première partie traite des \emph{algèbres enveloppantes quantiques généralisées}. La deuxième partie est consacrée aux \emph{algèbres de Kac-Moody colorées} et à l'\emph{interpolation de Langlands}. Quant à la troisième partie, elle présente une construction explicite d'un cas particulier; elle correspond à un article \cite{Bou1} publié en 2012 dans la revue \emph{International Mathematics Research Notices} et a inspiré quelques-unes des idées à partir desquelles le cas général a pu être développé. \\

Cette introduction donne en premier lieu un rapide rappel du contexte dans lequel s'inscrivent les travaux de cette thèse, on y présente entre autres la \emph{dualité de Langlands entre représentations d'algèbres de Lie}. Le \emph{principe d'interpolation de Langlands} est introduit dans un second temps. On présente par la suite les principales constructions et idées qui composent la théorie des algèbres de Kac-Moody colorées et des algèbres enveloppantes quantiques généralisées. Les énoncés des principaux théorèmes de cette thèse sont regroupés dans une section consacrée à cet usage. On discute enfin des différentes ouvertures que laissent entrevoir les travaux présentés ici. \\

Cette introduction peut aussi être utilisée comme un guide à travers les deux premières parties de cette thèse. Sont en effet expliquées ici les principales idées sur lesquelles s'appuient les constructions qui composent chacune des deux parties.

\section{Contexte}
\subsection{Algèbres de Lie}
Les \emph{algèbres de Lie} sont l'approximation linéaire des groupes de Lie, objets de nature à la fois algébrique et géométrique. Cette simplification fournit par suite la possibilité de définir un remplacement des groupes de Lie en des termes purement algébriques. Il est dans le même temps remarquable que les algèbres de Lie retiennent beaucoup des informations des objets a priori plus riches dont elles sont issues. \\

Parmi les algèbres de Lie, on peut distinguer la classe des algèbres de Lie \emph{simples} (et \emph{semi-simples}). Les algèbres de Lie (complexes) semi-simples de dimension finie sont classifiées par certaines matrices définies positives, appelées les \emph{matrices de Cartan}. La \emph{présentation de Serre} indique de quelle manière les matrices de Cartan encodent ces algèbres de Lie, en distinguant à l'intérieur de celles-ci générateurs et relations (voir \cite{Ser}).

\subsection{Algèbres de Kac-Moody}
La présentation de Serre des algèbres de Lie semi-simples de dimension finie admet des généralisations naturelles. Les \emph{algèbres de Kac-Moody} (voir \cite{Kac, Moo}) sont par exemple définies à partir d'une généralisation des matrices de Cartan, pour lesquelles la condition de positivité est levée. Les algèbres de Kac-Moody apparaissent alors comme un analogue de dimension infinie des algèbres de Lie semi-simples de dimension finie. \\

Soit $\glie$ une algèbre de Kac-Moody (complexe), on note $C= {(C_{ij})}_{i,j \in I}$ sa matrice de Cartan généralisée. On supposera toujours que $\glie$ est \emph{symétrisable}, i.e. qu'il existe une famille ${(d_i)}_{i \in I}$ d'entiers strictement positifs tels que ${d_i \: \! C_{ij}} = {d_j \: \! C_{ji}}$ pour tout $i,j \in I$. On note $d$ le ppcm des entiers $d_i$ ($i \in I$). \\

La présentation de Serre distingue à l'intérieur de l'algèbre de Kac-Moody $\glie$ une sous-algèbre commutative maximale $\hlie$, appelée la \emph{sous-algèbre de Cartan} \hbox{de $\glie$}, ainsi que certains éléments $X_i^\pm$ ($i \in I$), appelés les \emph{générateurs de Chevalley} \hbox{de $\glie$}. Sont de surcroît distingués des éléments $\alpha_i \in \hlie$ \hbox{et $\check \alpha_i \in \hlied$ ($i \in I$)}, appelés les \emph{racines simples} et les \emph{coracines simples} de $\glie$, respectivement. \\

L'algèbre $\glie$ est engendrée par l'espace vectoriel $\hlie$ et par les générateurs de Chevalley. On donne ci-dessous les relations prescrites par la présentation de Serre.

\begin{equation} \label{eq_KM_nondeformableintro} \left\{ {\renewcommand{\arraystretch}{1.3} \begin{array}{lr}
\big[ \check \mu, \check \mu' \big] = 0 & (\check \mu, \check \mu' \in \hlie) \, , \\
\big[ {\check \mu}, X^{\pm}_i \big] = \pm \alpha_i(\check \mu) \, X_i^{\pm} & (i \in I, \ \check \mu \in \hlie) \, , \\
\big[ X^+_i, X^-_j \big] = 0 & \quad (i,j \in I, \ i \neq j) \, ,
\end{array}} \right. \end{equation}
\begin{equation} \label{eq_KM1_intro}
\big[ X^+_i, X^-_i \big] = \check \alpha_i \quad \quad \quad (i \in I) \, ,
\end{equation}
\begin{equation} \label{eq_KM2_intro}
\sum_{k+k' = 1 - C_{ij}} {(-1)}^k \, {1 - C_{ij} \choose k} \, {(X_i^{\pm})}^k \; \! X^{\pm}_j \; \!  {(X_i^{\pm})}^{k'} = 0 \quad \quad \quad (i,j \in I, \ i \neq j) \, .
\end{equation}

Les relations \eqref{eq_KM2_intro} sont appelées les \emph{relations de Serre}.

\subsection{Algèbres quantiques}
La forme concise de la présentation de Serre autorise d'autres généralisations. Les \emph{groupes quantiques} sont par exemple définies en déformant les relations de la présentation de Serre. Ces déformations dépendent d'un paramètre, désigné habituellement par $q$ ou par $h$ (ces deux choix vérifient la relation $q = \exp (h)$). Le paramètre $h$ est directement emprunté à la physique quantique et fait référence à la constante de Planck. \\

Les groupes quantiques ont été introduits par Drinfeld \cite{Dri,DrP} et \hbox{Jimbo \cite{Jim}}, ils sont considérés pour différentes raisons comme des analogues quantiques des \emph{algèbres universelles enveloppantes} des algèbres de Kac-Moody. On note $\Uq(\glie)$ l'algèbre quantique associée à l'algèbre de Kac-Moody $\glie$, quand $U(\glie)$ désigne l'algèbre enveloppante universelle de $\glie$. \\

Les générateurs de Chevalley de $\glie$ et la sous-algèbre de Cartan $\hlie$, dont la structure n'est par ailleurs pas déformée dans l'algèbre quantique, constituent comme pour $\glie$ des éléments distinctifs de l'algèbre $\Uq(\glie)$; ils sont à nouveau désignés par $X_i^\pm$ ($i \in I$) et $\hlie$, respectivement.

\subsection{Dualité de Langlands entre sous-algèbres de Cartan}
L'\emph{algèbre de Kac-Moody Langlands duale} de $\glie$, que l'on note $\Lglie$, est l'algèbre de Kac-Moody (symétrisable) dont la matrice de Cartan est la transposée de la matrice de Cartan de $\glie$. On note $\Lhlie$ la sous-algèbre de Cartan de $\Lglie$. Les racines et coracines de $\Lglie$ sont quant à elle désignées par $\Lalpha$ et $\Lcalpha$ ($i \in I$), respectivement. \\

Il existe pour les sous-algèbres de Cartan $\hlie$ et $\Lhlie$ une correspondance naturelle. Les deux algèbres de Lie (abéliennes) $\hlie$ \hbox{et $\Lhlie$} ont en effet la même dimension et il est possible de construire entre les deux un isomorphisme qui induisent les correspondances suivantes entre racines et coracines :
\begin{equation} \label{eq_CartanLanglands}
{\renewcommand{\arraystretch}{1.3}
\begin{array}{ccc}
\hlie & \longleftrightarrow & \Lhlie \\
\check \alpha_i & \leftrightarrow & \frac{d}{d_i} \; \! \Lcalpha \, ,
\end{array}
\quad \quad \quad \quad \quad
\begin{array}{ccc}
\hlied & \longleftrightarrow & \Lhlied \\
\frac{d}{d_i} \; \! \alpha_i & \leftrightarrow & \Lalpha \, .
\end{array}
}
\end{equation}

\subsection{Dualité de Langlands entre représentations d'algèbres de Lie}
La correspondance \eqref{eq_CartanLanglands} entre les deux sous-algèbres de Cartan Langlands duales permet d'établir une correspondance entre les représentations linéaires de la première et les représentations linéaires de la seconde. On peut alors se poser la question suivante : est-il possible d'étendre cette correspondance entre les représentations de $\glie$ et les représentations de $\Lglie$ ? \\

On se restreint ici aux \emph{représentations intégrables} de $\glie$ et $\Lglie$. À cela, trois raisons a priori peuvent être données. Premièrement, dans le cas où $\glie$ est de dimension finie, ces représentations peuvent être, comme leur nom l'indique, intégrées en des représentations des groupes de Lie $G$ et ${{}^L \! \; \! G}$, dont $\glie$ et $\Lglie$ sont les approximations linéaires respectives. Deuxièmement, les actions des sous-algèbres de Cartan sont faciles à décrire : elles sont diagonalisables et les valeurs propres associées aux actions des coracines simples sont \hbox{dans $\ZZ$}. Enfin, il existe suffisamment de représentations intégrables pour obtenir par reconstruction Tannakienne toute l'algèbre enveloppante universelle $U(\glie)$ de $\glie$. \\

Une représentation intégrable $V$ admet donc la décomposition suivante, où l'action de $\hlie$ sur un espace $V_{\lambda}$ est scalaire et donnée par la forme \hbox{linéaire $\lambda$}:
\begin{equation} \label{eq_decompV}
V \ = \ \bigoplus_{\lambda \in \hliedZ} V_\lambda \, , \quad \quad \text{avec} \quad \hliedZ \ := \ \Big \{ \lambda \in \hlied \, ; \ \lambda(\check \alpha_i) \in \ZZ \, , \ \forall \, i \in I  \Big\} \, .
\end{equation}
Une forme linéaire $\lambda \in \hlied$ telle que $V_\lambda$ est non nul est appelée un \hbox{\emph{poids} de $V$}. Le \emph{caractère} de la représentation $V$ résume l'action de $\hlie$ sur $V$, en retenant pour chaque $\lambda \in \hliedZ$ la dimension du sous-espace $V_\lambda$. \\

La correspondance de Langlands \eqref{eq_CartanLanglands} pour les sous-algèbres de Cartan fournit une action de $\Lhlie$ sur $V$. Toutefois, on remarque que les valeurs propres associées aux actions des coracines ${}^L \! \check \alpha_i$ ($i \in I$) ne sont pas toutes dans $\ZZ$. S'il l'on espère étendre la correspondance entre les représentations de $\hlie$ et $\Lhlie$ à une correspondance entre les représentations intégrables de $\glie$ et $\Lglie$, il apparaît donc nécessaire de considérer la sous-représentation $\LV := \bigoplus_{\lambda \in \LhliedZ} V_\lambda$ de $\hlie$. \\

On peut à présent apporter une réponse au problème précédemment posé. On restreint pour cela une nouvelle fois les catégories de représentations de $\glie$ \hbox{et $\Lglie$} mises en jeu, en imposant quelques restrictions supplémentaires sur l'action des sous-algèbres de Cartan. Les nouvelles catégories considérées sont \hbox{notées $\Oint[]{\glie}$} et $\Oint[]{\Lglie}$, respectivement (dans le cas où $\glie$ est de dimension finie, $\Oint[]{\glie}$ est la catégorie des représentations de dimension finie de $\glie$). Notons que ces catégories ne présentent pas une réduction trop simpliste du problème: elles contiennent par exemple encore suffisamment de représentations pour obtenir par reconstruction Tannakienne les algèbres enveloppantes universelles de $\glie$ \hbox{et $\Lglie$} (voir à ce propos la proposition I.\ref{prop_UBperfect}).

\begin{introthm}[Dualité de Langlands entre les catégories ${\Oint[]{\glie}}$ et ${\Oint[]{\Lglie}}$] \label{introthm_dualite}
Pour toute représentation $V$ de $\glie$ dans $\Oint[]{\glie}$, il existe une représentation de $\Lglie$ dans $\Oint[]{\Lglie}$, dont la restriction à la sous-algèbre de Cartan $\Lhlie$ est la représentation $\LV$.
\end{introthm}

Toute représentation dans $\Oint[]{\Lglie}$ étant déterminée par sa restriction à $\Lhlie$, notons qu'il existe une unique représentation de $\Lglie$ qui vérifie les conditions du précédent théorème. \\

Dans le cas où $\glie$ est de dimension finie, Frenkel-Hernandez \cite{FH1} ont prouvé qu'il existe une action virtuelle de $\Lglie$ sur $\LV$ qui étend l'action de $\Lhlie$ (i.e. il existe un élément dans le groupe de Grothendieck de $\Oint[]{\Lglie}$ dont la restriction à $\Lhlie$ est la classe de $\LV$ dans le groupe de Grothendieck des représentations de $\Lhlie$). \\

Le théorème \ref{introthm_dualite} a été établi par Littelmann \cite{Lit1,Lit2} en toute généralité, ainsi que par McGerty \cite{McG} dans le cas où le diagramme de Coxeter associé à $\glie$ ne contient aucun cycle de longueur impaire.

\subsection{Dualité de Langlands entre représentations quantiques}
La sous-algèbre de Cartan $\hlie$ n'est pas déformée dans l'algèbre quantique associé à $\glie$. On peut alors définir de la même manière que dans le cas classique une notion de représentation intégrable de $\Uq(\glie)$, ainsi qu'une catégorie $\Oint[\! q]{\glie}$ de représentations de $\Uq(\glie)$. Dans le cas où $q$ n'est pas une racine de l'unité (on fera toujours cette hypothèse), la théorie des représentations de $\Uq(\glie)$ dans la catégorie $\Oint[\! q]{\glie}$ s'avère être parallèle à la théorie des représentations de $\glie$ dans la catégorie $\Oint[]{\glie}$ (voir \cite{LuR, Ros} entres autres). Plus précisément, on sait que toute représentation de l'algèbre quantique $\Uq(\glie)$ dans $\Oint[\! q]{\glie}$ se spécialise \hbox{à ${q \! = \! 1}$} en une représentation de $\glie$ dans $\Oint[]{\glie}$ et, réciproquement, que toute représentation \hbox{de $\glie$} dans $\Oint[]{\glie}$ admet une unique déformation dans $\Oint[\! q]{\glie}$. \\

Rappelons que la dualité de Langland \eqref{eq_CartanLanglands} entre les sous-algèbres de Cartan $\hlie$ et $\Lhlie$ induit une dualité de Langlands entre les représentations de ces algèbres. Le problème d'étendre la dualité de Langlands entre les représentations de ces algèbres à une dualité de Langlands entre les représentations intégrables des groupes quantiques $\Uq(\glie)$ et $\Uq(\Lglie)$ se pose à nouveau. Au vu de la grande similitude entre les catégories $\Oint[]{\glie}$ et $\Oint[\! q]{\glie}$, le problème posé dans le cadre des catégories $\Oint[\! q]{\glie}$ et $\Oint[\! q]{\Lglie}$ apparaît équivalent au problème classique. \\

Soit $\Vq$ une représentation intégrable de $\Uq(\glie)$, on définit à partir de la restriction de $\Vq$ à $\hlie$, une sous-représentation $\LVq$ de $\Lhlie$, de la même manière que dans le cas classique. Le théorème \ref{introthm_dualite} implique son analogue quantique.

\begin{introthm} [Dualité de Langlands entre les catégories ${\Oint[\! q]{\glie}}$ et ${\Oint[\! q]{\Lglie}}$] \label{introthm_dualiteq}
Pour toute représentation $\Vq$ de $\Uq(\glie)$ dans $\Oint[\! q]{\glie}$, il existe une représentation de $\Uq(\Lglie)$ dans $\Oint[\! q]{\Lglie}$, dont la restriction à la sous-algèbre de Cartan $\Lhlie$ est la représentation $\LVq$.
\end{introthm}

De la même manière que dans le cas classique, toute représentation dans la catégorie $\Oint[\! q]{\Lglie}$ étant déterminée par sa restriction à $\Lhlie$, notons qu'il existe une unique représentation de $\Uq(\Lglie)$ qui vérifie les conditions du précédent théorème.

\section{Interpolation de Langlands}
\subsection{Principe d'interpolation de Langlands classique}
La dualité de Langlands entre les catégories $\Oint[]{\glie}$ et $\Oint[]{\Lglie}$, établie dans  le théorème \ref{introthm_dualite}, est un résultat d'existence : étant donné une représentation $V$ \hbox{de $\glie$} dans $\Oint[]{\glie}$, il existe \hbox{sur $\LV$} une action de l'algèbre $\Lglie$ qui étend l'action de la sous-algèbre de Cartan $\Lhlie$. On connaît explicitement l'action de $\Lhlie$ sur $\LV$, toutefois le théorème ne fournit aucune description explicite quant à l'action du reste de l'algèbre $\Lglie$ (le théorème étant un résultat de dualité entre $\glie$ et $\Lglie$, nous entendons par description explicite une expression de l'action de $\Lglie$ \hbox{sur $\LV$} obtenue à partir de l'action de $\glie$ \hbox{sur $V$}). \\

On désigne dans la suite par $\LXpm$ ($i \in I$) les générateurs de Chevalley de l'algèbre de Kac-Moody Langlands duale $\Lglie$. Au vu des relations \hbox{définissant $\glie$}, on remarque que le générateur $X_i^\pm$ envoie le sous-espace $V_{\lambda}$ \hbox{de $V$} \hbox{dans $V_{\lambda \pm \alpha_i}$}. La forme de la correspondance \eqref{eq_CartanLanglands} implique alors que l'élément ${(X_i^{\pm})}^{d/d_i}$ envoie le sous-espace $V_{\lambda}$ \hbox{dans $V_{\lambda \pm \Lalpha}$} et, en particulier, que $\LV$ est stable sous l'action de ${(X_i^{\pm})}^{d/d_i}$ (notons que $\LV$ n'est  pas stable sous l'action de l'algèbre $\glie$ toute entière). Il apparaît donc naturel de se demander si l'action de $\LXpm$ \hbox{sur $\LV$} peut être décrite à partir de l'action de ${(X_i^{\pm})}^{d/d_i}$ sur $V$. \\

Répondre par l'affirmative à cette question, et ce de manière générale, est un problème difficile, si ce n'est impossible : les catégories $\Oint[]{\glie}$ et $\Oint[]{\Lglie}$ permettant d'obtenir les algèbres enveloppantes universelles $U(\Lglie$) et $U(\Lglie)$ par reconstruction Tannakienne, il serait en effet peu ou prou équivalent d'exhiber un morphisme d'algèbre injectif \hbox{de $U(\Lglie)$} dans $U(\glie)$ (dont la restriction à $\Lhlie$ serait de surcroît déterminée par la dualité de Langlands \eqref{eq_CartanLanglands} entre $\hlie$ et $\Lhlie$). \\

On propose dans cette thèse de considérer l'énoncé (plus faible) suivant.
\begin{introprinc}[Principe d'interpolation de Langlands classique] \label{introprinc_inter} $\phantom{}$ \\
Pour toute représentation $V$ de $\glie$ dans $\Oint[]{\glie}$, il existe une déformation à un paramètre de l'action \hbox{sur $V$} des générateurs de Chevalley $X_i^\pm$ telle que la spécialisation à une première valeur dudit paramètre restaure l'action initiale des générateurs $X^\pm_i$ sur $V$ et telle que la spécialisation à une seconde valeur permette d'obtenir les actions des générateurs $\LXpm$ sur $\LV$ à partir de l'action déformée des éléments ${(X_i^{\pm})}^{d/d_i}$ sur $V$.
\end{introprinc}

On obtient ainsi un lien entre les actions de $\glie$ sur $V$ et de $\Lglie$ sur $\LV$, non via la description de cette dernière à partir de la première, mais plutôt en les termes d'une interpolation, les deux actions apparaissant comme autant de spécialisations d'une seule et même action, unificatrice. On donne ici à ce principe le nom d'\emph{interpolation de Langlands classique}. Notons qu'il est par ailleurs sous-entendu que l'action de $\hlie$ sur $V$ n'est pas déformée : on rappelle que l'action de la sous-algèbre $\Lhlie$ sur $\LV$ est obtenue à partir de l'action de $\hlie$ sur $V$, via la dualité de Langlands \eqref{eq_CartanLanglands} entre les deux sous-algèbres de Cartan. \\

C'est ce principe d'interpolation dont ont fait implicitement usage Littelmann et McGerty dans leur preuves du théorème \ref{introthm_dualite} : à la représentation de $V$ de $\glie$, ils ont associé sa déformation quantique $\Vq$, représentation du groupe \hbox{quantique $\Uq(\glie)$}, et ont prouvé que la spécialisation à ${q \! = \! \epsilon}$ de $\Vq$ (où $\varepsilon$ désigne une racine primitive de l'unité d'ordre ${2d}$) permet de reconstruire l'action de $\LXpm$ sur $\LV$ à partir de l'action de ${(X_i^{\pm})}^{d/d_i}$ sur $V$, la spécialisation à ${q \! = \! 1}$ permettant quant à elle de retrouver l'action de $X_i^\pm$ sur $V$ à partir de l'action de $X^\pm_i$ sur $\Vq$.

\subsection{Principe d'interpolation de Langlands quantique}
Le principe d'interpolation de Langlands classique est associé à la dualité de Langlands entre les catégories $\Oint[]{\glie}$ et $\Oint[]{\Lglie}$. Qu'en est-il dans le cas quantique ? \\

Inspirés par un projet entrepris par \hbox{Frenkel-Reshetikhin \cite{FR3}}, Frenkel-Hernandez ont introduit, dans le cas où $\glie$ est simple et de dimension finie, une déformation à deux paramètres $\Uqt(\tglie)$ de l'algèbre enveloppante universelle de $\tilde \glie$ (l'algèbre de Lie vérifiant les mêmes relations que $\glie$, hormis les relations de Serre). Ils ont conjecturé que les représentations de $\Uqt(\tglie)$ interpolent entre certaines représentations des groupes quantiques Langlands duaux $\Uq(\glie)$ et $\Ut (\Lglie)$ (on rappelle que $\varepsilon$ désigne une racine primitive de l'unité d'ordre ${2d}$).

\begin{introconj} \emph{\cite[conjecture 5.1]{FH1}} \label{introconj_FH}
On suppose ici que $\glie$ est simple et de dimension finie. Pour toute représentation de dimension finie irréductible $\Vq$ de $\Uq (\glie)$ dont le plus haut poids est dans $\LhliedZ$, il existe une représentation de dimension finie $\Vqt$ de l'algèbre $\Uqt(\tglie)$, telle que la spécialisation à ${t \! = \! 1}$ de l'action de $\Uqt(\tglie)$ sur $\Vqt$ restaure l'action de $\Uq (\glie)$ sur $\Vq$ et telle que la spécialisation à ${q \! = \! \varepsilon}$ permette d'obtenir une action de $\Ut (\Lglie)$ sur $\LVq$.
\end{introconj}

Nous proposons, à partir de la précédente conjecture, de définir l'analogue quantique du principe d'interpolation de Langlands classique.

\begin{introprinc}[Principe d'interpolation de Langlands quantique] \label{introprinc_interq} $\phantom{}$ \\
Pour toute représentation $\Vq$ de $\Uq (\glie)$ dans $\Oint[\! q]{\glie}$, il existe une déformation à un paramètre de l'action \hbox{sur $\Vq$} des générateurs de Chevalley $X_i^\pm$ telle que la spécialisation à une première valeur dudit paramètre restaure l'action initiale des générateurs $X^\pm_i$ sur $\Vq$ et telle que la spécialisation à une seconde valeur permette d'obtenir les actions des générateurs $\LXpm$ sur $\LVq$ à partir de l'action déformée des éléments ${(X_i^{\pm})}^{d/d_i}$ sur $\Vq$.
\end{introprinc}

\section{Déformation colorée de la catégorie ${\Oint[]{\glie}}$}
Le programme de Frenkel-Hernandez évoqué précédemment a été le point de départ de cette thèse, qui propose dans cette optique d'aborder d'une manière générale les principes d'interpolation de Langlands classique et quantique (on rappelle que le principe classique a été implicitement illustré par Littelmann et Mcgerty, tandis qu'une réalisation du principe quantique a été conjecturée sous une forme particulière par Frenkel-Hernandez).


\subsection{Le cristal global de rang 1}
On sait que pour chaque $i \in I$, les trois générateurs $X^-_i$, $\check \alpha_i$ \hbox{et $X^+_i$} de l'algèbre de Kac-Moody $\glie$ engendrent une sous-algèbre isomorphe à la plus simple des algèbres de Kac-Moody, i.e. l'algèbre de Lie $\slt$ (l'algèbre des matrices carrées de taille $2$ et de trace nulle). Il est par ailleurs connu que pour chaque $i \in I$, une représentation intégrable de $\glie$, considérée comme représentation de $\slt$ par restriction de l'action de $\glie$ aux générateurs $X^\pm_i$ et $\check \alpha_i$, peut être décomposée en une somme directe (possiblement infinie) de représentations irréductibles de dimension finie. \\

Les représentations irréductibles de dimension finie de $\slt$ sont explicitement connues. L'action des générateurs de Chevalley de $\slt$ est naturellement donnée par une certaine représentation (de dimension un en chacun des sommets), dite \emph{classique} et \hbox{notée $\psicl$}, du carquois défini ci-dessous.

\begin{introdefi}
Le cristal global de rang 1, noté $\BB$, est le carquois infini suivant.
\begin{equation*}
\shorthandoff{;:!?} \xymatrix @!0 @C=3pc @R=1.1pc {
&&&&& \quad \ \! \; \bullet \! \; \scriptstyle{b_{n,0}} \ar@<4pt>[dd] && \\
&&&&&& \quad \ \ \iddots & \\
&&& \quad \ \! \; \bullet \! \; \scriptstyle{b_{3,0}} \ar@<4pt>[dd] && \quad \ \! \; \bullet \! \; \scriptstyle{b_{n,1}} \ar@<4pt>[uu] \ar@<4pt>[dd] && \\
&& \quad \ \! \; \bullet \! \; \scriptstyle{b_{2,0}} \ar@<4pt>[dd] && \iddots &&& \\
& \quad \ \! \; \bullet \! \; \scriptstyle{b_{1,0}} \ar@<4pt>[dd] && \quad \ \! \; \bullet \! \; \scriptstyle{b_{3,1}} \ar@<4pt>[dd] \ar@<4pt>[uu] && \quad \ \! \; \bullet \! \; \scriptstyle{b_{n,2}} \ar@<4pt>[uu] && \\
\quad \ \! \; \bullet \! \; \scriptstyle{b_{0,0}} && \quad \ \! \; \bullet \! \; \scriptstyle{b_{2,1}} \ar@<4pt>[dd] \ar@<4pt>[uu] && \cdots & \vdots & \quad \ \ \cdots & \\
& \quad \ \! \; \bullet \! \; \scriptstyle{b_{1,1}} \ar@<4pt>[uu] && \quad \ \! \; \bullet \! \; \scriptstyle{b_{3,2}} \ar@<4pt>[dd] \ar@<4pt>[uu] && \quad \quad \ \, \bullet \! \; \scriptstyle{b_{n,n-2}} \ar@<4pt>[dd] && \\
&& \quad \ \! \; \bullet \! \; \scriptstyle{b_{2,2}} \ar@<4pt>[uu] && \ddots &&& \\
&&& \quad \ \! \; \bullet \! \; \scriptstyle{b_{3,3}} \ar@<4pt>[uu] && \quad \quad \ \, \bullet \! \; \scriptstyle{b_{n,n-1}} \ar@<4pt>[uu] \ar@<4pt>[dd] & \\
&&&&&& \quad \ \ \ddots & \\
&&&&& \quad \ \, \! \; \bullet \! \; \scriptstyle{b_{n,n}} \ar@<4pt>[uu] && \\
}
\end{equation*}
\end{introdefi}

Le carquois $\BB$ est dessiné dans le réseau $\ZZ^2$ et a pour support l'ensemble des combinaisons linéaires, à coefficients dans $\NN$, des vecteurs $(1,1)$ et $(1,-1)$, de telle sorte que le sommet $b_{n,p}$ a pour coordonnées $(n,n-2p)$. Pour chaque $n \in \NN$, l'action des générateurs de Chevalley sur l'unique représentation irréductible de dimension $n+1$ de $\slt$ est déterminée par la restriction de la représentation $\psicl$ à la composante connexe d'abscisse $n$ du carquois $\BB$, l'action de la sous-algèbre de Cartan étant quant à elle déterminée par les ordonnées des sommets. \\

Rappelons à présent que l'algèbre quantique $\Uq(\glie)$ associé à $\glie$ est construit en déformant la présentation de Serre de $\glie$. Distinguons, parmi l'ensemble des relations qui composent la présentation de Serre, le groupe de relations \eqref{eq_KM_nondeformableintro}, non déformées dans $\Uq(\glie)$, qu'on appelle alors les \emph{relations non déformables} de $\glie$. On sait par ailleurs qu'il existe une déformation ${\psi_{\! \! \: q^{d_i}}}$ de la représentation $\psicl$ telle que pour chaque $i \in I$, l'action du générateur de Chevalley $X_i^\pm$, sur une représentation intégrable $\Vq$ de $\Uq(\glie)$, est déterminée par la déformation ${\psi_{\! \! \: q^{d_i}}}$.

\subsection{Coloriages}
On se propose de considérer les principes d'interpolation de Langlands classique et quantique à la lumière de l'exemple offert par les groupes quantiques (on rappelle que c'est via ces derniers que le principe classique a été une première fois approché par Littelmann et McGerty). \\

Considérons pour chaque $i \in I$, une déformation $\psi_i$ de la représentation classique du carquois $\BB$. Par déformation, on entend ici que les représentations $\psi_i$ du carquois $\BB$ sont définies au-dessus de l'anneau $R := \CC[u]$ et que les dimensions de $\psicl$ sont conservées.

\begin{introdefi}
\begin{itemize}
\item[\textbullet] Un coloriage de $\BB$ à valeurs dans $R$ est une représentation du carquois $\BB$ définie au-dessus de $R$ et libre de \hbox{rang 1} sur chaque sommet.
\item[\textbullet] Un ${I \! \! \;}$-coloriage de $\BB$ est une famille $\psi = {(\psi_i)}_{i \in I}$ de coloriages \hbox{de $\BB$}.
\end{itemize}
\end{introdefi}

Pour chaque $i \in I$, la représentation $\psi_i$ détermine naturellement une déformation de l'action du générateur $X^\pm_i$ sur les représentations irréductibles de dimension finie de $\slt$. Apportons quelques précisions. \\

La donnée de la \hbox{représentation $\psi_i$} du carquois $\BB$ est équivalente à la donnée d'une famille de scalaires \hbox{$\psi_i^\pm(n,k)$ dans $R$}, où $n \in \NNI$ désigne l'abscisse d'une composante connexe du carquois $\BB$ et où l'entier $k$ (compris entre $1$ et $n$) désigne la $k\up{ème}$ flèche au sein de cette $n\up{ième}$ composante, selon la numérotation fixée par la convention suivante.
$$ \shorthandoff{;:!?} \xymatrix @!0 @R=1.1pc {
\quad \ \, \! \; \bullet \! \; \scriptstyle{b_{n,0}} \ar@<4pt>[dd]^{\scriptstyle \psi^-(n,1)} \\
\\
\quad \ \, \! \; \bullet \! \; \scriptstyle{b_{n,1}} \ar@<4pt>[uu]^{\scriptstyle \psi^+(n,n)} \ar@<4pt>[dd]^{\scriptstyle \psi^-(n,2)}\\
 \\
\quad \ \, \! \; \bullet \! \; \scriptstyle{b_{n,2}} \ar@<4pt>[uu]^{\scriptstyle \psi^+(n,n-1)} \\
\vdots \\
\quad \quad \ \, \bullet \! \; \scriptstyle{b_{n,n-2}} \ar@<4pt>[dd]^{\scriptstyle \psi^-(n,n-1)} \\
\\
\quad \quad \ \, \bullet \! \; \scriptstyle{b_{n,n-1}} \ar@<4pt>[uu]^{\scriptstyle \psi^+(n,2)} \ar@<4pt>[dd]^{\scriptstyle \psi^-(n,n)} \\
\\
\quad \ \, \! \; \bullet \! \; \scriptstyle{b_{n,n}} \ar@<4pt>[uu]^{\scriptstyle \psi^+(n,1)}
} $$

À titre d'exemple, la représentation classique $\psicl$ du carquois $\BB$ est caractérisée par les nombres entiers ${\psi^\pm_{\! \! \; \mathrm{cl}}}(n,k) = k$, tandis que sa déformation quantique $\psiq$ est caractérisée par les \emph{nombres quantiques} ${\psi^\pm_{\! \! \; q}}(n,k) = {[k]}_q := \frac{q^k - q^{-k}}{q-q^{-1}}$ (regardés comme déformations quantiques des nombres entiers). \\

Les déformations induites par la représentation $\psi_i$ sont décrites dans la définition suivante.

\begin{introdefi}
On note $\Lrep{n,\psi_i}$ le $R$-module $\bigoplus_{0 \leq p \leq n} {R \; \! b_{n,p}}$, muni d'actions de $X^\pm_i$ et $\check \alpha_i$, déterminées par $\psi_i$ de la manière suivante.
$$ \left\{ {\renewcommand{\arraystretch}{0.75} \begin{array}{rcl}
\check \alpha_i . \! \; b_{n,p} & := & \ (n- 2p) \, b_{n,p} \, , \\
&& \\
X_i^- . \! \; b_{n,p} & := & \left\{ {\renewcommand{\arraystretch}{1.1} \begin{array}{cl}
\psi_i^-(n,p+1) \, b_{n,p+1} & \quad \text{si $p \neq n$,} \\
0 & \quad \text{sinon,}
\end{array}} \right. \\
&& \\
X_i^+ . \! \; b_{n,p} & := &  \left\{ {\renewcommand{\arraystretch}{1.1} \begin{array}{cl}
\psi_i^+(n,n-p+1) \, b_{n,p-1} & \quad \text{si $p \neq 0$,} \\
0 & \quad \text{sinon.}
\end{array}} \right.
\end{array}} \right.
$$
\end{introdefi}

\subsection{Représentations $\psi$-intégrables}
Soit maintenant $V$ une représentation de $\glie$ dans la catégorie $\Oint[]{\glie}$. On se propose de déformer, suivant les directions fixées par les représentations $\psi_i$, l'action des générateurs $X_i^\pm$ sur l'espace $\VR := {V \! \otimes_\CC \! \CC[u]}$, de telle sorte que les conditions suivantes, dites de $\psi$-intégrabilité, soient satisfaites. Notons que, conformément aux principes d'interpolation de Langlands, on demande que l'action de la sous-algèbre de Cartan $\hlie$ ne soit quant à elle pas déformée \hbox{sur $\VR$}.

\begin{introdefi}
On appelle représentation \hbox{$\psi$-intégrable} un \hbox{$R$-module} $V$ muni d'une action des générateurs de Chevalley $X_i^\pm$ et d'une action diagonalisable de la sous-algèbre de Cartan $\hlie$, telles que les conditions suivantes sont satisfaites.
\begin{enumerate}
\item[a)] Les relations non déformables de $\glie$ sont vérifiées sur $V$.
\item[b)] Pour chaque $i \in I$, les actions de $X_i^\pm$ et $\check \alpha_i$ sont déterminées par la \hbox{représentation $\psi_i$} de $\BB$, de telle sorte que $V$ est isomorphe à une somme directe de représentations $\Lrep{n,\psi_i}$.
\end{enumerate}
\end{introdefi}

Admettant dans un premier temps la possibilité qu'une telle déformation existe, examinons l'action déformée de ${(X_i^{\pm})}^{d/d_i}$ sur le sous-espace $\LVR := { {}^L \! \! \: V} \otimes_\CC \CC[u]$. \\

Fixons $i \in I$ et considérons le carquois dont les sommets sont ceux d'ordonnées divisibles par $d/d_i$ dans le carquois $\BB$ et dont les flèches sont obtenues comme composées de $d/d_i$ flèches pointant dans une même direction. Remarquons que ce nouveau carquois peut être identifié avec la réunion de $d/d_i$ copies de $\BB$ (la figure suivante illustre ce fait).

\vsp
\vsp
\vsp

\begin{equation} \label{Louloufig}
\begin{tikzpicture}[scale = 0.6]
\begin{scope}[scale=2]
    \foreach \i in {0,...,7}
        {
            \foreach \j in {0,...,\i}
                {
                    \filldraw[blue] (\i/2, -\i/4 + \j/2 ) circle (1pt);
                }
        }
\end{scope}
\begin{scope}[scale= 2, xshift= 0.5cm]
    \foreach \i in {0,...,6}
        {
            \foreach \j in {0,...,\i}
                {
                    \filldraw[red] (\i/2, -\i/4 + \j/2 ) circle (1pt);
                }
        }
\end{scope}
\begin{scope}
    \foreach \i in {0,...,14}
        {
            \foreach \j in {0,...,\i}
                {
                    \draw[black] (\i/2, -\i/4 + \j/2 ) circle (2pt);
                }
        }
\end{scope}
\node at (3.75, -4.5 ) {$\scriptstyle d/d_i \, = \,  2$};
\begin{scope}[xshift= 10cm]
    \begin{scope}[scale=1]
        \foreach \i in {0,...,4}
            {
                \foreach \j in {0,...,\i}
                    {
                        \filldraw[blue] (3*\i/2, -3*\i/4 + 3*\j/2 ) circle (2pt);
                    }
            }
    \end{scope}
    \begin{scope}[scale= 1, xshift= 1cm]
        \foreach \i in {0,...,4}
            {
                \foreach \j in {0,...,\i}
                    {
                        \filldraw[green] (3*\i/2, -3*\i/4 + 3*\j/2 ) circle (2pt);
                    }
            }
    \end{scope}
    \begin{scope}[scale= 1, xshift= 2cm]
        \foreach \i in {0,...,3}
            {
                \foreach \j in {0,...,\i}
                    {
                        \filldraw[red] (3*\i/2, -3*\i/4 + 3*\j/2 ) circle (2pt);
                    }
            }
    \end{scope}
    \begin{scope}
        \foreach \i in {0,...,14}
            {
                \foreach \j in {0,...,\i}
                    {
                        \draw[black] (\i/2, -\i/4 + \j/2 ) circle (2pt);
                    }
            }
    \end{scope}
    \node at (3.75, -4.5 ) {$\scriptstyle d/d_i \, = \, 3$};
\end{scope}

\end{tikzpicture}
\end{equation}

\vsp
\vsp

La représentation $\psi_i$ induit alors $d/d_i$ représentations $\Lpsi(l)$ ($1 \leq l \leq d/d_i$) du carquois $\BB$, de telle sorte que l'action de l'élément ${(X_i^{\pm})}^{d/d_i}$ sur $\LVR$ est, au vu de la condition b) de $\psi$-intégrabilité, déterminée par l'ensemble des \hbox{représentations $\Lpsi(l)$}. \\

La condition a) implique par ailleurs que les relations non déformables \hbox{de $\Lglie$} sont vérifiées sur $\LVR$ par les éléments ${(X_i^{\pm})}^{d/d_i}$ et la sous-algèbre de \hbox{Cartan $\Lhlie$}. Les autres relations prescrites par la présentation de Serre de $\Lglie$ sont quant à elles satisfaites dans le cas où les représentations $\Lpsi(l)$ du carquois $\BB$ sont toutes égales à la représentation classique $\psicl$ : les relations \eqref{eq_KM1_intro} sont vérifiées car l'action de ${(X_i^{\pm})}^{d/d_i}$ sur $\LVR$ est déterminée par $\psicl$ et les relations de Serre sont vérifiées car l'action de $\Lhlie$ et ${(X_i^{\pm})}^{d/d_i}$ \hbox{sur $\LVR$} est intégrable. De la même manière, les relations du groupes quantiques $\Uq(\Lglie)$ sont vérifiées \hbox{sur $\LVR$} par les éléments ${(X_i^{\pm})}^{d/d_i}$ et la sous-algèbre de \hbox{Cartan $\Lhlie$}, dans le cas où pour chaque $i \in I$, les représentations $\Lpsi(l)$ du carquois $\BB$ sont toutes égales à la représentation quantique ${\psi_{\! q^{d \! \: \! / \! \: \! d_i}}}$. \\

La réalisation du principe d'interpolation de Langlands classique (ou quantique) se réduit donc à choisir des \hbox{déformations $\psi_i$} de $\psicl$, telles que pour chaque $i \in I$, la spécialisation de $\psi_i$ à une première valeur de $u$ soit égale à $\psicl$ (ou ${\psi_{\! \! \: q^{d_i}}}$) et telles que pour chaque $i \in I$, les spécialisations de $\Lpsi(l)$ à une seconde valeur de $u$ soient égales à $\psicl$ (ou ${\psi_{\! q^{d \! \: \! / \! \: \! d_i}}}$). Un tel choix ne présentant pas de difficulté, la question seule, et plus générale, de l'existence d'une déformation de l'action des générateurs de Chevalley sur $V$, suivant les directions imposées par les représentations $\psi_i$ du carquois $\BB$, est ainsi posée.

\begin{intropb}[Déformation colorée de la catégorie ${\Oint[]{\glie}}$] $\phantom{}$ \\
Est-il possible, pour chaque représentation $V$ dans $\Oint[]{\glie}$, d'étendre l'action de la sous-algèbre de Cartan \hbox{$\hlie$} sur $V$ en une action \hbox{$\psi$-intégrable} sur ${V \! \otimes_\KK \! R}$ ?
\end{intropb}

Il est nécessaire de sensiblement généraliser le problème. Plus précisément, l'algèbre de Kac-Moody $\glie$ est dorénavant construite au-dessus d'un corps $\KK$ quelconque, que l'on suppose de \textbf{caractéristique nulle}, et les \hbox{représentations $\psi_i$} du carquois $\BB$ sont désormais définies au-dessus d'une \hbox{$\KK$-algèbre}, que l'on suppose \textbf{intègre} et que l'on désigne à nouveau par $R$.

\section{Algèbres de Kac-Moody colorées}
Les représentations d'une algèbre peuvent être employées dans le but de mieux comprendre l'algèbre elle-même. Nous proposons dans cette thèse de considérer la stratégie inverse, dans le but de prouver l'existence d'une déformation colorée de la catégorie $\Ointc$. Plus précisément, nous proposons de collecter au sein d'une même algèbre les relations satisfaites par les générateurs de $\glie$ sur tous les \hbox{espaces $V_{\! R}$}, avec $V$ dans $\Oint[]{\glie}$, pour lesquels l'action des générateurs de Chevalley a pu être déformée dans les directions imposées par les \hbox{coloriages $\psi_i$}.

\subsection{Définition des algèbres de Kac-Moody colorées}
Considérons la $R$-algèbre $\AR(\glie,\psi)$ des endomorphismes du foncteur d'oubli de la catégorie des représentations $\psi$-intégrables dans la catégorie des $R$-modules. Il existe sur $\AR(\glie,\psi)$ une topologie naturelle, définie comme étant la plus petite topologie rendant continue chaque \hbox{application ${\AR(\glie,\psi) \! \times \! V \! \to \! V}$} associée à une représentation $\psi$-intégrable $V$.


\begin{introdefi}
On note $\Uhat{\glie,\psi}$ la sous-algèbre de $\AR(\glie,\psi)$ topologiquement engendrée par les générateurs de Chevalley $X_i^\pm$ et la sous-algèbre de \hbox{Cartan $\hlie$}.
\end{introdefi}

Autrement dit, la sous-algèbre $\Uhat{\glie,\psi}$ est composée des éléments dans $\AR(\glie,\psi)$ dont l'action sur chacun des vecteurs des représentations $\psi$-intégrables coïncide avec l'action d'une combinaison linéaire finie de monômes en $X^\pm_i$ et $\check \mu \in \hlie$. \\

L'algèbre $\Uhat{\glie,\psi}$ est regardée comme une déformation de l'algèbre enveloppante universelle $\Uc$, ou plus exactement, comme une déformation de sa \emph{complétion Tannakienne} $\Uck[\KK]{\glie}$ relativement à la catégorie des représentations intégrables.

\begin{defi}
Si le $I$-coloriage $\psi$ est admissible, i.e. si ses valeurs $\psi^\pm_i (n,k)$ sont inversibles dans $R$, l'algèbre $\Uhat{\glie,\psi}$ est appelée une algèbre de Kac-Moody colorée.
\end{defi}

\subsection{Structure des algèbres de Kac-Moody colorées}
L'algèbre $\Uhat{\glie,\psi}$ est une algèbre associative. Il est possible de munir celle-ci, et ce de manière naturel au vu des définitions, d'une structure algébrique plus riche. La nature-même des représentations \hbox{$\psi$-intégrables} implique en effet le résultat suivant.

\begin{introprop}
Il existe un morphisme d'algèbre \hbox{de $\Uck{\slt,\psi_i}$} dans $\Uhat{\glie,\psi}$, qui aux générateurs $X^\pm$ et $\check \alpha$ de la première, fait correspondre les \hbox{générateurs $X_i^\pm$} et $\check \alpha_i$ de la seconde.
\end{introprop}

Considérons alors le diagramme $\II$ obtenu, à partir de l'ensemble de \hbox{sommets $I$}, par l'ajout d'un sommet $t_I$ et d'une flèche reliant $i$ à $t_I$ pour \hbox{chaque $i \in I$}. On entend ici par \emph{diagramme d'algèbre (de forme $\II$)}, un foncteur de $\II$ dans la catégorie des $R$-algèbres.

\begin{introdefi}
Le diagramme de Kac-Moody coloré $\UUck{\glie,\psi}$ est le diagramme d'algèbre formé par les morphismes évoqués dans la précédente proposition.
\end{introdefi}

Il est également possible de définir un diagramme d'algèbre $\UUc$ à partir de l'algèbre $\Uc$ (on rappelle en effet que pour chaque $i \in I$, les \hbox{générateurs $X_i^\pm$} \hbox{et $\check \alpha_i$} engendrent dans $\glie$ une algèbre isomorphe à $\slt$). Le diagramme de Kac-Moody coloré $\UUck{\glie,\psi}$ est regardé comme une déformation de $\UUc$. \\

Rappelons par ailleurs que les relations non déformables de $\glie$ désignent le groupe de relations \eqref{eq_KM_nondeformableintro} de la présentation de Serre de $\glie$.

\begin{introdefi}
L'algèbre de Lie ${\glie_{\! \! \; F}}$ est l'algèbre de Lie engendrée par les mêmes générateurs que $\glie$ et sujette aux seules relations non déformables.
\end{introdefi}

La notation $F$ fait référence au terme anglais ``\emph{deformation-Free}''. \\

De nouveau, on associe naturellement à l'algèbre $\UF$ un diagramme $\FUU{{\glie_{\! \! \; F}}}$. Les relations non déformables de $\glie$ étant par définition satisfaites dans $\Uhat{\glie,\psi}$, remarquons alors qu'il existe un morphisme de diagramme naturel de $\FUU{{\glie_{\! \! \; F}}}$ dans le diagramme de Kac-Moody coloré $\UUck{\glie,\psi}$. \\

Ne sont considérées dans cette thèse que les déformations de $\Uc$ conservant les relations non déformables de $\glie$, et par ailleurs compatibles avec la structure de diagramme de $\UUc$.

\begin{introdefi}
Un diagramme ${\glie_{\! \! \; F}}$-pointé est un diagramme d'algèbre $\AA$, munie d'un épimorphisme de $\FUU{{\glie_{\! \! \; F}}}$ dans $\AA$.
\end{introdefi}






\section[Algèbres enveloppantes quantiques généralisées]{Algèbres enveloppantes quantiques \\ généralisées}
L'algèbre $\Uhat{\glie,\psi}$ a été introduite dans le but de prouver l'existence d'une déformation des représentations de la catégorie $\Ointc$, suivant les directions imposées par le \hbox{${I \! \! \;}$-coloriage} $\psi = {(\psi_i)}_{i \in I}$. Afin de rendre effective cette stratégie, il est dans un premier temps nécessaire d'imposer certaines restrictions à $\psi$.

\subsection{Complétion $h$-adique}
À l'anneau des séries formelles, noté $\Kh$, est associé la notion de déformation formelle : un objet algébrique ${X_{\! \! \: h}}$ (par exemple, une algèbre ou une représentation, etc.) est une déformation formelle d'un second objet ${X_{\! \! \: 0}}$ de même nature si la spécialisation à ${h \! = \! 0}$ de ${X_{\! \! \: h}}$ \hbox{est ${X_{\! \! \: 0}}$} et si l'objet ${X_{\! \! \: h}}$, considéré comme un module au-dessus de $\Kh$, admet une écriture de la forme ${X_{\! \! \: 0}} \cch$. Cette notion de déformation a le double avantage d'autoriser des développements à l'infini et d'être dans le même temps calculable  : il est en général possible d'aborder les questions de déformation à l'aide de récurrences menées à partir de la valuation de séries formelles, ce principe étant remarquablement illustré par certaines cohomologies pour ce genre de problème (voir par exemple \cite{Ger,GS}). \\

Notons $\UB{\glie,\psi}$ la sous-algèbre de $\AR(\glie,\psi)$ engendrée par les générateurs de Chevalley $X_i^\pm$ et la sous-algèbre de Cartan $\hlie$. On rappelle que $\Uhat{\glie,\psi}$ est la complétion Tannakienne de l'algèbre $\UB{\glie,\psi}$. On considère ici la situation où chaque coloriage $\psi_i$ dépend d'un paramètre formel $h$; en d'autres termes, on suppose à partir de maintenant que $R = \Kh$. Afin d'exploiter la notion de déformation formelle dans le cadre des algèbres de Kac-Moody colorées, nous proposons de remplacer la complétion Tannakienne de $\UB{\glie,\psi}$ par sa complétion $h$-adique.



\begin{introdefi}
L'algèbre $\Uh[\glie,\psi]$ est la complétion $h$-adique de $\UB{\glie,\psi}$, on note $\UUh[\glie,\psi]$ le diagramme d'algèbre associé.
\end{introdefi}

$$ \shorthandoff{;:!?} \xymatrix @!0 @C=4pc @R=1.4pc {
\Uhat{\glie,\psi} && \Uh[\glie,\psi] \\
\scriptstyle{\text{complétion Tannakienne}} \quad \quad \quad \quad \quad && \, \ \quad \quad \quad \scriptstyle{\text{complétion $h$-adique}} \\
& \UB{\glie,\psi} \ar@{_{(}->}[uul] \ar@{^{(}->}[uur] & \\
} $$




\subsection{Coloriages $h$-admissibles}
L'algèbre $\Uh[\glie,\psi]$ n'est toutefois pas toujours une déformation formelle de l'algèbre enveloppante universelle $\Uc$. La complétion $h$-adique étant moins permissive que la complétion Tannakienne, l'existence d'une déformation dans l'algèbre $\Uh[\glie,\psi]$ des relations de la présentation de Serre est sujette à certaines conditions, imposées au ${I \! \! \;}$-coloriage $\psi$.

\begin{introdefi}
Un $I$-coloriage $\psi = {(\psi_i)}_{i \in I}$ est dit $h$-admissible si les axiomes suivants sont vérifiés.
\begin{IEEEeqnarray*}{lCl}
\bullet \text{ (Axiome de régularité)} & \quad \quad & \text{$\psi^\pm_i(n,k)$ dépend polynômialement de $(n,k)$} \, . \\
\bullet \text{ (Axiome de déformation)} & \quad \quad & \text{$\psi^\pm_i(n,k)$, spécialisé à ${h \! = \! 0}$, est égal a $k$} \, . \\
\bullet \text{ (Axiome du quotient)} & \quad \quad & \psi_i^\pm (n,0) \, = \, 0 \, . \\
\bullet \text{ (Axiome de Verma)} & \quad \quad & \psi_i^\pm (n,-k) \ = \ - \, \psi_i^\mp (-n-2,k) \, .
\end{IEEEeqnarray*}
\end{introdefi}

Le caractère nécessaire des axiomes précédents peut être pressenti ici. Supposons que l'algèbre $\Uh[\glie,\psi]$ soit une déformation formelle de l'algèbre enveloppante universelle $\Uc$. L'axiome de régularité est alors une conséquence du fait que chaque élément de $\Uh[\glie,\psi]$ s'écrit sous la forme d'une combinaison linéaire (infinie et à coefficients dans $\Kh$) de monômes en $X^\pm_i$ et $\check \mu \in \hlie$. L'axiome de déformation doit quant à lui être vérifié afin que la spécialisation de $\Uh[\glie,\psi]$ à ${h \! = \! 0}$ soit égale à $\Uc$. Notons que les représentations de dimension finie de $\slt$ peuvent être obtenues comme les quotients de certaines représentations de dimension infinies, appelées modules de Verma. Cette propriété du quotient, préservée par déformation, implique l'axiome du quotient. Enfin, on sait que certains modules de Verma se réalisent comme sous-représentations d'autres modules de Verma. Cette dernière propriété étant également préservée par déformation, l'axiome de Verma s'en trouve déduit.

\subsection{Définition des algèbres GQE}
La dénomination suivante sera justifiée plus loin.

\begin{introdefi}
Si $\psi$ est un $I$-coloriage $h$-admissible, l'algèbre $\Uh$ est appelée une algèbre enveloppante quantique généralisée.
\end{introdefi}

La dénomination anglaise \emph{Generalised Quantum Enveloping} est symbolisée par le sigle GQE, que l'on utilise à partir de maintenant pour désigner les algèbres enveloppantes quantiques généralisées. \\

Notons qu'il est par ailleurs possible de définir l'algèbre quantique, que l'on note ici $\Uh[\glie]$ et que l'on qualifie de formel, en tant que déformation formelle de l'algèbre $\Uc$ (voir \cite{Dri}, ainsi que \cite[theorem 1]{Tan} pour une preuve détaillée). Les groupes quantiques apparaissent alors comme un exemple parmi les algèbres enveloppantes quantiques généralisées (voir le théorème \ref{introthm_GQEclq}, énoncé un peu plus loin dans cette introduction) : le $I$-coloriage $h$-admissible associé est la famille de représentations quantiques ${({\psi_{\! \! \: q^{d_i}}})}_{i \in I}$ évoquée plus tôt (on rappelle que les paramètres $q$ et $h$ vérifient la relation $q = \exp (h)$). \\

Les algèbres GQE sont l'analogue $h$-adique des algèbres de Kac-Moody colorées. Elles ont été introduites dans le but de répondre au problème de déformation colorée de la catégorie $\Ointc$. Cette stratégie est justifiée et précisée par les deux points suivants.

\vsp

\begin{itemize}
\item[\textbullet] Les algèbres GQE, déformations formelles de l'algèbre $\Uc$, permettent d'aborder le problème de déformation colorée de manière effective.
\item[\textbullet] Il existe beaucoup de coloriages $h$-admissibles du cristal $\BB$, de telle sorte la résolution du problème de déformation colorée dans le cas des $I$-coloriages \hbox{$h$-admissibles} permet de résoudre le même problème pour tous les \hbox{${I \! \! \;}$-coloriages} admissibles.
\end{itemize}

\section{Énoncés des théorèmes}

\subsection{Exemples fondamentaux de réalisations}
Les algèbres de Kac-Moody colorées, ainsi que les algèbres GQE, offrent un cadre unificateur, dans lequel s'inscrivent les algèbres de Kac-Moody classiques, les algèbres quantiques standards (formelles et rationnelles) et les algèbres de Frenkel-Hernandez $\Uqt (\tglie)$ (définies dans le cas où $\glie$ est de dimension finie). \\

Rappelons que l'algèbre quantique formelle $\Uh[\glie]$ est une déformation formelle de $\Uc$ et que l'algèbre quantique rationnelle $\Uq(\glie)$ est quant à elle définie au-dessus du corps $\KK(q)$. Les algèbres de Frenkel-Hernandez $\Uqt (\tglie)$ sont définies au-dessus du corps $\CC(q,t)$. \\

On désigne par $\psiclI$ le ${I \! \! \;}$-coloriage ${(\psicl)}_{i \in I}$ et on désigne par $\psiqI$ le \hbox{${I \! \! \;}$-coloriage} constitué des coloriages quantiques ${\psi_{\! \! \: q^{d_i}}}$.



\begin{introthm}
\label{introthm_GQEclq} $\phantom{}$
\begin{enumerate}
\item[1)] L'algèbre GQE ${\, \Uh[\glie,\psiclI]}$ est égale à la déformation formelle constante de l'algèbre enveloppante universelle $\Uc$.
\item[2)] L'algèbre GQE ${\, \Uh[\glie,\psiqI]}$ est égale à l'algèbre quantique formelle $\Uh[\glie]$.
\item[3)] L'algèbre de Kac-Moody colorée $\Uck[{\! \! \; \KK}]{\glie,\psiclI}$ contient, en tant que sous-algèbre, l'algèbre enveloppante universelle $\Uc$.
\item[4)] L'algèbre de Kac-Moody colorée $\Uck[{\! \! \; \KK(q)}]{\glie,\psiqI}$ contient, en tant que sous-algèbre, l'algèbre quantique rationnelle $\Uq(\glie)$.
\item[5)] Il existe une algèbre de Kac-Moody colorée contenant, en tant que sous-algèbre, un quotient de l'algèbre de Frenkel-Hernandez $\Uqt (\tglie)$.
\end{enumerate}
\end{introthm}

L'algèbre de Frenkel-Hernandez $\Uqt (\tglie)$ est une déformation (dans un sens autre que formel) de l'algèbre de Lie $\glie$. L'énoncé de la conjecture de Frenkel-Hernandez peut alors être reliée à un problème d'existence d'une déformation des relations de Serre dans $\Uqt (\tglie)$. La théorie des algèbres de Kac-Moody colorées apporte une réponse à cette question.

\begin{introcor}
Il existe des relations de Serre déformées, vérifiées sur toutes les représentations intégrables de $\Uqt (\tglie)$.
\end{introcor}





\subsection{Les équations GQE}
Le rang 1 joue un rôle particulier dans la théorie des algèbres de Kac-Moody colorées et des algèbres GQE; il constitue la brique fondamentale à partir de laquelle le cas général est développé (comme c'est le cas pour les algèbres de Kac-Moody classiques et les groupes quantiques). Le rang 1, à la différence du cas général, peut être traité explicitement : plusieurs phénomènes sont en \hbox{rang 1} directement reliés à certaines équations linéaires (de dimension infinie). À chaque paire de coloriages $(\psi_1,\psi_2)$ du \hbox{cristal $\BB$}, correspond une telle équation, que l'on note ${\widehat{\mathscr E}_{\! \! \: R}(\psi_1,\psi_2)}$. \\


La proposition suivante illustre un premier lien entre les algèbres de Kac-Moody de rang 1 et les \hbox{équations ${\widehat{\mathscr E}_{\! \! \: R}}$}.

\begin{introprop}
Soit $\psi$ un coloriage. Les relations de la présentation de Serre en rang 1 admettent une déformation dans $\Uck{\slt,\psi}$ si et seulement si ${\widehat{\mathscr E}_{\! \! \: R}(\psi,\psi)}$ admet une solution.
\end{introprop}

Dans le cas où la complétion Tannakienne est remplacée par la complétion \hbox{$h$-adique}, l'équation ${\widehat{\mathscr E}_{\! \! \: R}(\psi_1,\psi_2)}$ est remplacée par l'\emph{équation GQE} ${\mathscr E_{\! \! \; h}(\psi_1,\psi_2)}$.

\begin{introprop}
Soit $\psi$ un coloriage de $\BB$ à valeurs dans $\Kh$. Les relations de la présentation de Serre en rang 1 admettent une déformation dans $\Uh[\slt,\psi]$ si et seulement si l'équation GQE ${\mathscr E_{\! \! \; h}(\psi,\psi)}$ admet une solution.
\end{introprop}


Formellement, les équations ${\widehat{\mathscr E}_{\! \! \: R}}$ et les équations GQE sont identiques, seul l'espace dans lequel sont recherchées les solutions diffère. Dans le cas Tannakien, l'espace est suffisamment large, de telle sorte qu'une solution existe presque toujours (si $R$ est un corps, une solution existe toujours). Dans le cas $h$-adique, l'espace des solutions est très nettement restreint. Quel est l'analogue $h$-adique des coloriages admissibles ? Est-il possible d'en donner une caractérisation ? Nous avions déjà discuté à ce propos du caractère nécessaire des axiomes d'un coloriage $h$-admissible, le résultat suivant apporte une réponse complète.



\begin{introthm}
\label{introthm_GQEeq}
Soit $\psi$ un coloriage du cristal $\BB$ à valeurs dans $\Kh$. L'équation GQE ${\mathscr E_{\! \! \; h}(\psi,\psi)}$ admet une solution si et seulement \hbox{si $\psi$} est $h$-admissible.
\end{introthm}

Notons que ce dernier résultat, et plus généralement, l'étude des équations GQE, forment le coeur technique de cette thèse (voir l'appendice I.\ref{app_GQE}).

\subsection{Les algèbres GQE sans relations de Serre}
L'algèbre de Lie $\tilde \glie$ est définie à partir des mêmes relations que $\glie$, hormis les relations de Serre. Le théorème suivant affirme que toute algèbre GQE peut être obtenue comme quotient d'une déformation formelle de $\Utc$.

\begin{introthm}
Soit $\psi$ un $I$-coloriage $h$-admissible. Il existe une déformation formelle \hbox{${\glie_{\! \! \; F}}$-pointée} $\Uth$ de l'algèbre $\Utc$ qui admet l'algèbre GQE ${\, \Uh}$ comme quotient.
\end{introthm}

L'algèbre $\Uth$ est appelée une \emph{algèbre GQE sans relation de Serre}. Notons qu'il est possible de donner une présentation de l'algèbre $\Uth$ à partir de solutions des équations GQE ${\mathscr E_{\! \! \; h}(\psi_i,\psi_i)}$. \\

On désigne par $\FUh$ la complétion $h$-adique de l'algèbre $\UF$. Le diagramme suivant présente une hiérarchie entre les différentes algèbres introduites jusqu'à présent dans le cadre des algèbres GQE.
$$ \shorthandoff{;:!?} \xymatrix @!0 @C=6pc {
\FUh \ar@{->>}[r] & \, \Uth \ar@{->>}[r] & \, \Uh \, .
} $$

\subsection{La conjecture GQE}
Sont proposés dans cette thèse un programme et une méthode pour la résolution du problème de déformation de la catégorie $\Ointc$. Dans le cas où la direction de déformation est donnée par un $I$-coloriage $h$-admissible, une réponse positive au problème est énoncée en tant que conjecture. Notons que la formulation de la conjecture diffère sensiblement de celle du problème initial, et ce afin de souligner la simplification apportée à celui-ci par l'introduction de la notion de déformation formelle (il est possible de montrer l'équivalence des deux énoncés).

\begin{introconj}[Conjecture GQE]
Soit $\psi$ un $I$-coloriage $h$-admissible. Pour chaque représentation $V$ de $\glie$ dans $\Ointc$, l'action de ${\glie_{\! \! \; F}}$ sur $V$ admet une déformation formelle $\psi$-intégrable.
\end{introconj}

Au moment où ces lignes sont écrites, nous travaillons également à une résolution de la conjecture GQE (où les algèbre GQE sans relations de Serre jouent un rôle central). Notons que la conjecture est toutefois vérifiée en rang 1, les représentations $\Lrep{n,\psi_i}$ fournissant, par définition et au vu de l'axiome de déformation vérifié par $\psi$, une solution. \\


Les algèbres quantiques standards fournissent par ailleurs l'exemple d'une autre situation pour laquelle la conjecture GQE est établie. On rappelle en effet que les représentations de $\glie$ dans $\Ointc$ admettent des déformations formelles $\psiqI$-intégrables, obtenues comme représentations de l'algèbre quantique $\Uh[\glie]$. Une preuve complète de la conjecture GQE pourra fournir une nouvelle preuve de l'existence d'une déformation quantique des représentations dans $\Ointc$. Notons enfin que la conjecture GQE est trivialement vérifiée pour le $I$-coloriage $h$-admissible $\psiclI$. \\

Dans certains des énoncés suivants, nous supposons que nous sommes dans le cadre du programme proposé dans cette thèse, c'est-à-dire que la conjecture GQE est vérifiée pour toute algèbre de Kac-Moody $\glie$ et pour toute $I$-coloriage $h$-admissible $\psi$. Notons que les théorèmes sont toutefois pleinement établis en \hbox{rang 1}.

\subsection{Programme de résolution}
Rappelons que l'énoncé de la conjecture GQE apporte une réponse au problème de déformation colorée dans le cas des $I$-coloriages $h$-admissibles. Notons d'autre part que tous les coloriages du cristal $\BB$ peuvent s'obtenir à partir de coloriages $h$-admissibles, par spécialisation du paramètre $h$ à une valeur autre que zéro. Dans le cas des représentations $\psi$-intégrables, la spécialisation du paramètre formel $h$ à une valeur distincte de zéro doit toutefois être menée avec précaution (les spécialisations de l'anneau $\Kh$ à des valeurs non nulles sont triviales). La théorie des algèbres de Kac-Moody colorées permet alors de surmonter ces dernières difficultés et donne ainsi un cadre pour une solution générale au problème de déformation colorée.

\begin{introthm}
\label{introthm_deformpb}
Soit $\psi$ un ${I \! \! \;}$-coloriage admissible. On suppose que l'anneau $R$ est principal. Pour toute \hbox{représentation $V$} dans $\Ointc$, il est possible d'étendre l'action de \hbox{$\hlie$} sur $V$ en une action \hbox{$\psi$-intégrable} de $\glieF$ sur ${V \! \otimes_\KK \! R}$.
\end{introthm}


\subsection{La catégorie $\Oint{\glie,\psi}$}
On désigne \hbox{par $\Lc{\lambda}$} l'unique représentation irréductible de $\glie$ de plus haut \hbox{poids $\lambda \in \hlie^\ast$}. On rappelle que si $\lambda$ est \emph{dominant}, la représentation $\Lc{\lambda}$ peut être définie par générateurs et relations. Elle est engendrée par un vecteur $v_\lambda$ et sujette aux relations suivantes :
\begin{equation} \label{eq_Lcpres}
\check \mu \! \; . \; \! v_\lambda \ = \ \langle \lambda, \check \mu \rangle \, v_\lambda \, , \quad X^+_i \! \; . \! \; v \ = (X^-_i)^{\langle \lambda, \check \alpha_i \rangle +1} . \! \; v_\lambda \ = \ 0 \quad \quad \ (i \in I, \ \check \mu \in \hlie) \, .
\end{equation}
On rappelle que les représentations $\Lc{\lambda}$, avec $\lambda$ dominant, sont deux à deux distinctes et sont les uniques représentations irréductibles de la catégorie $\Ointc$, par ailleurs complètement réductible. \\


Soit maintenant $\psi$ un ${I \! \! \;}$-coloriage. La catégorie $\Oint{\glie,\psi}$ est composée des représentations $\psi$-intégrables dont les poids satisfont aux mêmes conditions que celles imposées aux représentations intégrables classiques dans $\Ointc$. Le théorème suivant affirme que la catégorie $\Oint{\glie,\psi}$ est parallèle à son analogue classique $\Ointc$.

\begin{introthm}
\label{introthm_Opsi}
Soit $\psi$ un ${I \! \! \;}$-coloriage admissible du cristal $\BB$. On suppose que l'anneau $R$ est principal.
\begin{enumerate}
\item[1)] Soit $\lambda$ dominant. La représentation $\Lc{\lambda}$ admet une unique déformation $\psi$-intégrable, que l'on note $\Lrep{\glie,\lambda,\psi}$.
\item[2)] Soit $\lambda$ dominant. La représentation $\Lrep{\glie,\lambda,\psi}$ de $\Uhat{\glie,\psi}$ est engendrée par un vecteur $v_\lambda$ et sujette aux relations \eqref{eq_Lcpres}.
\item[3)] Les représentations $\Lrep{\glie,\lambda,\psi}$ ($\lambda$ dominant) sont deux à deux distinctes et sont les uniques représentations indécomposables dans $\Oint{\glie,\psi}$.
\item[4)] Toute représentation dans $\Oint{\glie,\psi}$ est isomorphe à une somme directe de représentations $\Lrep{\glie,\lambda,\psi}$ avec $\lambda$ dominant.
\end{enumerate}
\end{introthm}

\subsection{Déformations formelles}
Le théorème suivant affirme, qu'outre la structure d'algèbre associative, c'est l'entière structure de diagramme ${\glie_{\! \! \; F}}$-pointé de $\UUc$  qui est déformée par les algèbres (ou diagrammes) GQE. Le caractère nécessaire et suffisant des axiomes d'une $I$-coloriage $h$-admissible est dans le même temps établi.

\begin{introthm}
\label{introthm_Uhdned}
Soit $\psi$ un ${I \! \! \;}$-coloriage du cristal ${\, \BB}$. Les trois assertions suivantes sont équivalentes.
\begin{enumerate}
\item[i)] Le ${I \! \! \;}$-coloriage $\psi$ est un $I$-coloriage $h$-admissible.
\item[ii)] L'algèbre $\Uh$ est une déformation formelle de $\Uc$.
\item[iii)] Le diagramme d'algèbre $\UUh$ est une déformation formelle de $\UUc$.
\end{enumerate}
\end{introthm}

\subsection{Présentation des algèbres GQE}
Lorsque les algèbres GQE sont des déformations formelles de $\Uc$, l'existence d'une déformation des relations de la présentation de Serre est assurée. Le résultat suivant précise ce fait et exhibe une présentation par générateurs et relations explicite.

\begin{introthm}
Soit $\psi$ un $I$-coloriage $h$-admissible. L'algèbre GQE ${\, \Uh}$ est topologiquement engendrée par les éléments $X_i^\pm$, l'espace vectoriel $\hlie$, sujette aux relations non déformables \eqref{eq_KM_nondeformableintro} de $\glie$ et aux relations suivantes:
\begin{equation*} \label{eq_GQESerre} \left\{ {\renewcommand{\arraystretch}{2}\begin{array}{l}
\big[ {}^\psi \! X^+_i, X^-_i \big] = \, \check \alpha_i \quad \quad (i \in I) \, , \\
\sum_{k+k' = 1 - C_{ij}} {(-1)}^k \, {1 - C_{ij} \choose k} \, \big( {}^\psi \! X_i^{\pm} \big)^k \! \; X^{\pm}_j \! \; \big( {}^\psi \! X_i^{\pm} \big)^{k'} \ = \ 0 \quad \quad (i,j \in I, \ i \neq j) \, , \\
\text{avec } \ {}^\psi \! X_i^{\pm} \ := \ \sum_{a \in \NN} \, \big( X_i^\mp \big)^a \! \; {\bar S^{\! \: \psi_i}_{\! \! \; h,a}} (\pm \check \alpha_i) \! \; \big( X^\pm_i \big)^{a+1}
\end{array}} \right.
\end{equation*}
où ${({\bar S^{\! \: \psi_i}_{\! \! \; h,a}})}_{a \in \NN}$ est une suite convergeant vers $0$ et à valeurs dans la sous-algèbre topologiquement engendrée par $\check \alpha_i$.
\end{introthm}

Les relations $\sum_{k+k' = 1 - C_{ij}} {(-1)}^k \, {1 - C_{ij} \choose k} \, \big( {}^\psi \! X_i^{\pm} \big)^k X^{\pm}_j \big( {}^\psi \! X_i^{\pm} \big)^{k'} = 0$ ($i \neq j$) sont appelées les \emph{relations de Serre GQE}. \\

Notons que les suites ${{({\bar S^{\! \: \psi_i}_{\! \! \; h,a}})}_{\! \! \: a \in \NN}}$ sont obtenues directement à partir des solutions de certaines équations GQE associées aux coloriages $h$-admissibles $\psi_i$.

\subsection{$\hlie$-trivialité}
Il est connu, dans le cas où $\glie$ est de dimension finie, que toute déformation formelle de $\Uc$ est triviale. Le théorème \ref{introthm_htrivial} énonce un résultat sensiblement plus fort dans le cas des algèbres GQE. Notons que ce résultat avait été établi par Drinfeld \cite{Dtrivial} dans le cas particulier de l'algèbre quantique $\Uh[\glie]$.

\begin{introthm}
\label{introthm_htrivial}
Si $\glie$ est de dimension finie, il existe un isomorphisme fixant $\hlie$ entre chaque algèbre GQE (associée à $\glie$) et la déformation formelle constante de $\Uc$.
\end{introthm}

Dans le cas des algèbres de Kac-Moody colorées, la $\hlie$-trivialité est à ce jour conjecturée.

\begin{introconj}
Si $\glie$ est de dimension finie, les algèbres de Kac-Moody colorées (associées à $\glie$) sont deux à deux $\hlie$-isomorphes.
\end{introconj}





\subsection{Non $\glieF$-trivialité}
La trivialité en rang fini (établie ou conjecturée) des algèbres GQE et des algèbre de Kac-Moody colorées mène à la remarque suivante : les objets introduits dans cette thèse sont, en tant qu'algèbres associatives, très proches de $\Uc$. Néanmoins, les identifications sont en général difficiles à mettre en oeuvre. L'intérêt des algèbres GQE et des algèbre de Kac-Moody colorées est confirmé par le problème  de déformation colorée, qui on le rappelle avait motivé leur introduction : nous ne recherchons pas seulement une déformation de l'action de l'algèbre $\Uc$ sur ses représentations, mais plus précisément une déformation de l'action de chacun des générateurs de Chevalley de $\Uc$ (rappelons par ailleurs que ces contraintes ont été imposées dans l'optique d'une réalisation du principe d'interpolation de Langlands). \\

La structure algébrique adaptée à ce contexte est ici la structure d'algèbre \hbox{${\glie_{\! \! \; F}}$-pointée}, qui retient, entre autres, et la structure d'algèbre associative, et le choix des générateurs de Chevalley. Les théorèmes suivant affirment alors que deux ${I \! \! \:}$-coloriages différents de $\BB$ mènent à deux déformations différentes de $\glie$.

\begin{introthm}
Les structures $\glieF$-pointées de deux algèbres de K-M colorées (respectivement de deux algèbres GQE) sont isomorphes si et seulement si, pour chaque $i \in I$, les deux représentations associées du carquois $\BB$ sont isomorphes.
\end{introthm}

\subsection{Classification}
Une déformation formelle $\AA_h$ du diagramme $\UUc$ est dite \emph{localement $\hlie$-triviale} si les algèbres $\AA_h(i)$ sont des déformations formelles \hbox{$H$-triviales} de $\Uc[\slt]$. \\

On rappelle (voir le théorème \ref{introthm_Uhdned}) que les diagrammes GQE sont des déformations formelles du diagramme $\glieF$-pointé $\UUc$. Le théorème suivant établit que les diagrammes GQE décrivent l'ensemble des déformations formelles localement $\hlie$-triviales de $\UUc$.

\begin{introthm}
L'application suivante est une bijection.
\begin{IEEEeqnarray*}{CCC}
\left \{ {\renewcommand{\arraystretch}{0.85} \begin{array}{c}
\text{${I \! \! \;}$-coloriages} \\
\text{$h$-admissibles}
\end{array}} \right \}
& \ \longrightarrow \ &
\left \{ {\renewcommand{\arraystretch}{0.85} \begin{array}{c}
\text{Déformations formelles} \\
\text{localement $\hlie$-triviales} \\
\text{du diagramme ${\glie_{\! \! \; F}}$-pointé $\UUc$}, \\
\text{à isomorphisme} \\
\text{de diagramme ${\glie_{\! \! \; F}}$-pointé près}
\end{array}} \right \} \\
&& \\
\dnh & \longmapsto & \UUh \, .
\end{IEEEeqnarray*}
\end{introthm}

\subsection{Dualité isogénique entre représentations d'algèbres de Lie}
La dualité de Langlands entre algèbres de Kac-Moody est un exemple d'\emph{isogénie}. Deux algèbres de \hbox{Kac-Moody $\glie$ et $\glie'$} sont dit isogéniques s'il existe une dualité entre les sous-algèbres de Cartan $\hlie$ et $\hlie'$, de forme identique à la dualité de \hbox{Langlands \eqref{eq_Cartanisogenie}} entre $\hlie$ et $\Lhlie$ et où les nombres $d/d_i$ sont remplacés par des nombres $\xi_i \in \NNI$.

\begin{equation} \label{eq_Cartanisogenie}
{\renewcommand{\arraystretch}{1.3}
\begin{array}{ccc}
\hlie & \longleftrightarrow & \hlie' \\
\check \alpha_i & \leftrightarrow & \xi_i \; \! \check \alpha'_i \, .
\end{array}
\quad \quad \quad \quad \quad
\begin{array}{ccc}
\hlied & \longleftrightarrow & {(\hlie')}^\ast \\
\xi_i \; \! \alpha_i & \leftrightarrow & \alpha'_i \, ,
\end{array}
}
\end{equation}

La correspondance entre les algèbres $\hlie$ et $\hlie'$ induit une correspondance entre les représentations de $\hlie$ et $\hlie'$. On peut alors à nouveau se demander si cette dernière correspondance peut être étendue entre les représentations de $\glie$ et $\glie'$. La théorie des algèbres de Kac-Moody colorées permet de répondre à cette question dans le cas des catégories $\Ointc$ et $\Ointc[(\glie')]$.

\begin{introthm}
Pour toute représentation $V$ dans la catégorie $\Ointc$, il existe une représentation de $\glie'$ dans la catégorie $\Ointc[(\glie')]$, dont la restriction à $\hlie'$ est la représentation $\xiV := \bigoplus_{\lambda \in {(\hlie'_\ZZ)}^\ast} V_\lambda$.
\end{introthm}

\subsection{Interpolations isogéniques classique et quantique}
Le principe d'interpolation de Langlands classique (voir le principe \ref{introprinc_inter}) peut être formulé dans le cadre plus général d'une isogénie. La dualité isogénique, énoncée précédemment, est établie comme conséquence d'une réalisation de ce \emph{principe d'interpolation isogénique classique}. Le théorème suivant, à l'aide de la théorie des algèbres de Kac-Moody colorées, affirme que cette réalisation existe et donne en particulier une nouvelle preuve de la dualité de Langlands entre représentations d'algèbres de Lie, telle que formulée dans le théorème \ref{introthm_dualite}. \\

On désigne ici par $R$ l'anneau des fractions rationnelles en la variable $u$, à coefficients dans $\KK$, sans pôle en $0$ et $1$.

\begin{introthm}
\label{introthm_intercl}
Il existe un ${I \! \! \;}$-coloriage $\psi_u$, à valeurs dans $R$, tel que pour toute \hbox{représentation $V$} dans $\Ointc$, il est possible d'étendre l'action de \hbox{$\hlie$} \hbox{sur $V$} en une action $\psi_u$-intégrable de ${\glie_{\! \! \; F}}$ \hbox{sur $\VR = {V \! \otimes_\KK \! R}$} vérifiant les deux points suivants.
\begin{enumerate}
\item[\textbullet] L'action de $X^\pm_i$ sur $\VR$, spécialisée à ${u \! = \! 0}$, est égale à celle de $X_i^\pm$ \hbox{sur $V$}.
\item[\textbullet] L'action de $(X^\pm_i)^{d/d_i}$ sur ${{}^\xi \! \! \; V_{\! R}}$, spécialisée à ${u \! = \! 1}$, fait du $\hlie'$-module $\xiV$ une représentation de $\glie'$ dans $\Ointc[(\glie')]$.
\end{enumerate}
\end{introthm}

Il existe, et ce manière évidente (on rappelle en effet que les catégories $\Ointc$ \hbox{et $\Oint[\! q]{\glie}$} sont similaires), un analogue quantique de la dualité isogénique entre représentations d'algèbre de Lie : pour toute représentation $\Vq$ dans $\Oint[\! q]{\glie}$, il existe une représentation de $\Uq (\glie')$ dans $\Oint[\! q]{\glie}$, dont la restriction à $\hlie'$ est la représentation ${{}^\xi \! \! \; V_{\! \! \: q}}$. Le résultat suivant établit l'existence d'une réalisation du \emph{principe d'interpolation isogénique quantique}. \\


On désigne ici par $R_q$ l'anneau des fractions rationnelles en la \hbox{variable $u$}, à coefficients dans $\KK(q)$, sans pôle en $0$ et $1$.

\begin{introthm}
\label{introthm_interq}
Il existe un ${I \! \! \;}$-coloriage $\psi_u$, à valeurs dans $R_q$, tel que pour toute représentation $\Vq$ dans $\Oint[\! q]{\glie}$, il est possible d'étendre l'action de \hbox{$\hlie$} sur $\Vq$ en une action $\psi_u$-intégrable de ${\glie_{\! \! \; F}}$ \hbox{sur $\VR = {V \! \otimes_{\KK(q)} \! R_q}$} vérifiant les deux points suivants.
\begin{enumerate}
\item[\textbullet] L'action de $X^\pm_i$ sur $\VR$, spécialisée à ${u \! = \! 0}$, est égale à celle de $X_i^\pm$ \hbox{sur $\Vq$}.
\item[\textbullet] L'action de $(X^\pm_i)^{d/d_i}$ sur ${{}^\xi \! \! \; V_{\! R}}$, spécialisée à ${u \! = \! 1}$, fait du \hbox{$\hlie'$-module ${{}^\xi \! \! \; V_{\! \! \: q}}$} une représentation de $\Uq (\glie')$ dans $\Oint[\! q]{\glie'}$.
\end{enumerate}
\end{introthm}



\subsection{Interpolation isogénique colorée}

Le résultat suivant généralise les deux théorèmes précédents et établit l'existence d'une réalisation d'un \emph{principe d'interpolation isogénique colorée}. \\

On désigne ici par $R$ l'anneau des fractions rationnelles en la \hbox{variable $u$}, à coefficients dans $\KK$, sans pôle en $0$ et $1$.

\begin{introthm}
Soient $\psi, \psi'$ des ${I \! \! \;}$-coloriages de $\BB$, à valeurs dans le \hbox{corps $\KK$}. Il existe un ${I \! \! \;}$-coloriage $\psi_u$ de $\BB$, à valeurs dans l'anneau $R$, tel que pour toute \hbox{représentation $V$} dans $\Oint[\KK]{\glie,\psi}$, il est possible d'étendre l'action de \hbox{$\hlie$} sur $V$ en une action $\psi_u$-intégrable de ${\glie_{\! \! \; F}}$ \hbox{sur $\VR = {V \! \otimes_\KK \! R}$} vérifiant les deux points suivants.
\begin{enumerate}
\item[\textbullet] L'action de $X^\pm_i$ sur $\VR$, spécialisée à ${u \! = \! 0}$, est égale à celle de $X_i^\pm$ \hbox{sur $V$}.
\item[\textbullet] L'action de $(X^\pm_i)^{d/d_i}$ sur ${{}^\xi \! \! \; V_{\! R}}$, spécialisée à ${u \! = \! 1}$, fait du $\hlie'$-module $\xiV$ une représentation $\psi'$-intégrable de ${\glie'_{\! \! \; F}}$.
\end{enumerate}
\end{introthm}

Signalons ici que la réalisation des principes d'interpolation (isogénique, de Langlands, classique, quantique et colorée) ne peut pas être réalisée directement à partir des algèbres GQE. Rappelons en effet que le paramètre formel $h$ ne peut pas être spécialisé à une valeur autre que zéro en l'absente de certaines précautions, rendant difficile à première vue toute interpolation. Le relais est à cette fin transmis aux algèbres de \hbox{Kac-Moody} colorées.

\subsection{Contribution au programme de Frenkel-Hernandez}
La réalisation du principe d'interpolation de Langlands quantique, cas particulier du théorème \ref{introthm_interq} en considérant l'\emph{isogénie de Langlands}, apporte une solution à une extension de la conjecture de Frenkel-Hernandez : on rappelle en effet qu'un quotient de l'algèbre de Frenkel-Hernandez $\Uqt (\tglie)$ est contenu dans une algèbre de Kac-Moody colorée (voir le théorème \ref{introthm_GQEclq}). Le résultat suivant affirme ainsi que la conjecture GQE implique la conjecture de Frenkel-Hernandez. \\

On rappelle que $\varepsilon$ désigne une racine primitive de l'unité d'ordre $2d$.

\begin{introthm}
Pour toute représentation $\Vq$ dans la catégorie $\Oint[\! q]{\glie}$, il existe une représentation $\Vqt$ de $\Uqt(\tglie)$ vérifiant les deux points suivants.
\begin{enumerate}
\item[\textbullet] L'action de $X^\pm_i$ sur $\Vqt$ peut être spécialisée à ${t \! = \! 1}$ et est alors égale à l'action de $X_i^\pm$ \hbox{sur $\Vq$}.
\item[\textbullet] L'action de $(X^\pm_i)^{d/d_i}$ sur ${{}^\xi \! \! \; V_{\! \! \: q,t}}$, peut être spécialisée à ${q \! = \! \varepsilon}$ et fait alors du \hbox{$\hlie'$-module ${{}^\xi \! \! \; V_{\! \! \: q,t}}$} une représentation de ${U_{\! \! \; t}} (\Lglie)$ dans $\Oint[\! \! \; t]{\Lglie}$.
\end{enumerate}
\end{introthm}

\subsection{Interpolation de Langlands classique via les groupes quantiques}
La théorie des algèbres de Kac-Moody colorées permet de réaliser le principe d'interpolation de Langlands classique à partir des représentations de l'algèbre quantique $\Uq (\glie)$. Notons que ce résultat suivant est pleinement établi, dans la mesure où la conjecture GQE est vérifiée pour le $I$-coloriage quantique $\psiqI$, et apporte par ailleurs une nouvelle preuve de la dualité de Langlands entre les catégories $\Ointc$ et $\Ointc[(\Lglie)]$ (cf. théorème \ref{introthm_dualite}). \\ 

On rappelle que $\varepsilon$ désigne une racine primitive de l'unité d'ordre $2d$.

\begin{introthm}
\label{introthm_intergrpq}
Soit $V$ une représentation dans $\Ointc$ et notons par $\Vq$ la représentation de $\Uq(\glie)$ déformant $V$.
\begin{enumerate}
\item[\textbullet] L'action de $X^\pm_i$ sur $\Vq$ peut être spécialisée à ${q \! = \! 1}$ et est alors égale à l'action de $X_i^\pm$ \hbox{sur $V$}.
\item[\textbullet] L'action de $(X^\pm_i)^{d/d_i}$ sur $\LVq$ peut être spécialisée à ${q \! = \! \varepsilon}$ et fait alors du $\Lhlie$-module $\LVq$ une représentation de $\Lglie$ dans $\Ointc[(\Lglie)]$.
\end{enumerate}
\end{introthm}

\section{Ouvertures}
\subsection[Dualité de Langlands entre représentations affines classiques]{Dualité de Langlands entre représentations affines \\ classiques}
Soit ici $\glie$ une algèbre de Lie simple de dimension finie. On note $\hat \glie^\sigma$ l'algèbre de Lie affine associée à $\glie$ et à un automorphisme $\sigma$ d'ordre $m \in \NNI$ du diagramme de Dynkin de $\glie$ : c'est l'extension centrale non triviale de l'algèbre de Lie des applications polynomiales contravariantes de $\CC^\times$ dans $\glie$ (on considère les actions de $\ZZ / m \ZZ$ sur $\CC^\times$ et $\glie$, induites respectivement par une racine primitive \hbox{dans $\CC$} d'ordre $m$ et par $\sigma$). Lorsque $m = 1$, l'algèbre affine $\hat \glie^\sigma$ est dite non tordue et notée plus simplement $\hat \glie$. Lorsque $m > 1$, l'algèbre est dite tordue. Remarquons que $\glie$ est une sous-algèbre de Lie de l'algèbre de Lie affine non tordue $\hat \glie$. \\

L'algèbre de Lie affine $\hat \glie^\sigma$ n'est pas exactement une algèbre de Kac-Moody, mais plus exactement l'algèbre dérivée d'une algèbre de Kac-Moody $\hat \glie^\sigma_{\mathrm{KM}}$. La catégorie des représentations de dimension finie de $\hat \glie^\sigma$ est riche et est très différente de la catégorie $\mathcal O^{\mathrm{int}}(\hat \glie^\sigma_{\mathrm{KM}})$ (toute représentation dans $\mathcal O^{\mathrm{int}}(\hat \glie^\sigma_{\mathrm{KM}})$ est par exemple complètement réductible, tandis qu'une représentation de dimension finie de $\hat \glie^\sigma$ ne l'est en général pas). \\

On note ${{}^L \! \! \: \hat \glie^\sigma}$ l'algèbre Langlands duale de l'algèbre affine $\hat \glie^\sigma$ : c'est l'algèbre de Lie affine dont l'algèbre de Kac-Moody associée est l'algèbre Langlands duale de l'algèbre de Kac-Moody $\hat \glie^\sigma_{\mathrm{KM}}$. Dans le cas où l'algèbre $\hat \glie^\sigma$ est différente de son algèbre Langlands duale, cette dernière est une algèbre de Lie affine tordue si $\hat \glie^\sigma$ est non tordue et réciproquement. \\

Notons que ${{}^L \! \! \: \hat \glie^\sigma}$ n'est pas en général une algèbre de Lie affine associée à $\Lglie$ (les opérations ``duale de Langlands'' et ``affinisation'' ne commutent pas). Toutefois, l'algèbre $\Lsglie$ contient $\Lglie$ (comme sous-algèbre des applications constantes de l'algèbre de Lie affine $\Lsglie$). \\

Soit maintenant $V$ une représentation de dimension finie de $\hat \glie$. On rappelle que la dualité de Langlands, telle que formulée dans le théorème \ref{introthm_dualite}, affirme l'existence d'une extension à $\Lglie$ de l'action de $\Lhlie$ sur l'espace $\LV$ (où $V$ est regardé comme une représentation de $\glie$ par restriction).

\begin{introquest} \label{introquest_affine1}
Soit $V$ une représentation de dimension finie de l'algèbre de Lie affine non tordue $\hat \glie$. Existe-t-il une représentation de $\Lsglie$ de dimension finie, dont la restriction à la sous-algèbre de Cartan $\Lhlie$ est la représentation $\LV$ ?
\end{introquest}

Supposons que $\hat \glie$ n'est pas égale à son algèbre de Langlands duale (dans le cas contraire, la question \ref{introquest_affine1} devient triviale). L'algèbre $\Lsglie$ est une algèbre de Lie affine tordue, on note alors $\glie_1$ l'algèbre de Lie simple de dimension finie et $\sigma_1$ l'automorphisme du diagramme de Dynkin de $\glie_1$ tels que $\Lsglie = \hat \glie_1^{\sigma_1}$. En remarquant que l'algèbre $\Lglie$ est la sous-algèbre de $\glie_1$ constituée des points fixes de $\sigma_1$ et en remarquant que toute représentation de dimension finie de $\glie_1$ peut être étendue en une représentation de $\hat \glie_1^{\sigma_1}$ (par spécialisation à une valeur de $\CC^\times$ des éléments de $\hat \glie_1^{\sigma_1}$), la question \ref{introquest_affine1} apparaît équivalente à la question suivante : soit $V$ une représentation de dimension finie de $\glie$, est-il possible d'étendre à $\glie_1$ l'action de $\Lglie$ sur $\LV$ ? \\

Plus généralement, on peut se demander si la dualité de Langlands entre les catégories $\mathcal O^{\mathrm{int}}(\hat \glie^\sigma_{\mathrm{KM}})$ et $\mathcal O^{\mathrm{int}}({{}^L \! \! \: \hat \glie^\sigma_{\mathrm{KM}}})$ admet un analogue pour les représentations de dimension finie de $\hat \glie^\sigma$ et ${{}^L \! \! \: \hat \glie^\sigma}$. Notons pour cela que la correspondance \eqref{eq_CartanLanglands} entre les sous-algèbres de Cartan de $\hat \glie^\sigma_{\mathrm{KM}}$ et ${{}^L \! \! \: \hat \glie^\sigma_{\mathrm{KM}}}$ peut être restreinte aux sous-algèbres de Cartan $\hat \hlie^\sigma$ et ${{}^L \! \! \; \hat \hlie^\sigma}$ de $\hat \glie^\sigma$ et ${{}^L \! \! \: \hat \glie^\sigma}$. Notons également que l'action \hbox{de $\hat \hlie^\sigma$} sur une représentation de dimension finie de $\hat \glie^\sigma$ est diagonalisable, on définit alors la sous-représentation $\LV$ de ${{}^L \! \! \; \hat \hlie^\sigma}$ de la même manière que précédemment.

\begin{introquest} \label{introquest_affine2}
Soit $V$ une représentation de dimension finie de l'algèbre de Lie affine $\hat \glie^\sigma$. Existe-t-il une représentation de dimension finie de ${{}^L \! \! \: \hat \glie^\sigma}$, dont la restriction à la sous-algèbre de Cartan ${{}^L \! \! \; \hat \hlie^\sigma}$ est la représentation $\LV$.
\end{introquest}

Rappelons que pour toute représentation $V$ dans $\Ointc$, il existe une unique représentation dans $\Ointc[(\Lglie)]$ dont la restriction à $\Lhlie$ est $\LV$. Une propriété analogue d'unicité n'est toutefois pas envisageable dans le cas des questions \ref{introquest_affine1} et \ref{introquest_affine2}. Remarquons enfin qu'une réponse à la question \ref{introquest_affine2} apporte en particulier une réponse à la question \ref{introquest_affine1}, la dualité de Langlands entre les sous-algèbres de Cartan $\hat \hlie$ et ${{}^L \! \! \; \hat \hlie}$ étant une extension de la dualité de Langlands entre les sous-algèbres de Cartan de $\glie$ et $\Lglie$.

\subsection[Dualité de Langlands entre représentations affines quantiques]{Dualité de Langlands entre représentations affines \\ quantiques}
Les algèbres de Lie affines admettent également des déformations quantiques. Soit $\glie$ une algèbre de Lie simple de dimension finie, on note $\Uq(\hat \glie^\sigma)$ l'algèbre affine quantique associée à l'algèbre de Lie affine $\hat \glie^\sigma$. Les questions \ref{introquest_affine1} et \ref{introquest_affine2}, à propos d'une extension de la dualité de Langlands pour les catégories $\mathcal O^{\mathrm{int}}$ à une dualité de Langlands entre les représentations de dimensions finies des algèbres de Lie affines, peuvent être alors reformulées dans le cas quantique.

\begin{introquest} \label{introquest_affine1q}
Soit $\Vq$ une représentation de dimension finie de l'algèbre affine quantique non tordue $\Uq(\hat \glie)$. \hbox{Existe-t-il} une représentation de dimension finie \hbox{de $\Uq(\Lsglie)$}, dont la restriction à $\Uq(\Lhlie)$ est la représentation $\LVq$ ?
\end{introquest}

\begin{introquest} \label{introquest_affine2q}
Soit $\Vq$ une représentation de dimension finie de l'algèbre affine quantique $\Uq(\hat \glie^\sigma)$. \hbox{Existe-t-il} une représentation de dimension finie de $\Uq({{}^L \! \! \: \hat \glie^\sigma})$, dont la restriction à ${{}^L \! \! \; \hat \hlie^\sigma}$ est la représentation $\LVq$ ?
\end{introquest}

La catégorie des représentations de dimension finie de $\Uq(\hat \glie^\sigma)$, contrairement à la catégorie $\mathcal O^{\mathrm{int}}(\hat \glie^\sigma_{\mathrm{KM}})$, s'avère être très différente de son analogue classique. Par conséquent, les questions \ref{introquest_affine1q} et \ref{introquest_affine2q} ne peuvent pas être considérées cette fois-ci équivalentes au problème classique (posé par les questions \ref{introquest_affine1} et \ref{introquest_affine2}). \\

Frenkel-Hernandez \cite{FH2} apportent un début de réponse aux questions \ref{introquest_affine1q} et \ref{introquest_affine2q}. Ils prouvent en effet l'existence de relations entre certaines représentations de dimension finie des algèbres affines quantiques duales $\Uq(\hat \glie^\sigma)$ et $\Uq({{}^L \! \! \: \hat \glie^\sigma})$. Ces relations sont décrites en termes des $q$-caractères desdites représentations (le \hbox{$q$-caractère} d'une représentation de dimension finie de $\Uq(\hat \glie^\sigma)$ résume une grande partie de la décomposition de Jordan de la représentation sous l'action de la sous-algèbre de Cartan de $\Uq(\hat \glie^\sigma)$).

\subsection{Correspondance de Langlands géométrique}
La dualité de Langlands entre représentations d'algèbres de Kac-Moody peut être reliée au programme de Langlands géométrique. Ce lien a été proposé par Frenkel-Hernandez \cite{FH1}, nous reprenons ci-dessous certains des éléments d'explication présentés dans l'introduction \hbox{de \cite{FH1}}, où $\glie$ désigne toujours une algèbre de Lie simple de dimension finie. \\

[\emph{Un des résultats clés dans la correspondance géométrique de Langlands est un isomorphisme entre le centre $Z(\hat \glie)$ de l'algèbre enveloppante complétée \hbox{de $\hat \glie$} au niveau critique ($\hat \glie$ désigne l'algèbre de Lie affine associée à l'algèbre de dimension \hbox{finie $\glie$}) et la ${W \!}$-algèbre classique $W (\Lglie)$ associée à l'algèbre de Lie Langlands duale de $\glie$ (voir \cite{FF,Fre1} ainsi que \cite{FR1} pour plus de détails). Ce résultat constitue le socle de la correspondance de Langlands géométrique locale (\hbox{voir \cite{Fre4,FG1}}) et il est aussi un ingrédient principal employé par \hbox{Beilinson-Drinfeld \cite{BD}} (voir \hbox{aussi \cite{Fre2}}) dans la construction de la correspondance de Langlands géométrique globale. Toutefois, cet isomorphisme reste mystérieux : son existence est connue mais mal comprise.} \\

\emph{Dans le but de mieux comprendre cet isomorphisme, Frenkel-Reshetikhin \cite{FR2} ont proposé de $q$-déformer la situation, en considérant le centre $Z_q(\hat \glie)$ de l'algèbre affine quantique $\Uq(\hat \glie)$ au niveau critique. On sait par ailleurs que le centre $Z_q(\hat \glie)$ est relié à l'anneau de Grothendieck ${\mathrm{Rep} \, \Uq(\hat \glie)}$ des représentations de dimension finie de $\Uq(\hat \glie )$ (à chaque représentation de dimension finie, il est en effet possible d'associer une série génératrice composée d'éléments dans $Z_q(\hat \glie)$, en utilisant ce que l'on appelle la construction transfer-matrix). On espère donc obtenir un nouvel éclairage sur l'isomorphisme $Z(\hat \glie ) \simeq W(\Lglie)$ en étudiant les possibles liens existant entre $Z_q(\hat \glie )$, ${\mathrm{Rep} \, \Uq(\hat \glie)}$ et la ${W \!}$-algèbre classique \hbox{$q$-déformée} ${W_{\!q}}(\Lglie)$.} \\

\emph{L'idée de Frenkel-Reshetikhin \cite{FR3} a été de déformer un nouvelle fois la situation, en introduisant une déformation (non commutative) à deux paramètres ${W_{\! q,t}} (\glie)$. La spécialisation ${W_{\! q,1}}$ à $t=1$ est le centre $Z_q(\hat \glie)$, de telle sorte que ${W_{\! q,t}} (\glie)$ est une déformation à un paramètre de $Z_q(\hat \glie)$ et une déformation à deux paramètres du centre initial $Z(\hat \glie)$. Le travail de Frenkel-Reshetikhin \cite{FR3} a été motivé par l'espoir qu'à travers l'analyse de différentes dualités et limites de ${W_{\! q,t}} (\glie)$, il soit possible d'apporter de nouveaux éléments de compréhension à propos de l'isomorphisme $Z(\hat \glie ) \simeq W(\Lglie)$ et par suite à propos de la correspondance de Langlands. En particulier, il a été suggéré que la spécialisation ${W_{\! \varepsilon,t}} (\glie)$ à $q=\varepsilon$ (où $\varepsilon$ désigne une racine primitive de l'unité d'ordre $2d$) contient comme sous-algèbre le centre $Z_t(\Lsglie)$ de l'algèbre affine quantique $U_t(\Lsglie)$ au niveau critique. Le centre $Z_t(\Lsglie)$ est par ailleurs lié à l'anneau de Grothendieck $\mathrm{Rep} \, U_t(\Lsglie)$ des représentations de dimension finie de $U_t(\Lsglie)$ (via la construction transfer-matrix). La spécialisation ${W_{\! q,1}}(\glie)$ à $t=1$ est quant à elle liée à l'anneau de Grothendieck des représentations de dimension finie de $\Uq(\hat \glie)$ (toujours via la construction transfer-matrix). La ${W \!}$-algèbre ${W_{\! q,t}}(\glie)$ apparaît ainsi interpoler les anneaux de Grothendieck des représentations de dimension finie des algèbres affines quantiques associées à $\hat \glie$ et $\Lsglie$. Ceci suggère en particulier l'existence de certaines relations entre ces représentations (des exemples ont été donnés par Frenkel-Reshetikhin \cite{FR3}).}] \\

On retrouve à travers ces relations la dualité de Langlands entre représentations de dimension finie des algèbres affines quantiques duales $\Uq(\hat \glie)$ et $\Uq(\Lsglie)$, évoquée précédemment. Précisons toutefois que ces relations sont exprimées au niveau des $q$-caractères et que l'on ignore à ce jour si elles correspondent précisément aux propos de la question \ref{introquest_affine1q}. \\

Remarquons que la valeur de spécialisation $\varepsilon$ ait apparut dans trois situations.

\vsp

\begin{itemize}
\item[\textbullet] Une première fois chez Frenkel-Reshetikhin afin de retrouver le centre $Z_t(\Lsglie)$ de l'algèbre affine quantique $U_t(\Lsglie)$ au niveau critique à partir de la $W$-algèbre déformée ${W_{\! q,t}} (\glie)$ (rappelons qu'il s'agit à ce jour d'une conjecture).
\item[\textbullet] Une deuxième fois chez Littelmann et McGerty afin d'établir le \hbox{théorème \ref{introthm_dualite}}.
\item[\textbullet] Une troisième fois  dans la conjecture \ref{introconj_FH} de Frenkel-Hernandez.
\end{itemize}


\newpage
\thispagestyle{empty}
\mbox{}
\newpage
\cleardoublepage
\part{Generalised Quantum Enveloping Algebras}
\label{part_GQE}
\setcounter{section}{0}
\begin{otherlanguage}{english}
  \righthyphenmin=62
\lefthyphenmin=62

\addcontentsline{toc}{section}{Notations and conventions}
\section*{Notations and conventions}
Rings and associative algebras are unitary, their morphisms are unity-preserving. \\

We denote by $\KK$ a field of characteristic zero. We denote by $\KK[u]$ and by $\KK[u,v]$ the polynomial rings \hbox{over $\KK$} in one variable $u$ and in two variables $u,v$, respectively. \\


Let $\Lambda$ be a set. We denote by $\KK \Lambda$ the $\KK$-vector space which admits $\Lambda$ as a basis. We denote by $\langle \Lambda \rangle$ the free monoid generated by $\Lambda$. We denote by $\KK \langle \Lambda \rangle$ the \hbox{$\KK$-algebra} of this monoid. Let $R$ be a ring, we denote by ${\Lambda^{\! R}}$ the ring of functions from $R$ to $\Lambda$. \\


Let $A$ be a $\KK$-algebra and $\glie$ be Lie $\KK$-algebra. Left $A$-modules and left $\glie$-modules are called representations of $A$ and $\glie$, respectively. Morphisms of left $A$-modules and of left $\glie$-modules are called $A$-morphisms and $\glie$-morphisms, respectively. The universal enveloping algebra of $\glie$ is denoted by $\Uc$. \\

We identify as usual representations of $\glie$ with representations of $\Uc$. \\

A sum indexed by a empty set is zero, a product indexed by an empty set is equal to $1$.

\nomenclature[AN]{$\NN$}{Natural numbers}
\nomenclature[AZ]{$\ZZ$}{Integers}
\nomenclature[AQ]{$\QQ$}{Rational numbers}
\nomenclature[AK]{$\KK$}{Field of characteristic zero}
\nomenclature[AK(u)]{$\KK[u]$}{Polynomial ring over $\KK$ in one variable $u$}
\nomenclature[AK(u,v)]{$\KK[u,v]$}{Polynomial ring over $\KK$ in two variables $u,v$}
\nomenclature[AKLambda]{$\KK \Lambda$}{$\KK$-vector space with basis $\Lambda$}
\nomenclature[A(Lambda)]{$\langle \Lambda \rangle$}{Free monoid generated by a set $\Lambda$}
\nomenclature[AK(Lambda)]{$\KK \langle \Lambda \rangle$}{Free algebra generated by a set $\Lambda$}

\section{Kac-Moody algebras}
We recall in this section the definition of Kac-Moody algebras. Note that the definition given in this paper is slightly different (but mainly equivalent) to the definition of Kac-Moody algebras usually present in the literature. The difference is about the Cartan subalgebra, which is built here from what we call a root datum (we borrow the name and the concept from Lusztig: see \cite{Lus}). We also recall some results on the representation theory of Kac-Moody algebras.


\subsection{Generalised Cartan matrices and root data}
A \emph{generalised Cartan matrix} is a matrix $C = {(C_{ij})}_{i,j \in I}$ indexed by a finite \hbox{set $I$}, with entries in $\ZZ$ and which satisfies the following conditions for all $i,j \in I$.
\begin{enumerate}
\item[\textbullet] $C_{ii} = 2$.
\item[\textbullet] $C_{ij} \leq 0$ if $i \neq j$.
\item[\textbullet] $C_{ij} = 0$ if and only if $C_{ji} = 0$.
\end{enumerate}
The set $I$ is called the \emph{Dynkin set}. The elements of $I$ are called the \emph{Dynkin vertices}. The generalised Cartan matrix $C$ is called \emph{symmetrisable} if there exists a family ${(d_i)}_{i \in I}$, called a \emph{symmetrising vector} of $C$, with values in $\ZZ_{> 0}$ and such that $d_i \: \! C_{ij} = d_j \: \! C_{ji}$ for \hbox{all $i,j$} \hbox{in $I$}. We assume that the integers $d_i$ ($i \in I'$) are coprime for every connected component $I'$ of $I$ (two elements $i,j \in I$ are linked if $C_{ij} \neq 0$). In this paper, \textbf{we always assume that a generalised Cartan matrix is symmetrisable}. \\

A generalised Cartan matrix is of \emph{finite type} if it is positive-definite. Generalised Cartan matrices of finite type are also called Cartan matrices. \\

A \emph{root datum} $\Xdat$ associated to a generalised Cartan matrix $C$ consists of the following.
\begin{enumerate}
\item[\textbullet] A pair $(X,Y)$ where $X$, $Y$ are two finitely generated free $\ZZ$-modules.
\item[\textbullet] Families $(\alpha_i \: \! ; \, i \in I)$ and ${({\check \alpha_i} \: \! ; \, i \in I)}$ with values in $X$ and $Y$, respectively. We demand that the maps $i \mapsto \alpha_i$ and $i \mapsto \check \alpha_i$ are injective.
\item[\textbullet] A perfect bilinear pairing $\langle \cdot \, , \cdot \rangle : Y \times X \to \ZZ$ such that $\langle {\check \alpha_i}, \alpha_j \rangle = C_{ij}$ for all $i,j \in I$.
\end{enumerate}
The $\ZZ$-modules $X$ and $Y$ are called the \emph{root lattice} and the \emph{coroot lattice}, respectively. We call $\alpha_i$ and ${\check \alpha_i}$ ($i \in I$) the \emph{simple roots} and the \emph{simple coroots}, respectively. Where there is no possibility for confusion, we normally abuse notation slightly and write $(X, Y, I)$ to designate the root datum $\Xdat$. The root datum is said of \emph{finite type} if the generalised Cartan matrix $C$ is of finite type. \\

We say that the root datum $\Xdat$ is \emph{$X \!$-regular} if the family of simple roots is linearly independent and \emph{$Y \!$-regular} if the family of simple coroots is. When $\Xdat$ is of finite type, $X \!$-regularity and $Y \!$-regularity are automatic by the non-degeneracy of $C$, but in general it is an additional assumption, which we make in this paper. \\

An element $\lambda$ in $X$ is said \emph{dominant} if $\langle {\check \alpha_i}, \lambda \rangle \geq 0$ for all $i \in I$. The subset of dominant elements in $X$ is denoted by $\Xdom$ and called the \emph{dominant cone}. \\

On the root lattice $X$, we define a partial order (remark that the antisymmetry axiom is satisfied thanks to the $X \!$-regularity):
$$ \lambda \ \leq \ \lambda' \quad \ \text{ if } \ \quad \lambda' - \lambda \ \in \ \sum_{i \in I} \, \NN \: \! \alpha_i \, . $$

\nomenclature[BI]{$I$}{Dynkin set \hfill \ref{para_KMmatrix}}
\nomenclature[Bdi]{${(d_i)}_{i \in I}$}{Symmetrising vector of $C$ \hfill \ref{para_KMmatrix}}
\nomenclature[CFinitetype1]{}{Finite type $\ $(generalised Cartan matrix) \hfill \ref{para_KMmatrix}}
\nomenclature[BR]{$\Xdat$}{Root datum \hfill \ref{para_KMroot}}
\nomenclature[BX]{$X$}{Root lattice \hfill \ref{para_KMroot}}
\nomenclature[Balphai]{$\alpha_i \ \ (i \in I)$}{Simple roots of $\Xdat$ \hfill \ref{para_KMroot}}
\nomenclature[BY]{$Y$}{Coroot lattice \hfill \ref{para_KMroot}}
\nomenclature[Balphaicheck]{$\check \alpha_i \ \ (i \in I)$}{Simple coroots of $\Xdat$ \hfill \ref{para_KMroot}}
\nomenclature[B<,>]{${\langle \, \cdot \, , \cdot \, \rangle}$}{Pairing between $Y$ and $X$ \hfill \ref{para_KMroot}}
\nomenclature[BX_dom]{$\Xdom$}{Dominant cone \hfill \ref{para_KMroot}}
\nomenclature[Blambdalessthanorequalto]{$\lambda \leq \lambda' \ \ (\lambda, \lambda' \in X)$}{Partial order on $X$ \hfill \ref{para_KMroot}}
\nomenclature[CFinitetype2]{}{$\phantom{Finite type }$ $ \, \; $(root datum) \hfill \ref{para_KMroot}}
\nomenclature[CXregularYregular]{}{$X$-regular, $Y$-regular \hfill \ref{para_KMroot}}
\nomenclature[Cdominant]{}{Dominant element in $X$ \hfill \ref{para_KMroot}}

\subsection{Kac-Moody algebras with and without Serre relations} \label{subs_KM}
The \emph{Kac-Moody algebra} $\glie = \glie (\Xdat)$ associated to the root datum $\Xdat$ is the Lie \hbox{$\KK$-algebra} generated by the elements $X^-_i, X^+_i$ ($i \in I$), the \hbox{$\ZZ$-module $Y$} and subject to the following relations:
\begin{equation} \label{eq_KM_nondeformable} \left\{ {\renewcommand{\arraystretch}{1.3} \begin{array}{rclr}
\big[ \check \mu, \check \mu' \big] & = & 0 & (\check \mu, \check \mu' \in Y) \, , \\
\big[ {\check \mu}, X^{\pm}_i \big] & = & \pm \; \! \langle {\check \mu} , \alpha_i \rangle \, X_i^{\pm} & (i \in I, \ \check \mu \in Y) \, , \\
\big[ X^+_i, X^-_j \big] & = & 0 & \quad \quad (i,j \in I, \ i \neq j) \, ,
\end{array}} \right. \end{equation}
\begin{equation} \label{eq_KM_deformable} \left\{ {\renewcommand{\arraystretch}{1.3} \begin{array}{lr}
\big[ X^+_i, X^-_i \big] = \check \alpha_i & (i \in I) \, , \\
\sum_{k+k' = 1 - C_{ij}} (-1)^k \, {1 - C_{ij} \choose k} \, {(X_i^{\pm})}^k \; \! X^{\pm}_j \; \!  {(X_i^{\pm})}^{k'} = 0 & \quad (i,j \in I, \ i \neq j) \, .
\end{array}} \right.
\end{equation}

The relations $\sum_{k+k' = 1 - C_{ij}} (-1)^k {1 - C_{ij} \choose k} {(X_i^{\pm})}^k X^{\pm}_j {(X_i^{\pm})}^{k'} = 0$ ($i, j \in I$, $i \neq j$) are called the \emph{Serre relations}. The elements $X^-_i, X^+_i$ ($i \in I$) are called the \emph{Chevalley generators} of $\glie$. \\

We denote by $\Unnc$ and $\Unpc$ the $\KK$-subalgebras of $\Uc$ generated by the elements $X^-_i$ ($i \in I$) and $X^+_i$ ($i \in I$), respectively. We denote by $\UCac$ the \hbox{$\KK$-subalgebra} of $\Uc$ generated by $Y$. The algebra $\Uc$, when viewed as a $\KK$-vector space, admits the following decomposition, where the isomorphism is given by the multiplication of $\Uc$:
$$ \Uc \ \simeq \ \, \Unnc \, \otimes_\KK \, \UCac \, \otimes_\KK \, \UCac \, . $$
We call this decomposition the \emph{triangular decomposition} of $\Uc$. \\


\emph{The Kac-Moody algebra without Serre relations} $\tilde \glie = \tilde \glie (\Xdat)$ associated to the root datum $\Xdat$ is the Lie \hbox{$\KK$-algebra} generated by the elements $X^-_i, X^+_i$ ($i \in I$), the \hbox{$\ZZ$-module $Y$}, subject to the relations \eqref{eq_KM_nondeformable} and to the following ones:
\begin{equation} \label{eq_KM_deformablet}
\big[ X^+_i, X^-_i \big] = \check \alpha_i \quad \quad (i \in I) \, .
\end{equation}

We denote by $\vpitc$ the projection map from $\Utc$ \hbox{to $\Uc$}. The algebra $\Utc$ also admits a triangular decomposition.

\nomenclature[BUKlessthan0]{$\Unnc$}{\hfill \ref{para_triangl}}
\nomenclature[BUKgreaterthan0]{$\Unpc$}{\hfill \ref{para_triangl}}
\nomenclature[BUK0]{$\UCac$}{\hfill \ref{para_triangl}}
\nomenclature[Ctriangulardecomposition]{}{Triangular decomposition of $\Uc$ \hfill \ref{para_triangl}}

\nomenclature[BgK(R)]{$\glie_{\KK} (\Xdat)$}{Kac-Moody algebra \hfill \ref{para_KM}}
\nomenclature[BXimoinsXiplus1]{$X^-_i, X^+_i \ \ (i \in I)$}{Chevalley generators of $\ \glie_\KK(\Xdat)$, $\Uc$ \hfill \ref{para_KM}}
\nomenclature[BUK(R)]{$\Uc$}{Enveloping algebra of $\glie_{\KK} (\Xdat)$ \hfill \ref{para_KM}}
\nomenclature[CSerrerelations]{}{Serre relations \hfill \ref{para_KM}}
\nomenclature[BgtildeK(R)]{$\tilde \glie_{\KK} (\Xdat)$}{Kac-Moody algebra without Serre relations \hfill \ref{para_KMwithout}}
\nomenclature[BXimoinsXiplus2]{}{$\phantom{Chevalley generators of } \, \; \tilde \glie_{\KK} (\Xdat)$, $\Utc$ \hfill \ref{para_KMwithout}}
\nomenclature[BUtildeK(R)]{$\Utc$}{Enveloping algebra of $\tilde \glie_{\KK} (\Xdat)$ \hfill \ref{para_KMwithout}}
\nomenclature[BpiK(R)]{$\vpitc$}{Projection map from $\Utc$ to $\Uc$ \hfill \ref{para_KMwithout}}

\subsection{The Lie algebra $\slt$ and the simple $\slt$-directions} \label{subs_KMone}
We designate by $\Aone$ the root datum defined by the following.
\begin{enumerate}
\item[\textbullet] The root and coroot lattices are both equal to $\ZZ$.
\item[\textbullet] The simple root and the simple coroot are $2 \in \ZZ$ and $1 \in \ZZ$, respectively.
\item[\textbullet] The pairing is the product of $\ZZ$.
\end{enumerate}
Note that the Kac-Moody algebra associated to $\Aone$ is naturally isomorphic to the Lie algebra of 2-by-2 matrices with coefficients in $\KK$ and whose trace is zero; we denote it by $\slt$. \\

The two Chevalley generators of $\slt$ are designated by $X^-$ and $X^+$, we besides denote by $H$ the simple coroot of $\Aone$. The Lie $\KK$-algebra $\slt$ is generated by the elements $X^-, X^+$, $H$ and subject to the following relations:
\begin{IEEEeqnarray}{rCl}
\label{eq_KM_nondeformableone} \big[ H, X^{\pm} \big] & \ = \ & \pm \; \! 2 \, X^{\pm} \, , \\
\label{eq_KM_deformableone} \big[ X^+, X^- \big] & = & H \, .
\end{IEEEeqnarray}
Remark that there is no Serre relations for $\slt$. \\

Let $i \in I$. There exist unique Lie $\KK$-algebra morphisms from $\slt$ to $\glie$ \hbox{and $\tglie$}, which \hbox{sends $X^{\pm}$} and $H$ to $X_i^{\pm}$ and $\check \alpha_i$, respectively. The induced $\KK$-algebra morphisms from $\Uc[\slt]$ to $\Uc$ and $\Utc$ are denoted by $\Deltac$ and $\Deltac[\tglie]$, respectively. The morphisms $\Deltac$ and $\Deltac[\tglie]$ ($i \in I$) are called the \emph{simple $\slt$-directions} \hbox{of $\Uc$} and $\Utc$, respectively.

\nomenclature[BA1]{$\slt$}{Root datum of rank one \hfill \ref{para_KMone}}
\nomenclature[BX_moinsX_plus]{$X^-, X^+$}{Chevalley generators of $\glie_\KK(\slt)$ \hfill \ref{para_KMone}}
\nomenclature[BH1]{$H$}{Simple coroot of $\glie_\KK(\slt)$ \hfill \ref{para_KMone}}
\nomenclature[BDeltaKi(R)]{$\Deltac \ \ (i \in I)$}{Simple $\slt$-directions of $\Uc$ \hfill \ref{para_KMdirections}}
\nomenclature[Bwt(V)]{$\wt (V)$}{Set of weights of a representation $V$ \hfill \ref{para_KMweight}}
\nomenclature[BVlambda]{$V_\lambda \ \ (\lambda \in X)$}{Weight spaces of a representation $V$ \hfill \ref{para_KMweight}}
\nomenclature[Cweight]{}{Weight vector, weight representation \hfill \ref{para_KMweight}}

\subsection{Highest weight representations and the category $\mathcal O$}
Let $V$ be a representation of $\glie$ and let $\lambda \in X$. A vector $v$ in $V$ is called a \emph{weight vector} of weight $\lambda$ if $\check \mu.v = \langle \check \mu, \lambda \rangle \; \! v$ for all $\check \mu \in Y$. The \emph{weight space} \hbox{of $V$} of \hbox{weight $\lambda$} is the $\KK$-vector subspace of $V$ whose vectors are of weight $\lambda$, we denote it \hbox{by $V_{\lambda}$}. Remark that the sum of the weight spaces of $V$ is direct. The representation $V$ is called a \emph{weight representation} \hbox{of $\glie$} \hbox{if $V$} is equal to the sum of its weights spaces.  If $V$ is a weight representation and if $V_\lambda$ is nonzero, we call $\lambda$ a \emph{weight} of $V$. The set of weights of a weight \hbox{representation $V$} is denoted \hbox{by $\wt (V)$}. \\

A representation $V$ of $\glie$ is called a \emph{highest weight representation} of $\glie$ of highest weight $\lambda$ if the following two conditions hold.
\begin{enumerate}
\item[\textbullet] The representation $V$ is generated by a nonzero vector $v_{\lambda}$ in $V_{\lambda}$. The vector $v_{\lambda}$ is called a \emph{highest weight vector} of $V$.
\item[\textbullet] The weights of $V$ are less than or equal to $\lambda$.
\end{enumerate}
Remark that a highest weight representation of $\glie$ is a weight representation \hbox{of $\glie$} with finite-dimensional weight spaces. \\

We denote by $\Ocatc$ the category of representations $V$ of $\glie$ which satisfy the following conditions.
\begin{enumerate}
\item[\textbullet] The representation $V$ is a weight representation of $\glie$ with finite-dimensional weight spaces.
\item[\textbullet] There exist finitely many elements $\lambda_1, \dots, \lambda_k$ in $X$ such that $\wt (V)$ is contained in $\bigcup_{s=1}^k (\lambda_s - \sum_{i \in I} \NN \alpha_i)$.
\end{enumerate}


\nomenclature[BOK(R)]{$\Ocatc$}{Category $\mathcal O$ for $\ \Uc$ \hfill \ref{para_KMcatO}}

\subsection{Some known results} \label{subs_KMresults}
A weight representation $V$ of $\glie$ is called \emph{integrable} if the actions of the Chevalley generators on $V$ are locally nilpotent. We denote by $\Ointc$ the category of integrable representations of $\glie$. When the root datum $\Xdat$ is of finite type (equivalently, when $\glie$ is finite-dimensional), a representation of $\glie$ is in $\Ointc$ if and only if it is finite-dimensional. \\

Let us recall some results on the representation theory of $\glie$.
\begin{enumerate}
\item[\textbullet] For every $\lambda \in X$, there is a unique, up to isomorphism, irreducible representation $\Lc{\lambda}$ of $\glie$ of highest weight $\lambda$.
\item[\textbullet] Let $\lambda \in \Xdom$. The representation $\Lc{\lambda}$ is generated by a vector $v_{\lambda}$ and subject to the following relations:
$$ \check \mu . v_{\lambda} \, = \, \langle \lambda, \check \mu \rangle \; \! v_{\lambda} \, , \quad X^+_i . v_{\lambda} = 0 \, , \quad (X^-_i)^{\langle \lambda, \check \alpha_i \rangle +1} . v_{\lambda} \, = \, 0  \quad \ (i \in I, \ \check \mu \in Y) \, . $$
\item[\textbullet] The representations $\Lc{\lambda}$ ($\lambda \in X$) are pairwise distinct and are the unique, up to isomorphism, irreducible representations in $\Ocatc$.
\item[\textbullet] Let $\lambda,\lambda' \in X$. Every $\glie$-morphism from $\Lc{\lambda}$ to $\Lc{\lambda'}$ is equal to a scalar multiple of the identity map if $\lambda = \lambda'$ and is zero else.
\item[\textbullet] Let $\lambda \in X$. The representation $\Lc{\lambda}$ is integrable if and only if $\lambda \in \Xdom$.
\item[\textbullet] Every representation in the category $\Ointc$ is isomorphic to a direct sum of representations $\Lc{\lambda}$ with $\lambda \in \Xdom$.
\end{enumerate}

\vsp

We give below, for the reader's convenience, a proof of a complete-reducibility type result for the category of integrable representations of $\slt$.

\begin{prop} \label{prop_KMintone}
An integrable representation of $\slt$ is isomorphic to a direct sum of representations $\Lc{n}$ with $n \in \NN$. Besides, a representation of $\slt$ is integrable if and only if the actions of $X^-,X^+,H$ are locally finite.
\end{prop}

\proof
Let $V$ be a representation of $\slt$ on which the actions of $X^-,X^+,H$ are locally finite and let $v$ be a vector in $V$. Denote by $V'$ the subrepresentation of $V$ generated by $v$. In view of the triangular decomposition of $\slt$ and since the actions of $X^-,X^+,H$ on $V$ are locally finite, the subrepresentation $V'$ is finite-dimensional. Therefore, $V'$ is a sum of representations $\Lc{n}$ with $n \in \NN$. This being true for \hbox{every $v$} in $V$, the representation $V$ is a sum of irreducible representations $\Lc{n}$ with $n \in \NN$. Hence, the representation $V$ is a direct sum of irreducible representations $\Lc{n}$ with $n \in \NN$.
\qed \\

The author ignores whether there exists (when $\glie$ is not finite-dimensional) a reference for the following proposition. The proof given below follows from a private communication with Victor Kac. Note that it is possible in the case \hbox{of $\slt$} to give a more direct and elementary proof.

\begin{prop} \label{prop_UBperfect}
An element in $\Uc$ is zero if and only if it acts as zero on all the representations of $\glie$ in $\Ointc$.
\end{prop}

\proof
Consider the ${X \!}$-gradation of the $\KK$-algebra $\Uc$ defined by the following:
$$ \mathrm{deg} \; \! (\mu) \, := \, 0 \, , \quad \quad \mathrm{deg} \; \!  (X^\pm_i) \, := \, \alpha_i \quad \quad \quad (\mu \in X, \ i \in I) \, . $$
Let $x$ be a nonzero element in $\Uc$. Recall the triangular decomposition of $\Uc$ and pick nonzero homogeneous \hbox{elements $x^-_k$, $x^0_k$} \hbox{and $x^+_k$} \hbox{in $\Unnc$}, $\UCac$ and $\Unpc$ respectively \hbox{($1 \leq k \leq l$)} such that the following conditions are satisfied:
\begin{enumerate}
\item[a)] $x = \sum_{1 \leq k \leq l} x^-_k \; \! x^0_k \; \! x_k^+$,
\item[b)] the family ${(x^-_k)}_{1 \leq k \leq l}$ is $\KK$-linearly independent,
\item[c)] $\deg (x_1^+) \ngtr \deg (x_k^+)$ for every $1 < k \leq l$.
\end{enumerate}
Denote by $\tau$ the involutive anti-automorphism of the $\KK$-algebra $\Uc$ which sends $X_i^\pm$ \hbox{to $-X_i^\mp$} for every $i \in I$ and which fixes pointwise $Y$. Pick $\lambda$ in $\Xdom$ such that the following conditions hold (we denote by $v_{\lambda}$ a fixed choice of highest weight vector of the representation $\Lc{\lambda}$):
\begin{enumerate}
\item[a')] $x^0_1 (\lambda)$ is nonzero,
\item[b')] the family ${(x^-_k . v_{\lambda})}_{1 \leq k \leq l}$ is $\KK$-linearly independent,
\item[c')] the vector $\tau(x^+_1) . v_{\lambda}$ is nonzero.
\end{enumerate}
Pick now a nonzero $\KK$-bilinear form $\sigma$ on $\Lc{\lambda}$ which has the following properties:
\begin{enumerate}
\item[a'')] $\sigma$ is non-degenerate,
\item[b'')] $\sigma(v,w) = 0$ for every weight vectors $v,w$ in $\Lc{\lambda}$ of distinct weights,
\item[c'')] $\sigma$ is contravariant, i.e. $\sigma (y.v,w) = \sigma(v,\tau(y).w)$ for every $v,w \in \Lc{\lambda}$ and for every $y \in \Uc$.
\end{enumerate}
In view of the non-degeneracy of $\sigma$ and in view of the orthogonality property $b'')$, pick a weight vector $v$ in $\Lc{\lambda}$ of weight $\lambda - \deg(x_1^+)$ such that $\sigma(v, \tau(x^+_1).v_{\lambda})$ is nonzero. Since $\sigma$ is contravariant, the vector $x_1^+.v$ is thus a nonzero scalar multiple of $v_{\lambda}$. Thanks to condition $c)$ above, the following holds:
$$ x.v \ = \ \sum_{\substack{1 \leq k \leq l, \\ \deg(x^+_k) \, = \, \deg (x_1^+)}} x^0_k(\lambda) \, x^-_k . (x^+_k . v) \, . $$
In view of conditions $a')$ and $b')$, one concludes that $x.v$ is nonzero.
\qed

\nomenclature[BOKint(R)]{$\Ointc$}{Category $\mathcal O^{\mathrm{int}}$ for $\Uc$ \hfill \ref{para_KM}}
\nomenclature[Cintegrablerepresentation]{}{Integrable representation \hfill \ref{para_KMint}}
\nomenclature[BLK(R,lambda)]{$\Lc{\lambda} \ \ (\lambda \in \Xdom)$}{Irreducible representations in $\Ointc$ \hfill \ref{para_KMresults}}
\nomenclature[BChK(R,V)1]{$\Chc{V}$}{Character of a representation $V$ of $\Uc$ \hfill \ref{para_KMcharacter}}
\nomenclature[BChK(R,V)2]{$\Ch{\lambda}$}{Character of the representation $\Lc{\lambda}$ \hfill \ref{para_KMcharacter}}

\section{Modules and algebras over a power series ring}
\subsection{The power series ring $\Kh$}
We denote by $\Kh$ the \emph{power series ring} in the variable $h$ over $\KK$. We use the following notation to describe an element $\lambda_h$ in $\Kh$:
$$ \lambda_h \ = \ \sum_{n \in \NN} \, \lambda_m \, h^m \, ; \quad \text{ with } \ \lambda_m \in \KK \, , \ \ \forall \, m \in \NN \, . $$
The formal power series $\lambda_h$ is invertible in $\Kh$ if and only if the scalar $\lambda_0$ is nonzero.

\nomenclature[Dlambdah=sum]{$\lambda_h = \sum_{m \in \NN} \lambda_m h^m$}{Decomposition of $\lambda_h$ in the ring $\Kh$ \hfill \ref{para_Kh}}

\subsection{$h$-adic topology}
Let $V$ be a $\Kh$-module. The submodules ${h^m \: \! V}$ ($m \in \NN$) form a local basis at the origin of a linear topology, called the \emph{$h$-adic topology} on the $\Kh$-module $V$. \textbf{In this first part of the thesis, topological statements are always made with regards to the $h$-adic topology.} We recall below some definitions and facts about the $h$-adic topology. \\

A $\Kh$-linear map is automatically continuous. In particular, the action by every element in the ring $\Kh$ on a $\Kh$-module is continuous. The closure in a $\Kh$-module of a $\Kh$-submodule is again a $\Kh$-submodule. \\

Let $A$ be a $\Kh$-algebra, let $\Lambda$ be a set and let $x_l, y_l$ ($l \in \Lambda$) be elements \hbox{in $A$}. The representation of $A$ \emph{topologically generated} by the vector $v$ and subject to the set of relations ${\{ x_l . v = y_l .v \; \! ; \, l \in \Lambda \}}$ is the quotient of the left regular representation of $A$ by the closure of the subrepresentation generated by the vectors $x_l - y_l$ ($l \in \Lambda$). \\


The space $V$ is \emph{separated} if and only if $\bigcap_{n \in \NN} {h^m \: \! V} = (0)$, i.e. if and only if the submodule $(0)$ is closed in $V$. As a consequence, a $\Kh$-submodule $W$ is closed in $V$ if and only if the $\Kh$-module $V / W$ is separated. \\

A sequence ${(v_l)}_{l \in \NN}$ with values in $V$  is said to be \emph{of Cauchy} if it satisfies the following condition:
$$ \forall \, m \in \NN \, , \ \exists \, l(m) \in \NN \, , \ \forall \, l \geq l(m) \, , \quad  (v_{l+1} \, - \, v_l) \ \in \ h^m \; \! V \, . $$
The $\Kh$-module $V$ is \emph{complete} if all the Cauchy sequences with values in $V$ are convergent. Note that a quotient of a complete $\Kh$-module is also complete. The \emph{completion} of $V$ is the unique (up to a unique isomorphism) complete and separated $\Kh$-module which contains $V$ as a dense $\Kh$-submodule. \\

The \emph{topologically direct sum} of a family ${(V_l)}_{l \in \Lambda}$ of $\Kh$-modules is the completion of the usual direct sum $\bigoplus_{l \in \Lambda} V_l$.

\nomenclature[Ehadic]{}{$h$-adic topology \hfill \ref{para_adic}}
\nomenclature[EVanishatinf]{}{Vanish at infinity \hfill \ref{para_adic}}




\nomenclature[Ehseparated]{}{separated \hfill \ref{para_hseparated}}




\nomenclature[EhCauchy]{}{Cauchy sequence \hfill \ref{para_complete}}
\nomenclature[Ehcomplete]{}{complete \hfill \ref{para_complete}}




\nomenclature[Ehtorsion]{}{torsion-free \hfill \ref{para_htorsionfree}}

\subsection{Topologically free $\Kh$-modules}
Let $V_0$ be a $\KK$-vector space. We denote by $V_0 \cch$ the $\KK$-vector space of functions from $\NN$ to $V_0$. We use the following notation to describe an element $v_h$ in $V_0 \cch$\nomenclature[Dvh=sum]{$v_h = \sum_{m \in \NN} v_m h^m$}{Decomposition of $v_h$ in a module $V \cch$ \hfill \ref{para_hfree}}:
$$ v_h \ = \ \sum_{m \in \NN} \, v_m \, h^m \, , \quad \text{ where } \ v_m \ := \ v_h (m) \quad (m \in \NN) \, . $$
We endow $V_0 \cch$ with a structure of $\Kh$-module:
$$ \lambda_h \, v_h \ := \ \sum_{m \in \NN} \left( \ \sum_{p + q = m} \, \lambda_p \, v_q \, \right) h^m \quad \ \ \big( \lambda_h \in \Kh \, , \ v_h \in V_0 \cch \big) \, . $$
We identify $V_0$ in a canonical way with a $\KK$-submodule of $V_0 \cch$. \\

A $\Kh$-module $V$ is said \emph{topologically free} if there exists a $\KK$-vector space $V_0$ and a $\Kh$-linear isomorphism $f$ from $V_0 \cch$ to $V$. The image by $f$ of a basis of the $\KK$-vector space $V_0$ is called a \emph{basis} of topologically free $\Kh$-module $V$. A topologically free $\Kh$-module is torsion-free, separated and complete. The converse is also true: see for example \cite[proposition XVI.2.4]{Kas}. \\

Let $f_0 : V_0 \to V_0'$ be a $\KK$-linear map. We denote by $f_0 \cch$ the $\Kh$-linear map from $V_0 \cch$ to $V'_0 \cch$ which makes the following diagram commute.
$$ \shorthandoff{;:!?} \xymatrix @R=1.5pc @C=3.5pc {
V_0 \cch \ar[r]^-{f_0 \cch} & V_0' \cch \\
V_0 \ar[u] \ar[r]_-{f_0} & \, V'_0 \ar[u]
} $$

Let $A_0$ be a $\KK$-algebra. We endow the $\Kh$-module $A_0 \cch$ with a structure of $\Kh$-algebra:
$$ x_h \, x'_h \ := \ \sum_{n \in \NN} \left( \ \sum_{p + q = n} \, x_p \, x_q \, \right) h^m \quad \ \ \big( x_h, x'_h \in A_0 \cch \big) \, . $$
Remark that $A$ is a $\KK$-subalgebra of $A \cch$. Remark also that if $f : A \to A'$ is a $\KK$-algebra morphism, then $f \cch : A \cch \to A' \cch$ is a $\Kh$-algebra morphism. \\

Let $\Lambda_0, \Lambda$ be two sets and let $x_l, y_l$ ($l \in \Lambda$) be elements in the algebra $(\KK \Lambda_0) \cch$. The $\Kh$-algebra \emph{topologically generated} by the set $\Lambda_0$ and subject to the set of relations ${\{ x_l = y_l \; \! ; \, l \in \Lambda \}}$ is the quotient of the $\Kh$-algebra $(\KK \Lambda_0) \cch$ by the closure of the two-sided ideal generated by the elements $y_l - x_l$ ($l \in \Lambda$). \\

Let $V_0$ be a representation of $A_0$. We endow the $\Kh$-module $V_0 \cch$ with a structure of representation of $A_0 \cch$:
$$ x_h \! \: . \! \: v_h \ := \ \sum_{n \in \NN} \left( \ \sum_{p + q = n} \, x_p \! \: . \! \: v_q \, \right) h^m \quad \ \ \big( x_h \in A_0 \cch, \ v_h \in V_0 \cch \big) \, . $$

\nomenclature[DV((h))]{$V \cch$}{\hfill \ref{para_Vhh}}
\nomenclature[DA((h))]{$A \cch$}{\hfill \ref{para_Vhh}}
\nomenclature[Df((h))]{$f \cch : V \cch \to V' \cch$}{\hfill \ref{para_Vhh}}





\nomenclature[Ehfree]{}{topologically free \hfill \ref{para_hfree}}
\nomenclature[Drkh(V)]{$\mathrm{rk}_h (V)$}{$h$-rank of a topologically free $\Kh$-module $V$ \hfill \ref{para_hfree}}



\nomenclature[EspecialisationofanelementinK(u)]{}{Specialisation of an element in $\KK[u] \cch$\hfill \ref{para_KKucch}}
\nomenclature[EspecialisationofanelementinK(u,v)]{}{Specialisation of an element in $\KK[u,v] \cch$\hfill \ref{para_KKucch}}



\nomenclature[Drkh(V)]{$\mathrm{rk}_h(V)$}{Rank of a free $\Kh$-module $V$ \hfill \ref{para_hrep}}
\nomenclature[Ehrepresentation]{}{representation \hfill \ref{para_hrep}}
\nomenclature[Eindecomposablehrepresentation]{}{Indecomposable representation \hfill \ref{para_hrep}}

\nomenclature[DsumlinLambdah]{$\sum_{l \in \Lambda}^h V_l$}{$h$-sum over $\Lambda$ of $\Kh$-submodules \hfill \ref{para_hspan}}
\nomenclature[Ehspan]{}{$h$-span \hfill \ref{para_hspan}}
\nomenclature[Edirecthsum]{}{Direct $h$-sum \hfill \ref{para_hspan}}
\nomenclature[Ehindependent]{}{$h$-independent family \hfill \ref{para_hspan}}
\nomenclature[Ehbasis]{}{topological basis \hfill \ref{para_hspan}}

\subsection{Specialisation at ${h \! = \! 0}$}
Let $V$ be a $\Kh$-module. The $\KK$-vector space $V / h V$ is denoted by $\hzero V$ and we call it the \emph{specialisation at ${h \! = \! 0}$} of $V$. We denote by $\hzero \pi^V$ the projection map from $V$ to $\hzero V$. Let $V_0$ be a $\KK$-vector space, we identify in a canonical way the $\KK$-vectors spaces $\hzero{V_0 \cch}$ and $V_0$. \\

Let ${f \! : \! V \! \to \! V'}$ be a $\Kh$-linear map. The \emph{specialisation at ${h \! = \! 0}$} of $f$ is the $\KK$-linear map $\hzero f$ from $\hzero V$ to $\hzero{V'}$ induced by $f$. Here are some facts.
\begin{enumerate}
\item[\textbullet] Assume that $V$ is complete and that $V'$ is separated. The map $f$ is surjective if and only if $\hzero f$ is surjective.
\item[\textbullet] Assume that $V$ is separated and that $V'$ is torsion-free. If the map $\hzero f$ is injective, then $f$ is injective. The converse statement is false in general (consider for example the action of $h$ on $V'$).
\item[\textbullet]  Assume that $V$ is topologically free. A family ${(v_l)}_{l \in \Lambda}$ with values \hbox{in $V$} is a basis of the topologically free $\Kh$-module $V$ if and only if $(\hzero \pi^V{(v_l))}_{l \in \Lambda}$ is a basis of the $\KK$-vector space $\hzero V$.
\end{enumerate}


\nomenclature[DVhequals0]{$\hzero V$}{Specialisation at ${h \! = \! 0}$ of a $\Kh$-module $V$ \hfill \ref{para_hzeromod}}
\nomenclature[Dfhequals0]{$\hzero f$}{Specialisation at ${h \! = \! 0}$ of a map $f$ \hfill \ref{para_hzeromod}}

\vsp
\vsp
\vsp

Let $A$ be a $\Kh$-algebra. We regard $\hzero A$ as a $\KK$-algebra and we call it the \emph{specialisation at ${h \! = \! 0}$} of the $\Kh$-algebra $A$. The projection \hbox{map $\hzero \pi^A$} from $A$ to $\hzero A$ is then a $\KK$-algebra morphism. Let $A_0$ be a $\KK$-algebra, remark that the \hbox{$\KK$-algebras $\hzero{A_0 \cch}$} and $A_0$ are equal. Let $g : A \to A'$ be a $\Kh$-algebra morphism, remark that the map $\hzero g$ is a $\KK$-algebra morphism. Let $V$ be a representation of $A$. We regard $\hzero V$ as a representation of $\hzero A$ and we call it the \emph{the specialisation at ${h \! = \! 0}$} of the representation $V$. Here are some conventions.

\begin{enumerate}
\item[\textbullet] The specialisation at ${h \! = \! 0}$ of the $\Kh$-algebra \emph{topologically generated} by a set $\Lambda_0$ and subject to a set of relations ${\{ x_l = y_l \; \! ; \, l \in \Lambda \}}$ is identified with the quotient of the $\KK$-algebra $\KK \Lambda_0$ by the two-sided ideal generated by the elements $\hzero \pi^{(\KK \Lambda_0) \cch} (y_l - x_l)$ ($l \in \Lambda'$).
\item[\textbullet] The specialisation at ${h \! = \! 0}$ of the representation of $A$ \emph{topologically generated} by the vector $v$ and subject to a set of relations ${\{ x_l . v = y_l . v \; \! ; \, l \in \Lambda \}}$ is identified with the quotient of the left regular representation of $\hzero A$ by the subrepresentation generated by the vectors $\hzero \pi^A (y_l - x_l)$ ($l \in \Lambda$).
\end{enumerate}

\nomenclature[DAhequals0]{$\hzero A$}{Specialisation at ${h \! = \! 0}$ of a $\Kh$-algebra $A$ \hfill \ref{para_hzeroalg}}
\nomenclature[Dpihequals0V]{$\hzero \pi^V$}{Projection map from $V$ to $\hzero V$ \hfill \ref{para_hzeroalg}}

\section[Pointed diagrams of algebras and formal deformations]{Pointed diagrams of algebras \\ and formal deformations}
Let $R$ be a $\KK$-algebra. We denote by $\AlgR$ the category of $R$-algebras and by ${\FF_{\! R}}$ the forgetful functor \hbox{from $\AlgR$} to $\Algc$. \\

We denote \hbox{by $\circ$} the horizontal composition for natural transformations and \hbox{by $\ast$} the vertical composition. We denote also by $\circ$ the composition for functors. \\

We denote by $\Cat$ the category of categories (the morphisms are the functors). We denote by $\Fun(\mathscr C',\mathscr C'')$ the category of functors from a category $\mathscr C'$ to another category $\mathscr C''$ (the morphisms are the natural transformations).

\nomenclature[FR]{$R$}{$\KK$-algebra \hfill \ref{para_diagramnot}}
\nomenclature[FAR]{$\AlgR$}{Category of $R$-algebras \hfill \ref{para_diagramnot}}
\nomenclature[FF]{${\FF_{R}}$}{Forgetful functor from $\AlgR$ to $\AlgR$ \hfill \ref{para_diagramnot}}
\nomenclature[F01]{$\circ$}{Horizontal composition for nat. transform. \hfill \ref{para_diagramnot}}
\nomenclature[F02]{}{Composition for functors \hfill \ref{para_diagramnot}}
\nomenclature[F*]{$\ast$}{Vertical composition for nat. transform. \hfill \ref{para_diagramnot}}
\nomenclature[FFun(C',C'')]{$\Fun(\mathscr C',\mathscr C'')$}{Category of functors from $\mathscr C$ to $\mathscr C'$ \hfill \ref{para_diagramnot}}
\nomenclature[FCat]{$\Cat$}{Category of categories \hfill \ref{para_diagramnot}}

\subsection{Pointed diagrams of algebras}
Let $\Shape$ be a category. A \emph{$R$-diagram} (of algebras) of shape $\Shape$ is a functor \hbox{from $\Shape$} to $\AlgR$. When there is no risk of confusion, we do not specify the shape of a diagram.

\begin{defi}
Let $\XX$ be a $\KK$-diagram. We say that a $R$-diagram $\AA$ is \emph{$\XX$-pointed} if it is endowed with a natural \hbox{transformation $\dot{\AA}$} from $\XX$ \hbox{to ${\FF_{\! R}} \circ \AA$}. We call $\dot{\AA}$ the \emph{$\XX$-point} \hbox{of $\AA$}.
\end{defi}

The identity map on a $\KK$-diagram $\XX$ endows $\XX$ with a structure of $\XX$-pointed $\KK$-diagram. We always refer to this canonical structure, when dealing with $\XX$ as a $\XX$-pointed $\KK$-diagram. \\

Let $\AA, \AA'$ be two $\XX$-pointed $R$-diagrams. A \emph{morphism of $\XX$-pointed $R$-diagram} from $\AA$ to $\AA'$ is a natural transformation $\alpha$ from $\AA$ to $\AA'$ such that $({\FF_{\! R}} \circ \alpha) \ast \dot{\AA}$ and $\dot{\AA}'$ are equal. \\

We denote by $\pAlgR[\XX]$ the category of $\XX$-pointed $R$-diagrams. The morphisms of $\pAlgR[\XX]$ are the morphisms of $\XX$-pointed $R$-diagrams, their composition is the natural one. \\

Remark that $\pAlgR[\XX]$ is a comma category. Remark also that $\pAlgR[\bullet]$ defines a contravariant functor from $\Fun(\Shape,\Algc)$ to $\Cat$. Let $\XX'$ be another $\KK$-diagram and let $\frac{\XX'}{\XX}$ be a natural transformation from $\XX'$ to $\XX$. The following holds.
\begin{enumerate}
\item[1)] The (covariant) functor $\pAlgR[\frac{\XX'}{\XX}]$ is faithful,
\item[2)] If $\frac{\XX'}{\XX}$ is an epimorphism in $\Fun(\Shape,\Algc)$, then the functor $\pAlgR[\frac{\XX'}{\XX}]$ is full.
\end{enumerate}

\nomenclature[FApoint]{$\dot{\AA}$}{$\XX$-point of a $\XX$-pointed $R$-diagram $\AA$ \hfill \ref{para_pdiagram}}
\nomenclature[FApointR(X)]{$\pAlgR[\XX]$}{Category of $\XX$-pointed $R$-diagrams \hfill \ref{para_pdiagram}}
\nomenclature[GRdiagram of (algebras)]{}{$R$-diagram (of algebras) of shape $\Shape$ \hfill \ref{para_pdiagram}}
\nomenclature[GmorphismofXXpointed]{}{Morphism of $\XX$-pointed $R$-diagram \hfill \ref{para_pdiagram}}

\subsection{Formal deformation of pointed diagrams}
Let $\AA$ be a $\Kh$-diagram. The $\KK$-diagram $\hzero \bullet \circ \AA$ is denoted by $\hzero{\AA}$ and called the \emph{specialisation at ${h \! = \! 0}$ of the $\Kh$-diagram} $\AA$. \\

Denote by $\hzero \varpi$ the natural transformation from ${\FF_{\! \! \: \Kh}}$ to $\hzero \bullet$ defined by the following: let $A$ be a $\Kh$-algebra, the $\KK$-algebra morphism ${{(\hzero \varpi)}_{\! A}}$ is the projection map $\hzero \pi^A$. \\

Let $\XX$ be a $\KK$-diagram and let $\dot{\AA}$ be a $\XX$-point of $\AA$. We \hbox{endow $\hzero{\AA}$} with the \hbox{$\XX$-point} $(\hzero \varpi \circ \AA) \ast \dot{\AA}$ and we call the $\XX$-pointed \hbox{$\KK$-diagram} $\hzero{\AA}$ the \emph{specialisation at ${h \! = \! 0}$} of the $\XX$-pointed $\Kh$-diagram $\AA$. \\

Let $\alpha : \AA \Rightarrow \AA'$ be a natural transformation between two \hbox{$\Kh$-diagrams}. We denote by $\hzero \alpha$ the natural transformation $\hzero \bullet \circ \alpha$ from $\hzero{\AA}$ to $\hzero{\AA'}$ and we call it the \emph{specialisation at ${h \! = \! 0}$} of the natural transformation $\alpha$. Remark that the specialisation at ${h \! = \! 0}$ of a morphism of $\XX$-pointed $\Kh$-diagram is a morphism of $\XX$-pointed $\KK$-diagram.

\begin{defi}
Let $\AA_0$ be a $\XX$-pointed $\KK$-diagram of shape $\Shape$. A formal deformation of the $\XX$-pointed \hbox{$\KK$-diagram} $\AA_0$ is a $\XX$-pointed \hbox{$\Kh$-diagram $\AA_h$} which satisfy the following two conditions.
\begin{enumerate}
\item[\textbullet] The $\Kh$-module $\AA_h(s)$ is topologically free for every object $s$ of $\Shape$.
\item[\textbullet] The $\XX$-pointed $\KK$-diagrams $\hzero{(\AA_h)}$ and $\AA_0$ are isomorphic.
\end{enumerate}
\end{defi}

\nomenclature[FAhequals02]{$\hzero{\AA}$}{Eval. at ${h \! = \! 0}$ of a point. $\Kh$-diagram $\AA$ \hfill \ref{para_hzerodiagram}}
\nomenclature[Falphahequals0]{$\hzero \alpha$}{Specialisation at ${h \! = \! 0}$ of a nat. transform. $\alpha$ \hfill \ref{para_hzerodiagram}}
\nomenclature[Gformaldeformation]{}{Formal deformation of a pointed diagram \hfill \ref{para_pdiagdeform}}





\subsection{Trivial formal deformations}
Denote by $\iota$ the natural transformation from the identity functor on $\Algc$ to the functor ${\FF_{\! \! \: \Kh}} \circ ({\bullet \: \! \cch})$, defined by the following: let $A$ be a $\KK$-algebra, the \hbox{$\KK$-algebra} \hbox{morphism ${\iota_{\! A}}$} is the inclusion map from $A$ to $A \cch$. \\

The \emph{constant formal deformation} $\AA_0 \cch$ of a $\XX$-pointed $\KK$-diagram $\AA_0$ is the $\Kh$-diagram $({\bullet \: \! \cch}) \circ \AA_0$, endowed with the $\XX$-point \hbox{$(\iota \circ \AA_0) \ast \dot{\AA}_0$}.

\begin{defi}
Let $\XX, \XX'$ be two $\KK$-diagrams of shape $\Shape$, let $\frac{\XX'}{\XX}$ be a natural transformation \hbox{from $\XX'$} to $\XX$ and let $\AA_h$ be a formal deformation of a $\XX$-pointed $\KK$-diagram $\AA_0$.
\begin{enumerate}
\item[1)] An isomorphism $\alpha_h$ of \hbox{$\XX'$-pointed} \hbox{$\Kh$-diagram} \hbox{from $(\pAlgh[\frac{\XX'}{\XX}]) (\AA_0 \cch)$} \hbox{to $(\pAlgh[\frac{\XX'}{\XX}]) (\AA_h)$} whose specialisation \hbox{at ${h \! = \! 0}$} is a morphism of \hbox{$\XX$-pointed} \hbox{$\KK$-diagram} is called a $\frac{\XX'}{\XX}$-trivialisation of the formal deformation $\AA_h$.
\item[2)] The formal deformation $\AA_h$ is said $\frac{\XX'}{\XX}$-trivial if it admits a $\frac{\XX'}{\XX}$-trivialisation.
\item[3)] The formal deformation $\AA_h$ is said locally $\frac{\XX'}{\XX}$-trivial if for every object $s$ of $\Shape$, the formal deformation $\AA_h(s)$ of the $\XX(s)$-pointed \hbox{$\KK$-algebra $\AA_0(s)$} is $\frac{\XX'}{\XX}(s)$-trivial.
\end{enumerate}
\end{defi}

\subsection{Formal deformation of representations}
Denote by $\mathbf 1$ a category with one element and one morphism. We identify in a canonical way the category $\Fun (\mathbf 1, \AlgR)$ with the category $\AlgR$: a $R$-diagram of shape $\mathbf 1$ is a $R$-algebra. Thanks to this identification, one can view $R$-algebras as a particular case of $R$-diagrams. In particular, the previous subsections yield the notions of \emph{$D$-pointed $R$-algebras}, where $D$ is a $\KK$-algebra, and of \emph{formal deformations of pointed algebras}. We use the following notational convention: an element $x$ in $D$ can be employed to designate the element $\dot A(x)$ in $A$. \\

Let $A_0$ be a $D$-pointed $\KK$-algebra, let $A_h$ be a formal deformation of $A_0$ and \hbox{let $V_0$} be a representation of $A_0$. We assume that the $D$-point $\dot A_0$ of $A_0$ is surjective and we denote by $f_0$ the unique morphism of $D$-pointed $\KK$-algebra from $A_0$ to the specialisation of $A_h$ at ${h \! = \! 0}$.

\begin{defi}
A formal deformation of $V_0$ along $A_h$ is a representation ${V_{\! \! \; h}}$ of $A_h$ which satisfies the following two conditions.
\begin{enumerate}
\item[\textbullet] The $\Kh$-module ${V_{\! \! \; h}}$ is topologically free.
\item[\textbullet] The pull-back of $\hzero{(V_h)}$ by $f_0$ is isomorphic to $V_0$.
\end{enumerate}
\end{defi}

Two formal deformations along $A_h$ of $V_0$ are said \emph{isomorphic} if they are isomorphic representations of $A_h$.

\nomenclature[GXpointedRalgebra]{}{$X$-pointed $R$-algebra \hfill \ref{para_palg}}

\subsection{Relative Hochschild cohomology}
Let $A_0$ be a $\KK$-algebra, let $B_0$ be a $\KK$-subalgebra of $A_0$, let $M$ be a $A_0$-bimodule and let $n \in \NN$. We denote by $C^n(A_0,M)$ the $\KK$-vector space formed by the \hbox{$\KK$-multilinear} maps from $(A_0)^n$ to $M$ and we denote by $C^n(A_0,B_0,M)$ the \hbox{$\KK$-vector} subspace of $C^n(A_0,M)$ formed by the maps $f$ for which the following condition holds:
\begin{multline*}
f(y_1 \! \: x_1, y_2 \! \: x_2, \dots, y_n \! \: x_n) . \! \: y_{n+1} \ = \ y_1 \! \: . f(x_1 \! \: y_2, \dots, x_{n-1} \! \: y_n, x_n \! \: y_{n+1}) \, , \\
\forall \, x_1, \dots, x_n \in A_0 \, , \ \forall \, y_1, \dots, y_{n+1} \in B_0 \, .
\end{multline*}
We denote by $d_n$ the $\KK$-linear map from $C^n(A_0,M)$ to $C^{n+1}(A_0,M)$ defined by the following for $f \in C^n(A_0,M)$ and for $x_1, \dots, x_{n+1} \in A_0$:
\begin{IEEEeqnarray*}{rCCl}
\big[ d_n(f) \big] \! \: (x_1, \cdots, x_{n+1}) & \ := && x_1 \! \: . f(x_2, \cdots, x_{n+1}) \\
&& + \ & \sum_{l=1}^{n-1} f(x_1, \dots, x_{l-1}, x_l \! \: x_{l+1}, x_{l+2}, \cdots, x_n) \\
&& + \ & f(x_1, \cdots, x_n) . \! \: x_{n+1} \, .
\end{IEEEeqnarray*}
The $\KK$-vector spaces $C^n(A_0,M)$ ($n \in \NN$) and the maps $d_n$ ($n \in \NN$) form a complex $C^\bullet(A_0,M)$. The $\KK$-vector subspaces $C^n(A_0,B_0,M)$ ($n \in \NN$) form a subcomplex $C^\bullet(A_0,B_0,M)$. The cohomology $H^\bullet(A_0,B_0,M)$ of $C^\bullet(A_0,B_0,M)$ is called the \emph{relative Hochschild cohomology} of the pair $(A_0,B_0)$ with coefficients in $M$. \\

Let $\mu_h$ be a $\Kh$-algebra product on the $\Kh$-module $A_0 \cch$. We denote then \hbox{by $A_0 [| \mu_h |]$} the $\Kh$-algebra whose underlying $\Kh$-module \hbox{is $A_0 \cch$} and whose product is equal to $\mu_h$. We denote \hbox{by $\mu_m$} ($m \in \NN$) the maps in $C^2(A_0,A_0)$ such that the following holds:
$$ \mu_h(x_1,x_2) \ = \ \sum_{m \in \NN} \mu_m(x_1, x_2) \! \; h^m \, , \quad \quad \ \forall \, x_1, x_2 \in A_0 \, . $$
We assume that the product $\mu_h$ satisfies the following two conditions.
\begin{enumerate}
\item[\textbullet] The map $\mu_0$ is equal to the product of the $\KK$-algebra $A_0$.
\item[\textbullet] The map $\mu_m$ belongs to $C^2(A_0,B_0,A_0)$ for every $m \in \NN$.
\end{enumerate}

\vsp
\vsp
\vsp

Let now $V_0$ be a representation of $A_0$ and let ${V_{\! \! \; h}}, {V'_{\! \! \: h}}$ be two representations \hbox{of $A_0 [| \mu_h |]$} whose underlying $\Kh$-modules are both equal to $V_0 \cch$. We denote by $\pi_h$ and $\pi'_h$ the $\Kh$-algebra morphisms from $A_0 [| \mu_h |]$ to $\Endh{(V_0 \cch)}$ associated to the representations ${V_{\! \! \; h}}$ and ${V'_{\! \! \: h}}$, respectively. We denote \hbox{by $\pi_m$} \hbox{and $\pi'_m$} ($m \in \NN$) the maps in $C^2(A_0,\Endc{(V_0)})$ such that the following holds:
\begin{eqnarray*}
\big[ \pi_h (x) \big] (v) & = & \sum_{m \in \NN} \big[ \pi_m(x) \big] \! \! \; (v) \! \; h^m \, , \quad \quad \ \forall \, x \in A_0 \, , \ \forall \, v \in V_0 \, , \\
\big[ \pi'_h (x) \big] (v) & = & \sum_{m \in \NN} \big[ \pi'_m(x) \big] \! \! \; (v) \! \; h^m \, , \quad \quad \ \forall \, x \in A_0 \, , \ \forall \, v \in V_0 \, .
\end{eqnarray*}
We assume that the representations ${V_{\! \! \; h}}, {V'_{\! \! \: h}}$ satisfy the following two conditions.
\begin{enumerate}
\item[\textbullet] The maps $\pi_0$ and $\pi'_0$ are both equal to the $\KK$-algebra morphism from $A_0$ to $\Endc{(V_0)}$ associated to the representation $V_0$.
\item[\textbullet] The maps $\pi_m$ and $\pi'_m$ belong to $C^2(A_0,B_0,\Endc{(V_0)})$ for \hbox{every $m \in \NN$}.
\end{enumerate}




\vsp
\vsp

An extension of two representations of $A_0$ is said \hbox{\emph{$B_0$-trivial}} if the restriction to $B_0$ is a trivial extension in the category of representations of $B_0$.

\begin{prop} \label{prop_Hochschild}
If every $B_0$-trivial extension of $V_0$ by $V_0$ is $A_0$-trivial, then there exists a $A_0 [| \mu_h |]$-isomorphism between  $V_h$ and $V'_h$ whose specialisation \hbox{at ${h \! = \! 0}$} is the identity map on $V_0$.
\end{prop}

\proof
Suppose that every $B_0$-trivial extension of the representation $V_0$ by itself is $A_0$-trivial; i.e. suppose that the cohomology group $H^1(A_0,B_0,\Endc{(V_0}))$ is zero. Let $f_0$ be the identity map on $V_0$ and let us construct by induction $B_0$-endomorphisms ${f_{\! \! \; m}}$ of $V_0$ ($m \in \NN$) such that the following assertion holds for every $m \in \NN$.
\begin{equation} \tag{$\ast_m$}
\sum_{m_1 + m_2=m} \, {f_{\! \! \; m_1}} \circ \big[ \pi_{m_2} (x) \big] \ = \ \sum_{m_1+m_2=m} \, \big[ \pi'_{m_2}(x) \big] \circ {f_{\! \! \; m_1}} \, , \quad \ \forall \, x \in A_0 \, .
\end{equation}
Fix $m \in \NN$ and suppose that we have constructed $B_0$-endomorphisms ${f_{\! \! \; m'}}$ of $V_0$ such that assertion $(\ast_{m'})$ holds for \hbox{every $m' \in \NNI_{< m}$}. Since $H^1(A_0,B_0,\Endc{(V_0}))$ is zero, since the maps ${f_{\! \! \; m'}}$ are $B_0$-morphisms for all $m' \in \NN_{< m}$ by induction hypothesis and in view of the conditions satisfied by $\pi_h,\pi'_h$, there exists a \hbox{$B_0$-endomorphism ${f_{\! \! \; m}}$} of $V_0$ such that assertion $(\ast_m)$ holds if and only if the following equality holds for \hbox{every $x_1,x_2 \in A_0$}:
\begin{multline*}
\sum_{\substack{m_1+m_2=m, \\ m_2 \neq 0}} \bigg( \big[ \pi_0 (x_1) \big] \circ {f_{\! \! \; m_1}} \circ \big[ \pi_{m_2} (x_2) \big] \ - \ \big[ \pi'_0 (x_1) \big] \circ \big[ \pi'_{m_2} (x_2) \big] \circ {f_{\! \! \; m_1}} \bigg) \\
+ \, \sum_{\substack{m_1+m_2=m, \\ m_2 \neq 0}} \bigg( {f_{\! \! \; m_1}} \circ \big[ \pi_{m_2} (x_1) \big] \circ \big[ \pi_0 (x_2) \big] \ - \ \big[ \pi'_{m_2} (x_1) \big] \circ {f_{\! \! \; m_1}} \circ \big[ \pi'_0 (x_2) \big] \bigg)  \\
= \, \sum_{\substack{m_1+m_2=m, \\ m_2 \neq 0}} \bigg( {f_{\! \! \; m_1}} \circ \big[ \pi_{m_2} (x_1 \! \: x_2) \big] \ - \  \big[ \pi'_{m_2} (x_1 \! \: x_2) \big] \circ {f_{\! \! \; m_1}} \bigg) \, .
\end{multline*}
Besides, since ${V_{\! \! \; h}}$ and ${V'_{\! \! \: h}}$ are representations of $A_0 [| \mu_h |]$, the following equalities hold for every $m' \in \NN$ and for every $x_1,x_2 \in A_0$:
\begin{eqnarray*}
\sum_{m_1+m_2=m'} \pi_{m_1} \big( \mu_{m_2} (x_1,x_2) \big) & = & \sum_{m_1+m_2=m'} \big[ \pi_{m_1} (x_1) \big] \circ \big[ \pi_{m_2} (x_2) \big] \, , \\
\sum_{m_1+m_2=m'} \pi'_{m_1} \big( \mu_{m_2} (x_1,x_2) \big) & = & \sum_{m_1+m_2=m'} \big[ \pi'_{m_1} (x_1) \big] \circ \big[ \pi'_{m_2} (x_2) \big] \, .
\end{eqnarray*}
As a consequence, the following equality holds for every $x_1,x_2 \in A_0$:
\begin{multline*}
\sum_{\substack{m_1+m_2=m, \\ m_2 \neq 0}} \bigg( {f_{\! \! \; m_1}} \circ \big[ \pi_{m_2} (x_1 \! \: x_2) \big] \ - \  \big[ \pi'_{m_2} (x_1 \! \: x_2) \big] \circ {f_{\! \! \; m_1}} \bigg) \\
= \, \sum_{\substack{m_1+m_2+m_3=m \\ m_2+m_3 \neq 0}} \bigg( {f_{\! \! \; m_1}} \circ \big[ \pi_{m_2} (x_1) \big] \circ \big[ \pi_{m_3} (x_2) \big] \ - \ \big[ \pi'_{m_2} (x_1) \big] \circ \big[ \pi'_{m_3} (x_2) \big] \circ {f_{\! \! \; m_1}} \bigg) \\
- \,  \sum_{\substack{m_1+m_2+m_3=m \\ m_3 \neq m}} \bigg( {f_{\! \! \; m_1}} \circ \big[ \pi_{m_2} \big( \mu_{m_3} (x_1,x_2) \big) \big] \ - \ \big[ \pi'_{m_2} \big( \mu_{m_3} (x_1,x_2) \big) \big] \circ {f_{\! \! \; m_1}} \bigg) \, .
\end{multline*}
Since by induction hypothesis assertion $(\ast_{m'})$ holds for every $m' \in \NN_{< m}$, the following equalities then do for every $x_1,x_2 \in A_0$:
\begin{multline*}
\sum_{\substack{m_1+m_2+m_3=m \\ m_3 \neq m}} \bigg( {f_{\! \! \; m_1}} \circ \big[ \pi_{m_2} \big( \mu_{m_3} (x_1,x_2) \big) \big] \ - \ \big[ \pi'_{m_2} \big( \mu_{m_3} (x_1,x_2) \big) \big] \circ {f_{\! \! \; m_1}} \bigg) \ = \ 0 \, , \\
\sum_{\substack{m_1+m_2+m_3=m \\ m_1 \neq m}} \bigg( {f_{\! \! \; m_1}} \circ \big[ \pi_{m_2} (x_1) \big] \circ \big[ \pi_{m_3} (x_2) \big] \ - \ \big[ \pi'_{m_3} (x_1) \big] \circ \big[ \pi'_{m_2} (x_2) \big] \circ {f_{\! \! \; m_1}} \bigg) \\
= \, \sum_{\substack{m_1+m_2=m, \\ m_1 \neq m}} \bigg( {f_{\! \! \; m_1}} \circ \big[ \pi_{m_2} (x_1) \big] \circ \big[ \pi_0 (x_2) \big] \ - \ \big[ \pi'_0 (x_1) \big] \circ \big[ \pi'_{m_2} (x_2) \big] \circ {f_{\! \! \; m_1}} \bigg) \\
+ \, \sum_{\substack{m_1+m_2+m_3=m \\ m_3 \neq 0}} \bigg( \big[ \pi'_{m_2} (x_1) \big] \circ {f_{\! \! \; m_1}} \circ \big[ \pi_{m_3} (x_2) \big] \ - \ \big[ \pi'_{m_3} (x_1) \big] \circ {f_{\! \! \; m_1}} \circ \big[ \pi_{m_2} (x_2) \big] \bigg) \\
= \, \sum_{\substack{m_1+m_2=m, \\ m_1 \neq m}} \bigg( {f_{\! \! \; m_1}} \circ \big[ \pi_{m_2} (x_1) \big] \circ \big[ \pi_0 (x_2) \big] \ - \ \big[ \pi'_0 (x_1) \big] \circ \big[ \pi'_{m_2} (x_2) \big] \circ {f_{\! \! \; m_1}} \bigg) \\
+ \, \sum_{\substack{m_1+m_3=m, \\ m_3 \neq 0}} \bigg( \big[ \pi'_0 (x_1) \big] \circ {f_{\! \! \; m_1}} \circ \big[ \pi_{m_3} (x_2) \big] \ - \ \big[ \pi'_{m_3} (x_1) \big] \circ {f_{\! \! \; m_1}} \circ \big[ \pi_0 (x_2) \big] \bigg) \, .
\end{multline*}
We thus have proved the existence of $B_0$-endomorphisms ${f_{\! \! \; m}}$ of $V_0$ $(m \in \NN)$ such that assertion $(\ast_m)$ holds for every $m \in \NN$. One then concludes by considering the $\Kh$-linear endomorphism ${f_{\! \! \; h}}$ of $V_0 \cch$ defined by the following:
$$ {f_{\! \! \; h}}(v) \ := \ \sum_{m \in \NN} {f_{\! \! \; m}}(v) \! \; h^m \quad \quad \ (v \in V_0) \, . $$
\qed


\nomenclature[FA((h))]{$\AA_0 \cch$}{Constant formal deformation of $\AA_0$ \hfill \ref{para_constantdeformation}}
\nomenclature[Gtrivialformal]{}{Trivial formal deformation \hfill \ref{para_trivialdeform}}
\nomenclature[Gtrivialisationofapointed]{}{Trivialisation of a pointed $\Kh$-diagram \hfill \ref{para_trivialdeform}}
\nomenclature[Glocallytrivialformal]{}{Locally trivial formal deformation \hfill \ref{para_trivialdeform}}
\nomenclature[Gpointedformaldeform]{}{Pointed f. deformation of a representation \hfill \ref{para_deformrep}}

\section{Definition of GQE algebras in rank 1}
\subsection{Non-deformable relations of $\slt$}
Recall the definition of the Lie algebra $\slt$: see subsection \ref{subs_KMone}. The \hbox{relations \eqref{eq_KM_nondeformableone}} are called the \emph{non-deformable relations} \hbox{of $\slt$}.

\begin{defi}
The algebra $\sltF$ is the Lie $\KK$-algebra generated by \hbox{$X^-, X^+, H$} and subject to the non-deformable relations of $\sltF$:
\begin{equation*} \label{eq_FUcAone}
\big[ H, X^{\pm} \big] \ = \ \pm \: \! 2 \, X^{\pm} \, .
\end{equation*}
\end{defi}

The universal enveloping algebra $\FUc[\sltF]$ of $\sltF$ is the \hbox{$\KK$-algebra} generated by the elements $X^-, X^+, H$ and subject to the non-deformable relations of $\sltF$. We designate also by $H$, when the context makes it unambiguous, the \hbox{$\KK$-algebra} morphism from $\KK[H]$ to $\FUc[\sltF]$ which sends $H$ to $H$. \\

A \emph{$\sltF$-pointed} algebra designates a $\FUc[\sltF]$-pointed algebra. We regard the algebra $\Uc[\slt]$ as a $\sltF$-pointed $\KK$-algebra: the $\FUc[\sltF]$-point, which we denote by $\pic[\slt]$, is the projection map from $\FUc[\sltF]$ \hbox{to $\Uc[\slt]$}.

\vsp

\begin{defi}
We denote by $\FUh[\sltF]$ the constant formal deformation of the \hbox{${\sltF}$-pointed} \hbox{$\KK$-algebra} $\FUc[\sltF]$.
\end{defi}

\nomenclature[XimoinsXiplus3]{}{$\phantom{Chevalley generators of } \, \; \FUh$ \hfill \ref{para_nondeformable}}

\subsection{The global crystal of rank 1}
\begin{defi}
The global crystal of rank 1, denoted by $\BB$, is the quiver defined by the following.
\begin{enumerate}
\item[\textbullet] The set $\BBv$ of vertices of $\BB$ is ${\big \{ b_{n,p} := (n,n-2p) \; \! ; \, n,i \in \NN, \, p \leq n \big \}}$.
\item[\textbullet] The set $\BBe$ of edges of $\BB$ is the disjoint union ${\BBe[-] \sqcup \BBe[+]}$ of two copies \hbox{of $[\BBe] := {\big \{ (n,k) \; \! ; \, n,k \in \NNI, \, k \leq n \big \}}$}. The source of the edge $(n,k) \in \BBe[\pm]$ is $(n,\pm (n-2k+2))$ and its target is $(n,\pm (n-2k))$.
\end{enumerate}
\end{defi}

\vsp
\vsp

Here is an illustration of the global crystal of rank 1. Remark that there is no arrow for $n=0$.
\begin{equation*}
\shorthandoff{;:!?} \xymatrix @!0 @C=3pc @R=1.1pc {
&&&&& \quad \ \! \; \bullet \! \; \scriptstyle{b_{n,0}} \ar@<4pt>[dd] && \\
&&&&&& \quad \ \ \iddots & \\
&&& \quad \ \! \; \bullet \! \; \scriptstyle{b_{3,0}} \ar@<4pt>[dd] && \quad \ \! \; \bullet \! \; \scriptstyle{b_{n,1}} \ar@<4pt>[uu] \ar@<4pt>[dd] && \\
&& \quad \ \! \; \bullet \! \; \scriptstyle{b_{2,0}} \ar@<4pt>[dd] && \iddots &&& \\
& \quad \ \! \; \bullet \! \; \scriptstyle{b_{1,0}} \ar@<4pt>[dd] && \quad \ \! \; \bullet \! \; \scriptstyle{b_{3,1}} \ar@<4pt>[dd] \ar@<4pt>[uu] && \quad \ \! \; \bullet \! \; \scriptstyle{b_{n,2}} \ar@<4pt>[uu] && \\
\quad \ \! \; \bullet \! \; \scriptstyle{b_{0,0}} && \quad \ \! \; \bullet \! \; \scriptstyle{b_{2,1}} \ar@<4pt>[dd] \ar@<4pt>[uu] && \cdots & \vdots & \quad \ \ \cdots & \\
& \quad \ \! \; \bullet \! \; \scriptstyle{b_{1,1}} \ar@<4pt>[uu] && \quad \ \! \; \bullet \! \; \scriptstyle{b_{3,2}} \ar@<4pt>[dd] \ar@<4pt>[uu] && \quad \quad \ \, \bullet \! \; \scriptstyle{b_{n,n-2}} \ar@<4pt>[dd] && \\
&& \quad \ \! \; \bullet \! \; \scriptstyle{b_{2,2}} \ar@<4pt>[uu] && \ddots &&& \\
&&& \quad \ \! \; \bullet \! \; \scriptstyle{b_{3,3}} \ar@<4pt>[uu] && \quad \quad \ \, \bullet \! \; \scriptstyle{b_{n,n-1}} \ar@<4pt>[uu] \ar@<4pt>[dd] & \\
&&&&&& \quad \ \ \ddots & \\
&&&&& \quad \ \, \! \; \bullet \! \; \scriptstyle{b_{n,n}} \ar@<4pt>[uu] && \\
}
\end{equation*}

Let $n \in \NN$. We denote by $\BB_n$ the connected component of $\BB$ containing the vertex $b_{n,0}$. We denote by $\BBv_n$ and $\BBe_n$ the sets of vertices and edges in $\BB_n$, respectively. The quiver $\BB_n$ is called the \emph{$n$th component of $\BB$}.

\subsection{Colourings}
Let $n \in \NN$, remark that the irreducible representation $\Lc{n}$ of $\slt$ can be defined as the $\KK$-vector space ${\KK \BBv_n}$, where the actions of $X^\pm, H$ are given by the following, for $p \in \NN_{\leq n}$:
$$ \left\{ {\renewcommand{\arraystretch}{0.75} \begin{array}{rcl}
H . \! \; b_{n,p} & := & \ (n- 2p) \, b_{n,p} \, , \\
&& \\
X^- . \! \; b_{n,p} & := & \left\{ {\renewcommand{\arraystretch}{1.1} \begin{array}{cl}
(p+1) \, b_{n,p+1} & \quad \text{if $p \neq n$,} \\
0 & \quad \text{else,}
\end{array}} \right. \\
&& \\
X^+ . \! \; b_{n,p} & := &  \left\{ {\renewcommand{\arraystretch}{1.1} \begin{array}{cl}
(n-p+1) \, b_{n,p-1} & \quad \text{if $p \neq 0$,} \\
0 & \quad \text{else.}
\end{array}} \right.
\end{array}} \right.
$$
We regard $\Lc{n}$ also as a representation of $\sltF$.

\begin{defi}
We denote by $\Coloh$ the ring $\Kh^{\BBe}$. An element in $\Coloh$ is called a colouring of the crystal $\BB$ with values in $\Kh$.
\end{defi}

Let $\psi \in \Coloh$. We denote by $\psi^\pm$ the restriction of $\psi$ to $\BBe[\pm]$. Here is an illustration of $\psi$ restricted to $E_n$ ($n \in \NN$):

$$ \shorthandoff{;:!?} \xymatrix @!0 @R=1.1pc {
\quad \ \, \! \; \bullet \! \; \scriptstyle{b_{n,0}} \ar@<4pt>[dd]^{\scriptstyle \psi^-(n,1)} \\
\\
\quad \ \, \! \; \bullet \! \; \scriptstyle{b_{n,1}} \ar@<4pt>[uu]^{\scriptstyle \psi^+(n,n)} \ar@<4pt>[dd]^{\scriptstyle \psi^-(n,2)}\\
 \\
\quad \ \, \! \; \bullet \! \; \scriptstyle{b_{n,2}} \ar@<4pt>[uu]^{\scriptstyle \psi^+(n,n-1)} \\
\vdots \\
\quad \quad \ \, \bullet \! \; \scriptstyle{b_{n,n-2}} \ar@<4pt>[dd]^{\scriptstyle \psi^-(n,n-1)} \\
\\
\quad \quad \ \, \bullet \! \; \scriptstyle{b_{n,n-1}} \ar@<4pt>[uu]^{\scriptstyle \psi^+(n,2)} \ar@<4pt>[dd]^{\scriptstyle \psi^-(n,n)} \\
\\
\quad \ \, \! \; \bullet \! \; \scriptstyle{b_{n,n}} \ar@<4pt>[uu]^{\scriptstyle \psi^+(n,1)}
} $$

\vsp

\begin{defi}
Let $\psi \in \Coloh$ and let $n \in \NN$. We denote \hbox{by $\Lh{n,\psi}$} the representation \hbox{of $\FUh[\sltF]$} whose underlying $\Kh$-module is ${(\KK \BBv_n) \cch}$ and where the actions of $X^{\pm}, H$ are defined by the following, for $p \in \NN_{\leq n}$:
$$ \left\{ {\renewcommand{\arraystretch}{0.75} \begin{array}{rcl}
H . \! \; b_{n,p} & := & \ (n- 2p) \, b_{n,p} \, , \\
&& \\
X^- . \! \; b_{n,p} & := & \left\{ {\renewcommand{\arraystretch}{1.1} \begin{array}{cl}
\psi^-(n,p+1) \, b_{n,p+1} & \quad \text{if $p \neq n$,} \\
0 & \quad \text{else,}
\end{array}} \right. \\
&& \\
X^+ . \! \; b_{n,p} & := &  \left\{ {\renewcommand{\arraystretch}{1.1} \begin{array}{cl}
\psi^+(n,n-p+1) \, b_{n,p-1} & \quad \text{if $p \neq 0$,} \\
0 & \quad \text{else.}
\end{array}} \right.
\end{array}} \right.
$$
\end{defi}

\begin{lem} \label{lem_indecompAone}
Let $\psi \in \Coloh$ and let $n \in \NN$. If the restriction of $\psi$ to $\BBe_n$ has no zero, then the representation $\Lh{n,\psi}$ ($n \in \NN$) is indecomposable.
\end{lem}

\proof
Suppose that the restriction of $\psi$ to $E_n$ has no zero. Suppose that there exist nonzero subrepresentations $V, V'$ of $\Lh{n,\psi}$ such that $\Lh{n,\psi} = V \oplus V'$. Remark that the pull-back of $\Lh{n,\psi}$ via the morphism $H$ is isomorphic to a direct sum of one-dimensional irreducible representations of $\KK[H]$. Hence, the \hbox{pull-backs} of $V,V'$ by $H$ are also isomorphic to direct sums of one-dimensional representations of $\KK[H]$. As a consequence, there exists $p,p' \in \NN_{\leq n}$ such that $b_{n,p} \in V$ and $b_{n,p'} \in V'$. Since the restriction of $\psi$ to $E_n$ has no zero and in view of the action of $X^\pm$ on $\Lh{n,\psi}$, one then obtains a contradiction.
\qed

\begin{defi}
Let $\psi \in \Coloh$.
\begin{itemize}
\item[\textbullet] The colouring $\psi$ is said non-degenerate if the function $\psi$ has no zero.
\item[\textbullet] The colouring $\psi$ is said invertible if $\psi$ is invertible in the ring $\Coloh$.
\end{itemize}
\end{defi}

\subsection{Notations and conventions}
The $\Kh$-algebras $\Kh^{[\BBe]}$ and $\KK^{[\BBe]} \cch$ are identified: an element $f \in \Kh^{[\BBe]}$ is identified with the formal power series ${f_{\! \! \; h}} = \sum_{m \in \NN} {f_{\! \! \: m} \: \! h^m}$,  where ${f_{\! \! \: m}} \in \KK^{[\BBe]}$ for all $m \in \NN$ and such that the following holds:
$$ f(n,k) \ = \ \sum_{m \in \NN} {f_{\! \! \: m}}(n,k) \, h^m \, , \quad \quad \forall \, (n,k) \in [\BBe] \, . $$
Let now $P_h = \sum_{m \in \NN} P_m \, h^m$ be an element in $\KK[u,v] \cch$ and let $n,k \in \ZZ$. We denote by $P_h(n,k)$ the power series in $\Kh$ defined by the following:
$$ P_h(n,k) \ := \ \sum_{m \in \NN} P_m(n,k) \, h^m \, . $$

Remark that a polynomial in $\KK[u,v]$ is zero if and only if its evaluations at $(n,k)$ are zero for all $(n,k) \in [\BBe]$. We then regard the $\Kh$-algebra $\KK[u,v] \cch$ as a \hbox{$\Kh$-subalgebra} of $\Kh^{[\BBe]}$: an element $P_h \in \KK[u,v] \cch$ is identified with the function in $\Kh^{[\BBe]}$ whose value $(n,k)$ is $P_h(n,k)$ for every $(n,k) \in [\BBe]$. \\

Let $f \in \Kh^{[\BBe]}$ and let $n,k \in \NN$ such that $k \leq n$. We denote by $f(n,k) !$ the power series defined by the following:
$$ f(n,k) ! \ := \ \left\{ {\renewcommand{\arraystretch}{1.3} \begin{array}{cl}
\prod_{k'=1}^k f(n,k') & \quad \ \text{if } k \geq 1 \, , \\
0 & \quad \ \text{else.}
\end{array}} \right. $$



\subsection{The algebra $\Uh[\slt,\psi]$}
In this subsection, $\psi$ designates a colouring of the crystal $\BB$ with values in $\Kh$.

\begin{defi} \label{defi_IkerhAone}
$\phantom{}$
\begin{itemize}
\item[\textbullet] We denote by $\Ikerh[\slt,\psi]$ the two-sided ideal of $\FUh[\sltF]$ formed by the elements acting by zero on all the representations $\Lh{n,\psi}$ $(n \in \NN)$.
\item[\textbullet] We denote by $\Uh[\slt,\psi]$ the quotient of $\FUh[\sltF]$ by $\Ikerh[\slt,\psi]$.
\end{itemize}
\end{defi}

We denote by $\pih[{\slt,\psi}]$ the projection map from $\FUh[\sltF]$ to $\Uh[{\slt,\psi}]$. We endow the $\Kh$-algebra $\Uh[{\slt,\psi}]$ with the $\FUc[\sltF]$-point $\pih[{\slt,\psi}] \circ \iota$, where $\iota$ designates the inclusion map from $\FUc[\sltF]$ to $\FUh[\sltF]$.

\begin{rem} \label{rem_Uhoneperfect}
We regard $\Lh{n,\psi}$ ($n \in \NN$) also as a representation of the algebra $\Uh[\slt,\psi]$. By definition then, an element in $\Uh[\slt,\psi]$ is zero if and only if it acts as zero on all the representations $\Lh{n,\psi}$ $(n \in \NN)$.
\end{rem}

The algebra $\Uh[\slt,\psi]$ is built from the topologically free $\Kh$-modules $\FUh[\sltF]$ and $\Lh{n,\psi}$ ($n \in \NN$) and therefore inherits from them the property.

\begin{lem} \label{lem_Uhonehfree}
The $\Kh$-module $\Uh[\slt,\psi]$ is topologically free.
\end{lem}

\proof
The $\Kh$-module $\FUh[\sltF]$ is by definition topologically free. It is in particular complete. Being a quotient of $\FUh[\sltF]$, the $\Kh$-module $\Uh[{\slt,\psi}]$ is thus also complete. \\

Let $n \in \NN$, we denote by $A_n$ the \hbox{$\Kh$-algebra} $\Endh (\Lh{n,\psi})$ and we denote by $f_n$ the \hbox{$\Kh$-algebra} morphism from $\FUh[\sltF]$ to $A_n$, associated with the representation $\Lh{n,\psi}$. We denote then by $f$ the product of the morphisms $f_n$ ($n \in \NN$), it is a $\Kh$-algebra morphism from $\FUh[\sltF]$ \hbox{to $\prod_{n \in \NN} A_n$}. The definition of the ideal $\Ikerh[{\slt,\psi}]$ can be rephrased by the following equation:
\begin{equation} \label{eq_UhAonehfree}
\Ikerh[{\slt,\psi}] \ = \ \ker (f) \, .
\end{equation}
Let $n \in \NN$, the $\Kh$-module $\Lh{n,\psi}$ is by definition topologically free. It is in particular separated and torsion-free. Hence, the \hbox{$\Kh$-module $\prod_{n \in \NN} A_n$} is also \hbox{separated} and torsion-free. Besides, in view of equation \eqref{eq_UhAonehfree}, one can identify the $\Kh$-module $\Uh[{\slt,\psi}]$ with a $\Kh$-submodule of $\prod_{n \in \NN} A_n$. As a consequence, the $\Kh$-module $\Uh[{\slt,\psi}]$ is \hbox{separated} and torsion-free. \\

To sum up, the $\Kh$-module $\Uh[{\slt,\psi}]$ is separated, \hbox{complete} and torsion-free. It is thus topologically free (see for example \cite[proposition XVI.2.4]{Kas}).
\qed

\subsection{Congruence}
\begin{defi}
$\phantom{}$
\begin{itemize}
\item[\textbullet] Let $\psi \in \Coloh$. The congruent class $[\psi]$ \hbox{of $\psi$} is the function from $[\BBe]$ \hbox{to $\Kh$} defined by the following:
$$ [\psi] (n,k) \ := \ \psi^-(n,k) \ \psi^+(n,n-k+1) \quad \quad (n,k \in \NNI, \ k \leq n) \, . $$
\item[\textbullet] Two colourings $\psi_1, \psi_2 \in \Coloh$ are said congruent if $[\psi_1] = [\psi_2]$. We denote by $\equiv$ the congruence relation on $\Coloh$.
\end{itemize}
\end{defi}

The following lemma gives an interpretation of the congruence relation.

\begin{lem} \label{lem_congruentAone}
Let $\psi_1,\psi_2$ be two admissible colourings of $\BB$ with values in $\Kh$.
\begin{itemize}
\item[1)] If $\psi_1 \equiv \psi_2$, then the \hbox{$\sltF$-pointed} $\Kh$-algebras $\Uh[\slt,\psi_1]$ and $\Uh[\slt,\psi_2]$ are isomorphic.
\item[2)] We assume that $\psi_1, \psi_2$ divide each other in the ring $\Coloh$. The following two assertions are equivalent.
\begin{enumerate}
\item[(i)] The colourings $\psi_1$ and $\psi_2$ are congruent.
\item[(ii)] For every $n \in \NN$, the representations $\Lh{n,\psi_1}$ and $\Lh{n,\psi_2}$ of $\FUh[\sltF]$ are isomorphic.
\end{enumerate}
\end{itemize}
\end{lem}

\proof
Let $\psi, \psi'$ be two admissible colourings of $\BB$ with values in $\Kh$ such that $\psi^-$ divide $(\psi')^-$ in the ring $\Coloh$. Let us first prove the following assertion.
\begin{equation} \tag{$\ast$}
\begin{tabular}{| p{24pc}}
The colourings $\psi$ and $\psi'$ are congruent if and only if there exists an injective $\FUh[\sltF]$-morphism from $\Lh{n,\psi}$ to $\Lh{n,\psi'}$ for every $n \in \NN$.
\end{tabular}
\end{equation}
Fix $n \in \NN$. Since $H . b_{n,p} = {(n-2p) \! \: b_{n,p}}$ in both $\Lh{n,\psi}$ \hbox{and $\Lh{n,\psi'}$} for every $p \in \NN_{\leq n}$, a $\FUh[\sltF]$-morphism from $\Lh{n,\psi}$ to $\Lh{n,\psi'}$ must send, for every $p \in \NN_{\leq n}$, the vector $b_{n,p}$ to a scalar multiple of it. Since besides the vector $b_{n,p}$ is a scalar multiple of $(X^-)^p.b_{n,0}$ in both $\Lh{n,\psi}$ \hbox{and $\Lh{n,\psi'}$} for every $p \in \NN_{\leq n}$, a $\FUh[\sltF]$-morphism from $\Lh{n,\psi}$ to $\Lh{n,\psi'}$ is thus a scalar multiple of the $\Kh$-linear endomorphism ${f_{\! \! \: n}}$ of ${(\KK \BBv_n) \cch}$, defined by the following:
$$ f_{\! \! \: n} (b_{n,p}) \ := \ \frac{(\psi')^-(n,p) !}{\psi^- (n,p) !} \ b_{n,p} \quad \quad \ (p \in \NN_{\leq n}) \, . $$
Remark that since the colouring $\psi'$ is admissible, the map ${f_{\! \! \: n}}$ is injective. Remark then that for every $p \in \NNI_{\leq n}$, the following equalities hold in $\Lh{n,\psi}$ \hbox{and $\Lh{n,\psi'}$}, respectively:
\begin{eqnarray*}
f_{\! \! \: n} \big( X^+ . \! \: b_{n,p} \big) & = & \psi^+ (n,n-p+1) \ \frac{(\psi')^-(n,p-1) !}{\psi^- (n,p-1) !} \ b_{n,p} \, , \\
X^+ . \big( f_{\! \! \: n} (b_{n,p}) \big) & = & (\psi')^+ (n,n-p+1) \ \frac{(\psi')^-(n,p) !}{\psi^- (n,p) !} \ b_{n,p} \, .
\end{eqnarray*}
Assertion $(\ast)$ follows. Point 2 of the lemma is a consequence. Let us now prove point 1. Denote by $\psi$ the colouring of $\BB$ defined by the following
$$ \psi^-(n,k) \ := \ 1 \, , \quad \quad \psi^+(n,k) \ := \ [\psi_1] (n,k) \quad \quad \quad (n,k \in \NNI, \ k \leq n) \, . $$
Remark that $\psi \equiv \psi_1$. Assertion $(\ast)$ then implies that there exists an injective $\FUh[\sltF]$-morphism from $\Lh{n,\psi}$ to $\Lh{n,\psi_1}$ for every $n \in \NN$. As a consequence, in view of the definition of the algebras $\Uh[\slt,\psi]$ and $\Uh[\slt,\psi_1]$, there exists an isomorphism of $\sltF$-pointed $\Kh$-algebra between the two latter. One concludes.
\qed

\subsection{$h$-admissible colourings}
\begin{defi} \label{defi_hadmissible}
A colouring $\psi \in \Coloh$ is said $h$-admissible if it satisfies the following axioms.
\begin{IEEEeqnarray*}{lCl}
\bullet \text{ (Deformation axiom)} & \quad \quad & {[\psi]}_0 (n,k) \, = \, k \! \; (n-k+1) \, , \ \ \forall \, (n,k) \in [\BBe] \, . \\
\bullet \text{ (Regularity axiom)} & \quad \quad & [\psi] \, \in \, \KK[u,v] \cch \, . \\
\bullet \text{ (Quotient axiom)} & \quad \quad & [\psi] (n,n+1) \, = \, 0 \, , \ \ \forall \, n \in \NN \, . \\
\bullet \text{ (Verma axiom)} & \quad \quad & [\psi] (-n-2,k) \ = \ [\psi] (n,n+k+1) \, , \ \ \forall \, n,k \in \NNI .
\end{IEEEeqnarray*}
\end{defi}

Note, in view of the deformation axiom, that a $h$-admissible colouring is in particular admissible.

\begin{rem} \label{rem_dnd}
Let $\psi \in \Coloh$ such that $\psi^- = \psi^+$ and which depends only on the second variable: $\psi^\pm(n,k) = \psi^\pm(n',k)$ for all $n,n',k \in \NNI$ such \hbox{that $k \leq n,n'$}. Let $k \in \NNI$, we denote by $\psi(k)$ the common value of $\psi^-, \psi^+$ at $(n,k)$ \hbox{with $n \in \NN_{\geq k}$} and we regard $\psi$ as a function from $[\BBe]$ to $\Kh$. The \hbox{colouring $\psi$} is $h$-admissible if and only if the following conditions hold:
\begin{equation} \label{eq_hadminv}
\psi \, \in \, \KK[v] \cch \, , \quad \quad \psi(-v) \, = \, - \, \psi(v) \, , \quad \quad \psi_0 (v) \, = \, \pm \! \; v \, .
\end{equation}
Let $\psi_1,\psi_2$ be two $h$-admissible colourings such that $\psi_1^- = \psi_1^+$, $\psi_2^- = \psi_2^+$ and which depend only on the second variable. The colourings $\psi_1$ and $\psi_2$ are congruent if and only if $\psi_1 = \pm \; \! \psi_2$.
\end{rem}

\vsp
\vsp

\begin{lem} \label{lem_defLcAone}
Let $\psi$ be colouring of $\BB$ with values in $\Kh$ and let $n \in \NN$. The representation $\Lh{n,\psi}$ is a formal deformation of $\Lc{n}$ along $\FUh[\sltF]$ if and only if ${[\psi]}_0 (n,k) = {k \! \: (n-k+1)}$ for all $k \in \NNI_{\leq n}$.
\end{lem}

\proof
Denote by $\Lh{n,\psi}_0$ the pull-back of $\hzero{\Lh{n,\psi}}$ by the unique morphism of $\sltF$-pointed $\KK$-algebra \hbox{from $\FUc[\sltF]$} to $\hzero{\FUh[\sltF]}$. Denote then by $g_n$ the $\KK$-linear endomorphism of ${\KK \BBv_n}$, defined by the following:
$$ g_n (b_{n,p}) \ := \ \frac{{(\psi^-)}_0(n,p) !}{p \! \: !} \ b_{n,p} \quad \quad \ (p \in \NN_{\leq p}) \, . $$
Suppose that ${[\psi]}_0 (n,k) = {k \! \: (n-k+1)}$ for all $k \in \NNI_{\leq n}$. The map $g_n$ then is a $\FUc[\sltF]$-isomorphism from $\Lc{n}$ to $\Lh{n,\psi}_0$. Remark besides that the $\Kh$-module $\Lh{n,\psi}$ is by definition topologically free. The representation is therefore $\Lh{n,\psi}$ is a formal deformation of $\Lc{n}$ along $\FUh[\sltF]$. \\

Let us prove the converse implication. Let us then suppose that there exists a $\FUc[\sltF]$-isomorphism $g'_n$ from $\Lc{n}$ to $\Lh{n,\psi}_0$. The vector $b_{n,p}$ being of weight $n-2p$ in both $\Lh{n,\psi}$ \hbox{and $\Lh{n,\psi}_0$} for every $p \in \NN_{\leq n}$, the map $g'_n$ must send, for every $p \in \NN_{\leq n}$, the vector $b_{n,p}$ to a (nonzero) scalar multiple of it. By considering the actions of $X^+ X^-$ on the two representations, one concludes.
\qed \\

The following definition will be justified a posteriori by theorem \ref{thm_Uhonednd}.

\begin{defi}
Let $\psi$ be a $h$-admissible colouring of $\BB$. The \hbox{$\sltF$-pointed} \hbox{$\Kh$-algebra} $\Uh[\slt,\psi]$ is called the generalised quantum enveloping algebra of rank 1 associated to $\psi$.
\end{defi}

\section{Main results in rank 1}
\subsection{Formal deformation}

\begin{thm} \label{thm_Uhonednd}
Let $\psi \in \Coloh$. The following assertions are equivalent.
\begin{enumerate}
\item[(i)] The colouring $\psi$ is $h$-admissible.
\item[(ii)] The $\sltF$-pointed $\Kh$-algebra $\Uh[\slt,\psi]$ is a formal deformation of $\Uc[\slt]$.
\end{enumerate}
\end{thm}

\proof
\setcounter{claimn}{0}

\begin{claimn} \label{claim_Uhonednd_deform}
If assertion (ii) holds, then $\psi$ satisfies the deformation axiom.
\end{claimn}

\proof
Let $n \in \NN$. The following equalities hold in the specialisation at ${h \! = \! 0}$ of the representation $\Lh{n,\psi}$:
\begin{IEEEeqnarray*}{rCl}
\hzero  \pi^{\Uh[\slt,\psi]} (H) . \! \; b_{n,p} & \ := \ & \ (n- 2p) \, b_{n,p} \, , \\
\hzero  \pi^{\Uh[\slt,\psi]} (X^+ X^-) . \! \; b_{n,p} & := & \left\{ {\renewcommand{\arraystretch}{1.1} \begin{array}{cl}
{[\psi]}_0(n,p+1) \, b_{n,p+1} & \quad \text{if $p \neq 0$,} \\
0 & \quad \text{else,}
\end{array}} \right. \\
\hzero  \pi^{\Uh[\slt,\psi]} (X^+ X^-) . \! \; b_{n,p} & := &  \left\{ {\renewcommand{\arraystretch}{1.1} \begin{array}{cl}
{[\psi]}_0(n,p) \, b_{n,p-1} & \quad \text{if $p \neq n$,} \\
0 & \quad \text{else.}
\end{array}} \right.
\end{IEEEeqnarray*}

Suppose that assertion (ii) holds. The following relation then holds in the specialisation $\hzero{\Uh[\slt,\psi]}$:
\begin{equation} \label{eq_thmUhonedneddeform}
\Big[ \hzero \pi^{\Uh[\slt,\psi]} (X^+), \hzero  \pi^{\Uh[\slt,\psi]} (X^-) \Big] \ = \ \hzero  \pi^{\Uh[\slt,\psi]} (H) \, .
\end{equation}
As a consequence, the following equalities hold:
$$ {[\psi]}_0 (n,1) \ = \ n \, , \quad \quad \ {[\psi]}_0 (n,p+1) \ = \ {[\psi]}_0 (n,p) \, + \, (n-2p) \, , \quad \forall \, p \in \NNI_{\leq n} \, . $$
One concludes.
\qed \\

\begin{claimn} \label{claim_Uhonednd_span}
We assume that the colouring $\psi$ satisfies the deformation axiom. The following two assertions are equivalent.
\begin{enumerate}
\item[(i)] The colouring $\psi$ is $h$-admissible.
\item[ii')] The relation \eqref{eq_thmUhonedneddeform} holds in $\hzero{\Uh[\slt,\psi]}$.
\end{enumerate}
\end{claimn}

\proof
Suppose that $\psi$ is $h$-admissible. There then exists a solution $M$ of the GQE equation $\GQEeqh[\psi,\psi]{-1}$: see theorem \ref{thm_GQEeq}. This is turn implies that the following relation holds in $\Uh[\slt,\psi]$ (see proposition \ref{prop_GQEeqspan}):
\begin{equation} \label{eq_thmUhonednedspan}
X^+ X^- \ = \ \sum_{a \in \NN} \, {(X^-)}^a \! \; M_a (H) \! \; {(X^+)}^a \, .
\end{equation}
Besides, since $\psi$ satisfies by assumption the deformation axiom, the following equalities hold in $\KK[u]$ (see proposition \ref{prop_GQEeqclassic}):
$$ \hzero \pi^{\KK[u] \cch} (M_0) \ = \ u \, , \quad \ \hzero \pi^{\KK[u] \cch} (M_1) \ = \ 1 \, , \quad \ \hzero \pi^{\KK[u] \cch} (M_p) \ = \ u \, , \ \ \forall p \in \NN_{\geq 2} \, . $$
Assertion ii') follows. \\

Let us prove the converse. Denote by $T$ the subset ${\{ (X^-)^a H^b (X^+)^c \; \!; \, a,b,c \in \NN \}}$ of $\Uh[\slt,\psi]$ and denote by $f$ the $\Kh$-linear map from ${(\KK \! \; T) \cch}$ \hbox{to $\Uh[\slt,\psi]$} induced by the inclusion map from $T$ to $\Uh[\slt,\psi]$ (note that the map $f$ exists since the $\Kh$-module $\Uh[\slt,\psi]$ is separated and complete: we know from lemma \ref{lem_Uhonehfree} that it is topologically free). The $\Kh$-algebra $\Uh[\slt,\psi]$ being $\sltF$-pointed, the following relations hold in the specialisation $\hzero{\Uh[\slt,\psi]}$:
$$ \Big[ \hzero \pi^{\Uh[\slt,\psi]} (H), \hzero  \pi^{\Uh[\slt,\psi]} (X^{\pm}) \Big] \ = \ \pm \: \! 2 \ \hzero  \pi^{\Uh[\slt,\psi]} (X^{\pm}) \, . $$
Remark besides that the $\KK$-algebra $\hzero{\Uh[\slt,\psi]}$ is generated by $\hzero  \pi^{\Uh[\slt,\psi]} (X^{\pm})$ and $\hzero  \pi^{\Uh[\slt,\psi]} (H)$. Suppose now that assertion ii') holds. It follows that the \hbox{$\KK$-vector} space $\hzero{\Uh[\slt,\psi]}$ is then spanned by $\hzero  \pi^{\Uh[\slt,\psi]}(T)$. In other words, the specialisation at ${h \! = \! 0}$ of $f$ is surjective. The $\Kh$-module $\Uh[\slt,\psi]$ being separated and complete, the map $f$ is thus also surjective. There exists in particular a \hbox{sequence ${(M_a)}_{a \in \NN}$} with values in $\KK[u] \cch$, which converges to zero and such that the relation \eqref{eq_thmUhonednedspan} holds in $\Uh[\slt,\psi]$. This implies that the GQE equation $\GQEeqh[\psi,\psi]{-1}$ admits a solution and that $\psi$ is therefore $h$-admissible: see again proposition \ref{prop_GQEeqspan} and theorem \ref{thm_GQEeq}, respectively.
\qed \\

Suppose that assertion (ii) holds. The relation \eqref{eq_thmUhonedneddeform} then holds in $\hzero{\Uh[\slt,\psi]}$. Assertion (i) follows from claims \ref{claim_Uhonednd_deform} and \ref{claim_Uhonednd_span}. Let us prove the reverse implication. \\

Let $x_h$ be an element in the ideal $\Ikerh[\slt,\psi]$ \hbox{of $\FUh[\sltF]$}. Denote by $\bar x_0$ the image \hbox{in $\Uc[\slt]$} of $x_0 \in \FUc[\sltF]$ by $\pic[\slt]$. Since the colouring $\psi$ satisfies the deformation axiom, the representation $\Lh{n,\psi}$ is a formal deformation along $\FUh[\sltF]$ of the pull-back of $\Lc{n}$ \hbox{via $\pic[\slt]$}, for every $n \in \NN$. Since $x_h \in \Ikerh[\slt,\psi]$, the element $\bar x_0 \in \Uc[\slt]$ then acts as zero on all the representations $\Lc{n}$ \hbox{with $n \in \NN$} and is thus zero: see subsection \ref{subs_KMresults}. As a consequence, there exists a $\KK$-algebra \hbox{morphism $f_0$} from $\hzero{\Uh[\slt,\psi]}$ \hbox{to $\Uc[\slt]$} which makes the following diagram commute:

$$ \shorthandoff{;:!?} \xymatrix @C=5pc @R=1.5pc {
\FUc[\sltF] \ar@{->>}[dr]^-{\pic[\slt]} \ar@{->>}[d]_-{\hzero{\pih[\slt,\psi]} \ } & \\
\hzero{\Uh[\slt,\psi]} \ar[r]_-{f_0} & \Uc[\slt]
} $$

Suppose now that $\psi$ is $h$-admissible. We know from claim \ref{claim_Uhonednd_span} that the \hbox{relation \eqref{eq_thmUhonedneddeform}} then holds in $\hzero{\Uh[\slt,\psi]}$. As a consequence, and in view of the presentation of the algebra $\Uc[\slt]$ (see subsection \ref{subs_KMone}), there exists a $\KK$-algebra \hbox{morphism $g_0$} from $\Uc[\slt]$ to $\hzero{\Uh[\slt,\psi]}$ which makes the following diagram commute

$$ \shorthandoff{;:!?} \xymatrix @C=5pc @R=1.5pc {
\FUc[\sltF] \ar@{->>}[dr]^-{\pic[\slt]} \ar@{->>}[d]_-{\hzero{\pih[\slt,\psi]} \ } & \\
\hzero{\Uh[\slt,\psi]} & \Uc[\slt] \ar[l]^-{g_0}
} $$
Since the maps $\hzero{\pih[\slt,\psi]}$ and $\pic[\slt]$ are surjective, the morphisms $g_0 \circ f_0$ and $f_0 \circ g_0$ are therefore equal to the identity maps. The $\Kh$-module $\Uh[\slt,\psi]$ being topologically free (see \hbox{lemma \ref{lem_Uhonehfree}}), one concludes.
\qed \\



\subsection{Classical and standard quantum realisations} \label{subs_exclqUhone}
Let us recall the definition of the \emph{formal quantum algebra} $\Uh[\slt]$ associated to the classical Lie algebra $\slt$. The following elements are defined in the power series ring $\Kh$:
$$ q \ := \ \exp (h) \, , \quad \quad \quad {[k]}_q \ := \ \frac{q^k - q^{-k}}{q - q^{-1}} \quad \ (k \in \ZZ) \, . $$

The formal quantum algebra $\Uh[\slt]$ is the $\Kh$-algebra topologically generated by $X^-, X^+, H$ and subject to the following relations:
\begin{IEEEeqnarray*}{rCl}
\big[ H, X^{\pm} \big] & \ = \ & \pm \; \! 2 \, X^{\pm} \, , \\
\big[ X^+, X^- \big] & \ = \ & \frac{q^H  - q^{-H}}{q - q^{-1}} \, .
\end{IEEEeqnarray*}
We denote by $\pih[\slt]$ the $\Kh$-algebra morphism from $\FUh[\sltF]$ to $\Uh[\slt]$, which sends $X^{\pm}$ to $X^{\pm}$ \hbox{and $H$} to $H$. We endow the $\Kh$-algebra $\Uh[\slt]$ with the $\FUc[\sltF]$-point $\pih[\slt] \circ \iota$, where $\iota$ designates the inclusion map from $\FUc[\sltF]$ to $\FUh[\sltF]$. \\

Let us now recall the description of the fundamental representations of the quantum algebra $\Uh[\slt]$. Let $n \in \NN$, the representation $\Lh{n}$ can be defined as the $\Kh$-module ${(\KK \BBv_n) \cch}$, where the actions of $X^\pm, H$ are given by the following, for $p \in \NN_{\leq n}$:
\begin{IEEEeqnarray*}{rCl}
H . \! \; b_{n,p} & \ := \ & \ (n- 2p) \, b_{n,p} \, , \\
X^- . \! \; b_{n,p} & := & \left\{ {\renewcommand{\arraystretch}{1.1} \begin{array}{cl}
{[p+1]}_q \, b_{n,p+1} & \quad \text{if $p \neq n$,} \\
0 & \quad \text{else,}
\end{array}} \right. \\
X^+ . \! \; b_{n,p} & := &  \left\{ {\renewcommand{\arraystretch}{1.1} \begin{array}{cl}
{[n-p+1]}_q \, b_{n,p-1} & \quad \text{if $p \neq 0$,} \\
0 & \quad \text{else.}
\end{array}} \right.
\end{IEEEeqnarray*}

The \emph{classical colouring} and the \emph{quantum colouring}, denoted by $\psicl$ and $\psiq$, respectively, are the colourings of $\BB$ with values in $\Kh$, defined by the following:
$$ (\psicl)^\pm (n,k) \ := \ k \, , \quad \quad (\psiq)^\pm (n,k) \ := \ {[k]}_q \quad \quad \quad (n,k \in \NNI, \ k \leq n) \, . $$
Note, in view of remark \ref{rem_dnd}, that the colourings $\psicl$ and $\psiq$ are \hbox{$h$-admissible}.

\vsp
\vsp

\begin{prop} \label{prop_exUhone}
\begin{enumerate}
\item[1)] The GQE algebra $\Uh[\slt,\psicl]$ and the constant formal deformation $\Uc[\slt] \cch$ are isomorphic $\sltF$-pointed $\Kh$-algebras.
\item[2)] The GQE algebra $\Uh[\slt,\psiq]$ and the formal quantum algebra $\Uh[\slt]$ are isomorphic $\sltF$-pointed $\Kh$-algebras.
\end{enumerate}
\end{prop}

\proof
Remark that for every $n \in \NN$ the pull-back of $\Lc{n} \cch$ by $\pic[\slt] \cch$ is the representation $\Lh{n,\psicl}$ of $\FUh[\sltF]$. As a consequence, there exists a morphism ${f_{\! \! \; h}}$ of $\sltF$-pointed $\Kh$-algebra from $\Uc[\slt] \cch$ to $\Uh[\slt,\psi]$ which makes the following diagram commute:
$$ \shorthandoff{;:!?} \xymatrix @C=5pc @R=1.5pc {
\FUh[\sltF] \ar@{->>}[dr]^-{\ \ \pih[\slt,\psicl]} \ar@{->>}[d]_-{\pic[\slt] \cch \, } & \\
\Uc[\slt] \cch \ar[r]_-{{f_{\! \! \; h}}} & \Uh[\slt,\psicl]
} $$
Since $\Uh[\slt,\psicl]$ is a formal deformation of the $\sltF$-pointed $\KK$-algebra $\Uc[\slt]$ (see theorem \ref{thm_Uhonednd}), the map $\hzero{({f_{\! \! \; h}})}$ is bijective. The $\Kh$-module $\Uh[\slt,\psi]$ being topologically free, the map ${f_{\! \! \; h}}$ is then also bijective. In other words, the $\sltF$-pointed algebras $\Uh[\slt,\psicl]$ and $\Uc[\slt] \cch$ are isomorphic. The proof in the quantum case is similar: remark first that for every $n \in \NN$ the pull-back of the representation $\Lh{n}$ by $\pih[\slt]$ is the representation $\Lh{n,\psiq}$ of $\FUh[\sltF]$ and remark that the $\sltF$-pointed $\Kh$-algebras $\hzero{\Uh[\slt]}, \Uc[\slt]$ are isomorphic. The proof in the quantum case is analog.
\qed

\subsection{Presentation by generators and relations}
In this subsection, $\psi$ designates a $h$-admissible colouring of $\BB$. We denote by ${S^{\! \: \psi}_{\! \! \; h}} = {({S^{\! \: \psi}_{\! \! \; h,p}})}_{p \in \NN}$ the solution of the GQE equation $\GQEeqh[\psi,\psi]{-1}$: see proposition \ref{prop_GQEequnicity} and \hbox{theorem \ref{thm_GQEeq}}. We denote by $\Uph{\slt,\psi}$ the \hbox{$\Kh$-algebra} topologically generated by $X^-, X^+, H$ and subject to the following relations:
\begin{IEEEeqnarray*}{CCl}
\label{eq_Uhonepres1} \big[ H,X^\pm \big] & \ = \ & 2 \: \! X^\pm \, , \\
\label{eq_Uhonepres2} X^+ X^- & \, = \ & \sum_{a \in \NN} \, {(X^-)}^a \: \! {S^{\! \: \psi}_{\! \! \; h,a}} (H) \; \! {(X^+)}^a \, .
\end{IEEEeqnarray*}

Note that the series $\sum_{a \in \NN} {(X^-)}^a {S^{\! \: \psi}_{\! \! \; h,a}} (H) {(X^+)}^a$ converges since the \hbox{sequence ${S^{\! \: \psi}_{\! \! \; h}}$} converges to zero. The projection map from $\FUh[\sltF]$ to $\Uph{\slt,\psi}$ endows the latter with a structure of $\sltF$-pointed \hbox{$\Kh$-algebra}.

\begin{thm} \label{thm_presentationUhone}
The GQE algebra $\Uh[\slt,\psi]$ is isomorphic to the $\sltF$-pointed $\Kh$-algebra $\Uph{\slt,\psi}$.
\end{thm}

\proof
Since ${S^{\! \: \psi}_{\! \! \; h}}$ is a solution of the GQE equation $\GQEeqh[\psi,\psi]{-1}$ and since the colouring $\psi$ satisfies the deformation axiom, the following equalities hold in $\KK[u]$ (see proposition \ref{prop_GQEeqclassic}):
$$ \hzero \pi^{\KK[u] \cch} ({S^{\! \: \psi}_{\! \! \; h,0}}) \ = \ u \, , \quad \hzero \pi^{\KK[u] \cch} ({S^{\! \: \psi}_{\! \! \; h,1}}) \ = \ 1 \, , \quad \hzero \pi^{\KK[u] \cch} ({S^{\! \: \psi}_{\! \! \; h,p}}) \ = \ u \, , \ \, \forall p \in \NN_{\geq 2} \, . $$
As a consequence, the $\sltF$-pointed $\KK$-algebras $\Uc[\slt]$ and $\hzero{{U'_{\! \! \; h}(\slt,\psi)}}$ are isomorphic (recall the presentation of $\Uc[\slt]$: see subsection \ref{subs_KMone}). Besides, since ${S^{\! \: \psi}_{\! \! \; h}}$ is a solution of the GQE equation $\GQEeqh[\psi,\psi]{-1}$, the following relation holds in $\Uh[\slt,\psi]$ (see proposition \ref{prop_GQEeqspan}):
$$ X^+ X^- \ = \ \sum_{a \in \NN} \, {(X^-)}^a \! \; {S^{\! \: \psi}_{\! \! \; h,a}} (H) \! \; {(X^+)}^a \, . $$
Hence, there exists a morphism ${f_{\! \! \; h}}$ of $\sltF$-pointed $\Kh$-algebra from ${U'_{\! \! \; h}(\slt,\psi)}$ to $\Uh[\slt,\psi]$ (the $\Kh$-module $\Uh[\slt,\psi]$ is separated and complete). On the one hand, we know that $\Uh[\slt,\psi]$ is a formal deformation of the \hbox{$\sltF$-pointed} $\KK$-algebra $\Uc[\slt]$ (see theorem \ref{thm_Uhonednd}). On the other hand, we have proved that the $\sltF$-pointed $\KK$-algebras $\Uc[\slt]$ and $\hzero{{U'_{\! \! \; h}(\slt,\psi)}}$ are isomorphic. As a consequence, the map $\hzero{({f_{\! \! \; h}})}$ is bijective. The $\Kh$-module ${U'_{\! \! \; h}(\slt,\psi)}$ being complete and separated, the $\Kh$-module $\Uh[\slt,\psi]$ being separated and torsion-free, the map ${f_{\! \! \; h}}$ is then bijective.
\qed

\subsection{Integrable representations}
Let $\psi$ be a colouring of $\BB$ with values in $\Kh$. We denote by $\Cinth{\slt,\psi}$ the category of integrable representations of $\Uh[\slt,\psi]$ which are topologically free when viewed as $\Kh$-modules.

\vsp

\begin{thm} \label{thm_Uhintegrable}
Let $\psi$ be a $h$-admissible colouring of $\BB$.
\begin{enumerate}
\item[1)] For every $n \in \NN$, the representation $\Lh{n,\psi}$ is a formal deformation of $\Lc{n}$ along $\Uh[\slt,\psi]$.
\item[2)] For every $n \in \NN$, the representation $\Lh{n,\psi}$ is, up to isomorphism, equal to the representation of $\Uh[\slt,\psi]$ topologically generated by $v$ and subject to the following relations:
$$ H . \! \: v \ = \ n \; \! v \, , \quad \quad X^+ . \! \: v \ = \ 0 \, , \quad \quad (X^-)^{n+1} . \! \: v \ = \ 0 \, . $$
\item[3)] The representations $\Lh{n,\psi}$ ($n \in \NN$) are pairwise distinct and are the unique, up to isomorphism, indecomposable representations in $\Cinth{\slt,\psi}$.
\item[4)] A representation of $\Uh[\slt,\psi]$ is in the category $\Cinth{\slt,\psi}$ if and only if it is $\psi$-integrable.
\item[5)] Let $n,n' \in \NN$. Every $\Uh[\slt,\psi]$-morphism from $\Lh{n,\psi}$ to $\Lh{n',\psi}$ is equal to a scalar multiple of the identity map if $n = n'$ and is zero else.
\item[6)] Every integrable representation of the algebra $\Uc[\slt]$ admits a unique, up to isomorphism, formal deformation along $\Uh[\slt,\psi]$.
\end{enumerate}
\end{thm}

\proof
Denote by $f_0$ the unique morphism of $\sltF$-pointed $\KK$-algebra from $\Uc[\slt]$ to the specialisation at ${h \! = \! 0}$ of $\Uh[\slt,\psi]$ and let $n \in \NN$. Recall besides \hbox{lemma \ref{lem_defLcAone}}: since $\psi$ satisfies the deformation axiom, the representation $\Lh{n,\psi}$ is a formal deformation of $\Lc{n}$ along $\Uh[\slt,\psi]$. \\

According to the definition of $\Lh{n,\psi}$, there exists a \hbox{$\Uh[\slt,\psi]$-morphism $g_h$} from $\Lph{n,\psi}$ to $\Lh{n,\psi}$. Besides, according to the definition of the representation $\Lph{n,\psi}$, the representation $\Lc{n}$ of $\Uc[\slt]$ is isomorphic to the pull-back of $\hzero{\Lph{n,\psi}}$ via $f_0$ (recall the presentation of $\Lc{n}$: see subsection \ref{subs_KMresults}). On the other hand, we have proved that $\Lh{n,\psi}$ is a formal deformation of $\Lc{n}$ along $\Uh[\slt,\psi]$. Hence, the map $\hzero{(g_h)}$ is bijective. The $\Kh$-module $\Lph{n,\psi}$ being complete and separated, the $\Kh$-module $\Lh{n,\psi}$ being separated and torsion-free, the map $g_h$ is then bijective. \\

Remark that the representation $\Lh{n,\psi}$ is integrable. We besides now from lemma \ref{lem_indecompAone} that the representation is indecomposable. Let now $n_1, n_2 \in \NN$ such that $n_1 \neq n_2$. Remark that the representations $\Lh{n_1,\psi}$ and $\Lh{n_2,\psi}$ are not isomorphic since the $\Kh$-modules $\Lh{n_1,\psi}$ and $\Lh{n_2,\psi}$ are not. \\

Let $V_0$ be an integrable representation of $\Uc[\slt]$ and let ${V_{\! \! \; h}}, {V'_{\! \! \; h}}$ be two formal deformations of $V_0$ along $\Uc[\slt,\psi]$. Let $W_0$ be an extension of $V_0$ by $V_0$. \hbox{Since $V_0$} is an integrable representation of $\Uc[\slt]$, the actions of $X^-,X^+,H$ on $W_0$ are locally finite and the representation $W_0$ is then isomorphic to a direct sum of representations $\Lc{n}$ with $n \in \NN$ (see proposition \ref{prop_KMintone}). The representations $\Lc{n}$ ($n \in \NN$) being irreducible, the representation $W_0$ is in particular semisimple and hence a trivial extension of $V_0$ by $V_0$. This being true for every extension $W_0$ of $V_0$ by $V_0$, the representations ${V_{\! \! \; h}}$ \hbox{and ${V'_{\! \! \; h}}$} are thus isomorphic (see proposition \ref{prop_Hochschild}). As a consequence, and since again every integrable representation of $\Uc[\slt]$ is isomorphic to a direct sum of representations $\Lc{n}$ with $n \in \NN$ (see proposition \ref{prop_KMintone}), there exists for every integrable representation $V_0$ of $\Uc[\slt]$, a unique, up to isomorphism, formal deformation of $V_0$ along $\Uh[\slt,\psi]$, which besides is isomorphic to a topologically direct sums of representations $\Lh{n,\psi}$ with $n \in \NN$. Since besides every integrable representation of $\Uh[\slt,\psi]$ is a formal deformation along $\Uh[\slt,\psi]$ of an integrable representation of $\Uc[\slt]$, point 3 of the theorem follows and every indecomposable integrable representation of $\Uh[\slt,\psi]$ is then isomorphic to some representation $\Lh{n,\psi}$ with $n \in \NN$.
\qed

\subsection{$H$-triviality}
In this subsection, $\psi$ designates a $h$-admissible colouring of $\BB$. Let ${\bar S^{\! \: \psi}_{\! \! \; h}} = {({\bar S^{\! \: \psi}_{\! \! \; h,p}})}_{p \in \NN}$ be the solution of the GQE equation $\GQEeqh[\psi,\psicl]{0}$: see proposition \ref{prop_GQEequnicity} and theorem \ref{thm_GQEeq}. \\

We recall that $H$ designates also the \hbox{$\KK$-algebra} morphism from $\KK[H]$ to $\FUc[\sltF]$ which sends $H$ to $H$.

\begin{thm} \label{thm_UhoneHtrivial}
Let $\psi$ be a $h$-admissible colouring of the crystal $\BB$. There exists a ${H \! \! \:}$-trivialisation $\gamma_h$ of the $\sltF$-pointed $\Kh$-algebra $\Uh[\slt,\psi]$ which satisfies the following:
$$ \left\{ \begin{array}{lcl}
\gamma_h(X^-) & = & X^- \, , \\
\gamma_h(H) & = & H \, , \\
\gamma_h(X^+) & = & \sum_{a \in \NN} \! \: {(X^-)}^a \: \! {\bar S^{\! \: \psi}_{\! \! \; h,a}} (H) \; \! {(X^+)}^{a+1} \, .
\end{array} \right. $$
\end{thm}

Note that the series $\sum_{a \in \NN} {(X^-)}^a {\bar S^{\! \: \psi}_{\! \! \; h,a}} (H) {(X^+)}^{a+1}$ converges since the \hbox{sequence ${\bar S^{\! \: \psi}_{\! \! \; h}}$} converges to zero.

\proof
Let us prove the existence of a $\Kh$-algebra morphism $\gamma'_h$ from $\Uh[\slt,\psicl]$ to $\Uh[\slt,\psi]$ which satisfies the following:
\begin{equation} \label{eq_thmUhonetrivial} \left\{ \begin{array}{lcl}
\gamma'_h(X^-) & = & X^- \, , \\
\gamma'_h(H) & = & H \, , \\
\gamma'_h(X^+) & = & \sum_{a \in \NN} \! \: {(X^-)}^a \: \! {\bar S^{\! \: \psi}_{\! \! \; h,a}} (H) \; \! {(X^+)}^{a+1} \, .
\end{array} \right.
\end{equation}
The theorem will follows since the GQE algebra $\Uh[\slt,\psicl]$ and the constant formal deformation $\Uc[\slt] \cch$ are isomorphic $\sltF$-pointed $\Kh$-algebras (see proposition \ref{prop_exUhone}).
Since the algebra $\Uh[\slt,\psi]$ and the GQE equation $\GQEeqh[\psi,\psicl]{0}$ depend only on the congruence class of $\psi$ (see lemma \ref{lem_congruentAone} concerning $\Uh[\slt,\psi]$), one can without loss of generality assume that $\psi$ has values in $\Kh^\times$ (remark that $\psi$ satisfies the deformation axiom). Since ${\bar S^{\! \: \psi}_{\! \! \; h}}$ is a the solution of $\GQEeqh[\psi,\psicl]{0}$, there then exists a $\Kh$-algebra morphism $\gamma'_h$ from $\Uh[\slt,\psicl]$ to $\Uh[\slt,\psi]$ which satisfies the conditions \eqref{eq_thmUhonetrivial}: see proposition \ref{prop_GQEeqiso}. Besides, since the colouring $\psi$ satisfies the deformation axiom, the following equalities hold \hbox{in $\KK[u]$} (see proposition \ref{prop_GQEeqclassic}):
$$ \hzero \pi^{\KK[u] \cch} ({\bar S^{\! \: \psi}_{\! \! \; h,0}}) \ = \ 0 \, , \quad \hzero \pi^{\KK[u] \cch} ({\bar S^{\! \: \psi}_{\! \! \; h,1}}) \ = \ 1 \, , \quad \hzero \pi^{\KK[u] \cch} ({\bar S^{\! \: \psi}_{\! \! \; h,p}}) \ = \ 0 \, , \ \, \forall p \in \NN_{\geq 2} \, . $$
As a consequence the specialisation at ${h \! = \! 0}$ of $\gamma'_h$ is a morphism of $\sltF$-pointed $\KK$-algebra between $\Uh[\slt,\psicl]$ and $\Uh[\slt,\psi]$. The latter algebras being both generated by $X^\pm$ and $H$, the specialisation at ${h \! = \! 0}$ of $\gamma'_h$ is therefore bijective. The $\Kh$-modules $\Uh[\slt,\psicl]$ and $\Uh[\slt,\psi]$ being both topologically free, the map $\gamma'_h$ is thus also bijective. We are done.
\qed

\subsection{Classification}
\begin{thm} \label{thm_GQEoneclassification}
The following map is a bijection:
\begin{IEEEeqnarray*}{CCC}
\left \{ {\renewcommand{\arraystretch}{0.85} \begin{array}{c}
\text{$h$-admissible} \\
\text{colourings of $\BB$,} \\
\text{up to congruence}
\end{array}} \right \}
& \ \longrightarrow \ &
\left \{ {\renewcommand{\arraystretch}{0.85} \begin{array}{c}
\text{${H \! \! \:}$-trivial formal deformations} \\
\text{of the $\sltF$-pointed $\KK$-algebra $\Uc[\slt]$}, \\
\text{up to isomorphism} \\
\text{of $\sltF$-pointed $\Kh$-algebra}
\end{array}} \right \} \\
&& \\
\begin{array}{c}
\text{congruence} \\
\text{class of $\psi$}
\end{array} & \longmapsto & \begin{array}{c}
\text{isomorphism} \\
\text{class of $\Uh[\slt,\psi]$.}
\end{array}
\end{IEEEeqnarray*}
\end{thm}

Remark that the map of theorem \ref{thm_GQEoneclassification} is well defined since GQE algebras associated to congruent colourings are isomorphic $\sltF$-pointed $\Kh$-algebras: see lemma \ref{lem_congruentAone}.

\proof
Let $A_h$ be a $H$-trivial formal deformation of $\Uc[\slt]$. We know from proposition \ref{prop_exUhone} that the constant formal deformation $\Uc[\slt] \cch$ and the GQE algebra $\Uh[\slt,\psicl]$ are isomorphic $\sltF$-pointed $\Kh$-algebras. Hence, there exists a morphism $\gamma_h$ of $\sltF$-pointed $\Kh$-algebra from $A_h$ to $\Uh[\slt,\psicl]$. Denote then by $V_n$ ($n \in \NN$) the pull-back via $\gamma_h$ of $\Lc{n,\psicl}$. Since the morphism $\gamma_h$ sends $H \in A_h$ to $H \in \Uc[\slt,\psicl]$, there exists $\psi \in \Coloh$ such that, for every $n \in \NN$, the actions of $X^-,X^+,H \in A_h$ on the representation $V_n$ are given by the following, for $p \in \NN_{\leq n}$:
\begin{IEEEeqnarray*}{rCl}
H . \! \; b_{n,p} & \ := \ & \ (n- 2p) \, b_{n,p} \, , \\
X^- . \! \; b_{n,p} & := & \left\{ {\renewcommand{\arraystretch}{1.1} \begin{array}{cl}
\psi^-(n,p+1) \, b_{n,p+1} & \quad \text{if $p \neq n$,} \\
0 & \quad \text{else,}
\end{array}} \right. \\
X^+ . \! \; b_{n,p} & := &  \left\{ {\renewcommand{\arraystretch}{1.1} \begin{array}{cl}
\psi^+(n,n-p+1) \, b_{n,p-1} & \quad \text{if $p \neq 0$,} \\
0 & \quad \text{else.}
\end{array}} \right.
\end{IEEEeqnarray*}
Denote now by ${\pi_{\! \! \; h}}$ the $\Kh$-algebra morphism from $\FUh[\sltF]$ to $A_h$ which makes the following diagram commutes, where $\iota$ designates the inclusion map from $\FUc[\sltF]$ to $\FUh[\sltF]$ (the morphism ${\pi_{\! \! \; h}}$ exists since the $\Kh$-module $A_h$ is complete and separated):
$$ \shorthandoff{;:!?} \xymatrix @C=5pc @R=1.5pc {
\FUh[\sltF] \ar[dr]^-{{\pi_{\! \! \; h}}} & \\
\FUc[\sltF] \ar[u]^{\iota} \ar[r]_-{\dot{A}_h} & \, A_h
} $$
In view of the definition of the algebra $\Uh[\slt,\psicl]$ (see remark \ref{rem_Uhoneperfect}), an element in the algebra $A_h$ is zero if and only if it acts as zero on all $V_n$ ($n \in \NN$). As a consequence, there exists an injective morphism ${f_{\! \! \; h}}$ of $\sltF$-pointed $\Kh$-algebra \hbox{from $\Uh[\slt,\psi]$} to $A_h$ which makes the following diagram commute:
$$ \shorthandoff{;:!?} \xymatrix @C=5pc @R=1.5pc {
\FUh[\sltF] \ar[dr]^-{{\pi_{\! \! \; h}}} \ar@{->>}[d]_-{\pih[\slt,\psi]} \ar@{->>}[r]^-{\pih[\slt,\psicl]} & \Uh[\slt,\psicl] \\
\Uh[\slt,\psi] \ar[r]_-{{f_{\! \! \; h}}} & A_h \ar[u]_-{\gamma_h}
} $$
Since the map $\pih[\slt,\psicl]$ is surjective, so then is the map ${f_{\! \! \; h}}$. We thus have proved that the $\sltF$-pointed $\Kh$-algebras $\Uh[\slt,\psi]$ and $A_h$ are isomorphic. Since $A_h$ is a formal deformation of the $\sltF$-pointed $\KK$-algebra $\Uc[\slt]$, so then is $\Uh[\slt,\psi]$. Therefore, the colouring $\psi$ is $h$-admissible: see theorem \ref{thm_Uhonednd}. \\

Let now $\psi_1$ and $\psi_2$ be two $h$-admissible colourings such that the $\sltF$-pointed \hbox{$\Kh$-algebras} $\Uh[\slt,\psi_1]$ and $\Uh[\slt,\psi_2]$ are isomorphic. Let us prove that the colourings $\psi_1$ and $\psi_2$ are congruent. Since $\Uh[\slt,\psi_1]$ and $\Uh[\slt,\psi_2]$ depend only on the congruent classes of $\psi_1$ and $\psi_2$, respectively (see \hbox{lemma  \ref{lem_congruentAone}}), one can assume without of generality that $\psi_1$ and $\psi_2$ divide each other in the \hbox{ring $\Coloh$} (remark that the two colourings satisfy the deformation axiom). Denote by $g_h$ an isomorphism of $\sltF$-pointed \hbox{$\Kh$-algebras} from $\Uh[\slt,\psi_1]$ to $\Uh[\slt,\psi_2]$. For every $n \in \NN$, the pull-back of $\Lc{n,\psi_2}$ via $g_h$ is an integrable and indecomposable representation of $\Uh[\slt,\psi_1]$ and is thus, in view of theorem \ref{thm_Uhintegrable}, isomorphic to the representation $\Lc{n,\psi_1}$ (the ranks of the free $\Kh$-modules $\Lc{n,\psi_1}$ and $\Lc{n,\psi_2}$ are equal). As a consequence, the colourings $\psi_1$ and $\psi_2$ are congruent: see lemma \ref{lem_congruentAone}.
\qed

\section{Definition of GQE algebras}
\subsection{Non-deformable relations of $\glie$} \label{para_nondeformable}
Recall the definition of the Kac-Moody algebra $\glie$: see subsection \ref{subs_KM}. The \hbox{relations \eqref{eq_KM_nondeformable}} are called the \emph{non-deformable relations} \hbox{of $\glie$}.

\begin{defi}
The algebra $\glieF$ is the Lie $\KK$-algebra generated by \hbox{$X^\pm_i$ ($i \in I$)}, the \hbox{$\ZZ$-module $Y$} and subject to the non-deformable relations of $\glie$:
\begin{equation*}
\left\{ {\renewcommand{\arraystretch}{1.3} \begin{array}{rclr}
\big[ \check \mu, \check \mu' \big] & = & 0 & (\check \mu, \check \mu' \in Y) \, , \\
\big[ {\check \mu}, X^{\pm}_i \big] & = & \pm \; \! \langle {\check \mu} , \alpha_i \rangle \, X_i^{\pm} & (i \in I, \ \check \mu \in Y) \, , \\
\big[ X^+_i, X^-_j \big] & = & 0 & \quad \quad (i,j \in I, \ i \neq j) \, .
\end{array}} \right.
\end{equation*}
\end{defi}

The universal enveloping algebra $\FUc[\glieF]$ of $\glieF$ is the \hbox{$\KK$-algebra} generated by the elements \hbox{$X^\pm_i$ ($i \in I$)}, the \hbox{$\ZZ$-module $Y$} and subject to the non-deformable relations of $\glie$. We denote by $\KK[Y]$ the symmetric algebra on the $\KK$-vector space $Y \otimes_\ZZ \KK$. We designate also by $Y$, when the context makes it unambiguous, the \hbox{$\KK$-algebra} morphism from $\KK[Y]$ to $\FUc[\glieF]$ which sends $\check \mu$ to $\check \mu$ for every $\check \mu \in Y$. \\

A \emph{$\glieF$-pointed} algebra designates a $\FUc$-pointed algebra. We regard the algebra $\Uc$ as a $\glieF$-pointed $\KK$-algebra: the $\FUc$-point, which we denote by $\pic$, is the projection map from $\FUc[\glieF]$ \hbox{to $\Uc$}.

\vsp

\begin{defi}
We denote by $\FUh[\glieF]$ the constant formal deformation of the \hbox{${\glieF}$-pointed} \hbox{$\KK$-algebra} $\FUc[\glieF]$.
\end{defi}

\subsection{$\psi$-integrable representations}
\begin{defi}
We denote by $\Coloh[I]$ the ring of functions from $I$ to $\Coloh$. An element $\psi = {(\psi_i)}_{i \in I}$ in $\Coloh[I]$ is called an ${I \! \! \:}$-colouring of the crystal $\BB$ with values in $\Kh$, it is said admissible if the colouring $\psi_i$ is for every $i \in I$.
\end{defi}

Let $i \in I$. We denote by $\FDeltah$ the \hbox{$\Kh$-algebra} morphism $\FDeltac \cch$. It is a morphism from $\FUh[\sltF]$ to $\FUh$.

\begin{defi} \label{defi_phiint}
Let $\psi \in \Coloh[I]$. A representation ${V_{\! \! \; h}}$ of $\FUh$ is said \hbox{$\psi$-integrable} if the following two conditions are satisfied.
\begin{enumerate}
\item[a)] The action of every element $\check \mu \in Y$ on ${V_{\! \! \; h}}$ is diagonalisable.
\item[b)] For every $i \in I$, the pull-back of ${V_{\! \! \; h}}$ via $\FDeltah$ is isomorphic to a topologically direct sum of representations $\Lh{n,\psi_i}$ with $n \in \NN$.
\end{enumerate}
\end{defi}

Note, in view of condition b) in the previous definition, that the $\psi$-integrable representations of $\FUh$ are topologically free $\Kh$-modules.

\subsection{The algebra $\Uh$}
In this subsection, $\psi = {(\psi_i)}_{i \in I}$ designates a $I$-colouring of the crystal $\BB$ with values in $\Kh$.


\begin{defi}
$\phantom{}$
\begin{enumerate}
\item[\textbullet] We denote by $\Ikerh$ the two-sided ideal of $\FUh$ formed by the elements acting by zero on all the $\psi$-integrable representations of $\FUh$.
\item[\textbullet] We denote by $\Uh$ the quotient of $\FUh$ by $\Ikerh$.
\end{enumerate}
\end{defi}

Note that when $\glie$ is the Lie algebra $\slt$, the previous definition is equivalent to the one given in rank one: see definition \ref{defi_IkerhAone}. \\

We denote by $\pih$ the projection map from $\FUh$ to $\Uh$. We endow the $\Kh$-algebra $\Uh$ with the $\FUc$-point $\pih \circ \iota$, where $\iota$ designates the inclusion map from $\FUc$ to $\FUh$.

\begin{rem} \label{rem_Uhperfect}
We regard the $\psi$-integrable representations of $\FUh$ also as representations of $\Uh$. By definition then, an element in $\Uh$ is zero if and only if it acts as zero on all the $\psi$-integrable representations.
\end{rem}

The algebra $\Uh$ is built from the $\psi$-integrable representations, which are topologically free $\Kh$-modules, and therefore inherits from them the property.

\begin{prop} \label{prop_Uhhfree}
The $\Kh$-module $\Uh$ is topologically free.
\end{prop}

\proof
The $\Kh$-module $\FUh$ is by definition topologically free. It is in particular complete. Being a quotient of $\FUh$, the $\Kh$-module $\Uh$ is thus also complete. \\

Let us denote by $\mathrm{Int}(\psi)$ the class of $\psi$-integrable representations of $\FUh$. Let now $V$ be a $\psi$-integrable representation of $\FUh$, we denote by $A_V$ the \hbox{$\Kh$-algebra} $\Endh (V)$ and we denote by $f_V$ the \hbox{$\Kh$-algebra} morphism from $\FUh$ to $A_V$, associated with the representation $V$. We denote by $f$ the product of the morphisms $f_V$ ($V \in \mathrm{Int}(\psi)$), it is a $\Kh$-algebra morphism \hbox{from $\FUh$} \hbox{to $\prod_{V \in \mathrm{Int}(\psi)} A_V$}. The definition of $\Ikerh$ can be rephrased by the following equation:
\begin{equation} \label{eq_Uhhfree}
\Ikerh \ = \ \ker (f) \, .
\end{equation}
Let $V$ be a $\psi$-integrable representation of $\FUh$, the $\Kh$-module $V$ is by definition topologically free. It is in particular separated and torsion-free. Hence, the \hbox{$\Kh$-module $\prod_{V \in \mathrm{Int}(\psi)} A_V$} is also \hbox{separated} and torsion-free. Besides, in view of equation \eqref{eq_Uhhfree}, one can identify $\Uh$ with a $\Kh$-submodule \hbox{of $\prod_{V \in \mathrm{Int}(\psi)} A_V$}. As a consequence, the $\Kh$-module $\Uh$ is \hbox{separated} and torsion-free. \\

To sum up, the $\Kh$-module $\Uh$ is separated, \hbox{complete} and torsion-free. It is thus topologically free (see for example \cite[proposition XVI.2.4]{Kas}).
\qed

\subsection{The diagram $\UUh$}
In this subsection, $\psi = {(\psi_i)}_{i \in I}$ designates a $I$-colouring of the crystal $\BB$ with values in $\Kh$. \\

We denote by $\II$ the category defined by the following.
\begin{enumerate}
\item[\textbullet] The objects of $\II$ are the elements of the Dynkin \hbox{set $I$} together with an additional object $t_I$.
\item[\textbullet] The morphisms of $\II$ are the identity morphisms, together with a morphism denoted $\Delta_i$ from $i$ to $t_I$, for each $i \in I$.
\end{enumerate}

\vsp
\vsp

Let $i \in I$, we denote by $\FDeltac$ the $\KK$-algebra morphism from $\FUc[\sltF]$ \hbox{to $\FUc$} which \hbox{sends $X^{\pm}$} to $X_i^\pm$ and $H$ to $\check \alpha_i$.

\begin{defi}
We denote by $\FUUc$ the $\KK$-diagram of shape $\II$ defined by the following:
\begin{IEEEeqnarray*}{rCll}
\FUUc (i) & \ := \ & \FUc[\sltF] & \quad \quad (i \in I) \, , \\
\FUUc (t_I) & \ := \ & \FUc \, , & \\
\FUUc (\Delta_i) & \ := \ & \FDeltac & \quad \quad (i \in I) \, .
\end{IEEEeqnarray*}
\end{defi}

Thanks to the simple $\slt$-directions (see subsection \ref{subs_KMone}), we also endow the algebra $\Uc$ with a structure of $\KK$-diagram.

\begin{defi}
We denote by $\UUc$ the $\KK$-diagram of shape $\II$, defined by the following:
\begin{IEEEeqnarray*}{rCll}
\UUc (i) & \ := \ & \Uc[\slt] & \quad \quad (i \in I) \, , \\
\UUc (t_I) & \ := \ & \Uc \, , & \\
\UUc (\Delta_i) & \ := \ & \Deltac & \quad \quad (i \in I) \, .
\end{IEEEeqnarray*}
\end{defi}

A \emph{$\glieF$-pointed diagram} designates a $\FUUc$-pointed diagram. We regard the diagram $\UUc$ as a $\glieF$-pointed $\KK$-diagram, whose $\FUUc$-point $\Pic$ is defined by the following:
\begin{IEEEeqnarray*}{rCl}
{(\Pic)}_i & \ := \ & \pic[\slt] \quad \ (i \in I) \, , \\
{(\Pic)}_{t_I} & \ := \ & \pic \, .
\end{IEEEeqnarray*}

\begin{defi-prop}
Let $i \in I$. We denote by $\Deltah$ the unique $\Kh$-algebra morphism from $\Uh[\slt,\psi_i]$ to $\Uh$ which makes the following diagram commute:
\begin{equation} \label{eq_iotah}
\shorthandoff{;:!?} \xymatrix @C=5pc @R=1.5pc {
\FUh[\sltF] \ar[r]^-{\FDeltah \ } \ar@{->>}[d]_-{\pih[\slt,\psi_i] \ } & \FUh \ar@{->>}[d]^-{\ \pih} \\
\Uh[\slt,\psi_i] \ar[r]_-{\Deltah} & \ \Uh \, .
} \end{equation}
\end{defi-prop}

\proof
Let $x \in \Ikerh[\slt,\psi_i]$ and denote by $y$ the image of $x$ by $\FDeltah$. \hbox{Let $V$} be a $\psi$-integrable representation of $\FUh$. By definition, the pull-back of $V$ via the morphism $\FDeltah$ is isomorphic to a direct sum of representations $\Lh{n,\psi_i}$ with $n \in \NN$. Since $x$ belongs to the ideal $\Ikerh[\slt,\psi_i]$, the element $y$ than acts as zero on $V$. This being true for every $\psi$-integrable representation of $\FUh$, the image of $y$ by $\pih$ is zero. Hence, the map $\pih \circ \FDeltah$ can be factorised through $\Uh[\slt,\psi_i]$, i.e. there exists a $\Kh$-algebra morphism morphism $\Deltah$ which makes the diagram \eqref{eq_iotah} commutes. The unicity part of the proposition follows from the surjectivity of $\pih[\slt,\psi_i]$.
\qed \\

Thanks to the morphisms $\Deltah$ ($i \in I$), we endow the algebra $\Uh$ with a structure of diagram.

\begin{defi}
We denote by $\UUh$ the $\KK$-diagram of shape $\II$, defined by the following:
\begin{IEEEeqnarray*}{rCll}
\UUh (i) & \ := \ & \Uh[\slt,\psi_i] & \quad \quad (i \in I) \, , \\
\UUh (t_I) & \ := \ & \Uh \, , & \\
\UUh (\Delta_i) & \ := \ & \Deltah & \quad \quad (i \in I) \, .
\end{IEEEeqnarray*}
\end{defi}

We endow the $\Kh$-diagram $\UUh$ with the $\FUUc$-point $\Pih$, defined by the following:
\begin{IEEEeqnarray*}{rCl}
{(\Pih)}_i & \ := \ & \pih[\slt,\psi_i] \, \circ \, \iota_i  \quad \ (i \in I) \, , \\
{(\Pih)}_{t_I} & \ := \ & \pih \, \circ \, \iota \, ,
\end{IEEEeqnarray*}
where $\iota_i$ ($i \in I$) designates the inclusion map from $\FUc[\slt]$ to $\FUh[\slt]$ and \hbox{where $\iota$} designates the inclusion map from $\FUc$ to $\FUh$.


\subsection{$h$-admissible $I$-colourings}
\begin{defi}
Let $\psi \in \Coloh[I]$. The ${I \! \! \; }$-colouring $\psi$ is said $h$-admissible if the \hbox{colouring $\psi_i$} is $h$-admissible for every $i \in I$.
\end{defi}

The following definition will be justified a posteriori by theorem \ref{thm_Uhdned}.

\begin{defi}
Let $\psi$ be a $h$-admissible ${I \! \! \; }$-colouring of the crystal $\BB$. The \hbox{$\glieF$-pointed} \hbox{$\Kh$-algebra} $\Uh$ is called the generalised quantum enveloping algebra associated to $\psi$.
\end{defi}



\section{Classical and standard quantum realisations}



\subsection{Standard quantum algebras}

Let $k,k' \in \NN$ and let $i \in I$. The following elements are defined in the ring $\Kh$, where we recall that $q$ designates the formal power series $\exp(h)$:
\begin{IEEEeqnarray*}{rClCrCl}
q_i & \ := \ & q^{d_i} \, , & \quad \quad \quad & {[k]}_{q_i} & \ := \ & \frac{q_i^k - q_i^{-k}}{q_i - q_i^{-1}} \ , \\
{[k]}_{q_i} ! & \ := \ & \prod_{k'=0}^k \, {[k']}_{q_i} \ , & \quad \quad \quad & {k \brack k'}_{q_i} & \ := \ & \frac{[k+k']_{q_i} !}{[k]_{q_i} ! \ [k-k']_{q_i} !} \ \cdot
\end{IEEEeqnarray*}

The \emph{(formal) standard quantum algebra} $\Uh[\glie]$ is the $\Kh$-algebra topologically generated by the elements $X^-_i, X^+_i$ ($i \in I$), the $\ZZ$-module $Y$, subject to the non-deformable relations \eqref{eq_KM_nondeformable} of $\glie$ and to the following relations:
$$ \left\{ {\renewcommand{\arraystretch}{2.5} \begin{array}{lr}
\displaystyle{ \big[ X^+_i, X^-_i \big]} = \displaystyle{ \frac{K_i - K_i^{-1}}{q_i - q_i^{-1}}} \, , \quad \quad \text{where } \, K_i := \exp (d_i \: \! h \; \! {\check \alpha_i}) & (i \in I) \, , \\
\displaystyle{ \sum_{k+k' = 1 - C_{ij}} (-1)^k \, {1 - C_{ij} \brack k}_{q_i} \, {(X_i^{\pm})}^k \; \! X^{\pm}_j \; \!  {(X_i^{\pm})}^{k'}} = \, \displaystyle 0 & \quad \quad (i,j \in I, \ i \neq j) \, .
\end{array}} \right. $$



\subsection{Theorem}
Recall the definition of the classical and quantum colouring of $\BB$:
$$ (\psicl)^\pm (n,k) \ := \ k \, , \quad \quad (\psiq)^\pm (n,k) \ := \ {[k]}_q \quad \quad \quad (n,k \in \NNI, \ k \leq n) \, . $$
We denote by $\psiclI$ and $\psiqI$ the $h$-admissible $I$-colourings of $\BB$ with values in $\Kh$ defined by the following:
$$  {(\psiclI)}_i \ := \ \psicl \, , \quad \quad {(\psiqI)}_i \ := \ \psi_{q_i} \quad \quad \quad (i \in I) \, . $$

\begin{thm} \label{thm_exUh}
$\phantom{}$
\begin{itemize}
\item[1)] The GQE algebra $\Uh[\glie,\psiclI]$ and the constant formal deformation $\Uc \cch$ are isomorphic $\glieF$-pointed $\Kh$-algebras.
\item[2)] The GQE algebra $\Uh[\glie,\psiqI]$ and the standard quantum algebra $\Uh[\glie]$ are isomorphic $\glieF$-pointed $\Kh$-algebras.
\end{itemize}
\end{thm}

\proof
The proof is analog to the one given in rank 1(see proposition \ref{prop_exUhone}) and makes use of proposition \ref{prop_UBperfect}.
\qed

\section{GQE algebras without Serre relations} \label{sect_GQEwithout}
\subsection{Definition}
In this subsection, $\psi$ designates a $h$-admissible $I$-colouring of $\BB$. Let $i \in I$, we denote by ${S^{\! \: \psi_i}_{\! \! \; h}} = {({S^{\! \: \psi_i}_{\! \! \; h,p}})}_{p \in \NN}$ the solution of the GQE equation $\GQEeqh[\psi_i,\psi_i]{-1}$: see proposition \ref{prop_GQEequnicity} and \hbox{theorem \ref{thm_GQEeq}}.

\vsp
\vsp

\begin{defi}
The GQE algebra without Serre relations associated to $\psi$, denoted by $\Uth$, is the $\Kh$-algebra topologically generated by $X^\pm_i$ ($i \in I$), the \hbox{$\ZZ$-module $Y$}, subject to the non-deformable relations \eqref{eq_KM_nondeformable} of $\glie$ and to the following relations:
\begin{equation} \label{eq_UthXii}
\big[ X_i^+,  X_i^- \big] \ = \ \sum_{a \in \NN} \, {(X_i^-)}^a \! \; {S^{\! \: \psi_i}_{\! \! \; h,a}} (\check \alpha_i) \! \; {(X_i^+)}^a \quad \quad \quad (i \in I) \, .
\end{equation}
\end{defi}

Note that the series in the relations \eqref{eq_UthXii} converge since the \hbox{sequences ${S^{\! \: \psi_i}_{\! \! \; h}}$} converge to zero for all $i \in I$. \\

We denote by $\pith$ the projection map from $\FUh$ to $\Uth$, and we endow the algebra $\Uth$ with a structure of $\glieF$-pointed $\Kh$-algebra thanks to $\pith$.

\begin{defi-prop} \label{defi-prop_vpith}
There exists a unique and surjective morphism of $\glieF$-pointed $\Kh$-algebra from $\Uth$ to $\Uh$. We denote it by $\vpith$.
\end{defi-prop}

\proof
Thanks to the morphism $\Deltah$ ($i \in I$) and in view of theorem \ref{thm_presentationUhone}, remark that the relations \eqref{eq_UthXii} hold in $\Uh$.
\qed \\

Proposition \ref{defi-prop_vpith} asserts that the following diagram commutes, and thus presents a hierarchy between the algebra $\FUh$, $\Uth$ and $\Uh$:
\begin{equation} \label{diag_rhoth}
\shorthandoff{;:!?} \xymatrix @C=5pc @R=1.5pc {
\FUh \ar@{->>}[r]^-{\pith} \ar@{->>}[dr]_-{\pih} & \, \Uth \ar@{->>}[d]^-{\, \vpith} \\
& \, \Uh \, .
}
\end{equation}





\subsection{Formal deformations}
It is possible to give different proofs of the theorem below. We have chosen here to use the diamond lemma. The reader is referred to \cite{Ber} where the lemma is originally proved and to \cite{Gui} for an extension in the $h$-adic case. We use the same terminology than the one employed in the two references.

\begin{thm} \label{thm_Uth}
Let $\psi = {(\psi_i)}_{i \in I}$ be a $h$-admissible $I$-colouring. The \hbox{$\glieF$-pointed} $\Kh$-algebra $\Uth$ is a formal deformation of $\Utc$.
\end{thm}

\proof
Pick a family ${(\check \mu_k)}_{1 \leq k \leq r}$ with values in $Y$ which contains the linearly independent family ${(\check \alpha_i)}_{i \in I}$ and such that the induced family ${(\check \mu_k \otimes 1)}_{1 \leq k \leq r}$ is a basis of the $\KK$-vector space $Y \otimes_\ZZ \KK$. \\

Denote by $\Lambda$ the set $\bigcup_{i \in I} \{ X^-_i, X^+_i \} \cup {\{ \check \mu_k \; \! ; \, 1 \leq k \leq r \}}$ and denote by $\prec$ the partial order on $\Lambda$ defined by the following:
$$ \check \mu_k \, \prec \, X^-_i \, \prec \, X^+_j \quad (i,j \in I, \ 1 \leq k \leq r) \quad \text{ and } \quad \check \mu_k \, \prec \, \check \mu_{k'} \quad (1 \leq k < k' \leq r) \, . $$
Let $M = x_1 x_2 \cdots x_l$ be a monomial in $\langle \Lambda \rangle$ (with $l \in \NN$). The number $\mathrm{Inv} (M)$ of inversions of $M$ and the length $\mathrm{lg} (M)$ of $M$ are defined by the following:
$$ \mathrm{Inv} (M) \ := \ \# \Big\{ (s,t) \, ; \ s < t \text{ and } x_s \succ x_t \Big\} \, , \quad \ \mathrm{lg} (M) \ := \ l \, . $$

The partial order $\prec$ on $\Lambda$ induces a semigroup partial order on $\KK \langle \Lambda \rangle$ again denoted \hbox{by $\prec$} and defined by the following for $M, M'$ two monomials in $\langle \Lambda \rangle$:
$$ M \, \prec \, M' \quad \text{if} \quad \big( \mathrm{Inv} (M) , \, \mathrm{lg} (M) \big) \, <_{\: \! \text{lexicographic}} \, \big( \mathrm{Inv} (M') , \, \mathrm{lg} (M') \big) \, . $$

We define now a reduction system $\mathcal S$, compatible with the semigroup partial order $\prec$ previously defined and accordingly to the non-deformable relations \eqref{eq_KM_nondeformable} of $\glie$ and to the relations \eqref{eq_UthXii}:
\begin{IEEEeqnarray*}{CCl}
\mathcal S_1 & \ := \ & \Bigg\{ \big( \check \mu_{k'} \; \! \check \mu_k \, , \, \check \mu_k \; \! \check \mu_{k'} \big) \, ; \ 1 \leq k < k' \leq r \Bigg\} \, , \\
\mathcal S_2 & := & \Bigg\{ \big( X_i^{\pm} \: \! \check \mu_k \, , \, \check \mu_k \: \! X^{\pm}_i \; \! \mp \; \! \langle \check \mu_k, \alpha_i  \rangle \: \! X^{\pm}_i \big) \, ; \ i \in I, \ 1 \leq k \leq r \Bigg\} \, , \\
\mathcal S_3 & := & \Bigg\{ \big( X_j^+ \: \! X^-_i \, , \, X^-_i \: \! X_j^+ \big) \, ; \ i,j \in I, \ i \neq j \Bigg\} \, , \\
\mathcal S_4 & := & \Bigg\{ \big( X_i^+ \: \! X^-_i \, , \, X^-_i \: \! X_i^+ \, + \, \sum_{a \in \NN} \, P_a^{\: \! \psi_i,\psi_i} (\check \alpha_i) \left(  X_i^- \right)^a \left(  X_i^+ \right)^a \big) \, ; \ i \in I \Bigg\} \, , \\
&& \\
\mathcal S & := & \mathcal S_1 \, \sqcup \, \mathcal S_2 \, \sqcup \, \mathcal S_3  \, \sqcup \, \mathcal S_4 \, .
\end{IEEEeqnarray*}

Let us now prove that the ambiguities of $\mathcal S$ are resolvable. We treat both the inclusion ambiguities and the overlap ambiguities. The proofs of resolvability are given below in terms of commutative diagrams, where the arrows stand for the reduction maps of $\KK \langle \Lambda \rangle$ induced by $\mathcal S$. The symbol `$\infty$' above an arrow indicates an infinite composition of a reduction map, where the convergence holds with respect to the $h$-adic topology. Let $1 \leq k \leq r$ and let $i \in I$, we denote by $a_{ki}$ the number $\langle \check \mu_k, \alpha_i  \rangle$. \\

$\bullet$ Let $1 \leq k < k' < k'' \leq r$,

$$ \shorthandoff{;:!?} \xymatrix @C=2pc @R=1.8pc {
& \check \mu_{k''} \: \! \check \mu_{k'} \: \! \check \mu_k \ar[dl]_-{\mathcal S_1} \ar[dr]^-{\mathcal S_1} & \\
\check \mu_{k'} \: \! \check \mu_{k''} \: \! \check \mu_k \ar[d]_-{\mathcal S_1} && \check \mu_{k''} \: \! \check \mu_k \: \! \check \mu_{k'} \ar[d]^-{\mathcal S_1} \\
\check \mu_{k'} \: \! \check \mu_k \: \! \check \mu_{k''} \ar[dr]_-{\mathcal S_1} && \check \mu_k \: \! \check \mu_{k''} \: \! \check \mu_{k'} \ar[dl]^-{\mathcal S_1} \\
& \check \mu_k \: \! \check \mu_{k'} \: \! \check \mu_{k''} \, . & \\
&&
}
$$

$\bullet$ Let $i \in I$ and let $1 \leq k < k' \leq r$,
$$ \shorthandoff{;:!?} \xymatrix @!0 @C=7pc @R=1.5pc {
& X^{\pm}_i \: \! \check \mu_{k'} \: \! \check \mu_k \ar[ddl]_-{\mathcal S_2} \ar[dddr]^-{\mathcal S_1} & \\
&& \\
\check \mu_{k'} \: \! X^{\pm}_i \: \! \check \mu_k \mp \: \! a_{k'i} \: \! X^{\pm}_i \: \! \check \mu_k \ar[dd]_-{\mathcal S_2} && \\
&& X^{\pm}_i \: \! \check \mu_k \: \! \check \mu_{k'} \ar[dd]^-{\mathcal S_2} \\
\check \mu_{k'} \: \! \check \mu_k \: \! X^{\pm}_i \mp a_{ki} \: \! \check \mu_{k'} X^{\pm}_i \mp \: \! a_{k'i} \: \! X^{\pm}_i \: \! \check \mu_k \ar[dd]_-{\mathcal S_1} && \\
&& \check \mu_k \: \! X^{\pm}_i \: \! \check \mu_{k'} \mp \: \! a_{ki} \: \! X^{\pm}_i \: \! \check \mu_{k'} \ar[dddl]^-{\mathcal S_2} \\
\check \mu_k \: \! \check \mu_{k'} \: \! X^{\pm}_i \mp a_{ki} \: \! \check \mu_{k'} X^{\pm}_i \mp \: \! a_{k'i} \: \! X^{\pm}_i \: \! \check \mu_k \ar[ddr]_-{\mathcal S_2} && \\
&& \\
& \check \mu_k \: \! \check \mu_{k'} \: \! X^{\pm}_i \mp a_{ki} \: \! \check \mu_{k'} X^{\pm}_i \mp \: \! a_{k'i} \: \! \check \mu_k \; \! X^{\pm}_i \: \! + a_{ki} \: \! a_{k'i} \: \! X^{\pm}_i \, . & \\
&&
}
$$

$\bullet$ Let $i,j \in I$ such that $i \neq j$ and let $1 \leq k \leq r$,
$$ \shorthandoff{;:!?} \xymatrix @!0 @C=6.5pc @R=3pc {
& X_j^+ \: \! X_i^- \: \! \check \mu_k \ar[dl]_-{\mathcal S_3} \ar[dr]^-{\mathcal S_2} & \\
X_i^- \: \! X^+_j \: \! \check \mu_k \ar[d]_-{\mathcal S_2} && X_j^+ \: \! \check \mu_k \: \! X_i^- \, + \, a_{ki} \: \! X^+_j \: \! X^-_i \ar[d]^-{\mathcal S_2} \\
X^-_i \: \! \check \mu_k \: \! X^+_j \, - \, a_{kj} \: \! X^-_i \: \! X^+_j \ar[dr]_-{\mathcal S_2} && \check \mu_k \: \! X^+_j \: \! X^-_i \, + \, (a_{ki} - a_{kj}) \: \! X^+_j \: \! X_i^- \ar[dl]^-{\mathcal S_3} \\
& \check \mu_k \: \! X^-_i \: \! X^+_j \, + \, (a_{ki} - a_{kj}) \: \! X^-_i \: \! X_j^+ \, . & \\
&&
} $$

$\bullet$ Let $i \in I$ and let $1 \leq k \leq r$ such that $k(i) < k$, where $\check \mu_{k(i)} = \check \alpha_i$,
$$ \shorthandoff{;:!?} \xymatrix @!0 @C=8pc @R=2.5pc {
&& \check \mu_k \: \! X^-_i \: \! X_i^+ \, + \, \sum_{a \in \NN} \, \check \mu_k \, P_a^{\: \! \psi_i,\psi_i} (\check \alpha_i) \left(  X_i^- \right)^a \left(  X_i^+ \right)^a \ar[dddd]^-{\mathcal S_1^{\infty}} \\
& \check \mu_k \: \! X^+_i \: \! X^-_i \ar[ur]_-{\mathcal S_4} & \\
& X^+_i \: \! \check \mu_k \: \! X^-_i + a_{ki} \: \! X^+_i \: \! X^-_i \ar[u]^-{\mathcal S_2} & \\
X^+_i \: \! X_i^- \: \! \check \mu_k \ar[dddr]_-{\mathcal S_4} \ar[ur]_-{\mathcal S_2} && \\
&& \check \mu_k \: \! X^-_i \: \! X_i^+ \, + \, \sum_{a \in \NN} \, P_a^{\: \! \psi_i,\psi_i} (\check \alpha_i) \, \check \mu_k \: \! \left(  X_i^- \right)^a \left(  X_i^+ \right)^a \, . \\
&& \\
& X^-_i \: \! X_i^+ \: \! \check \mu_k \, + \, \sum_{a \in \NN} \, P_a^{\: \! \psi_i,\psi_i} (\check \alpha_i) \left(  X_i^- \right)^a \left(  X_i^+ \right)^a \check \mu_k \ar[uur]_-{\mathcal S_2^{\infty}} & \\
&&
}$$

$\bullet$ Let $i \in I$ and let $1 \leq k \leq r$ such that $k(i) \geq k$, where $\check \mu_{k(i)} = \check \alpha_i$,
$$ \shorthandoff{;:!?} \xymatrix @!0 @C=8pc @R=3pc {
&& \check \mu_k \: \! X^+_i \: \! X^-_i \ar[dd]^-{\mathcal S_4} \\
& X^+_i \: \! \check \mu_k \: \! X^-_i + a_{ki} \: \! X^+_i \: \! X^-_i \ar[ur]^-{\mathcal S_2} & \\
X^+_i \: \! X_i^- \: \! \check \mu_k \ar[ur]_-{\mathcal S_2} \ar[dr]^-{\mathcal S_4} && \check \mu_k \: \! X^-_i \: \! X_i^+ \, + \, \sum_{a \in \NN} \ \check \mu_k \, P_a^{\: \! \psi_i,\psi_i} (\check \alpha_i) \left(  X_i^- \right)^a \left(  X_i^+ \right)^a \, . \\
& X^-_i \: \! X_i^+ \: \! \check \mu_k \, + \, \sum_{a \in \NN} \, P_a^{\: \! \psi_i,\psi_i} (\check \alpha_i) \left(  X_i^- \right)^a \left(  X_i^+ \right)^a \check \mu_k \ar[dr]_-{\mathcal S_2^{\infty}} & \\
&& \check \mu_k \: \! X^-_i \: \! X_i^+ \, + \, \sum_{a \in \NN} \, P_a^{\: \! \psi_i,\psi_i} (\check \alpha_i) \, \check \mu_k \, \left(  X_i^- \right)^a \left(  X_i^+ \right)^a \ar@<-30pt>[uu]_-{\mathcal S_1^{\infty}}
}$$
\qed

\subsection{The GQE Serre relations}
In this subsection, $\psi$ designates a $h$-admissible $I$-colouring of $\BB$. Let $i \in I$, we denote by ${\bar S^{\! \: \psi_i}_{\! \! \; h}} = {({\bar S^{\! \: \psi_i}_{\! \! \; h,p}})}_{p \in \NN}$ the solution of the GQE equation $\GQEeqh[\psi_i,\psicl]{0}$: see proposition \ref{prop_GQEequnicity} and theorem \ref{thm_GQEeq}. We recall that $\psicl$ designates the classical colouring of $\BB$ (see subsection \ref{subs_exclqUhone}).

\begin{defi-prop} \label{defi-prop_GQESerre}
The following relations are called the GQE Serre relations associated to $\psi$, they are satisfied in $\Uh$:
\begin{equation*} \label{eq_GQESerre} \left\{ {\renewcommand{\arraystretch}{2}\begin{array}{l}
\sum_{k+k' = 1 - C_{ij}} {(-1)}^k \, {1 - C_{ij} \choose k} \, \big( {}^\psi \! X_i^{\pm} \big)^k \! \; X^{\pm}_j \! \; \big( {}^\psi \! X_i^{\pm} \big)^{k'} \ = \ 0 \quad \quad (i,j \in I, \ i \neq j) \, , \\
\text{where } \ {}^\psi \! X_i^{\pm} \ := \ \sum_{a \in \NN} \, \big( X_i^\mp \big)^a \! \; {\bar S^{\! \: \psi_i}_{\! \! \; h,a}} (\pm \check \alpha_i) \! \; \big( X^\pm_i \big)^{a+1}
\end{array}} \right.
\end{equation*}
\end{defi-prop}

Note that the series $\sum_{a \in \NN} {(X^-)}^a {\bar S^{\! \: \psi_i}_{\! \! \; h,a}} (H) {(X^+)}^{a+1}$ ($i \in I$) converge in $\Uh$ since the \hbox{sequences ${\bar S^{\! \: \psi_i}_{\! \! \; h}}$} converge to zero for all $i \in I$.

\proof
Fix $i \neq j$ in $I$ and let $V$ be an integrable representation of $\Uh$. Let $a \in \ZZ$, we denote by $\Ugeneric{h,a}{\slt,\psi_i}$ the $\Kh$-submodule \hbox{of $\Uh[\slt,\psi_i]$} whose elements $x$ satisfy $[\check \alpha_i,x] = 2ax$. Remark that $xy$ belong \hbox{to $\Ugeneric{h,a+b}{\slt,\psi_i}$} for every $x \in \Ugeneric{h,a}{\slt,\psi_i}$, for every $y \in \Ugeneric{h,b}{\slt,\psi_i}$ and for \hbox{every $a,b \in \ZZ$}. Recall theorem \ref{thm_UhoneHtrivial} and denote by $\gamma$ the $\KK$-algebra morphism from $\Uc[\slt]$ to $\Uh[\slt,\psi_i]$ defined by the following:
\begin{IEEEeqnarray*}{lCl}
\gamma (X^-) & \ := \ & X^- \, , \\
\gamma (H) \ & := & H \, , \\
\gamma (X^+) & := & \sum_{a \in \NN} \, \big( X_i^\mp \big)^a \! \; {\bar S^{\! \: \psi_i}_{\! \! \; h,a}} (\pm \check \alpha_i) \! \; \big( X^\pm_i \big)^{a+1} \, .
\end{IEEEeqnarray*}
Since $\gamma (X^+)$ lies in $\Ugeneric{h,2}{\slt,\psi}$, the element $\gamma((X^+)^k)$ belongs to $\Ugeneric{h,2k}{\slt,\psi}$ for \hbox{every $k \in \NN$}. Hence, the pull-back of $V$ via $\Deltah \circ \gamma$ is an integrable representation of $\Uc[\slt]$, we denote it by $V_i$. Thanks to proposition \ref{prop_KMintone}, we know that the representation $V_i$ is thus isomorphic to a direct sum of representations $\Lc{n}$ with $n \in \NN$. Regard the $\KK$-vector space $\mathrm{End}_\KK(V_i)$ as a representation of $\Uc[\slt]$, where the actions \hbox{of $X^\pm$} and $H$ on a vector $f$ are given by $[X^\pm,f]$ and $[H,f]$. Thanks to \hbox{proposition \ref{prop_KMintone}} again, on can identify the representation $\mathrm{End}_\KK(V_i)$ with a product of \hbox{representations $\Lc{n}$} with $n \in \NN$. The element $X_j^+$, when viewed as a vector in $\mathrm{End}_\KK(V_i)$, lies in a subproduct of representations $\Lc{C_{ij}}$, since the following relations hold in the algebra $\FUh$:
$$ [\check \alpha_i, X^+_j] \ = \ C_{ij} \; \! X^+_j \, , \quad \quad [X^-_i, X^+_j] \ = \ 0 \, . $$
Therefore, the element $(X^+)^{1 - C_{ij}}$ acts as zero on $X^+_j$, when viewed as a vector in $\mathrm{End}_\KK(V_i)$. This being true for every integrable representation $V$ \hbox{of $\Uh$}, remark \ref{rem_Uhperfect} implies that the following relation holds in $\Uh$:
$$ \sum_{s+t = 1 - C_{ij}} (-1)^s \, {1 - C_{ij} \choose s} \, \big( {}^\psi \! X_i^+ \big)^s \; \! X^+_j \; \!  \big( {}^\psi \! X_i^+ \big)^t \ = \ 0 \, . $$
By considering the involutive $\Kh$-algebra automorphism of $\Uh$ which sends $X^\pm_i$ to $-X^\mp_i$ and $\check \mu$ to $- \check \mu$ for every $i \in I$ and $\check \mu \in Y$, one concludes.
\qed

\vsp
\vsp

\begin{rem}
The GQE Serre relations associated to $\psi$ are formal deformations of the classical Serre relations, i.e. the following equality holds in $\hzero{\Uh}$:
\begin{IEEEeqnarray*}{rl}
\hzero \pi^{\Uh} \Bigg( \, & \sum_{k+k' = 1 - C_{ij}} (-1)^s \, {1 - C_{ij} \choose k} \, \big( {}^\psi \! X_i^{\pm} \big)^k X^{\pm}_j \big( {}^\psi \! X_i^{\pm} \big)^{k'} \Bigg) \\
= \ & \sum_{k+k' = 1 - C_{ij}} (-1)^k \, {1 - C_{ij} \choose k} \, {(X_i^{\pm})}^k \; \! X^{\pm}_j \; \!  {(X_i^{\pm})}^{k'} \, , \quad \quad \forall \, i, j \in I, \ i \neq j \, .
\end{IEEEeqnarray*}
\end{rem}

\section{The GQE conjecture and its consequences}
\subsection{The GQE conjecture}
\begin{conj}[GQE conjecture] \label{conj_GQE}
Let $\psi$ be a $h$-admissible $I$-colouring. Every representation in the category $\Ointc$ admits a $\psi$-integrable \hbox{formal deformation} along $\FUh$.
\end{conj}



In the following of this section, \textbf{we assume that the statement of the GQE conjecture holds}. The author believes that the GQE conjecture is true, a proof of it is currently in preparation and will be published in a forthcoming paper. Note that the GQE algebras without Serre relations introduced in \hbox{section \ref{sect_GQEwithout}} will play a central role.

\subsection{Formal deformations}
\begin{thm} \label{thm_Uhdned}
Let $\psi \in \Coloh[I]$. The following two assertions are equivalent.
\begin{enumerate}
\item[(i)] The $I$-colouring $\psi$ is $h$-admissible.
\item[(ii)] The $\glieF$-pointed $\Kh$-diagram $\UUh$ is a formal deformation of $\UUc$.
\end{enumerate}
\end{thm}

\proof
The implication $ii) \Rightarrow i)$ is a direct consequence of theorem \ref{thm_Uhonednd}. Let us prove the converse. Suppose that assertion (i) holds. \hbox{Let $x_h$} be an element in the ideal $\Ikerh$ \hbox{of $\FUh$} and recall that we identify the specialisation at ${h \! = \! 0}$ of $\FUh$ with $\FUc$. Denote by $x_0$ the image \hbox{in $\Uc$} of $\hzero \pi^{\FUh} (x_h)$ by $\pic$. In view of the GQE \hbox{conjecture \ref{conj_GQE}}, the element $x_0$ acts as zero on all the representations of $\Uc$ in $\Ointc$. Hence, thanks to proposition \ref{prop_UBperfect}, the element $x_0$ is zero. Since the functor $\hzero \bullet$ from the category of $\Kh$-modules to the category of $\KK$-vector spaces is right-exact, we have proved that there exists a morphism $f$ of $\glieF$-pointed $\KK$-algebra from $\hzero{\Uh}$ to $\Uc$ which makes the following diagram commute:
$$ \shorthandoff{;:!?} \xymatrix @C=7pc @R=3pc {
\FUc \ar@{->>}[dr]^-{\pic} \ar@{->>}[d]_-{\hzero{\pih}} & \\
\hzero{\Uh} \ar[r]_-f & \Uc
} $$
We thus have proved that there exists a morphism $g$ of $\glieF$-pointed $\KK$-algebra from $\Uc$ to $\hzero{\Uh}$ which makes the following diagram commute:
$$ \shorthandoff{;:!?} \xymatrix @C=7pc @R=3pc {
\FUc \ar@{->>}[dr]^-{\pic} \ar@{->>}[d]_-{\hzero{\pih}} & \\
\hzero{\Uh} & \Uc \ar[l]^-g
} $$
Since the maps $\hzero{\pih}$ and $\pic$ are surjective, the morphisms of $\glieF$-pointed $\KK$-algebra $g \circ f$ and $f \circ g$ are equal to identity maps. To sum up, we have proved that $\Uh$ is a formal deformation of the $\glieF$-pointed \hbox{$\KK$-algebra $\Uc$}. One concludes thanks to theorem \ref{thm_Uhonednd}.
\qed

\subsection{Presentation by generators and relations}
In this subsection, $\psi$ designates a $h$-admissible $I$-colouring of $\BB$. Let $i \in I$, we denote by ${\bar S^{\! \: \psi_i}_{\! \! \; h}} = {({\bar S^{\! \: \psi_i}_{\! \! \; h,p}})}_{p \in \NN}$ the solution of the GQE equation $\GQEeqh[\psi_i,\psicl]{0}$: see proposition \ref{prop_GQEequnicity} and theorem \ref{thm_GQEeq}. We recall that $\psicl$ designates the classical colouring of $\BB$ (see subsection \ref{subs_exclqUhone}). \\

We denote by $\Uph{\glie,\psi}$ the \hbox{$\Kh$-algebra} topologically generated by $X^\pm_i$ ($i \in I$), the $\ZZ$-module $Y$, subject to the non-deformable relations \eqref{eq_KM_nondeformable} of $\glie$ and to the GQE Serre relations. The projection map from $\Uh$ to $\Uph{\glie,\psi}$ endows the latter with a structure of $\glieF$-pointed $\Kh$-algebra.

\begin{thm} \label{thm_presentationUh}
The GQE algebra $\Uh$ is isomorphic to the \hbox{$\glieF$-pointed} $\Kh$-algebra $\Uph{\glie,\psi}$.
\end{thm}

\proof
The proof is analog to the one given in rank 1 (see theorem \ref{thm_presentationUhone}) and makes use of theorem \ref{thm_UhoneHtrivial} and proposition \ref{defi-prop_GQESerre}.
\qed

\subsection{$Y$-Triviality in finite type}

We recall that $Y$ designates also the \hbox{$\KK$-algebra} morphism from $\KK[Y]$ to $\FUc[\glieF]$ which sends $\check \mu$ to $\check \mu$ for every $\check \mu \in Y$.

\begin{thm} \label{thm_YtrivialUh}
Let $\psi$ be a $h$-admissible $I$-colouring of $\BB$. When the Lie algebra $\glie$ is finite-dimensional, the GQE algebra $\Uh$ is a $Y \!$-trivial formal deformation of $\Uc$.
\end{thm}

\proof
Suppose that $\glie$ is finite-dimensional.
Thanks to theorem \ref{thm_Uhdned}, we know that $\Uh$ is a formal deformation of the $\glieF$-pointed $\Kh$-algebra $\Uc$. Denote then by $\varphi_h$ an isomorphism of \hbox{$\glieF$-pointed} $\Kh$-algebra from $\Uh$ to $\Uc \cch$. Fix $\lambda \in \Xdom$ and denote by ${V_{\! \! \; h}}$ the pull-back via $\varphi_h$ of the representation $\Lc{\lambda}$ of $\Uc$. Let now $v_{\lambda}$ be a highest weight vector of $\Lc{\lambda}$ and let $i \in I$. In view of the relations of $\FUc$, the element $\check \alpha_i$ in $\Uh$ acts on $v_{\lambda} \in {V_{\! \! \; h}}$ by a scalar in $\Kh$. Since ${V_{\! \! \; h}}$ is a $\glieF$-pointed formal deformation of $\Lc{\lambda}$, the specialisation of this scalar to zero is $\langle \check \alpha_i, \lambda \rangle$. Besides, we know that ${V_{\! \! \; h}}$ is a weight representation of $\Uh$. Therefore this scalar is $\lambda$. One concludes thanks to remark \ref{rem_Uhperfect}.
\qed

\begin{rem}
It was already known that the quantum group $\Uh$ is a \hbox{${Y \!}$-trivial} deformation of $\Uc$: see \cite{Dtrivial}.
\end{rem}

\subsection{Classification}

Two $I$-colourings $\psi, \psi'$ of $\BB$ with values in $\Kh$ are said congruent if the colourings $\psi_i$ and $\psi'_i$ are for every $i \in I$.

\begin{thm} \label{thm_IGQEclassification}
The following map is a bijection:
\begin{IEEEeqnarray*}{CCC}
\left \{ {\renewcommand{\arraystretch}{0.85} \begin{array}{c}
\text{$h$-admissible} \\
\text{$I$-colourings of $\BB$,} \\
\text{up to congruence}
\end{array}} \right \}
& \ \longrightarrow \ &
\left \{ {\renewcommand{\arraystretch}{0.85} \begin{array}{c}
\text{Locally ${Y \!}$-trivial formal deformations} \\
\text{of the $\glieF$-pointed $\KK$-diagram $\UUc$}, \\
\text{up to isomorphism} \\
\text{of $\glieF$-pointed $\Kh$-diagram}
\end{array}} \right \} \\
&& \\
\begin{array}{c}
\text{congruence} \\
\text{class of $\psi$}
\end{array} & \longmapsto & \begin{array}{c}
\text{isomorphism} \\
\text{class of $\UUh$.}
\end{array}
\end{IEEEeqnarray*}
\end{thm}

\proof
Consequence of theorems \ref{thm_GQEoneclassification} and \ref{thm_Uhdned}.
\qed





\subsection{Representation theory}
Let $\psi$ be a $I$-colouring of $\BB$ with values in $\Kh$. We denote by $\Cinth{\glie,\psi}$ the category of integrable representations of $\Uh$ which are topologically free when viewed as $\Kh$-modules. \\

We denote by $\Ointh$ the full subcategory of $\Cinth{\glie,\psi}$ whose objects $V$ satisfy the following two conditions.
\begin{enumerate}
\item[\textbullet] The weight spaces of $V$ are of finite-rank over $\Kh$.
\item[\textbullet] There exist finitely many elements $\lambda_1, \dots, \lambda_k \in X$ such that $\wt (V)$ are contained in $\bigcup_{s=1}^k (\lambda_s - \sum_{i \in I} \NN \: \! \alpha_i)$.
\end{enumerate}

\vsp
\vsp
\vsp

Let $\lambda \in \Xdom$, we denote by $\Lh{\lambda,\psi}$ the representation of $\Uh$ topologically generated \hbox{by $v_\lambda$} and subject to the following relations:
$$ \check \mu \! \: . \! \: v_\lambda \ = \  \langle \lambda, \check \mu \rangle \! \; v_\lambda \, , \quad X^+_i . \! \: v_\lambda \ = \ 0 \, , \quad (X^-_i)^{\langle \lambda, \check \alpha_i \rangle +1} . \! \: v \ = \ 0  \quad \ (i \in I, \ \check \mu \in Y) \, . $$

\vsp
\vsp

\begin{thm} \label{thm_Uhrepresentations}
Let $\psi$ be a $h$-admissible $I$-colouring of $\BB$.
\begin{enumerate}
\item[1)] A representation of $\Uh$ is in $\Cinth{\glie,\psi}$ if and only it is is $\psi$-integrable.
\item[2)] For every $\lambda \in \Xdom$, the representation $\Lh{\lambda,\psi}$ is a formal deformation of $\Lc{\lambda}$ along $\Uh$.
\item[3)] The representations $\Lh{\lambda,\psi}$ ($\lambda \in \Xdom$) are pairwise distinct and are the unique, up to isomorphism, indecomposable representations in $\Ointh$.
\item[4)] Every representation in $\Ointh$ is isomorphic to a topologically direct sum of representations $\Lh{\lambda,\psi}$ with $\lambda \in \Xdom$.
\item[5)] Let $\lambda,\lambda' \in \Xdom$. Every $\Uh$-morphism from $\Lh{\lambda,\psi}$ to $\Lh{\lambda',\psi}$ is equal to a scalar multiple of the identity map if $\lambda = \lambda'$ and is zero else.
\item[6)] Every representation in $\Ointc$ admits a unique, up to isomorphism, formal deformation along $\Uh[\glie,\psi]$ in $\Ointh$.
\end{enumerate}
\end{thm}

\proof
The proof is analog to the one given in rank one (see theorem \ref{thm_Uhintegrable}) and makes use of proposition \ref{prop_Hochschild}.
\qed

\begin{appendices}
\section{GQE equations} \label{app_GQE}
\subsection{Notations}
We denote by $\Tinf$ the $\Kh$-algebra formed by the lower triangular matrices of infinite order with coefficients in $\Endh(\Kh^\ZZ)$ (remark that the matrix multiplication for $\Tinf$ is well defined since matrices in $\Tinf$ are lower triangular):
$$ \Tinf \ := \ \left\{ \mathbf T \in \left( \Endh(\Kh^\ZZ) \right)^{\! \: \! \NN^2} ; \, \ {\mathbf T_{\! \! \: p,a}} = 0 \, , \ \forall \, p,a \in \NN \, \text{ s.t. } p < a \right\} \, . $$

We denote by $\Sinf$ the $\Kh$-module formed by the column matrices of infinite order with coefficients in $\Kh^\ZZ$:
$$ \Sinf \ := \ \left( \Kh^\ZZ \right)^{\! \! \; \NN} \, . $$
 We regard $\Sinf$ as a left $\Tinf$-module, where the action of a matrix $\mathbf T$ \hbox{in $\Tinf$} on a column matrix $M$ in $\Sinf$ is defined by the following:
$$ {\big( \mathbf T \! \; . \! \; M \big)}_{\! \! \: p} \ := \ \sum_{0 \leq a \leq p} {\mathbf T_{\! \! \: p,a}} \! \; . \! \; M_a \quad \quad (p \in \NN) \, . $$

Let $M \in \Sinf$ and let $d \in \ZZ$. We denote by $M[d]$ the column matrix \hbox{in $\Sinf$} defined by the following:
$$ {\big( M[d] \big)}_{\! \! \; p} \ := \ \left\{ {\renewcommand{\arraystretch}{1.2} \begin{array}{ccl}
M_{p-d} & \quad & \text{if } p \geq d \, , \\
0 && \text{else} \end{array}} \right. \quad \quad \quad (p \in \NN) \, . $$

We denote by $\Sinfh$ the $\Kh$-module formed by the column matrices of infinite order with coefficients in $\KK[u] \cch$ converging to zero:
$$ \Sinfh \ := \ \left \{ M \in \left( \KK[u] \cch \right)^{\! \! \; \NN} \, ; \, \ \lim_{p \rightarrow \infty} M_p = 0 \right \} \, . $$
We regard $\Sinfh$ as a $\Kh$-submodule of $\Sinf$.

\subsection{Definition}
Let $\psi \in \Coloh$. We denote by $\mathbf T(\psi)$ the lower triangular matrix in $\Tinf$, defined by the following, \hbox{for $p,a \in \NN$}, $f \in \Kh^\ZZ$ and $n \in \ZZ$:
$$ \bigg[ {\big( \mathbf T(\psi) \big)}_{\! \! \: p,a} \, . \, f \bigg] (n) \ := \ \left\{
{\renewcommand{\arraystretch}{1.6} \begin{array}{cl}
\displaystyle \frac{[\psi] (n,p) !}{[\psi] (n,p-a) !} \ f(n-2p+2a) & \quad \text{if } a \leq p \leq n \, , \\
0 & \quad \text{else.}
\end{array}} \right. $$
We denote by $N(\psi)$ and the column matrix in $\Sinf$, defined by the following, \hbox{for $p \in \NN$} and $n \in \ZZ$:
$$ {\big( N(\psi) \big)}_{\! \! \: p} \! \; (n) \ := \ \left\{
{\renewcommand{\arraystretch}{1.3} \begin{array}{cl}
[\psi] (n,p) & \quad \ \text{if } 0 < p \leq n \, , \\
0 & \quad \ \text{else.}
\end{array}} \right. $$

\vsp

\begin{defi}
Let $\psi_1, \psi_2 \in \Coloh$ and let $d \in \ZZ$. The GQE equation of degree $d$ associated to $\psi_1$ and $\psi_2$, denoted by $\GQEeqh{d}$, is the following linear equation:
\begin{equation*} \label{eq_GQE}
\mathbf T(\psi_1) \, . \, M \ = \ N(\psi_2)[d] \, ,
\end{equation*}
where the unknown $M$ is a column matrix in $\Sinfh$.
\end{defi}

\subsection{First interpretation}
\begin{prop} \label{prop_GQEeqspan}
Let $\psi \in \Coloh$ and let $M \in \Sinfh$. The following two assertions are equivalent.
\begin{enumerate}
\item[(i)] The equality $X^+ X^- = \sum_{a \in \NN} {(X^-)}^a M_a (H) {(X^+)}^a$ holds in $\Uh[\slt,\psi]$.
\item[(ii)] The column matrix $M$ is a solution of the GQE equation $\GQEeqh[\psi,\psi]{-1}$.
\end{enumerate}
\end{prop}

\vsp

Remark that the series $\sum_{a \in \NN} {(X^-)}^a M_a (H) {(X^+)}^a$ converges in $\Uh[\slt,\psi]$: the sequence ${(M_p)}_{p \in \NN}$ converges to zero in $\KK[u] \cch$, the $\Kh$-module $\Uh[\slt,\psi]$ is separated and complete (we know from \hbox{lemma \ref{lem_Uhonehfree}} that it is topologically free).

\proof
Denote by $x$ and $y$ the elements $X^+ X^-$ and  $\sum_{a \in \NN} {(X^-)}^a M_a (H) {(X^+)}^a$ in the algebra $\Uh[\slt,\psi]$. According to remark \ref{rem_Uhoneperfect}, the equality $x = y$ holds in the algebra $\Uh[\slt,\psi]$ if and only if the actions of $x$ and $y$ are equal on all the representations $\Lh{n,\psi}$ with $n \in \NN$. Therefore, assertion (i) is equivalent to the following one.
\begin{equation} \tag{i'}
\begin{tabular}{| p{18pc}}
The equality $x.b_{n,p} = y.b_{n,p}$ holds in $\Lh{n,\psi}$ for every $n, p \in \NN$ such that $p \leq n$.
\end{tabular}
\end{equation}
Besides, for every $n, p \in \NN$ such that $p \leq n$, the following equalities hold in the representation $\Lh{n,\psi}$:
\begin{eqnarray*}
x. b_{n,p} & = & \left\{ {\renewcommand{\arraystretch}{1.3} \begin{array}{cr}
\psi^- (n,p+1) \ \psi^+ (n,n-p) \ b_{n,p} & \quad \text{if } p \neq n \, , \\
0 & \quad \text{else} \, ,
\end{array}} \right. \\
y.b_{n,p} & = & \left[ \sum_{0 \leq a \leq p} M_a (n-2p+2a) \ \frac{\psi^- (n,p) !}{\psi^- (n,p-a) !} \ \frac{\psi^+ (n,n-p+a) !}{\psi^+ (n,n-p) !} \right] b_{n,p} \, .
\end{eqnarray*}
As a consequence, assertion (i') holds if and only if the column matrix $M$ is a solution of the shifted GQE equation $\GQEeqh[\psi,\psi]{-1}$.
\qed

\subsection{Second interpretation}
\begin{prop} \label{prop_GQEeqiso}
Let $\psi_1, \psi_2$ be two invertible colourings of $\BB$ with values in $\Kh$ and let $M \in \Sinfh$. The following two assertions are equivalent.
\begin{enumerate}
\item[(i)] There exists a $\Kh$-algebra \hbox{morphism $f$} from $\Uh[\slt,\psi_2]$ to $\Uh[\slt,\psi_1]$ such that the following holds.
\begin{IEEEeqnarray*}{ll}
\bullet \ \, & \left\{ \! \! \! \: {\renewcommand{\arraystretch}{1.1} \begin{array}{lcl}
f(X^-) & = & X^- \, , \\
f(H) & = & H \, , \\
f(X^+) & = & \sum_{a \in \NN} {(X^-)}^a M_a (H) {(X^+)}^{a+1} \, .
\end{array}} \right. \\
& \\
\bullet \ \, & \! \: \ \begin{tabular}{| p{24pc}}
For every $n \in \NN$, the representation $\Lh{n,\psi_2}$ of $\Uh[\slt,\psi_2]$ and the pull-back of $\Lh{n,\psi_1}$ via $f$ are isomorphic.
\end{tabular}
\end{IEEEeqnarray*}
\item[(ii)] The column matrix $M[1]$ is a solution of the GQE equation $\GQEeqh{0}$.
\end{enumerate}
\end{prop}

\vsp

Remark that the series $\sum_{a \in \NN} {(X^-)}^a M_a (H) {(X^+)}^{a+1}$ converges in $\Uh[\slt,\psi_1]$: the sequence ${(M_p)}_{p \in \NN}$ converges to zero in $\KK[u] \cch$, the $\Kh$-module $\Uh[\slt,\psi_1]$ is separated and complete (we know from \hbox{lemma \ref{lem_Uhonehfree}} that it is topologically free).

\proof
Remark that the equation $\GQEeqh{0}$ depends only on the congruence classes of the colourings $\psi_1$ and $\psi_2$. The algebras $\Uh[\slt,\psi_1]$ and $\Uh[\slt,\psi_2]$ also do: see lemma \ref{lem_congruentAone}. We besides know that if $\psi, \psi' \in \Coloh$ are two congruent colourings which divide each other in the ring $\Coloh$, then the representations $\Lh{n,\psi}$ and $\Lh{n,\psi'}$ are isomorphic (see again lemma \ref{lem_congruentAone}). Therefore, and since the colourings $\psi_1, \psi_2$ are invertible, one can without loss of generality assume that $\psi_1^- = \psi^-_2 = 1$. Denote by $g$ the $\Kh$-algebra morphism from $\FUh[\sltF]$ to $\Uh[\slt,\psi_1]$ defined by the following:
$$ \left\{ {\renewcommand{\arraystretch}{1.2} \begin{array}{lcl}
g(X^-) & := & X^- \, , \\
g(H) & := & H \, , \\
g(X^+) & := & \sum_{a \in \NN} {(X^-)}^a M_a (H) {(X^+)}^{a+1} \, .
\end{array}} \right. $$
Let $n \in \NN$. Since $g(X^-) = X^-$ and $g(H) = H$, every $\FUh[\sltF]$-morphism between $\Lh{n,\psi_2}$ and the pull-back of $\Lh{n,\psi_1}$ via $g$ is a scalar multiple of the identity map on ${\KK \; \! \BBv(n) \cch}$. Therefore and according to the definition of the algebra $\Uh[\slt,\psi_2]$, assertion (i) is equivalent to the following assertion.
\begin{equation} \tag{i'}
\begin{tabular}{| p{20pc}}
For every $n \in \NN$, the identity map on ${\KK \; \! \BBv(n) \cch}$ is a $\FUh[\sltF]$-morphism from $\Lh{n,\psi_2}$ to the pull-back of $\Lh{n,\psi_1}$ via $g$.
\end{tabular}
\end{equation}
Besides, for every $n,p \in \NN$ such that $0 < p \leq n$, the following two equalities hold in the representations $\Lh{n,\psi_1}$ and $\Lh{n,\psi_2}$, respectively:
\begin{eqnarray*}
g(X^+) . \; \! b_{n,p} & = & \sum_{0 \leq a \leq p-1} \, M_{a} (n-2p+2a+2) \ \frac{[\psi_1](n,p) !}{[\psi_1](n,p-a-1) !} \ b_{n,p-1} \, , \\
X^+ . \; \! b_{n,p} & = & [\psi_2] (n,p) \ b_{n,p-1} \, .
\end{eqnarray*}
Therefore assertion (i') holds if and only if the column matrix $M[1]$ is a solution of the GQE equation $\GQEeqh{0}$.
\qed

\subsection{Unicity of the solutions}
\begin{prop} \label{prop_GQEequnicity}
Let $\psi_1, \psi_2 \in \Coloh$ and let $d \in \ZZ$. There is at most one solution of the GQE equation $\GQEeqh{d}$.
\end{prop}

\proof
Let $M$ be a column matrix in $(\KK[u] \cch)^\NN$ which satisfies the following homogeneous linear equation:
\begin{equation} \label{eq_GQEhomon}
\mathbf T(\psi_1) \, . \, M \ = \ 0 \, .
\end{equation}
Let us prove by induction on $p \in \NN$ that $M_p$ is zero. Let $p \in \NN$ and suppose that $M_{p'}$ is zero for every $p' \in \NN_{<p}$. Fix now $n \in \NN_{\geq p}$. In view of equation \eqref{eq_GQEhomon}, the following equality holds in $\Kh$.
$$ \sum_{0 \leq a \leq p} M_a (n-2p+2a) \ \frac{[\psi_1] (n,p) !}{[\psi_1] (n,p-a) !} \ = \ 0 \, . $$
Since $[\psi_1] (n,p) ! \neq 0$ (recall that a colouring has by definition no zero) and thanks to the induction hypothesis, the specialisation of $M_p$ at $n$ is thus zero. This being true for every $n \in \NN_{\geq p}$, the element $M_p \in \KK[u] \cch$ is zero. The proposition follows.
\qed

\subsection{Classical solutions}
We regard $\hzero{(\Sinfh)}$ as a $\KK$-vector subspace of $\Sinfh$. We denote by ${S^{\! \: [0]}_{\! \! \; \mathrm{cl}}}$ and ${S^{\! \: [-1]}_{\! \! \: \mathrm{cl}}}$ the column matrices in $\hzero{(\Sinfh)}$ defined by the following:
\begin{IEEEeqnarray*}{rClrClrCl}
{(S_{\mathrm{cl}})}_0 & \ := \ & u \, , \quad \quad & {(S_{\mathrm{cl}})}_1 & \ := \ & 1 \, , \quad \quad & {(S_{\mathrm{cl}})}_p & \ := \ & 0 \quad \ (p \in \NN_{\geq 2}) \, , \\
{(\bar S_{\mathrm{cl}})}_0 & \ := \ & 0 \, , \quad \quad & {(\bar S_{\mathrm{cl}})}_1 & \ := \ & 1 \, , \quad \quad & {(\bar S_{\mathrm{cl}})}_p & \ := \ & 0 \quad \ (p \in \NN_{\geq 2}) \, .
\end{IEEEeqnarray*}
Recall the definition of the classical colouring $\psicl$ of $\BB$:
$$ (\psicl)^\pm (n,k) \ := \ k \ \quad \quad (n,k \in \NNI, \ k \leq n) \, . $$

\begin{prop} \label{prop_GQEeqclassic}
Let $\psi_1, \psi_2 \in \Coloh$. Let ${S_{\! \! \; h}}$ and ${\bar S_{\! \! \; h}}$ be solutions of the GQE equations $\GQEeqh{-1}$ and $\GQEeqh{0}$, respectively. We assume that the colourings $\psi_1$ and $\psi_2$ satisfy the deformation axiom.
\begin{enumerate}
\item[1)] The column matrix $\hzero \pi^{\Sinf[\Kh]} ({S_{\! \! \; h}})$ is equal to $S_{\mathrm{cl}}$.
\item[2)] The column matrix $\hzero \pi^{\Sinf[\Kh]} ({\bar S_{\! \! \; h}})$ is equal to $\bar S_{\mathrm{cl}}$.
\end{enumerate}
\end{prop}

\proof
The following equalities hold for every $n,p \in \NN$:
\begin{IEEEeqnarray*}{rCl}
\sum_{0 \leq a \leq p} {(S_{\mathrm{cl}})}_a (n-2p+2a) \ \frac{p \! \; ! \, (n-p+1) !}{(p-a)! \, (n-p+a+1)!} \quad \quad && \\
= \ (n-2p) \, + \, p \, (n-p+1) & \ = \ & (p+1) \, (n-p) \, , \\
&& \\
\sum_{0 \leq a \leq p} {(\bar S_{\mathrm{cl}})}_a (n-2p+2a) \ \frac{p \! \; ! \, (n-p+1) !}{(p-a)! \, (n-p+a+1)!} & \ = \ & p \, (n-p+1) \, .
\end{IEEEeqnarray*}
As a consequence, the column matrices $S_{\mathrm{cl}}$ and $\bar S_{\mathrm{cl}}$ are solutions of the GQE equations $\GQEeqh[\psicl,\psicl]{-1}$ and $\GQEeqh[\psicl,\psicl]{0}$, respectively. Since $\psi_1$ \hbox{and $\psi_2$} satisfy the deformation axiom, the column matrices $\hzero \pi^{\Sinf[\Kh]} ({S_{\! \! \; h}})$ and $\hzero \pi^{\Sinf[\Kh]} ({\bar S_{\! \! \; h}})$ are also solutions of $\GQEeqh[\psicl,\psicl]{-1}$ and $\GQEeqh[\psicl,\psicl]{0}$, respectively. One concludes in view of the unicity of the solutions of GQE equations: see proposition \ref{prop_GQEequnicity}.
\qed

\subsection{Theorem}
This subsection is dedicated to theorem \ref{thm_GQEeq}, whose statement is given below. We establish some preparatory lemmas in the next paragraphs, before giving a proof of the theorem in paragraph \ref{para_paraproofGQEeq}.

\subsubsection{Statement}
\begin{thm} \label{thm_GQEeq}
Let $\psi_1, \psi_2 \in \Coloh$. We assume that $\psi_1$ and $\psi_2$ satisfy the deformation axiom.
\begin{enumerate}
\item[1)] The GQE equation $\GQEeqh[\psi_1,\psi_1]{-1}$ admits a solution if and only if $\psi_1$ is \hbox{$h$-admissible}.
\item[2)] If $\psi_1$ and $\psi_2$ are both $h$-admissible, then the GQE equation $\GQEeqh{0}$ admits a solution.
\end{enumerate}
\end{thm}

\subsubsection{Lemma: shift of the solutions}
\begin{lem} \label{lem_shift}
Let $\psi_1, \psi_2 \in \Coloh$ and let $M \in \Sinf[\Kh]$. We assume that the colouring $\psi_1$ is admissible. The following two assertions are equivalent.
\begin{enumerate}
\item[1)] The column matrix $M$ is a solution of the GQE equation $\GQEeqh{-1}$.
\item[2)] The column matrix $M[1]$ is a solution of the GQE equation $\GQEeqh[{\psi_1, \psi_1 \psi_2 }]{0}$.
\end{enumerate}
\end{lem}

\proof
On the one hand, assertion (i) holds if and only if the following equality does for every $n,p \in \NN$ such that $p \leq n$:
$$ \sum_{0 \leq a \leq p} \, M_a (n-2p+2a) \, \! \: \frac{[\psi_1](n,p) !}{[\psi_1] (n,p-a)!} \ = \ \left\{ {\renewcommand{\arraystretch}{1.2} \begin{array}{ccl}
[\psi_2] (n,p+1) & \quad & \text{if } 0 < p+1 \leq n \, , \\
0 && \text{else.} \end{array}} \right. $$
On the other hand, assertion (ii) holds if and only if the following equality does for every $n,p \in \NN$ such that $p \leq n$:
$$ \sum_{0 \leq a \leq p} \, {\big( M[1] \big)}_{\! \! \: a} (n-2p+2a) \, \! \: \frac{[\psi_1](n,p) !}{[\psi_1] (n,p-a)!} \ = \ \left\{ {\renewcommand{\arraystretch}{1.2} \begin{array}{ccl}
[\psi_1 \psi_2] (n,p) & \quad & \text{if } 0 < p \leq n \, , \\
0 && \text{else.} \end{array}} \right. $$
One then concludes in view of the following equality, satisfied for every $n,p \in \NN$ such that $0 < p \leq n$ (recall that $\psi_1$ is by assumption admissible):
\begin{multline*}
\sum_{0 \leq a \leq p} \, {\big( M[1] \big)}_{\! \! \: a} (n-2p-2a) \, \! \: \frac{[\psi_1](n,p) !}{[\psi_1] (n,p-a)!} \\
= \ [\psi_1](n,p) \sum_{0 \leq a \leq p-1} \, M_a (n-2(p-1)+2a) \, \! \: \frac{[\psi_1](n,p-1) !}{[\psi_1] (n,p-1-a)!} \, \cdot
\end{multline*}
\qed

\subsubsection{Lemma: semi-regularity}
Let $P_h = \sum_{m \in \NN} P_m \, h^m$ be an element in $\KK[u] \cch$ and let $n \in \ZZ$. We denote \hbox{by $P_h(n)$} the power series in $\Kh$ defined by the following:
$$ P_h(n) \ := \ \sum_{m \in \NN} P_m(n) \, h^m \, . $$
Let $n' \in \NN$, we then regard the $\Kh$-algebra $\KK[u] \cch$ as a \hbox{$\Kh$-subalgebra} \hbox{of $\Kh^{\NN_{\geq n'}}$}: an element $P_h \in \KK[u] \cch$ is identified with the function in $\Kh^{\NN_{\geq n'}}$ whose value at $n$ is $P(n)$ for every $n \in \NN_{\geq n'}$.

\begin{defi}
A colouring $\psi$ of $\BB$ with values in $\Kh$ is said semi-regular if $[\psi](\, \cdot \, ,k) \in \KK[u] \cch$, for every $k \in \NNI$.
\end{defi}

\begin{lem} \label{lem_semiregularity}
Let $\psi \in \Coloh$. If the GQE equation $\GQEeqh[\psi,\psi]{-1}$ has a solution, then $\psi$ is semi-regular.
\end{lem}

\proof
Suppose that the equation $\GQEeqh[\psi,\psi]{-1}$ has a solution $S_h = {(S_{h,p})}_{p \in \NN}$. Let us prove by induction on $k \in \NNI$ the following assertion.
\begin{equation} \tag{$\ast_k$}
[\psi](\, \cdot \, ,k) \ \in \ \KK[u] \cch \, .
\end{equation}
Let $k \in \NN$ and suppose that assertion $(\ast_{k'})$ holds for every $k' \in \NNI_{\leq k}$. The \hbox{vector $S_h$} being a solution of the GQE equation $\GQEeqh[\psi,\psi]{-1}$, the following equality holds for every $n \in \NN_{\geq k+1}$.
\begin{eqnarray}
\sum_{0 \leq a \leq k} S_{h,a}(n-2k+2a) \ \frac{[\psi] (n,k) !}{[\psi_1] (n,k-a) !} \ = \ [\psi] (n,k+1) \, .
\end{eqnarray}
Assertion $(\ast_{k+1})$ then holds by induction hypothesis.
\qed

\subsubsection{Lemma: quasi-regularity I}
\begin{defi}
A colouring $\psi$ of $\BB$ with values in $\Kh$ is said quasi-regular if $\psi$ is semi-regular and if the following condition holds:
$$ \sup \big\{ \deg_{u} \big( {[\psi]}_m(u,k) \big) \, ; \ k \in \NNI \big\} \ < \ \infty \, , \quad \quad \forall \, m \in \NN \, . $$
\end{defi}

\vsp

\begin{lem} \label{lem_quasiregularity}
Let $\psi \in \Coloh$. If the GQE equation $\GQEeqh[\psi,\psi]{-1}$ has a solution, then $\psi$ is quasi-regular.
\end{lem}

\proof
Suppose that the equation $\GQEeqh[\psi,\psi]{-1}$ has a solution $S_h = {(S_{h,p})}_{p \in \NN}$. We know from lemma \ref{lem_semiregularity} that the colouring $\psi$ is semi-regular. Let us prove by induction on $m \in \NN$ the following assertion.
\begin{equation} \tag{$\ast_m$}
\sup \big\{ \deg_{u} \big( {[\psi]}_m(u,k) \big) \, ; \ k \in \NNI \big\} \ < \ \infty \, .
\end{equation}
Let $m \in \NN$ and suppose that assertion $(\ast_{m'})$ holds for every $m' \in \NN_{< m}$. The vector $S_h$ being a solution of the GQE equation $\GQEeqh[\psi,\psi]{-1}$, the following equalities hold in $\KK[u] \cch$.
$$ \sum_{0 \leq a \leq k} S_{h,a}(u -2k +2a) \ \frac{[\psi] (u,k) !}{[\psi] (u,k-a) !} \ = \ [\psi] (u,k+1) \, , \quad \quad \forall \, k \in \NN \, . $$
Besides, since the sequence ${(S_{h,k})}_{k \in \NN}$ converges to zero, there exists $k_1 \in \NN$ such that $S_{h,k} \in h^{m+1} \KK[u] \cch$ for all $k > k_1$. Therefore, the previous equalities imply the following ones in $\KK[u] \cch / h^{m+1} \KK[u] \cch$.
$$ \sum_{0 \leq a \leq k_1} S_{h,a}(u-2k+2a) \ \frac{[\psi] (u,k) !}{[\psi] (u,k-a) !} \ = \ [\psi] (u,k+1) \, , \quad \quad \forall \, k \in \NN_{\geq k_1} \, . $$
Assertion $(\ast_{m})$ then holds by induction hypothesis.
\qed

\subsubsection{Lemma: quasi-regularity II}
\begin{lem} \label{lem_GQEeqvanish}
Let $\psi_1, \psi_2 \in \Coloh$. We assume that the colourings $\psi_1, \psi_2$ are quasi-regular and \hbox{that $\psi_1$} satisfies the deformation axiom. Let $M$ be a column matrix in $(\KK[u] \cch)^\NN$. If ${\mathbf T(\psi_1) \! \: . \! \: M} = N(\psi_2)$, then $M \in \Sinfh$.
\end{lem}

\proof
Suppose that ${\mathbf T(\psi_1) \! \: . \! \: M} = N(\psi_2)$. Let us prove by induction on $m \in \NNI$ the following assertion.
\begin{equation} \tag{$\ast_m$}
\text{There exists $p(m) \in \NN$ such that $M_p \in h^m \; \! \KK[u] \cch$ for all $p \in \NN_{\geq p(m)}$.}
\end{equation}
Let $m \in \NNI$ and suppose that assertion $(\ast_{m'})$ holds for every $m' \in \NNI_{< m}$. Thanks to the induction hypothesis, there exists $p(m-1) \in \NN$ such that $M_p$ belongs to ${h^{m-1} \: \! \KK[u] \cch}$ for every $p \in \NN_{> p(m-1)}$. For each $p \in \NN$, we denote \hbox{by $M_{m,p}$ ($m \in \NN$)} the elements in $\KK[u]$ such that $M_p = \sum_{m \in \NN} {M_{m,p} \: \! h^m}$. We besides denote by $D_1(m)$ and $D_2(m)$ the sets defined by the following.
\begin{IEEEeqnarray*}{lCl}
D_1(m) & := & \left \{ \deg_{u} \: \! (M_{p_0,m_0}) +  \sum_{l = 1}^a \deg_{u} \big( {[\psi_1]}_{m_l} (u,p_l) \big) \, ;
{\renewcommand{\arraystretch}{1.2} \begin{array}l
0 \leq a \leq p(m-1) \, , \\
p_l \in \NNI, \ \forall \, l \in \NN_{\leq a} \, , \\
\sum_{l = 0}^a m_l = m - 1
\end{array}} \right \} \, , \\
&& \\
D_2(m) & := & \bigg \{ \deg_{u} \big( {[\psi_2]}_{m-1} (u,p) \big) \, ; \ p \in \NNI \bigg \} \, .
\end{IEEEeqnarray*}
Let us prove by induction on $p \in \NNI$ the following assertion.
\begin{equation} \tag{$\star_p$}
\deg_{u} \: \! (M_{m-1,p}) \, + \, p \ \leq \ \sup \big( D_1(m) \cup D_2(m) \big) \, .
\end{equation}
Remark first that $\deg_{u} \big( {[\psi_1]}_0 (u,p) \big) = 1$ for all $p \in \NNI$, since by assumption the colouring $\psi_1$ satisfies the deformation axiom. Fix \hbox{now $p \in \NNI$}. \hbox{If $p \leq p(m-1)$}, assertion $(\star_p)$ follows from the definition of $D_1(m)$. Suppose \hbox{that $p > p(m-1)$} and that assertion $(\star_{p'})$ holds for \hbox{all $p' \in \NNI_{< p}$}. Since ${\mathbf T(\psi_1) \! \: . \! \: M} = N(\psi_2)$, the following equality holds in $\KK[u] \cch$.
$$ \sum_{0 \leq a \leq p} M_a (u-2p+2a) \ \frac{[\psi_1] (u,p) !}{[\psi_1] (u,p-a) !} \ = \ [\psi_2] (u,p+1) \, . $$
In view of the definition of $D_1(m), D_2(m)$, since besides $M_a \in {h^{m-1} \:\! \KK[u] \cch}$ for every $a \in \NN_{> p(m-1)}$, since $\deg_{u} \big( {[\psi_1]}_0 (u,p) \big) = 1$ for \hbox{all $p \in \NNI$} and thanks to the induction hypothesis, assertion thus $(\star_p)$ holds. Besides, thanks to quasi-regularity of the colourings $\psi_1$ and $\psi_2$, we know that the sets $D_1$ and $D_2$ are bounded. Assertion $(\star_p)$ being true for every $p \in \NNI$, \hbox{assertion $(\ast_m)$} then follows.
\qed

\subsubsection{Lemma: quotient axiom}
Remark that the quotient axiom (see definition \ref{defi_hadmissible} of a $h$-admissible colouring) makes sense for a semi-regular colouring.

\begin{lem} \label{lem_quotientaxiom}
Let $\psi \in \Coloh$. If the GQE equation $\GQEeqh[\psi,\psi]{-1}$ has a solution, then $\psi$ is semi-regular and satisfies the quotient axiom.
\end{lem}

\proof
Suppose that the equation $\GQEeqh[\psi,\psi]{-1}$ has a solution $S_h = {(S_{h,p})}_{p \in \NN}$ and let $k \in \NN$. We already know from lemma \ref{lem_semiregularity} that the colouring $\psi$ is semi-regular. Therefore, the following equality holds for every $n \in \NN$:
$$ \sum_{0 \leq a \leq k} S_{h,a}(n-2k+2a) \ \frac{[\psi] (n,k) !}{[\psi] (n,k-a) !} \ = \ [\psi] (n,k+1) \, . $$
Besides, since $S_h$ is a solution of $\GQEeqh[\psi,\psi]{-1}$, the left-hand side of the previous equation is equal to zero when $n=k$. One concludes.
\qed

\subsubsection{Lemma: regularity}
Recall the definition of the regularity axiom: see definition \ref{defi_hadmissible} of a $h$-admissible colouring.

\begin{lem} \label{lem_regularity}
Let $\psi \in \Coloh$. If the colouring $\psi$ is quasi-regular and satisfies the Verma axiom, then it satisfies the regularity axiom.
\end{lem}

\proof
Suppose that $\psi$ is quasi-regular and satisfies the quotient axiom. Let us fix $m \in \NN$. Since $\psi$ is quasi-regular, there exist $L \in \NN$ and $c_l \in \KK^{\NNI}$ ($l \in \NN_{\leq L}$) such that the following holds:
\begin{equation*} \label{eq_dnedregularity1}
{[\psi]}_m (n,k) \ = \ \sum_{l=0}^L \, c_l(k) \, n^l \, , \quad \quad \forall \, n,k \in \NNI .
\end{equation*}
Besides, since $\psi$ satisfies the Verma axiom, the following equalities hold:
$$ \sum_{l=0}^L \, c_l(k) \, (-n-2)^l \ = \ \sum_{l=0}^L \, c_l(n+k+1) \, n^l \, , \quad \quad \forall \, n,k \in \NNI . $$
As a consequence, the function ${\sum_{l=0}^L c_l(\, \cdot \, +n_1) \! \: u^l} \in \KK^{\NNI}$ belongs to $\KK[u]$, for every $n_1 \in \NN_{\geq 2}$. In order to establish the lemma, it is then sufficient to prove the following assertion
\begin{equation} \tag{$\ast_L$}
\begin{tabular}{| p{22.5pc}}
Let ${(f_l)}_{0 \leq l \leq L} \in (\KK[u])^L$ such that $\deg_u (f_l) = l$ for all $l \in \NN_{\leq L}$.
If ${\sum_{l=0}^L c_l(\, \cdot \, +n_1) \! \: f_l}$ belongs to $\KK[u]$ for every $n_1 \in \NN_{\geq 2}$, then so does $c_l$ for every $l \in \NN_{\leq L}$.
\end{tabular}
\end{equation}
Let us prove assertion $(\ast_L)$ by induction on $L \in \NN$. Let $L_1 \in \NN$ and suppose that assertion $(\ast_{L'_1})$ holds for every $L'_1 \in \NN_{< L_1}$. Let ${(f_l)}_{0 \leq l \leq L_1} \in (\KK[u])^{L_1}$ such that $\deg_u (f_l) = l$ for all $l \in \NN_{\leq L_1}$ and suppose that ${\sum_{l=0}^{L_1} c_l(\, \cdot \, +n_1) \! \: f_l}$ belongs to $\KK[u]$ for every $n_1 \in \NN_{\geq 2}$. The following then holds for every $n_1 \in \NN_{\geq 2}$:
$$ \sum_{l=0}^{L_1} \, c_l(\, \cdot \, + n_1) \, \Big( f_l(u) - f_l(u-1) \Big) \ \in \ \KK[u] \, . $$
Thanks to the induction hypothesis, the function $c_l$ thus belongs to $\KK[u]$ for every $l \in \NNI_{\leq L_1}$. Since ${\sum_{l=0}^{L_1} c_l(\, \cdot \, +2) \! \: f_l}$ belongs to $\KK[u]$, the function $c_0$ then also belongs to $\KK[u]$. One concludes.
\qed

\subsubsection{Proof} \label{para_paraproofGQEeq}
We give here a proof of theorem \ref{thm_GQEeq}.

\proof
Fix $n \in \NN$. Let $\mathbf T \in \Tinf$ and let $M \in \Sinf$, we denote by $\mathbf T_{\vert n}$ and $M_{\vert n}$ the matrices in $\Tinf$ and $\Sinf$ respectively, defined by the following:
\begin{IEEEeqnarray*}{CCll}
{(\mathbf T_{\vert n})}_{p,a} & \ := \ & \left\{
{\renewcommand{\arraystretch}{1.6} \begin{array}{cl}
{\mathbf T_{\! \! \: p,a}} & \quad \text{if } a, p \leq n \, , \\
0 & \quad \text{else}
\end{array}} \right. & \quad \quad \quad (p,a \in \NN) \, , \\
&&& \\
{(M_{\vert n})}_p & \ := \ & \left\{
{\renewcommand{\arraystretch}{1.3} \begin{array}{cl}
\, M_p \, & \quad \ \text{if } p \leq n \, , \\
0 & \quad \ \text{else}
\end{array}} \right. & \quad \quad \quad (p \in \NN) \, .
\end{IEEEeqnarray*}
Let $\psi, \psi' \in \Coloh$. We denote by $\GQEeqgeneric{\psi,\psi'}{-1}{h,n}$ the following equation:
\begin{equation*} \label{eq_GQEn}
{\mathbf T(\psi)}_{\vert n} \, . \, M_{\vert n} \ = \ {N(\psi')}_{\vert n} \, ,
\end{equation*}
where the unknown $M$ is a column matrix in in $(\KK[u] \cch)^\NN$.

\vsp

\setcounter{claimn}{0}

\begin{claimn} \label{claim_GQEequnicity}
If $M$ and $M'$ are two solutions of $\GQEeqgeneric{\psi,\psi'}{-1}{h,n}$, then $M_{\vert n} = M'_{\vert n}$.
\end{claimn}

\proof
We are going to adapt the proof of proposition \ref{prop_GQEequnicity}. Let $M$ be a column matrix in $(\KK[u] \cch)^\NN$ which satisfies the following homogeneous linear equation:
\begin{equation} \label{eq_GQEnhomon}
{\mathbf T(\psi)}_{\vert n} \, . \, M_{\vert n} \ = \ 0 \, .
\end{equation}
Let us prove by induction on $p \in \NN_{\leq n}$ that $M_p$ is zero. Let $p \in \NN_{\leq n}$ and suppose that $M_{p'}$ is zero for every $p' \in \NN_{< p}$. Fix now $n' \in \NN_{\geq p}$. In view of \hbox{equation \eqref{eq_GQEnhomon}}, the following equality holds in $\Kh$.
$$ \sum_{0 \leq a \leq p} M_a (n'-2p+2a) \ \frac{[\psi] (n',p) !}{[\psi] (n',p-a) !} \ = \ 0 \, . $$
Since $[\psi] (n',p) ! \neq 0$ (recall that a colouring has by definition no zero) and thanks to the induction hypothesis, the specialisation of $M_p$ at $n'$ is thus zero. This being true for every $n' \in \NN_{\geq p}$, the element $M_p \in \KK[u] \cch$ is zero. The claim follows.
\qed \\

Suppose now that the colourings $\psi$ and $\psi'$ are semi-regular and satisfy the quotient axiom. Consider then the following two assertions.
\begin{IEEEeqnarray*}{ll}
{[}\psi{]} (l,n+1) \ = \ {[}\psi{]}(-l-2,n-l) \, , \quad \ \forall \, l \in \NN_{< n} \, . & \quad \quad \quad (\text{i$_n$)} \\
\text{The equation $\GQEeqgeneric{\psi,\psi'}{-1}{h,n}$ has a solution.} & \quad \quad \quad (\text{ii$_n$)}
\end{IEEEeqnarray*}

\vsp

\begin{claimn} \label{claim_GQEeqrec}
We assume that $\psi$ and $\psi'$ are semi-regular and satisfy the quotient axiom. We moreover assume that $\psi$ satisfies the deformation axiom.
\begin{enumerate}
\item[1)] Assume that $\psi$ and $\psi'$ satisfy the Verma axiom. If assertion (ii${_{\! \; n}}$) holds, then assertion (ii${_{\! \; {n+1}}}$) holds.
\item[2)] Assume that $\psi$ and $\psi'$ are equal. If assertions (i${_{\! \; n}}$) and (i${_{\! \; n+1}}$) hold, then assertions (i${_{\! \; {n+1}}}$) \hbox{and (ii${_{\! \; {n+1}}}$)} are equivalent.
\end{enumerate}
\end{claimn}

\proof
Let $M$ be a solution of $\GQEeqgeneric{\psi,\psi'}{-1}{h,n}$. Assertion (ii$_{n+1}$) then holds if and only if there exists $x \in \KK[u] \cch$ such that the following equality holds in $\KK[u] \cch$ (recall that $\psi'$ satisfies the quotient axiom):
$$ \sum_{a=0}^n \, M_a (u-2n-2+2a) \, \frac{[\psi] (u,n+1) !}{[\psi] (u,n+1-a) !} \, + \, x \, \big( [\psi] (u,n+1) ! \big) \ = \ [\psi'] (u,n+2) \, . $$
Besides, since $\psi$ satisfies the deformation axiom, there exists an invertible element $y$ in $\KK[u] \cch$ such that $[\psi] (u,n+1) ! = {y \: \! \prod_{l=0}^n (u-l)}$. Therefore, assertion (ii$_{n+1}$)  holds if and only if the following equalities do.
$$ \sum_{0 \leq a \leq n} M_a (l-2n-2+2a) \ \frac{[\psi] (l,n+1) !}{[\psi] (l,n+1-a) !} \ = \ [\psi'] (l,n+2) \, , \quad \quad \forall \, l \in \NN_{\leq n} \, . $$
Suppose now that assertion (i${_{\! \; n}}$) holds (remark that assertion (i${_{\! \; n}}$) does if $\psi$ satisfies the Verma axiom). Since besides $M$ is a solution of $\GQEeqgeneric{\psi,\psi'}{-1}{h,n}$, the following equalities then hold for every $l \in \NN_{\leq n}$ (recall also that $\psi_1$ satisfies the quotient axiom):
\begin{IEEEeqnarray*}{lr}
\sum_{0 \leq a \leq n} M_a (l-2n-2+2a) \ \frac{[\psi] (l,n+1) !}{[\psi] (l,n+1-a) !} & \\
= \ \sum_{0 \leq a \leq n-l} M_a (l-2n-2+2a) \ \frac{[\psi] (l,n+1) !}{[\psi] (l,n+1-a) !} & [\text{quotient axiom}] \\
= \ \sum_{0 \leq a \leq n-l} M_a (-l-2-2(n-l)+2a) \ \frac{[\psi] (-l-2,n-l) !}{[\psi] (l,n-l-a) !} & [\text{assertion (\text{i${_{\! \; n}}$})}] \\
= \ [\psi'] (-l-2,n-l+1) \, . & \! \! \! \! \! \! \! \! \! \! \! \! \! \! \! \! \! \! [\text{$M$ solution of $\GQEeqgeneric{\psi,\psi'}{-1}{h,n}$}]
\end{IEEEeqnarray*}
Therefore, assertion assertion (ii$_{n+1}$) holds if and only if the following equality does for every $l \in \NN_{\leq n}$:
$$ [\psi'] (l,n+2) \ = \ [\psi'] (-l-2,n-l+1) \, . $$
One concludes.
\qed \\

\begin{claimn} \label{claim_GQEeqfinal}
Assume that $\psi$ and $\psi'$ are quasi-regular, satisfy the quotient axiom and the Verma axiom. The GQE equations $\GQEeqh[\psi,\psi']{-1}$ and $\GQEeqh[\psi,{\psi \! \; \psi'}]{0}$ admit a solution.
\end{claimn}

\proof
Remark that assertion (ii$_0$) straightforwardly holds. Claim \ref{claim_GQEeqrec} then implies by induction that the equation $\GQEeqgeneric{\psi,\psi'}{-1}{h,n}$ admits a solution ${(M_{n,p})}_{p \in \NN}$ for every $n \in \NN$. Denote by $S_h$ the column matrix in $(\KK[u] \cch)^\NN$ defined by the following: $S_{h,p} := {(M_{p,p})}_p$ ($p \in \NN$). Claim \ref{claim_GQEequnicity} implies that ${\mathbf T(\psi) \! \: . \! \: S_h = N(\psi')}$. Since $\psi$ and $\psi'$ are quasi-regular, the column matrix $S_h$ then converges to zero: see lemma \ref{lem_GQEeqvanish}. Therefore, the GQE equation $\GQEeqh[\psi,\psi']{-1}$ has a solution and so does the GQE equation $\GQEeqh[\psi,{\psi \! \; \psi'}]{0}$ thanks to lemma \ref{lem_shift}.
\qed \\

Let us prove point 1 of the theorem. Suppose that $\psi_1$ is $h$-admissible. Claim \ref{claim_GQEeqfinal} then implies that the GQE equation $\GQEeqh[\psi_1,\psi_1]{-1}$ admits a solution. Conversely, suppose that $\GQEeqh[\psi_1,\psi_1]{-1}$ admits a solution. We then know from lemmas \ref{lem_quasiregularity} and \ref{lem_quotientaxiom} that $\psi_1$ satisfies the quotient axiom and is quasi-regular. Remark besides that assertions (i$_0$) and (ii$_0$) straightforwardly hold. Claim \ref{claim_GQEeqrec} then implies by induction that $\psi_1$ satisfies the Verma axiom. Since $\psi_1$ satisfies by assumption the deformation axiom, it is then $h$-admissible. Let us now prove point 2 of the theorem. Suppose that the colourings $\psi_1$ and $\psi_2$ are $h$-admissible. Since the equation $\GQEeqh{0}$ depends only on the congruence classes of $\psi_1$ and $\psi_2$, one can assume without loss of generality that $\psi^-_1 = 1$. Since besides $\psi_1$ satisfies the deformation axiom, it is then invertible in the ring $\Coloh$. One concludes from claim \ref{claim_GQEeqfinal} with $\psi = \psi_1$ and $\psi' = \psi_2 / \psi_1$.
\qed

\end{appendices}

\end{otherlanguage}

\cleardoublepage
\part[Coloured Kac-Moody Algebras, Isogenic \\ \hspace*{1.65pc} and Langlands Interpolations]{Coloured Kac-Moody Algebras, Isogenic and Langlands Interpolations}
\label{part_BKM}
\setcounter{section}{0}
\begin{otherlanguage}{english}
  \righthyphenmin=62
\lefthyphenmin=62

\addcontentsline{toc}{section}{Notations and conventions}
\section*{Notations and conventions}
Rings and associative algebras are unitary, their morphisms are unity-preserving. \\

We denote by $\KK$ a field of \textbf{characteristic zero}. We denote by $R$ a (commutative) $\KK$-algebra which is an \textbf{integral domain} when viewed as a ring. We designate by $K(R)$ its field of fractions. When considering topological $R$-algebras, we \hbox{regard $R$} as a topological ring, endowed with the discrete topology. \\

We denote by $\KK[u]$ the polynomial $\KK$-algebra in one variable $u$ and we denote \hbox{by $\KK[u^{\pm 1}]$} the Laurent polynomial $\KK$-algebra.  \\


Let $\Lambda$ be a set. We denote by $\KK \Lambda$ the $\KK$-vector space which admits $\Lambda$ as a basis. We denote by $\langle \Lambda \rangle$ the free monoid generated by $\Lambda$. We denote by $\KK \langle \Lambda \rangle$ the $\KK$-algebra of this monoid. We denote by ${\Lambda^{\! R}}$ the ring of functions from $R$ to $\Lambda$. \\


Let $A$ be a $\KK$-algebra and let $\glie$ be a Lie $\KK$-algebra. Left $A$-modules and left $\glie$-modules are called representations of $A$ and $\glie$, respectively. Morphisms of left $A$-modules and of left $\glie$-modules are called $A$-morphisms and $\glie$-morphisms, respectively. \\

We identify as usual representations of a Lie algebra with representations of its universal enveloping algebra. The universal enveloping algebra of the Lie $R$-algebra ${\glie \otimes_\KK \! R}$ is denoted by $\UR{\glie}$. \\

A sum indexed by a empty set is zero, a product indexed by an empty set is equal to $1$.

\section{Kac-Moody algebras}

We recall here the definition of Kac-Moody algebras. Note that the definition given in this paper is slightly different (but mainly equivalent) to the definition of a Kac-Moody algebra usually present in the literature. The difference is about the Cartan subalgebra, which is built here from what we call a root datum (we borrow the name and the concept from Lusztig: see \cite{Lus}). We also recall some results on the representation theory of Kac-Moody algebras.

\subsection{Generalised Cartan matrices and root data} \label{subs_Cartanmatrixhat}
A \emph{generalised Cartan matrix} is a matrix $C = {(C_{ij})}_{i,j \in I}$ indexed by a finite \hbox{set $I$}, with entries in $\ZZ$ and which satisfies the following conditions for all $i,j \in I$.
\begin{enumerate}
\item[\textbullet] $C_{ii} = 2$.
\item[\textbullet] $C_{ij} \leq 0$ if $i \neq j$.
\item[\textbullet] $C_{ij} = 0$ if and only if $C_{ji} = 0$.
\end{enumerate}
The set $I$ is called the \emph{Dynkin set}. The elements of $I$ are called the \emph{Dynkin vertices}. The generalised Cartan matrix $C$ is called \emph{symmetrisable} if there exists a family ${(d_i)}_{i \in I}$, called a \emph{symmetrising vector} of $C$, with values in $\ZZ_{> 0}$ and such that $d_i \: \! C_{ij} = d_j \: \! C_{ji}$ for \hbox{all $i,j$} \hbox{in $I$}. We assume that the integers $d_i$ ($i \in I'$) are coprime for every connected component $I'$ of $I$ (two elements $i,j \in I$ are linked if $C_{ij} \neq 0$). In this paper, \textbf{we always assume that a generalised Cartan matrix is symmetrisable}. \\

A generalised Cartan matrix is of \emph{finite type} if it is non-degenerate (or positive-definite, equivalently). Generalised Cartan matrices of finite type are also called Cartan matrices. \\

A \emph{root datum} $\Xdat$ associated to a generalised Cartan matrix $C$ consists of the following.
\begin{enumerate}
\item[\textbullet] A pair $(X,Y)$ where $X$, $Y$ are two finitely generated free $\ZZ$-modules.
\item[\textbullet] Families $(\alpha_i \: \! ; \, i \in I)$ and ${({\check \alpha_i} \: \! ; \, i \in I)}$ with values in $X$ and $Y$, respectively. We demand that the maps $i \mapsto \alpha_i$ and $i \mapsto \check \alpha_i$ are injective.
\item[\textbullet] A perfect bilinear pairing $\langle \cdot \, , \cdot \rangle : Y \times X \to \ZZ$ such that $\langle {\check \alpha_i}, \alpha_j \rangle = C_{ij}$ for all $i,j \in I$.
\end{enumerate}
The $\ZZ$-modules $X$ and $Y$ are called the \emph{root lattice} and the \emph{coroot lattice}, respectively. We call $\alpha_i$ and ${\check \alpha_i}$ ($i \in I$) the \emph{simple roots} and the \emph{simple coroots}, respectively. Where there is no possibility for confusion, we normally abuse notation slightly and write $(X, Y, I)$ to designate the root datum $\Xdat$. The root datum is said of \emph{finite type} if the generalised Cartan matrix $C$ is of finite type. \\

We say that the root datum $\Xdat$ is \emph{$X \!$-regular} if the family of simple roots is linearly independent and \emph{$Y \!$-regular} if the family of simple coroots is. When $\Xdat$ is of finite type, $X \!$-regularity and $Y \!$-regularity are automatic by the non-degeneracy of $C$, but in general it is an additional assumption, which we make in this paper. \\

An element $\lambda$ in $X$ is said \emph{dominant} if $\langle {\check \alpha_i}, \lambda \rangle \geq 0$ for all $i \in I$. The subset of dominant elements in $X$ is denoted by $\Xdom$ and called the \emph{dominant cone}. \\

On the root lattice $X$, we define a partial order (remark that the antisymmetry axiom is satisfied thanks to the $X \!$-regularity):
$$ \lambda \ \leq \ \lambda' \quad \ \text{ if } \ \quad \lambda' - \lambda \ \in \ \sum_{i \in I} \, \NN \: \! \alpha_i \, . $$

\subsection{Kac-Moody algebras with and without Serre relations} \label{subs_KMhat}
The \emph{Kac-Moody algebra} $\glie = \glie_{\KK} (\Xdat)$ associated to the root datum $\Xdat$ is the Lie \hbox{$\KK$-algebra} generated by the elements $X^-_i, X^+_i$ ($i \in I$), the \hbox{$\ZZ$-module $Y$} and subject to the following relations:
\begin{equation} \label{eq_KM_nondeformablehat} \left\{ {\renewcommand{\arraystretch}{1.3} \begin{array}{rclr}
\big[ \check \mu, \check \mu' \big] & = & 0 & (\check \mu, \check \mu' \in Y) \, , \\
\big[ {\check \mu}, X^{\pm}_i \big] & = & \pm \; \! \langle {\check \mu} , \alpha_i \rangle \, X_i^{\pm} & (i \in I, \ \check \mu \in Y) \, , \\
\big[ X^+_i, X^-_j \big] & = & 0 & \quad \quad (i,j \in I, \ i \neq j) \, ,
\end{array}} \right. \end{equation}
\begin{equation} \label{eq_KM_deformablehat} \left\{ {\renewcommand{\arraystretch}{1.3} \begin{array}{lr}
\big[ X^+_i, X^-_i \big] = \check \alpha_i & (i \in I) \, , \\
\sum_{k+k' = 1 - C_{ij}} (-1)^k \, {1 - C_{ij} \choose k} \, {(X_i^{\pm})}^k \; \! X^{\pm}_j \; \!  {(X_i^{\pm})}^{k'} = 0 & \quad (i,j \in I, \ i \neq j) \, .
\end{array}} \right.
\end{equation}

The universal enveloping algebra $\Uc$ of $\glie$ is the \hbox{$\KK$-algebra} generated by the elements $X^-_i$, $X^+_i$ ($i \in I$), the $\ZZ$-module $Y$, subject to the relations \eqref{eq_KM_nondeformablehat} and \eqref{eq_KM_deformablehat}. \\

The relations $\sum_{k+k' = 1 - C_{ij}} (-1)^k {1 - C_{ij} \choose k} {(X_i^{\pm})}^k X^{\pm}_j {(X_i^{\pm})}^{k'} = 0$ ($i, j \in I$, $i \neq j$) are called the \emph{Serre relations}. The elements $X^-_i, X^+_i$ ($i \in I$) are called the \emph{Chevalley generators} of $\glie$. \\

We denote by $\Unnc$ and $\Unpc$ the $\KK$-subalgebras of $\Uc$ generated by the elements $X^-_i$ ($i \in I$) and $X^+_i$ ($i \in I$), respectively. We denote by $\UCac$ the \hbox{$\KK$-subalgebra} of $\Uc$ generated by $Y$. The algebra $\Uc$, when viewed as a $\KK$-vector space, admits the following decomposition, where the isomorphism is given by the multiplication of $\Uc$:
$$ \Uc \ \simeq \ \, \Unnc \, \otimes_\KK \, \UCac \, \otimes_\KK \, \UCac \, . $$
We call this decomposition the \emph{triangular decomposition} of $\Uc$. \\


\emph{The Kac-Moody algebra without Serre relations} $\tilde \glie = \tilde \glie_{\KK} (\Xdat)$ associated to the root datum $\Xdat$ is the Lie \hbox{$\KK$-algebra} generated by the elements $X^-_i, X^+_i$ ($i \in I$), the \hbox{$\ZZ$-module $Y$}, subject to the relations \eqref{eq_KM_nondeformablehat} and to the following ones:
\begin{equation} \label{eq_KM_deformablethat}
\big[ X^+_i, X^-_i \big] = \check \alpha_i \quad \quad (i \in I) \, .
\end{equation}

The universal enveloping algebra $\Utc$ of $\tglie$ is the \hbox{$\KK$-algebra} generated by the elements $X^-_i$, $X^+_i$ ($i \in I$), the $\ZZ$-module $Y$, subject to the relations \eqref{eq_KM_nondeformablehat} \hbox{and \eqref{eq_KM_deformablethat}}. We denote by $\vpitc$ the projection map from $\Utc$ \hbox{to $\Uc$}. The algebra $\Utc$ also admits a triangular decomposition.

\subsection{The Lie algebra $\slt$ and the simple $\slt$-directions} \label{subs_KMonehat}
We designate by $\Aone$ the root datum defined by the following.
\begin{enumerate}
\item[\textbullet] The root and coroot lattices are both equal to $\ZZ$.
\item[\textbullet] The simple root and the simple coroot are $2 \in \ZZ$ and $1 \in \ZZ$, respectively.
\item[\textbullet] The pairing is the product of $\ZZ$.
\end{enumerate}
Note that the Kac-Moody algebra associated to $\Aone$ is naturally isomorphic to the Lie algebra of 2-by-2 matrices with coefficients in $\KK$ and whose trace is zero; we denote it by $\slt$. \\

The two Chevalley generators of $\slt$ are designated by $X^-$ and $X^+$, we besides denote by $H$ the simple coroot of $\Aone$. The Lie $\KK$-algebra $\slt$ is generated by the elements $X^-, X^+$, $H$ and subject to the following relations:
\begin{IEEEeqnarray}{rCl}
\label{eq_KM_nondeformableonehat} \big[ H, X^{\pm} \big] & \ = \ & \pm \; \! 2 \, X^{\pm} \, , \\
\label{eq_KM_deformableonehat} \big[ X^+, X^- \big] & = & H \, .
\end{IEEEeqnarray}
Remark that there is no Serre relations for $\slt$. \\

Let $i \in I$. There exist unique Lie $\KK$-algebra morphisms from $\slt$ to $\glie$ \hbox{and $\tglie$}, which \hbox{sends $X^{\pm}$} and $H$ to $X_i^{\pm}$ and $\check \alpha_i$, respectively. The induced $\KK$-algebra morphisms from $\Uc[\slt]$ to $\Uc$ and $\Utc$ are denoted by $\Deltac$ and $\Deltac[\tglie]$, respectively. The morphisms $\Deltac$ and $\Deltac[\tglie]$ ($i \in I$) are called the \emph{simple $\slt$-directions} \hbox{of $\Uc$} and $\Utc$, respectively.

\subsection{Highest weight representations and the category $\mathcal O$}
Let $V$ be a representation of $\glie$ and let $\lambda \in X$. A vector $v$ in $V$ is called a \emph{weight vector} of weight $\lambda$ if $\check \mu.v = \langle \check \mu, \lambda \rangle \; \! v$ for all $\check \mu \in Y$. The \emph{weight space} \hbox{of $V$} of \hbox{weight $\lambda$} is the $\KK$-vector subspace of $V$ whose vectors are of weight $\lambda$, we denote it \hbox{by $V_{\lambda}$}. Remark that the sum of the weight spaces of $V$ is direct. The representation $V$ is called a \emph{weight representation} \hbox{of $\glie$} \hbox{if $V$} is equal to the sum of its weights spaces.  If $V$ is a weight representation and if $V_\lambda$ is nonzero, we call $\lambda$ a \emph{weight} of $V$. The set of weights of a weight \hbox{representation $V$} is denoted \hbox{by $\wt (V)$}. \\

A representation $V$ of $\glie$ is called a \emph{highest weight representation} of $\glie$ of highest weight $\lambda$ if the following two conditions hold.
\begin{enumerate}
\item[\textbullet] The representation $V$ is generated by a nonzero vector $v_{\lambda}$ in $V_{\lambda}$. The vector $v_{\lambda}$ is called a \emph{highest weight vector} of $V$.
\item[\textbullet] The weights of $V$ are less than or equal to $\lambda$.
\end{enumerate}
Remark that a highest weight representation of $\glie$ is a weight representation \hbox{of $\glie$} with finite-dimensional weight spaces. \\

We denote by $\Ocatc$ the category of representations $V$ of $\glie$ which satisfy the following conditions.
\begin{enumerate}
\item[\textbullet] The representation $V$ is a weight representation of $\glie$ with finite-dimensional weight spaces.
\item[\textbullet] There exist finitely many elements $\lambda_1, \dots, \lambda_k$ in $X$ such that $\wt (V)$ is contained in $\bigcup_{s=1}^k (\lambda_s - \sum_{i \in I} \NN \alpha_i)$.
\end{enumerate}


\subsection{Some known results} \label{subs_knownresultshat}
A weight representation $V$ of $\glie$ is called \emph{integrable} if the actions of the Chevalley generators on $V$ are locally nilpotent. We denote by $\Ointc$ the category of integrable representations of $\glie$. When the root datum $\Xdat$ is of finite type (equivalently, when $\glie$ is finite-dimensional), a representation of $\glie$ is in $\Ointc$ if and only if it is finite-dimensional. \\

Let us recall some results on the representation theory of $\glie$.
\begin{enumerate}
\item[\textbullet] For every $\lambda \in X$, there is a unique, up to isomorphism, irreducible representation $\Lc{\lambda}$ of $\glie$ of highest weight $\lambda$.
\item[\textbullet] Let $\lambda \in \Xdom$. The representation $\Lc{\lambda}$ is generated by a vector $v_{\lambda}$ and subject to the following relations:
$$ \check \mu . v_{\lambda} \, = \, \langle \lambda, \check \mu \rangle \; \! v_{\lambda} \, , \quad X^+_i . v_{\lambda} = 0 \, , \quad (X^-_i)^{\langle \lambda, \check \alpha_i \rangle +1} . v_{\lambda} \, = \, 0  \quad \ (i \in I, \ \check \mu \in Y) \, . $$
\item[\textbullet] The representations $\Lc{\lambda}$ ($\lambda \in X$) are pairwise distinct and are the unique, up to isomorphism, irreducible representations in $\Ocatc$.
\item[\textbullet] Let $\lambda,\lambda' \in X$. Every $\glie$-morphism from $\Lc{\lambda}$ to $\Lc{\lambda'}$ is equal to a scalar multiple of the identity map if $\lambda = \lambda'$ and is zero else.
\item[\textbullet] Let $\lambda \in X$. The representation $\Lc{\lambda}$ is integrable if and only if $\lambda \in \Xdom$.
\item[\textbullet] Every representation in the category $\Ointc$ is isomorphic to a direct sum of representations $\Lc{\lambda}$ with $\lambda \in \Xdom$.
\end{enumerate}

\vsp

We give below, for the reader's convenience, a proof of a complete-reducibility type result for the category of integrable representations of $\slt$.

\begin{prop} \label{prop_KMintonehat}
An integrable representation of $\slt$ is isomorphic to a direct sum of representations $\Lc{n}$ with $n \in \NN$. Besides, a representation of $\slt$ is integrable if and only if the actions of $X^-,X^+,H$ are locally finite.
\end{prop}

\proof
Let $V$ be a representation of $\slt$ on which the actions of $X^-,X^+,H$ are locally finite and let $v$ be a vector in $V$. Denote by $V'$ the subrepresentation of $V$ generated by $v$. In view of the triangular decomposition of $\slt$ and since the actions of $X^-,X^+,H$ on $V$ are locally finite, the subrepresentation $V'$ is finite-dimensional. Therefore, $V'$ is a sum of representations $\Lc{n}$ with $n \in \NN$. This being true for \hbox{every $v$} in $V$, the representation $V$ is a sum of irreducible representations $\Lc{n}$ with $n \in \NN$. Hence, the representation $V$ is a direct sum of irreducible representations $\Lc{n}$ with $n \in \NN$.
\qed \\

The author ignores whether there exists (when $\glie$ is not finite-dimensional) a reference for the following proposition. The proof given below follows from a private communication with Victor Kac. Note that it is possible in the case \hbox{of $\slt$} to give a more direct and elementary proof.

\begin{prop} \label{prop_UBperfecthat}
An element in $\Uc$ is zero if and only if it acts as zero on all the representations of $\glie$ in $\Ointc$.
\end{prop}

\proof
Consider the ${X \!}$-gradation of the $\KK$-algebra $\Uc$ defined by the following:
$$ \mathrm{deg} \; \! (\mu) \, := \, 0 \, , \quad \quad \mathrm{deg} \; \!  (X^\pm_i) \, := \, \alpha_i \quad \quad \quad (\mu \in X, \ i \in I) \, . $$
Let $x$ be a nonzero element in $\Uc$. Recall the triangular decomposition of $\Uc$ and pick nonzero homogeneous \hbox{elements $x^-_k$, $x^0_k$} \hbox{and $x^+_k$} \hbox{in $\Unnc$}, $\UCac$ and $\Unpc$ respectively \hbox{($1 \leq k \leq l$)} such that the following conditions are satisfied:
\begin{enumerate}
\item[a)] $x = \sum_{1 \leq k \leq l} x^-_k \; \! x^0_k \; \! x_k^+$,
\item[b)] the family ${(x^-_k)}_{1 \leq k \leq l}$ is $\KK$-linearly independent,
\item[c)] $\deg (x_1^+) \ngtr \deg (x_k^+)$ for every $1 < k \leq l$.
\end{enumerate}
Denote by $\tau$ the involutive anti-automorphism of the $\KK$-algebra $\Uc$ which sends $X_i^\pm$ \hbox{to $-X_i^\mp$} for every $i \in I$ and which fixes pointwise $Y$. Pick $\lambda$ in $\Xdom$ such that the following conditions hold (we denote by $v_{\lambda}$ a fixed choice of highest weight vector of the representation $\Lc{\lambda}$):
\begin{enumerate}
\item[a')] $x^0_1 (\lambda)$ is nonzero,
\item[b')] the family ${(x^-_k . v_{\lambda})}_{1 \leq k \leq l}$ is $\KK$-linearly independent,
\item[c')] the vector $\tau(x^+_1) . v_{\lambda}$ is nonzero.
\end{enumerate}
Pick now a nonzero $\KK$-bilinear form $\sigma$ on $\Lc{\lambda}$ which has the following properties:
\begin{enumerate}
\item[a'')] $\sigma$ is non-degenerate,
\item[b'')] $\sigma(v,w) = 0$ for every weight vectors $v,w$ in $\Lc{\lambda}$ of distinct weights,
\item[c'')] $\sigma$ is contravariant, i.e. $\sigma (y.v,w) = \sigma(v,\tau(y).w)$ for every $v,w \in \Lc{\lambda}$ and for every $y \in \Uc$.
\end{enumerate}
In view of the non-degeneracy of $\sigma$ and in view of the orthogonality property $b'')$, pick a weight vector $v$ in $\Lc{\lambda}$ of weight $\lambda - \deg(x_1^+)$ such that $\sigma(v, \tau(x^+_1).v_{\lambda})$ is nonzero. Since $\sigma$ is contravariant, the vector $x_1^+.v$ is thus a nonzero scalar multiple of $v_{\lambda}$. Thanks to condition $c)$ above, the following holds:
$$ x.v \ = \ \sum_{\substack{1 \leq k \leq l, \\ \deg(x^+_k) \, = \, \deg (x_1^+)}} x^0_k(\lambda) \, x^-_k . (x^+_k . v) \, . $$
In view of conditions $a')$ and $b')$, one concludes that $x.v$ is nonzero.
\qed

\section{Coloured Kac-Moody algebras}
\subsection{$\glieF$-pointed algebras}
Recall the definition of the Kac-Moody algebra $\glie$: see subsection \ref{subs_KMhat}. The \hbox{relations \eqref{eq_KM_nondeformablehat}} are called the \emph{non-deformable relations} \hbox{of $\glie$}.

\begin{defi}
We denote by $\glieF$ the Lie $\KK$-algebra generated by \hbox{$X^\pm_i$ ($i \in I$)}, the \hbox{$\ZZ$-module $Y$}, and subject to the non-deformable relations of $\glie$:
\begin{equation*}
\left\{ {\renewcommand{\arraystretch}{1.3} \begin{array}{rclr}
\big[ \check \mu, \check \mu' \big] & = & 0 & (\check \mu, \check \mu' \in Y) \, , \\
\big[ {\check \mu}, X^{\pm}_i \big] & = & \pm \; \! \langle {\check \mu} , \alpha_i \rangle \, X_i^{\pm} & (i \in I, \ \check \mu \in Y) \, , \\
\big[ X^+_i, X^-_j \big] & = & 0 & \quad \quad (i,j \in I, \ i \neq j) \, .
\end{array}} \right.
\end{equation*}
\end{defi}

When $\glie = \slt$, we denote $\glieF$ by $\sltF$. \\



Let $V$ be a representation of $\UF$ and let $\lambda \in X$. A vector $v$ in $V$ is called a \emph{vector of weight} $\lambda$ if $\check \mu.v = {\langle \check \mu \, ,  \lambda \rangle \! \: v}$ for all $\check \mu \in Y$. The \emph{weight space} \hbox{of $V$} of \hbox{weight $\lambda$} is the \hbox{$R$-submodule} of vectors of \hbox{weight $\lambda$} in $V$, we denote it \hbox{by $V_\lambda$}. We call $\lambda$ a \emph{weight} of $V$ \hbox{if $V_\lambda$} is nonzero, the set of weights of $V$ is denoted \hbox{by $\wt (V)$}. Remark that the sum of weight spaces of $V$ is direct. The representation $V$ is called a \emph{weight representation} if $V$ is equal to the sum of its weights spaces.

\begin{defi}
If $V$ is a weight representation and if $V$ is free when viewed as a $R$-module, we say that $V$ admits a character and we denote by $\chi(V)$ the function from $X$ to $\NN \cup \{ \infty \}$ which associates to $\lambda \in X$ the rank of the free \hbox{$R$-module $V_{\lambda}$}.
\end{defi}

A representation $V$ of $\UF$ is said \emph{integrable} if it is a weight representation and if the action of $X^\pm_i$ \hbox{on $V$} is locally nilpotent for every $i \in I$.

\begin{defi}
A topological $R$-algebra $A$ is said $\glieF$-pointed if $A$ is endowed with a $R$-algebra morphism $\dot A$ from $\UF$ to $A$, called the $\glieF$-point of $A$, and whose image is dense in $A$.
\end{defi}


Every element in $\UF$ is regarded as an element in $A$ as well and every representation $V$ of $A$ is viewed also as a representation of $\UF$, thanks to the pull-back of $V$ via the $\glieF$-point $\dot A$. The notions of weight representations, characters and integrability are then accordingly defined for representations of $\glieF$-pointed algebras.

\subsection{The global crystal of rank 1}
\begin{defi}
The global crystal of rank 1, denoted by $\BB$, is the quiver defined by the following.
\begin{enumerate}
\item[\textbullet] The set $\BBv$ of vertices of $\BB$ is ${\big \{ b_{n,p} := (n,n-2p) \; \! ; \, n,i \in \NN, \, p \leq n \big \}}$.
\item[\textbullet] The set $\BBe$ of edges of $\BB$ is the disjoint union ${\BBe[-] \sqcup \BBe[+]}$ of two copies \hbox{of $[\BBe] := {\big \{ (n,k) \; \! ; \, n,k \in \NNI, \, k \leq n \big \}}$}. The source of the edge $(n,k) \in \BBe[\pm]$ is $(n,\pm (n-2k+2))$ and its target is $(n,\pm (n-2k))$.
\end{enumerate}
\end{defi}

\vsp
\vsp

Here is an illustration of the global crystal of rank 1. Remark that there is no arrow for $n=0$.
\begin{equation*}
\shorthandoff{;:!?} \xymatrix @!0 @C=3pc @R=1.1pc {
&&&&& \quad \ \! \; \bullet \! \; \scriptstyle{b_{n,0}} \ar@<4pt>[dd] && \\
&&&&&& \quad \ \ \iddots & \\
&&& \quad \ \! \; \bullet \! \; \scriptstyle{b_{3,0}} \ar@<4pt>[dd] && \quad \ \! \; \bullet \! \; \scriptstyle{b_{n,1}} \ar@<4pt>[uu] \ar@<4pt>[dd] && \\
&& \quad \ \! \; \bullet \! \; \scriptstyle{b_{2,0}} \ar@<4pt>[dd] && \iddots &&& \\
& \quad \ \! \; \bullet \! \; \scriptstyle{b_{1,0}} \ar@<4pt>[dd] && \quad \ \! \; \bullet \! \; \scriptstyle{b_{3,1}} \ar@<4pt>[dd] \ar@<4pt>[uu] && \quad \ \! \; \bullet \! \; \scriptstyle{b_{n,2}} \ar@<4pt>[uu] && \\
\quad \ \! \; \bullet \! \; \scriptstyle{b_{0,0}} && \quad \ \! \; \bullet \! \; \scriptstyle{b_{2,1}} \ar@<4pt>[dd] \ar@<4pt>[uu] && \cdots & \vdots & \quad \ \ \cdots & \\
& \quad \ \! \; \bullet \! \; \scriptstyle{b_{1,1}} \ar@<4pt>[uu] && \quad \ \! \; \bullet \! \; \scriptstyle{b_{3,2}} \ar@<4pt>[dd] \ar@<4pt>[uu] && \quad \quad \ \, \bullet \! \; \scriptstyle{b_{n,n-2}} \ar@<4pt>[dd] && \\
&& \quad \ \! \; \bullet \! \; \scriptstyle{b_{2,2}} \ar@<4pt>[uu] && \ddots &&& \\
&&& \quad \ \! \; \bullet \! \; \scriptstyle{b_{3,3}} \ar@<4pt>[uu] && \quad \quad \ \, \bullet \! \; \scriptstyle{b_{n,n-1}} \ar@<4pt>[uu] \ar@<4pt>[dd] & \\
&&&&&& \quad \ \ \ddots & \\
&&&&& \quad \ \, \! \; \bullet \! \; \scriptstyle{b_{n,n}} \ar@<4pt>[uu] && \\
}
\end{equation*}

Let $n \in \NN$. We denote by $\BB_n$ the connected component of $\BB$ containing the vertex $b_{n,0}$. We denote by $\BBv_n$ and $\BBe_n$ the sets of vertices and edges in $\BB_n$, respectively. The quiver $\BB_n$ is called the \emph{$n$th component of $\BB$}. \\

Let $f \in R^{[\BBe]}$ and let $n,k \in \NN$ such that $k \leq n$. We denote by $f(n,k) !$ the element in $R$ defined by the following:
$$ f(n,k) ! \ := \ \left\{ {\renewcommand{\arraystretch}{1.3} \begin{array}{cl}
\prod_{k'=1}^k f(n,k') & \quad \ \text{if } k \geq 1 \, , \\
0 & \quad \ \text{else.}
\end{array}} \right. $$

\subsection{Colourings} \label{subs_colouringshat}
Let $n \in \NN$, remark that the irreducible representation $\Lc{n}$ of $\slt$ can be defined as the $\KK$-vector space ${\KK \BBv_n}$, where the actions of $X^\pm, H$ are given by the following, for $p \in \NN_{\leq n}$:
$$ \left\{ {\renewcommand{\arraystretch}{0.75} \begin{array}{rcl}
H . \! \; b_{n,p} & := & \ (n- 2p) \, b_{n,p} \, , \\
&& \\
X^- . \! \; b_{n,p} & := & \left\{ {\renewcommand{\arraystretch}{1.1} \begin{array}{cl}
(p+1) \, b_{n,p+1} & \quad \text{if $p \neq n$,} \\
0 & \quad \text{else,}
\end{array}} \right. \\
&& \\
X^+ . \! \; b_{n,p} & := &  \left\{ {\renewcommand{\arraystretch}{1.1} \begin{array}{cl}
(n-p+1) \, b_{n,p-1} & \quad \text{if $p \neq 0$,} \\
0 & \quad \text{else.}
\end{array}} \right.
\end{array}} \right.
$$

\vsp
\vsp

\begin{defi}
\begin{enumerate}
\item[\textbullet] We denote by $\Colo{R}$ the ring $R^{\BBe}$, an element is called a colouring of the crystal $\BB$ with values in $R$.
\item[\textbullet] We denote by $\Colo[I]{R}$ the ring $(\Colo{R})^I$, an element $\psi = {(\psi_i)}_{i \in I}$ is called an ${I \! \! \:}$-colouring of the crystal $\BB$ with values in $R$.
\end{enumerate}
\end{defi}

When the cardinality of $I$ is $1$, we identify the rings $\Colo{R}$ and $\Colo[I]{R}$ and we identify colourings with $I$-colourings. \\

Let $\psi \in \Colo{R}$. We denote by $\psi^\pm$ the restriction of $\psi$ to $\BBe[\pm]$. Here is an illustration of the colouring $\psi$ over the $n$th component of $\BB$.

$$ \shorthandoff{;:!?} \xymatrix @!0 @R=1.1pc {
\quad \ \, \! \; \bullet \! \; \scriptstyle{b_{n,0}} \ar@<4pt>[dd]^{\scriptstyle \psi^-(n,1)} \\
\\
\quad \ \, \! \; \bullet \! \; \scriptstyle{b_{n,1}} \ar@<4pt>[uu]^{\scriptstyle \psi^+(n,n)} \ar@<4pt>[dd]^{\scriptstyle \psi^-(n,2)}\\
 \\
\quad \ \, \! \; \bullet \! \; \scriptstyle{b_{n,2}} \ar@<4pt>[uu]^{\scriptstyle \psi^+(n,n-1)} \\
\vdots \\
\quad \quad \ \, \bullet \! \; \scriptstyle{b_{n,n-2}} \ar@<4pt>[dd]^{\scriptstyle \psi^-(n,n-1)} \\
\\
\quad \quad \ \, \bullet \! \; \scriptstyle{b_{n,n-1}} \ar@<4pt>[uu]^{\scriptstyle \psi^+(n,2)} \ar@<4pt>[dd]^{\scriptstyle \psi^-(n,n)} \\
\\
\quad \ \, \! \; \bullet \! \; \scriptstyle{b_{n,n}} \ar@<4pt>[uu]^{\scriptstyle \psi^+(n,1)}
} $$

\vsp

\begin{defi}
Let $\psi \in \Colo{R}$ and let $n \in \NN$. We denote \hbox{by $\Lrep{n,\psi}$} the representation \hbox{of $\UFslt$} whose underlying \hbox{$R$-module} is ${R \BBv_n}$ and where the actions of $X^{\pm}, H$ are defined by the following, for $p \in \NN_{\leq n}$:
$$ \left\{ {\renewcommand{\arraystretch}{0.75} \begin{array}{rcl}
H . \! \; b_{n,p} & := & \ (n- 2p) \, b_{n,p} \, , \\
&& \\
X^- . \! \; b_{n,p} & := & \left\{ {\renewcommand{\arraystretch}{1.1} \begin{array}{cl}
\psi^-(n,p+1) \, b_{n,p+1} & \quad \text{if $p \neq n$,} \\
0 & \quad \text{else,}
\end{array}} \right. \\
&& \\
X^+ . \! \; b_{n,p} & := &  \left\{ {\renewcommand{\arraystretch}{1.1} \begin{array}{cl}
\psi^+(n,n-p+1) \, b_{n,p-1} & \quad \text{if $p \neq 0$,} \\
0 & \quad \text{else.}
\end{array}} \right.
\end{array}} \right.
$$
\end{defi}

Remark at once that the representation $\Lrep{n,\psi}$ is a $R$-free integrable representation of $\UFslt$.

\begin{lem} \label{lem_indecompAonehat}
Let $\psi \in \Colo{R}$ and let $n \in \NN$. The representation $\Lrep[\! \! \: K(R)]{n,\psi}$ is irreducible if and only if the restriction of the function to $\psi$ to $\BBe_n$ has no zero.
\end{lem}

\proof
Suppose that the restriction of the function $\psi$ to $\BBe_n$ has no zero and let $V$ be a subrepresentation of $\Lrep[\! \! \: K(R)]{n,\psi}$. Since $\Lrep[\! \! \: K(R)]{n,\psi}$ is a weight representation of $\UFslt$, so is $V$. Hence, $V$ is nonzero if and only if there exists $p \in \NN_{\leq n}$ such that $b_{n,p} \in V$. Besides, in view of the action of $X^\pm$ on the representation $\Lrep[\! \! \: K(R)]{n,\psi}$, remark that the subrepresentation of $\Lrep[\! \! \: K(R)]{n,\psi}$ generated by the vector $b_{n,p}$ is proper if and only if the restriction of $\psi$ to $\BBe_n$ has no zero. One concludes.
\qed

\vsp
\vsp

\begin{defi}
Let $\psi \in \Colo[I]{R}$.
\begin{itemize}
\item[\textbullet] The $I$-colouring $\psi$ is said non-degenerate if $\psi_i$ has no zero for every $i \in I$.
\item[\textbullet] The $I$-colouring $\psi$ is said admissible (over $R$) if it is invertible in $\Colo[I]{R}$.
\end{itemize}
\end{defi}

Remark that an admissible $I$-colouring is in particular non-degenerate.

\subsection{Tannakian topologies}
Let $A$ be $R$-algebra, let $\mathcal C$ be a class of representations of $A$ and let $B$ be a topological \hbox{$R$-algebra}.

\begin{defi}
\begin{itemize}
\item[\textbullet] The Tannakian topology $\tau(\mathcal C)$ generated by the class $\mathcal C$ is the coarsest topology on $A$ which makes continuous every map ${A \! \times \! V \! \to \! V}$ associated to a representation $V$ in $\mathcal C$, where $V$ is endowed with the discrete topology.
\item[\textbullet] A representation $V$ of $B$ is said \emph{discrete} if the associated map ${B \! \times \! V \! \to \! V}$ is continuous, where $V$ is endowed with the discrete topology.
\end{itemize}
\end{defi}

Remark that the topology $\tau(\mathcal C)$ is an initial topology: more precisely, it is the coarsest topology on $A$ which makes continuous every map ${A \! = \! A \! \times \! \{ v \} \! \to \! V}$ associated to a representation $V \in \mathcal C$ and a vector $v \in V$, where $V$ is endowed with the discrete topology.

\begin{prop} \label{prop_Tantopo}
\begin{enumerate}
\item[1)] The $R$-algebra $A$ endowed with the topology $\tau(\mathcal C)$ is a topological $R$-algebra.
\item[2)] The following two assertions are equivalent.
\begin{enumerate}
\item[(i)] The topological $R$-algebra $A$ is separated.
\item[(ii)] An element in $A$ is zero if and only if it acts as zero on all the representations in $\mathcal C$.
\end{enumerate}
\item[3)] A sequence ${(x_n)}_{n \in \NN}$ with values in $A$ converges to zero if and only if the sequence ${(x_n.v)}_{n \in \NN}$ is eventually zero, for every $V \in \mathcal C$ and $v \in V$.
\end{enumerate}
\end{prop}

\proof
Let $V \in \mathcal C$, let $v \in V$ and denote by $I(v)$ the left ideal of $A$ formed by the elements acting by zero on the vector $v$. Let $x \in A$ and remark that $I(v)$ contains ${I(x.v) \! \: x}$. As a consequence, the $R$-submodules $I(v)$ of $A$, with $v \in V$ and $V \in \mathcal C$, form a local basis at the origin of a $R$-linear topology which endows the $R$-algebra $A$ with a structure of topological $R$-algebra. This topology is the Tannakian topology generated by $\mathcal C$. \hbox{Points 1} and 2 of the proposition then follow. Point 3 is a consequence of the initial form of the topology $\tau(\mathcal C)$.
\qed \\

We regard from now $A$ as a topological $R$-algebra, endowed with the Tannakian topology $\tau(\mathcal C)$. Remark that a representation in $\mathcal C$ is a discrete representation of $A$.

\begin{prop} \label{prop_Tantopo2}
Let $f$ be a $R$-algebra morphism from $B$ to $A$.
\begin{enumerate}
\item[1)] The map $f$ is continuous if and only if the pull-back of every representation in $\mathcal C$ via $f$ is a discrete representation of $B$.
\item[2)] Let $\mathcal C'$ be a class of representations of $B$ and suppose that $B$ is endowed with the topology $\tau(\mathcal C')$. If the pull-back via $f$ of every representation \hbox{in $\mathcal C$} is in $\mathcal C'$, then the map $f$ is continuous.
\end{enumerate}
\end{prop}

\proof
Remark that a representation $V$ of $B$ is continuous if and only if every map ${B \! \times \! \{ v \} \! \to \! V}$ associated to the representation $V$ and to a vector $v \in V$, is continuous, where $V$ is endowed with the discrete topology. Point 1 is then a consequence of the initial form of the topology $\tau(\mathcal C)$. Point 2 follows from \hbox{point 1} and from the definition of the topology $\tau(\mathcal C')$.
\qed \\

Denote by $\widehat A(\mathcal C)$ the (separated) completion of $A$. By definition, $\widehat A(\mathcal C)$ is a topological $R$-algebra, which comes with a continuous $R$-algebra morphism $\widehat \pi(\mathcal C)$ from $A$ to $\widehat A(\mathcal C)$.

\begin{prop} \label{prop_discreterep}
For every representation $V$ in $\mathcal C$, there exists a unique discrete representation of $\widehat A(\mathcal C)$ whose pull-back via $\widehat \pi(\mathcal C)$ \hbox{is $V$}.
\end{prop}

\proof
Let $V \in \mathcal C$ and denote by $f$ the \hbox{$R$-algebra} map from $A$ to $\EndR (V)$ associated to the \hbox{representation $V$}. We regard $V$ as a representation of $\EndR(V)$ and we endow the latter algebra with the Tannakian topology generated by the unique representation $V$. Remark then that $\EndR(V)$ is a separated topological $R$-algebra (see proposition \ref{prop_Tantopo}) and is moreover complete. Since $V$ is in $\mathcal C$, note that the map $f$ is continuous (see point 2 of proposition \ref{prop_Tantopo2}). Therefore, and by definition of $\widehat A(\mathcal C)$, there exists a unique continuous $R$-algebra \hbox{morphism $\widehat f$} from $\widehat A(\mathcal C)$ to $\EndR(V)$, which makes the following diagram commutes:
$$ \shorthandoff{;:!?} \xymatrix @C=5pc @R=1.5pc {
A \ar[r]^-{f} \ar[d]_-{\widehat \pi(\mathcal C) \, } & \, \EndR (V) \, . \\
\widehat A(\mathcal C) \ar[ur]_-{\widehat f} &
} $$
A $R$-algebra morphism $g$ from $\widehat A(\mathcal C)$ to $\EndR(V)$ being continuous if and only if the pull-back of $V$ via $g$ is a discrete representation of $\widehat A(\mathcal C)$ (see point 1 of proposition \ref{prop_Tantopo2}), we are done.
\qed

\subsection{$\psi$-integrable representations}
In this subsection, $\psi$ designates an \hbox{$I$-colouring} of $\BB$ with values in $R$. \\

Let $i \in I$, we denote by $\DeltaF$ the unique $R$-algebra morphism from $\UFslt$ \hbox{to $\UF$} which sends $X^\pm$ and $H$ \hbox{to $X_i^\pm$} and $\check \alpha_i$, respectively.

\begin{defi} \label{defi_psiinthat}
A representation $V$ of $\UF$ is said \hbox{$\psi$-integrable} if the following two conditions are satisfied.
\begin{enumerate}
\item[a)] The representation $V$ is a weight representation of $\UF$.
\item[b)] For every $i \in I$, the pull-back of $V$ via $\DeltaF$ is isomorphic to a direct sum of representations $\Lrep{n,\psi_i}$ with $n \in \NN$.
\end{enumerate}
\end{defi}

Note, in view of condition b) in the previous definition, that the $\psi$-integrable representations of $\UF$ are integrable and admit a character.

\begin{defi}
The $\psi$-integrable topology on $\UF$ is the Tannakian topology generated by the $\psi$-integrable representations of $\UF$.
\end{defi}

Note that $\UF$, endowed with the $\psi$-integrable topology, is a topological \hbox{$R$-algebra}, in view of point 1 of proposition \ref{prop_Tantopo}.


\subsection{The algebra $\Uhat{\glie,\psi}$}
In this subsection, $\psi$ designates an \hbox{$I$-colouring} of $\BB$ with values in $R$.

\begin{defi}
\begin{itemize}
\item[\textbullet] We denote by $\Uhat{\glie,\psi}$ the (separated) completion of $\UF$ with respect to the \hbox{$\psi$-integrable} topology.
\item[\textbullet] If the $I$-colouring $\psi$ is admissible, we call $\Uhat{\glie,\psi}$ the coloured Kac-Moody algebra associated \hbox{to $\psi$}.
\end{itemize}
\end{defi}

By definition, the $R$-algebra $\Uhat{\glie,\psi}$ is naturally $\glieF$-pointed, we denote its \hbox{$\glieF$-point} by $\pihat{\glie,\psi}$. \\

Remark that for every \hbox{$\psi$-integrable} representation $V$ of $\UF$, there is, in view of proposition \ref{prop_discreterep}, a unique discrete representation of $\Uhat{\glie,\psi}$ whose \hbox{pull-back} via $\pihat{\glie,\psi}$ \hbox{is $V$}. We regard from now \hbox{$\psi$-integrable} representations of $\UF$ also as discrete representations of $\Uhat{\glie,\psi}$. In the case $\glie = \slt$, the representations $\Lrep{n,\psi}$ \hbox{($n \in \NN$)} of $\UFslt$ are in particular regarded as discrete representations of $\Uhat{\slt,\psi}$.

\begin{rem} \label{rem_Uhatperfect}
\begin{itemize}
\item[\textbullet] In view of point 2 of proposition \ref{prop_Tantopo}, an element in $\Uhat{\glie,\psi}$ is zero if and only if it acts as zero on all the $\psi$-integrable representations \hbox{of $\UF$}.
\item[\textbullet] In particular, when $\glie = \slt$, an element in $\Uhat{\slt,\psi}$ is zero if and only if it acts as zero on all the representations $\Lrep{n,\psi}$ \hbox{($n \in \NN$)}.
\end{itemize}
\end{rem}

We regard $Y$ as a $\ZZ$-submodule $R^X$, thanks to the non-degenerate pairing between $Y$ and $X$. We denote by $R[Y]$ the $R$-subalgebra of $R^X$ generated by $Y$.

\begin{prop} \label{prop_Uhatconv}
\begin{enumerate}
\item[1)] Let ${(x_a)}_{a \in \NN}$ be a sequence with values in $\Uhat{\glie,\psi}$. For every $i \in I$, the series $\sum_{a \in \NN} x_a {(X_i^\pm)}^a$ converges to a unique element in $\Uhat{\glie,\psi}$.
\item[2)] There exists a unique $R$-algebra morphism from $R^X$ to $\Uhat{\glie,\psi}$ which sends every element $\check \mu$ in $Y$ to $\check \mu$ in $\Uhat{\glie,\psi}$.
\end{enumerate}
\vsp
We regard from now every element in $R^X$ as an element in $\Uhat{\glie,\psi}$ as well.
\begin{enumerate}
\item[3)] The image of $R[Y]$ in $\Uhat{\glie,\psi}$ is dense in the image of $R^X$.
\item[4)] An element $f$ in $R^X$ acts as $f(\lambda)$ on every weight vector of weight $\lambda$ in a \hbox{$\psi$-integrable} representation of $\Uhat{\glie,\psi}$.
\end{enumerate}
\end{prop}

\proof
Remark that the sequence ${(x_a{(X_i^\pm)}^a .v)}_{a \in \NN}$ is eventually zero, for every \hbox{vector $v$} in a \hbox{$\psi$-integrable} representation of $\UF$. As a consequence, and by definition of the \hbox{$\psi$-integrable} topology on $\UF$, the partial sums associated to the series $\sum_{a \in \NN} x_a {(X_i^\pm)}^a$ form a Cauchy sequence in $\UF$ (see point 3 of proposition \ref{prop_Tantopo}). The latter series then converges to a unique element in the (separated) completion $\Uhat{\glie,\psi}$. \\

The \hbox{$\ZZ$-module $X$} being free, the pairing between $Y$ and $X$ being perfect, one can identify $X$ with the $\ZZ$-module $\ZZ^d$ and $R[Y]$ with the $R$-subalgebra of ${R^{\! \: \ZZ^d}}$ formed by the polynomials in $d$ variables, where $d$ designates the rank of the $\ZZ$-module $X$. Let now $f \in R^X$ and let $N \in \NN$. Remark that there exists a polynomial ${f_{\! \! \: N}}$ in $R[Y]$ such that the following holds:
$$ f(k_1, \dots, k_d) \ = \ {f_{\! \! \: N}} (k_1, \dots, k_d) \, , \quad \quad \forall \, k_1, \dots, k_d \in \ZZ, \ \, -N \leq k_1, \dots, k_d \leq N \, . $$
Remark besides that ${f_{\! \! \: N}}$ acts as ${f_{\! \! \: N}}(\lambda)$ on every weight vector of weight $\lambda$ in a \hbox{$\psi$-integrable} representation of $\UF$. Hence, and by definition of the \hbox{$\psi$-integrable} topology on $\UF$, the sequence ${({f_{\! \! \: N}})}_{N \in \NN}$ is a Cauchy sequence in $\UF$ (see again point 3 of proposition \ref{prop_Tantopo}). The sequence then converges to a unique element $\bar f$ in the (separated) completion $\Uhat{\glie,\psi}$.
\qed

\subsection{Congruence}
\begin{defi}
\begin{itemize}
\item[\textbullet] Let $\psi \in \Colo{R}$. The congruent class $[\psi]$ \hbox{of $\psi$} is the function from $[\BBe]$ to $R$ defined by the following:
$$ [\psi] (n,k) \ := \ \psi^-(n,k) \ \psi^+(n,n-k+1) \quad \quad (n,k \in \NNI, \ k \leq n) \, . $$
\item[\textbullet] Two colourings $\psi, \psi'$ of $\BB$ with values in $R$ are said congruent if $[\psi] = [\psi']$. We denote by $\equiv$ the congruence relation on the set $\Colo{R}$.
\end{itemize}
\end{defi}

\vsp
\vsp

The following proposition gives an interpretation of the congruence relation.

\begin{prop} \label{prop_congruentAonehat}
Let $\psi_1,\psi_2$ be two non-degenerate colourings of the \hbox{crystal $\BB$} with values \hbox{in $R$}.
\begin{enumerate}
\item[1)] If $\psi_1 \equiv \psi_2$, then the \hbox{$\sltF$-pointed} $R$-algebras $\Uhat{\slt,\psi_1}$ and $\Uhat{\slt,\psi_2}$ are isomorphic.
\item[2)] If $\psi_1$ and $\psi_2$ divide each other in the ring $\Colo{R}$, then the following two assertions are equivalent.
\begin{enumerate}
\item[(i)] The colourings $\psi_1$ and $\psi_2$ are congruent.
\item[(ii)] For every $n \in \NN$, the representations $\Lrep{n,\psi_1}$ and $\Lrep{n,\psi_2}$ are isomorphic.
\end{enumerate}
\end{enumerate}
\end{prop}

\proof
Let $\psi, \psi'$ be two non-degenerate colourings of $\BB$ with values in $R$ such that $\psi^-$ divides $(\psi')^-$ in the ring $\Colo{R}$. Let us first prove the following assertion.
\begin{equation} \tag{$\ast$}
\begin{tabular}{| p{22.5pc}}
The colourings $\psi$ and $\psi'$ are congruent if and only if there exists an injective $\UFslt$-morphism from $\Lrep{n,\psi}$ to $\Lrep{n,\psi'}$ for every $n \in \NN$.
\end{tabular}
\end{equation}
Fix $n \in \NN$. Since $b_{n,p}$ is of weight vector of weigh $n-2p$ in both $\Lrep{n,\psi}$ \hbox{and $\Lrep{n,\psi'}$} for every $p \in \NN_{\leq n}$, every $\UFslt$-morphism from $\Lrep{n,\psi}$ \hbox{to $\Lrep{n,\psi'}$} must send, for every $p \in \NN_{\leq n}$, the vector $b_{n,p}$ to a scalar multiple of it. Since besides the vector $b_{n,p}$ is a scalar multiple of $(X^-)^p.b_{n,0}$ in $\Lrep{n,\psi}$ \hbox{and $\Lrep{n,\psi'}$} for every $p \in \NN_{\leq n}$, every $\FUh[\sltF]$-morphism from $\Lrep{n,\psi}$ \hbox{to $\Lrep{n,\psi'}$} is thus a scalar multiple of the $R$-linear endomorphism ${f_{\! \! \: n}}$ of ${R \BBv_n}$, defined by the following:
$$ f_{\! \! \: n} (b_{n,p}) \ := \ \frac{(\psi')^-(n,p) !}{\psi^- (n,p) !} \ b_{n,p} \quad \quad \ (p \in \NN_{\leq n}) \, . $$
The colouring $\psi'$ being non-degenerate, the map ${f_{\! \! \: n}}$ is injective. Remark then that for every $p \in \NNI_{\leq n}$, the following equalities hold in $\Lrep{n,\psi}$ \hbox{and $\Lrep{n,\psi'}$}, respectively:
\begin{eqnarray*}
f_{\! \! \: n} \big( X^+ . \! \: b_{n,p} \big) & = & \psi^+ (n,n-p+1) \ \frac{(\psi')^-(n,p-1) !}{\psi^- (n,p-1) !} \ b_{n,p} \, , \\
X^+ . \big( f_{\! \! \: n} (b_{n,p}) \big) & = & (\psi')^+ (n,n-p+1) \ \frac{(\psi')^-(n,p) !}{\psi^- (n,p) !} \ b_{n,p} \, .
\end{eqnarray*}
Assertion $(\ast)$ follows. \\

Point 2 is a consequence. Let us now prove point 1. Denote by $\psi$ the colouring of $\BB$ defined by the following:
$$ \psi^-(n,k) \ := \ 1 \, , \quad \quad \psi^+(n,k) \ := \ [\psi_1] (n,k) \quad \quad \quad (n,k \in \NNI, \ k \leq n) \, . $$
Remark that $\psi \equiv \psi_1$. Assertion $(\ast)$ then implies that there exists an injective $\UFslt$-morphism from $\Lrep{n,\psi}$ to $\Lrep{n,\psi_1}$ for every $n \in \NN$. In view of the definition of the algebras $\Uhat{\slt,\psi}$ and $\Uhat{\slt,\psi_1}$, there hence exists an isomorphism of $\sltF$-pointed $R$-algebra between the two latter. One concludes.
\qed

\subsection{Diagram structure}
In this subsection, $\psi$ designates an \hbox{$I$-colouring} of $\BB$ with values in $R$.

\begin{defi-prop} \label{prop_Deltahat}
Let $i \in I$. We denote by $\Deltahat$ the unique continuous $R$-algebra morphism from $\Uhat{\slt,\psi_i}$ to $\Uhat{\glie,\psi}$ which makes the following diagram commute:
\begin{equation*}
\shorthandoff{;:!?} \xymatrix @C=5pc @R=1.5pc {
\UFslt \ar[r]^-{\DeltaF \ } \ar[d]_-{\pihat{\slt,\psi_i} \ } & \, \UF \ar[d]^-{\ \pihat{\glie,\psi}} \\
\Uhat{\slt,\psi_i} \ar[r]_-{\Deltahat} & \, \Uhat{\glie,\psi} \, .
} \end{equation*}
\end{defi-prop}

\proof
The pull-back via $\DeltaF$ of a \hbox{$\psi$-integrable} representation of $\UF$ being by definition $\psi_i$-integrable, the $R$-algebra morphism $\DeltaF$ is continuous with respect to the $\psi_i$-integrable topology on $\UFslt$ and to the $\psi$-integrable topology \hbox{on $\UF$} (see \hbox{point 2} of proposition \ref{prop_Tantopo2}). The algebras $\Uhat{\slt,\psi_i}$ and $\Uhat{\glie,\psi}$ being the (separated) completions of the $R$-algebras $\UFslt$ \hbox{and $\UF$}, respectively, one concludes.
\qed \\

We regard every element in $Y$ as an element in $\ZZ^X$ as well, thanks to the pairing between $Y$ and $X$.

\begin{lem} \label{lem_Deltahat}
Let $i \in I$. The morphism $\Deltahat$ sends every element $f$ in $R^\ZZ$ to the element $f \circ \check \alpha_i$ in $R^X$.
\end{lem}

\proof
The morphism $\DeltaF$ sends by definition $H$ to $\check \alpha_i$. Recall besides that we regard $Y$, and in particular $\ZZ$, as $\ZZ$-submodules of $R^X$ and $R^\ZZ$, respectively, in a such a way that $H$ is identified with the ring morphism from $\ZZ$ to $R$. Remark then that the morphism $\DeltaF$ sends $f$ to $f \circ \check \alpha_i$, for every $f$ lying in the \hbox{$R$-subalgebra $R[H]$} of $R^\ZZ$ generated by $H$. The image of $R[H]$ in $\Uhat{\slt,\psi_i}$ being dense in the image of $R^\ZZ$ (see point 3 of proposition \ref{prop_Uhatconv}), one concludes.
\qed





\subsection{Cartan involution}
We denote by $\omegaF$ is the involutive $R$-algebra automorphism of $\UF$ which sends $X_i^\pm$ ($i \in I$) and $\check \mu \in Y$ to $-X_i^\mp$ and $- \check \mu$, respectively. Let $\psi \in \Colo[I]{R}$, we denote by $\psi^\ast$ the $I$-colouring of $\BB$ with values in $R$ defined by the following:
$$ {(\psi^\ast)}^\pm_i \ := \ \psi_i^\mp \quad \quad (i \in I) \, . $$

\begin{defi-prop} \label{defi-prop_Cartanhat}
We denote by $\omegahat$ the unique continuous $R$-algebra isomorphism from $\Uhat{\slt,\psi}$ to $\Uhat{\glie,\psi^\ast}$ which makes the following diagram commute:
\begin{equation*}
\shorthandoff{;:!?} \xymatrix @C=5pc @R=1.5pc {
\UF \ar[r]^-{\omegaF \ } \ar[d]_-{\pihat{\glie,\psi} \ } & \, \UF \ar[d]^-{\ \pihat{\glie,\psi^\ast}} \\
\Uhat{\glie,\psi} \ar[r]_-{\omegahat} & \, \Uhat{\glie,\psi^\ast} \, .
} \end{equation*}
\end{defi-prop}

\section{Classical and quantum realisations}
\subsection{Standard quantum algebras}




Let $k,k' \in \NN$ and let $i \in I$. The following elements are defined in $\KK[q^{\pm 1}]$:
\begin{IEEEeqnarray*}{rClCrCl}
q_i & \ := \ & q^{d_i} \, , & \quad \quad \quad & {[k]}_{q_i} & \ := \ & \frac{q_i^k - q_i^{-k}}{q_i - q_i^{-1}} \ , \\
{[k]}_{q_i} ! & \ := \ & \prod_{k'=0}^k \, {[k']}_{q_i} \ , & \quad \quad \quad & {k \brack k'}_{q_i} & \ := \ & \frac{[k+k']_{q_i} !}{[k]_{q_i} ! \ [k-k']_{q_i} !} \ \cdot
\end{IEEEeqnarray*}

The \emph{standard (rational) quantum algebra} $\Uq(\glie)$ is the $\KK(q)$-algebra generated by the elements $X^-_i, X^+_i, K_{\check \mu}$ ($i \in I, \check \mu \in Y$), and subject to the following relations:
\begin{equation*} \label{eq_KM_nondeformableq} \left\{ {\renewcommand{\arraystretch}{1.3} \begin{array}{rclr}
K_0 & = & 1 \, & \\
K_{\check \mu} \! \; K_{\check \mu'} & = & K_{\check \mu + \check \mu'} & (\check \mu, \check \mu' \in Y) \, , \\
K_{\check \mu} \! \: X^{\pm}_i & = & q^{\pm \langle {\check \mu} , \alpha_i \rangle} \! \; X_i^{\pm} \! \; K_{\check \mu} & (i \in I, \ \check \mu \in Y) \, , \\
\big[ X^+_i, X^-_j \big] & = & 0 & \quad \quad (i,j \in I, \ i \neq j) \, .
\end{array}} \right.
\end{equation*}
$$ \left\{ {\renewcommand{\arraystretch}{2.5} \begin{array}{lr}
\displaystyle{ \big[ X^+_i, X^-_i \big]} = \displaystyle{ \frac{K_i - K_i^{-1}}{q_i - q_i^{-1}}} \, , \quad \quad \text{where } \, K_i := K_{d_i \check \alpha_i} & (i \in I) \, , \\
\displaystyle{ \sum_{k+k' = 1 - C_{ij}} (-1)^k \, {1 - C_{ij} \brack k}_{q_i} \, {(X_i^{\pm})}^k \; \! X^{\pm}_j \; \!  {(X_i^{\pm})}^{k'}} = \, \displaystyle 0 & \quad \quad (i,j \in I, \ i \neq j) \, .
\end{array}} \right. $$


\subsection{Classical and quantum colourings}
The \emph{classical colouring} $\psicl$ is the colouring of $\BB$ with values in $\KK$ defined by the following:
$$ {\psi_{\! \! \: \mathrm{cl}}^\pm} (n,k) \ := \ k \quad \quad \quad (n,k \in \NNI, \ k \leq n) \, . $$
We denote by $\psiclI$ the $I$-colouring of $\BB$ with values in $\KK$ defined by the following:
$$ {(\psiclI)}_i \ := \ \psicl \quad \quad \quad (i \in I) \, . $$
Note that $\psiclI$ is admissible over $\KK$. \\

Let $n \in \NN$. The $(n+1)$-dimensional irreducible representation (of type 1) $\Lq(n)$ of the quantum algebra $\Uq(\slt)$ can be defined as the $\CC(q)$-vector space ${\KK(q) \BBv_n}$, where the actions of $X^\pm$, $K$ ($= {K_{\! \! \: H}}$) are given by the following, for $p \in \NN_{\leq n}$:
$$ \left\{ {\renewcommand{\arraystretch}{0.75} \begin{array}{rcl}
H . \! \; b_{n,p} & := & \ (n- 2p) \, b_{n,p} \, , \\
&& \\
X^- . \! \; b_{n,p} & := & \left\{ {\renewcommand{\arraystretch}{1.2} \begin{array}{cl}
{[p+1]}_q \, b_{n,p+1} & \quad \text{if $p \neq n$,} \\
0 & \quad \text{else,}
\end{array}} \right. \\
&& \\
X^+ . \! \; b_{n,p} & := &  \left\{ {\renewcommand{\arraystretch}{1.2} \begin{array}{cl}
{[n-p+1]}_q \, b_{n,p-1} & \quad \text{if $p \neq 0$,} \\
0 & \quad \text{else.}
\end{array}} \right.
\end{array}} \right.
$$

The \emph{quantum colouring} $\psiq$ is the \hbox{colouring} of $\BB$ with values in $\KK[q^{\pm 1}]$, defined by the following:
$$ \psiq^\pm (n,k) \ := \ {[k]}_q \quad \quad \quad (n,k \in \NNI, \ k \leq n) \, . $$
We besides denote by $\psiclI$ the $I$-colouring of $\BB$ with values in $\KK[q^{\pm 1}]$ defined by the following:
$$ {(\psiqI)}_i \ := \ \psi_{q_i} \quad \quad (i \in I) \, . $$
Note that $\psiqI$ is admissible over $\KK(q)$.



\subsection{Theorem}
We regard in the following theorem $\Uc$ as a $\glieF$-pointed $\KK$-algebra, endowed with the Tannakian topology generated by the integrable representations of $\glie$, and endowed with the projection map from $\FUc$ to $\Uc$. \\

We besides denote by $K_{\check \mu}$ ($\check \mu \in Y$) the function in ${\KK(q)^X \! \subset \Uhat[\! \! \; \KK(q)]{\glie,\psiqI}}$ whose value at $\lambda \in X$ is $q^{\langle \check \mu, \lambda \rangle}$.

\begin{thm} \label{thm_exhat}
\begin{enumerate}
\item[1)]
\begin{itemize}
\item[\textbullet] The coloured Kac-Moody algebra $\Uhat[\! \! \; \KK]{\glie,\psiclI}$ contains $\Uc$ as a dense $\glieF$-pointed $\KK$-subalgebra.
\item[\textbullet] The action of $\Uc$ on every integrable representation can be uniquely extended to an integrable and discrete action of $\Uhat[\! \! \; \KK]{\glie,\psiclI}$.
\end{itemize}
\item[2)]
\begin{itemize}
\item[\textbullet] The coloured Kac-Moody algebra $\Uhat[\! \! \; \KK(q)]{\glie,\psiqI}$ contains $\Uq(\glie)$ as a dense $\KK(q)$-subalgebra, in such a way that the elements $X_i^\pm, K_{\check \mu}$ ($i \in I, \check \mu \in Y$) of the two algebras are identified.
\item[\textbullet] The action of $\Uq(\glie)$ on every integrable representation can be uniquely extended to an integrable and discrete action of $\Uhat[\! \! \; \KK]{\glie,\psiclI}$.
\end{itemize}
\end{enumerate}
\end{thm}

\proof
Let $V$ be an integrable representation of $\glie$ and let $i \in I$. Remark that the pull-back of $V$ via $\Deltac$ is an integrable representation of $\slt$ and is then a direct sum of representations $\Lc{n}$ with $n \in \NN$ (see \hbox{subsection \ref{subs_knownresultshat}}). Remark besides that for every $n \in \NN$, the representation $\Lc{n}$, when viewed as a representation of $\Uc[\sltF]$, is by definition of the classical colouring $\psicl$ equal to the representation $\Lc{n,\psicl}$ (see subsection \ref{subs_colouringshat}). Hence, the \hbox{representation $V$}, when viewed as a representation of $\glieF$, is $\psiclI$-integrable. \hbox{Point 1} of the theorem follows. The proof in the quantum case is similar.
\qed

\section{Main results for coloured Kac-Moody algebras of rank 1}

\subsection{$H$-triviality}
\begin{thm} \label{thm_UhatoneHtrivial}
Let $\psi_1,\psi_2$ be two admissible colourings of $\BB$ with values in $R$. There exists a $R$-algebra isomorphism from $\Uhat{\slt,\psi_1}$ to $\Uhat{\slt,\psi_2}$ which fixes the \hbox{elements $X^-, H$} and such that the pull-back of $\Lrep{n,\psi_2}$ via $f$ is isomorphic \hbox{to $\Lrep{n,\psi_1}$}, for every $n \in \NN$.
\end{thm}

\proof
Let $\psi, \psi' \in \Colo{R}$ two congruent colourings which divide each other in the ring $\Colo{R}$. We know that the representations $\Lrep{n,\psi}$ and $\Lrep{n,\psi'}$ are isomorphic: see proposition \ref{prop_congruentAonehat}. We also know that $\Uhat{\slt,\psi}$ and $\Uhat{\slt,\psi'}$ are isomorphic $\sltF$-pointed algebras: see again proposition \ref{prop_congruentAonehat}. Therefore, and since the colourings $\psi_1, \psi_2$ are admissible, one can without loss of generality assume that $\psi_1^- = \psi^-_2 = 1$. \\

Since $\psi_2$ is invertible, there exists a sequence ${(M_a)}_{a \in \NN}$ with values in $R^\ZZ$ such that $M[1]$ is solution of the equation $\eqhat{\psi_2,\psi_1}{0}$: see proposition \ref{prop_eqhatexist}. Denote then by $g$ the $R$-algebra morphism from $\UFslt$ \hbox{to $\Uhat{\slt,\psi_2}$} defined by the following:
$$ \left\{ {\renewcommand{\arraystretch}{1.2} \begin{array}{lcl}
g(X^-) & := & X^- \, , \\
g(H) & := & H \, , \\
g(X^+) & := & \sum_{a \in \NN} {(X^-)}^a M_a (H) {(X^+)}^{a+1} \, .
\end{array}} \right. $$
Note that the series $\sum_{a \in \NN} {(X^-)}^a M_a (H) {(X^+)}^{a+1}$ converges in $\Uhat{\slt,\psi_2}$: see point 1 of proposition \ref{prop_Uhatconv}. Remark that for every $n,p \in \NN$ such that $0 < p \leq n$, the following two equalities hold in the representations $\Lrep{n,\psi_1}$ and $\Lrep{n,\psi_2}$, respectively:
\begin{eqnarray*}
X^+ . \; \! b_{n,p} & = & [\psi_1] (n,p) \ b_{n,p-1} \, , \\
g(X^+) . \; \! b_{n,p} & = & \sum_{0 \leq a \leq p-1} \, M_{a} (n-2p+2a+2) \ \frac{[\psi_2](n,p) !}{[\psi_2](n,p-a-1) !} \ b_{n,p-1} \, .
\end{eqnarray*}
Since $g$ fixes $X^-, H$ and since the column matrix $M[1]$ is a solution of the equation $\eqhat{\psi_2,\psi_1}{0}$, the pull-back of $\Lrep{n,\psi_2}$ via $g$ is in consequence equal to the representation $\Lrep{n,\psi_1}$ of $\UFslt$, for every $n \in \NN$. In particular, the pull-back via $g$ of every $\psi_2$-integrable representation of $\Uhat{\slt,\psi_2}$ is a \hbox{$\psi_1$-integrable} representation of $\UFslt$, and the map $g$ is thus continuous, where $\UFslt$ is endowed with the $\psi_1$-integrable topology (see point 5 of proposition \ref{prop_Tantopo}). By definition of $\Uhat{\slt,\psi_1}$, and since $\Uhat{\slt,\psi_2}$ is separated complete, there exists a $R$-algebra morphism $f$ from the former to the latter which makes the following diagram commutes:
$$ \shorthandoff{;:!?} \xymatrix @C=5pc @R=1.5pc {
\UFslt \ar[r]^-{g} \ar[d]_-{\pihat{\slt,\psi_1} \ } & \, \Uhat{\slt,\psi_2} \, . \\
\Uhat{\slt,\psi_1} \ar[ru]_-{f} &
} $$
Remark then that $f$, as $g$, fixes $X^-,H$ and that the pull-back of $\Lrep{n,\psi_2}$ \hbox{via $f$} is equal to the representation $\Lrep{n,\psi_1}$ of $\Uhat{\slt,\psi_1}$, for every $n \in \NN$. \\

By inverting the roles of $\psi_1$ and $\psi_2$, there also exists a $R$-algebra morphism $f'$ from $\Uhat{\slt,\psi_2}$ to $\Uhat{\slt,\psi_2}$, which fixes $X^-,H$ and such that the pull-back of $\Lrep{n,\psi_1}$ \hbox{via $f'$} is equal to the representation $\Lrep{n,\psi_2}$ of $\Uhat{\slt,\psi_2}$, for every $n \in \NN$. Hence, the actions of $(f' \circ f)(x)$ and $x$ on $\Lrep{n,\psi_1}$ are equal for every $x \in \Uhat{\slt,\psi_1}$ and for every $n \in \NN$. As a consequence, the map $f' \circ f$ is equal to the identify map on $\Uhat{\slt,\psi_1}$ (see remark \ref{rem_Uhatperfect}). One proves in the same way that $f \circ f'$ is the identity map on $\Uhat{\slt,\psi_2}$.
\qed

\subsection{Integrable representations}
Let $\psi$ be a colouring of $\BB$ with values in $R$. We denote by $\Cinthat{\slt,\psi}$ the category of integrable and discrete representations of $\Uhat{\slt,\psi}$ which are free when viewed as $R$-modules.

\vsp

\begin{thm} \label{thm_intrephat}
Let $\psi$ be an admissible colouring of $\BB$ with values \hbox{in $R$}. We assume that $R$ is a principal ideal domain.
\begin{enumerate}
\item[1)] For every $n \in \NN$, the representation $\Lrep{n,\psi}$ of $\Uhat{\slt,\psi}$ is topologically generated by a vector $v_n$ and subject to the following relations:
$$ H . \! \: v_n \ = \ n \; \! v_n \, , \quad \quad X^+ . \! \: v_n \ = \ 0 \, , \quad \quad (X^-)^{n+1} . \! \: v \ = \ 0 \, . $$
\item[2)] The representations $\Lrep{n,\psi}$ ($n \in \NN$) are pairwise distinct and are the unique, up to isomorphism, indecomposable representations in $\Cinthat{\slt,\psi}$.
\item[3)] A representation of $\Uhat{\slt,\psi}$ is in the category $\Cinthat{\slt,\psi}$ if and only if it is $\psi$-integrable and discrete.
\item[4)] Let $n,n' \in \NN$. Every $\UFslt$-morphism from $\Lrep{n,\psi}$ to $\Lrep{n',\psi}$ is equal to a scalar multiple of the identity map if $n = n'$ and is zero else.
\end{enumerate}
\end{thm}

\proof
Let us first prove points 2 and 3 of the theorem. Since $\psi$ and the classical colouring $\psicl$ are invertible over $R$, one can assume without loss of generality, in view of theorem \ref{thm_UhatoneHtrivial}, \hbox{that $\psi = \psicl$}. We regard $\UR{\slt}$ as a $\sltF$-pointed \hbox{$R$-algebra} in a natural way. The relation $[X^-,X^+] = H$ being satisfied in all the representations $\Lrep{n,\psicl}$ ($n \in \NN$), it is also \hbox{in $\Uhat{\slt,\psicl}$} (see \hbox{remark \ref{rem_Uhatperfect}}). As a consequence, there exists a $\sltF$-pointed $R$-algebra morphism $\iota$ from $\UR{\slt}$ to $\Uhat{\slt,\psicl}$ (recall the presentation of $\slt$: see subsection \ref{subs_KMonehat}). Let us then regard the representations $\Lrep{n,\psicl}$ ($n \in \NN$) as representations of $\UR{\slt}$ as well, thanks to the pull-back via $\iota$, and let us consider the following point.

\begin{enumerate}
\item[3')] Every free over $R$, integrable representation of $\UR{\slt}$ is isomorphic to a direct sum of representations $\Lrep{n,\psicl}$ with $n \in \NN$.
\end{enumerate}

The image of $\iota$ being dense in $\Uhat{\slt,\psicl}$, remark that point 3' implies point 3. \\

Let us prove point 2'. Let $V$ be an integrable, free over $R$, representation \hbox{of $\Uhat{\slt,\psicl}$}. In view of the relations of $\slt$, remark that the \hbox{representation $V$} is a sum of integrable representations of $\UR{\slt}$ which are of finite type \hbox{over $R$}. The \hbox{$R$-module} $V$ being free and the ring $R$ being by assumption a principal domain, one then can without loss of generality assume that $V$ is of finite type \hbox{over $R$}. Denote by ${V_{\! \! \; n}}$ ($n \in \NN$) the \hbox{$R$-submodule} of $V$ formed by the vectors of \hbox{weight $n$} on which $X^+$ acts as zero. Let $n \in \NN$ such that ${V_{\! \! \; n}}$ is nonzero. Since $V$ is a free \hbox{$R$-module} of finite type, and since the ring $R$ is by assumption a principal ideal domain, the \hbox{\hbox{$R$-module} ${V_{\! \! \; n}}$} is also free. Pick then a basis \hbox{$(v_{n,1}, \dots, v_{n,l(n)})$ ($l(n) \in \NNI$)} of it and denote by ${f_{\! \! \; n,k}}$ ($k \in \NNI_{\leq l(n)}$) the \hbox{$R$-linear} map from $\Lrep{n,\psicl}$ to ${V_{\! \! \; n}}$ which sends \hbox{$b_{n,p}$ ($p \in \NN_{\leq p}$)} \hbox{to $(X^-)^k.v_{n,k}$}. Since $X^+$ acts as zero on ${V_{\! \! \; n}}$, the \hbox{map ${f_{\! \! \; n,k}}$} ($k \in \NNI_{\leq l}$) is \hbox{a $\UR{\slt}$-morphism}. Denote then by ${V_\mathrm{ss}}$ the representation of $\UR{\slt}$ defined by the following:
$$ {V_\mathrm{ss}} \ := \ \bigoplus_{n \in \NN, \, {V_{\! \! \; n}} \neq 0} \Lrep{n,\psicl}^{\oplus \! \: l(n)} \, . $$
Denote by $f$ the $\UR{\slt}$-morphism from ${V_\mathrm{ss}}$ to $V$ defined as the direct sum of the morphisms ${f_{\! \! \; n,k}}$ ($n \in \NN$ such that ${V_{\! \! \; n}} \neq 0$, $k \in \NNI_{\leq l(n)}$). \\

Denote by $V'$ the $K(R)$-vector space ${V \! \otimes_R K(R)}$ and let us regard it as a representation of $\UR[\! K(R)]{\slt}$. It is a direct sum of representations $\Lrep[\! \! \; K(R)]{n}$ with $n \in \NN$ (see subsection \ref{subs_KMonehat}). Hence, the $\UR[\! K(R)]{\slt}$-morphism ${f \! \otimes_R \! K(R)}$ from ${V_1 \! \otimes_R \! K(R)}$ to $V'$ is bijective, and ${v \! \otimes \! 1}$ is a $\KK$-linear combination of the vectors ${v_{n,k} \! \otimes \! 1}$ ($n \in \NN$ such that ${V_{\! \! \; n}} \neq 0$, $k \in \NNI_{\leq l(n)}$) for every $v \in V$. The $\UR{\slt}$-morphism $f$ is in consequence also bijective. In other words, point 2' holds. \\

Remark that the representations $\Lrep{n,\psicl}$ ($n \in \NN$) are integrable and free \hbox{over $R$}. Since $\psicl$ is non-degenerate, we besides know from lemma \ref{lem_indecompAonehat} that they are also indecomposable. Let \hbox{$n_1, n_2 \in \NN$} such that $n_1 \neq n_2$. Remark that the representations $\Lrep{n_1,\psicl}$ and $\Lrep{n_2,\psicl}$ are not isomorphic since the \hbox{$R$-modules} $\Lrep{n_1,\psicl}$ and $\Lrep{n_2,\psi}$ are not. In view of point 2, there exists, up to isomorphism, no other representation of $\Uhat{\slt,\psicl}$ which are free \hbox{over $R$}, indecomposable and integrable.
\qed

\subsection{Non $\glieF$-triviality}
\begin{thm} \label{thm_congruentUonehat}
Let $\psi_1,\psi_2$ be two admissible colourings of $\BB$ with values in $R$. The following two assertions are equivalent.
\begin{enumerate}
\item[(i)] The colourings $\psi_1$ and $\psi_2$ are congruent.
\item[(ii)] Tthe $\sltF$-pointed algebras $\Uck{\slt,\psi_1}$ and $\Uck{\slt,\psi_2}$ are isomorphic.
\end{enumerate}
\end{thm}

\proof
Consequence of proposition \ref{prop_congruentAonehat} and of theorem \ref{thm_intrephat}.
\qed

\section[Representations of coloured Kac-Moody algebras]{Representations of coloured Kac-Moody \\ algebras}

\subsection{Integrable representations}
In this subsection, $\psi$ designates an $I$-colouring of $\BB$ with values in $R$. We denote by $\Cint{\glie,\psi}$ the category of $\psi$-integrable representations of $\UF$.


\begin{prop} \label{prop_Cinthat}
Let $V$ be a discrete representation of $\Uhat{\glie,\psi}$. We assume that $R$ is a principal ideal domain. If $\psi$ is admissible over $R$, then the following two assertions are equivalent.
\begin{enumerate}
\item[(i)] The representation $V$ is in $\Cint{\glie,\psi}$.
\item[(ii)] The representation $V$ is integrable and free over $R$.
\end{enumerate}
\end{prop}

\proof
Corollary of theorem \ref{thm_intrephat}.
\qed \\


\subsection{Change of rings}
Let $\varphi$ be a $\KK$-algebra morphism from $R$ to another $\KK$-algebra $R'$, which is an integral domain. We identify the $R'$-algebras ${\UF \! \otimes_\varphi \! R'}$ and $\UF[\! R']$. Let $V$ be a representation of $\UF$, we regard ${V \! \otimes_\varphi \! R'}$ as a representation of $\UF[\! R']$, thanks to the previous identification. \\

Let $\psi \in \Colo[I]{R}$, we denote by ${\psi \otimes_\varphi \! R'}$ the \hbox{$I$-colouring} of the crystal $\BB$ with values in $R'$, defined by the following:
$$ {\psi \otimes_\varphi \! R'} \ := \ \varphi \circ \psi \, . $$

\vsp

\begin{rem} \label{rem_changeofrings}
\begin{enumerate}
\item[1)] If $V$ admits a character, then ${V \! \otimes_\varphi \! R'}$ admits a character equal to the character of $V$.
\item[2)] if $V$ is $\psi$-integrable, then ${V \! \otimes_\varphi \! R'}$ is $({\psi \otimes_\varphi \! R'})$-integrable.
\end{enumerate}
\end{rem}


\subsection{Specialisation and $h$-admissible expansion}
We suppose in this subsection that $R$ is a $\KK$-algebra formed by rational functions over the field $\KK$ without pole at $\beta$:
$$ R \ \subset \ \Big \{ f/g \, ; \ f,g \in \KK[u], \, \, g(\beta) \neq 0 \Big \} \, . $$
Denote by $\varphi$ the $\KK$-algebra morphism from $R$ to $\KK$ which sends $u$ to $\beta$.

\begin{defi}
\begin{itemize}
\item[\textbullet] Let $\psi \in \Colo[I]{R}$. The $I$-colouring ${\psi \otimes_\varphi \! \KK}$ is called the specialisation of $\psi$ at ${u \! = \! \beta}$ and is denoted by $\psi_{\vert u = \beta}$.
\item[\textbullet] Let $V$ be a representation of $\UF$. The representation ${V \! \otimes_\varphi \! \KK}$ is called the specialisation of $V$ at ${u \! = \! \beta}$ and is denoted by $V_{\vert u = \beta}$.
\end{itemize}
\end{defi}



\vsp
\vsp

The following proposition establishes that every $I$-colouring of $\BB$ with values in a field admits an \emph{$h$-admissible expansion}.

\begin{prop} \label{prop_hadmissible}
For every \hbox{$I$-colouring} $\psi$ of $\BB$ with values in $\KK$, there exists a $h$-admissible \hbox{$I$-colouring} $\psi_h$ of $\BB$ with values in $\KK[h]$ such that $\psi$ is equal to the specialisation of $\psi_h$ at ${h \! = \! 1}$.
\end{prop}

\proof
Let $\psi \in \Colo{\, \KK}$. Pick, inductively on $m \in \NN$, polynomials ${P^-_{\! m}}, {P^+_{\! m}}$ \hbox{in $\KK[u,v]$}, such that $P^\pm_0(u,v) = v$, such that ${P^\pm_{\! m}}(-u-2,v) = - {P^\mp_{\! m}}(u,-v)$ and such that the following holds for every $n \in \NNI_{\leq m}$, $k \in \NNI_{\leq n+1}$:
$$ {P^\pm_{\! m}}(n,k) \ = \ \left\{ {\renewcommand{\arraystretch}{1.4} \begin{array}{cl}
\psi^\pm (m,k) \! \; - \, \sum_{m'=0}^{m-1} \! \; {P^\pm_{\! \! \! \; m'}} (m,k) & \quad \text{if } \, n=m, \ k \leq m, \\
0 & \quad \text{else.}
\end{array}} \right. $$
Denote now by $\psi_h$ be the colouring of $\BB$ with values in $\Kh$, defined by the following:
$$ (\psi_h)^\pm \ := \ \sum_{m \in \NN} {P^\pm_{\! m}} \; \! h^m \, . $$
By definition, the colouring $\psi_h$ has values in $\KK[h]$, its specialisation at $h=1$ is equal to $\psi$ and it is $h$-admissible.
\qed

\subsection{Forms and admissible ring extensions}
Let $R'$ be a $\KK$-subalgebra of $R$. We regard $\UF[\! R']$ as a \hbox{$R'$-subalgebra} of $\UF$.

\begin{defi} \label{defi_form}
Let $V$ be a representation of $\UF$. A $R'$-form of $V$ is a $R'$-submodule $V'$ of $V$ which satisfies the following two conditions.
\begin{enumerate}
\item[a)] The subset $V'$ is stable under the action of $\UF[\! R']$ on $V$.
\item[b)] The canonical $R$-linear map from ${V' \! \otimes_{R'} \! R}$ to $V$ is bijective.
\end{enumerate}
\end{defi}

We regard $R'$-forms of representations of $\UF$ as representations of $\UF[\! R']$, in view of condition a) of definition \ref{defi_form}. \\

We say that the ring extension $R' \subset R$ is admissible if the following condition is satisfied.

$$
\begin{tabular}{| p{22pc}}
If $V'$ is a $R'$-submodule of finite type of a free \hbox{$R$-module} $V$, then $V'$ is a free $R'$-module and the canonical $R$-linear map from ${V' \! \otimes_{R'} \! R}$ to $V$ is injective.
\end{tabular}
$$

\vsp

\begin{rem} \label{rem_admissiblepair}
The ring extension $R' \subset R$ is admissible if $R'$ is a principal ideal domain and if $R$ is contained in the fraction field of $R'$.
\end{rem}

\begin{lem} \label{lem_form}
Let $V$ be a representation of $\UF$ and let $V'$ be a $R'$-form \hbox{of $V$}. We assume that the ring extension $R' \subset R$ is admissible. If $V$ admits a character with finite values, then $V'$ admits a character, equal to the character of $V$.
\end{lem}

The following lemma gives another example of admissible pair of rings. We regard both $\KK(h)$ and $\Kh$ as $\KK$-subalgebras of $\KK((h))$.

\begin{lem} \label{lem_admissiblepair}
Suppose that the following $\KK$-algebra inclusions hold:
$$ \KK[h] \ \subset \ R' \ \subset \ \KK(h) \cap \Kh \, . $$
The pair of rings $(R', \Kh)$ is then admissible.
\end{lem}

\proof
Remark first that one can assume without loss of generality that $R'$ is equal to $\KK[h]$. Let us then prove that the ring $\Kh$ is flat over $\KK[h]$. To do so, it is sufficient to prove that for all finitely generated ideal $I$ of $\KK[h]$, the canonical map from $I \otimes_{\KK[h]} \Kh$ to $\Kh$ is injective. Since the ring $\KK[h]$ is a principal ideal domain, it is a consequence of the following fact: an element $x$ in $\KK[h]$ is a zero divisor in $\Kh$ only if $x$ is zero. \\

Let $V$ be a free $\Kh$-module and let $V'$ be a $\KK[h]$-submodule of finite type \hbox{of $V$}. Since $V$ is a free $\Kh$-module, $V$ is a torsion-free $\KK[h]$-module. \hbox{Since $\KK[h]$} is a principal ideal domain and since $V'$ is of finite type, the $\KK[h]$-module $V'$ is thus free. It remains to prove that the canonical $\Kh$-linear map $f$ from $V' \otimes_{\KK[h]} \Kh$ to $V$ is injective. Since $\Kh$ is flat over the ring $\KK[h]$, the \hbox{$\Kh$-linear} \hbox{map $\iota$} from $V' \otimes_{\KK[h]} \Kh$ to $V \otimes_{\KK[h]} \Kh$, induced by the inclusion map from $V'$ \hbox{to $V$}, is injective. Recall besides that $\Kh$ is a (projective) limit of the following diagram of $\KK[h]$-modules:
$$ \cdots \ \twoheadrightarrow \ \KK[h] / (h^{n+1}) \ \twoheadrightarrow \ \KK[h] / (h^n) \ \twoheadrightarrow \ \cdots \ \twoheadrightarrow \ \KK[h] / (h) \ \twoheadrightarrow \ (0) \quad \quad \quad (n \in \NN) \, . $$
By flatness again of $\Kh$ over the ring $\KK[h]$, the $(\KK[h])$-module $V \otimes_{\KK[h]} \Kh$, together with the projection maps $p_n$ from $V \otimes_{\KK[h]} \Kh$ to $V / (h^n V)$ ($n \in \NN$), is therefore a (projective) limit of the following diagram of $\KK[h]$-modules:
$$ \cdots \ \twoheadrightarrow \ V / (h^{n+1} V) \ \twoheadrightarrow \ V / (h^n V) \ \twoheadrightarrow \ \cdots \ \twoheadrightarrow \ V / (h V) \ \twoheadrightarrow \ (0) \quad \quad \quad (n \in \NN) \, . $$
Denote by $g$ the canonical $\KK[h]$-linear map from $V$ to $V \otimes_{\KK[h]} \Kh$. \hbox{Let $v \in V$}, let $x \in \Kh$ and remark that the following equalities hold:
$$ (p_n \circ g)(x \! \: v) \ = \ p_n (v \otimes x) \, , \quad \quad \forall \, n \in \NN \, . $$
The following diagram hence commutes:
$$ \shorthandoff{;:!?} \xymatrix @C=3pc @R=1.5pc {
V '\otimes_{\KK[h]} \Kh \ \ar@{^{(}->}[r]^-\iota \ar[d]_-f & \ V \otimes_{\KK[h]} \Kh \\
V \ \ar[ru]_-{g} &
} $$
As a consequence, the map $f$ is injective.
\qed

\subsection{Forms of $\psi$-integrable representations}
We still designate by $R'$ a $\KK$-subalgebra of $R$ in this subsection, and we recall that we regard $\UF[\! R']$ as a \hbox{$R'$-subalgebra} of $\UF$. Let $\psi$ be an $I$-colouring \hbox{of $\BB$} with values in $R'$.

\begin{lem} \label{lem_restrictionslt}
We assume that $\glie = \slt$. There is a unique continuous $R'$-algebra morphism $\widehat \iota$ from $\Uhat[\! R']{\slt,\psi}$ \hbox{to $\Uhat{\slt,\psi}$} which makes the following diagram commute, where $\iota$ designates the inclusion map from $\UFslt[\! R']$ to $\UFslt$:
$$ \shorthandoff{;:!?} \xymatrix @C=3pc @R=1.5pc {
\UFslt[\! R'] \ \ar[d]_-{\pihat[\! \! \: R']{\slt,\psi} \, } \ar@{^{(}->}[r]^-{\iota} & \, \UFslt \ar[d]^-{\, \pihat{\slt,\psi}} \\
\Uhat[\! R']{\slt,\psi} \, \ar[r]_-{\widehat \iota} & \, \Uhat{\slt,\psi} \, .
} $$
Moreover, the map $\widehat \iota$ is injective and the topology of $\Uhat[\! R']{\slt,\psi}$ coincide with the topology induced by $\widehat \iota$.
\end{lem}

\proof
Consequence of the definition of $\Uhat{\slt,\psi}$.
\qed \\

When $\glie = \slt$, we regard from now $\Uhat[\! R']{\slt,\psi}$ as a topological $R'$-subalgebra \hbox{of $\Uhat{\slt,\psi}$}, thanks to lemma \ref{lem_restrictionslt}.

\vsp

\begin{prop} \label{prop_form}
We assume that $\psi$ is admissible over $R'$ and we assume that the ring extension $R' \subset R$ is admissible. Every highest weight and $\psi$-integrable representation of $\UF$ admits a $\psi$-integrable $R'$-form.
\end{prop}

\proof
Let $\langle I \rangle$ be the free monoid generated by the set $I$. Denote by $\bar X^-$ the monoid morphism from $\langle I \rangle$ to $\Uhat{\glie,\psi}$ which sends $i$ to $X^-_i$ ($i \in I$). Let $V$ be a highest weight and \hbox{$\psi$-integrable} representation of $\UF$. Let $v$ be a highest weight vector of $V$. Denote by $V'$ the $R'$-submodule of $V$ which is spanned by the subset $\{ \bar X^-(b) . v \; \! ; \, b \in \langle I \rangle \}$. Remark that the subset $V'$ is by definition stable under the action of $X_i^-$ on $V$ for every $i \in I$ and under the action of $Y$ since $v$ is a weight vector. \\

The \hbox{\hbox{$I$-colouring}} $\psi$ being by assumption admissible over $R'$, there exists, for every $i \in I$, a sequence ${(S_{i,a})}_{a \in \NN}$ with values in $(R')^\ZZ$ such that the following equality holds in $\Uhat[\! R']{\slt,\psi_i}$ (see propositions \ref{prop_eqhatspan} and \ref{prop_eqhatexist}):
$$ X^+ X^- \ = \ \sum_{a \in \NN} \, {(X^-)}^a \! \; S_{i,a} \, {(X^+)}^a \, . $$
Fix now $i \in I$. Since $\Uhat[\! R']{\slt,\psi_i}$ is a topological $R'$-subalgebra of $\Uhat{\slt,\psi_i}$ (see lemma \ref{lem_restrictionslt}), the same equality holds in $\Uhat{\slt,\psi_i}$. Recall besides that one can \hbox{regard $V$} as a discrete representation of $\Uhat{\glie,\psi}$. Thanks to the $R$-algebra morphism $\Deltahat$ from $\Uhat{\slt,\psi_i}$ to $\Uhat{\glie,\psi}$ (see \hbox{proposition \ref{prop_Deltahat}}) and since the image of $(R')^\ZZ$ by $\Deltahat$ is included in $(R')^X$ (see lemma \ref{lem_Deltahat}), there hence exists a sequence ${(S^X_{i,a})}_{a \in \NN}$ with values in $(R')^X$ such that the following equality holds in $\Uhat{\glie,\psi}$:
$$ X_i^+ X_i^- \ = \ \sum_{a \in \NN} \, {(X_i^-)}^a \! \; S^X_{i,a} \, {(X_i^+)}^a \, . $$
As a consequence, and by definition of a discrete representation, there exists, for every $b \in \langle I \rangle$, an integer $a(b)$ such that the following equality holds in $V$:
$$ (X_i^+ X_i^-) \! \: . \! \; (X^-(b) . v) \ = \ \sum_{0 \leq a \leq a(b)} \Big[ {(X_i^-)}^a \! \; S^X_{i,a} \, {(X_i^+)}^a \Big] . \! \; (X^-(b) . v) \, . $$
This being true for every $i \in I$, since the subset $V'$ is besides stable under the action of $X_i^-$ on $V$ for every $i \in I$ and under the action of $Y$, one can prove by induction on the length of $b \in \langle I \rangle$ that the subset $V'$ is hence stable under the action of $X_i^+$ on $V$ for every $i \in I$. The subset $V'$ is in consequence stable under the action of $\UF[\! R']$. \\

Let $\lambda \in X$. Remark that the $R'$-submodule $V'_\lambda$ is finitely generated, in view of the definition of $V'$. The ring extension $R' \subset R$ being by assumption admissible, the canonical $R$-linear map $f_\lambda$ from $V'_\lambda \otimes_{R'} R$ to $V_\lambda$ is then injective. Besides, the subset $V'$ being stable under the action of $\UF[\! R']$, the \hbox{$R$-module} spanned by the subset $\{ \bar X^-(b) . v \; \! ; \, b \in \langle I \rangle \}$ is in particular stable under the action $\UF$, and therefore equal to $V$ since the latter is generated as a representation of $\UF$ by the vector $v$. As a consequence, the map $f_\lambda$ is also surjective, and thus bijective. This being true for every $\lambda \in X$, the canonical map from $V' \otimes_{R'} R$ to $V$ is then bijective.
\qed

\section{Mains consequences of the GQE conjecture}
We recall that the GQE conjecture predicts the existence of formal deformations along GQE algebras of representations in $\Ointc$: see conjecture I.\ref{conj_GQE}. In the following of this section, \textbf{we assume that the statement of the GQE conjecture holds}. The author believes that the GQE conjecture is true, a proof of it is currently in preparation and will be published in a forthcoming paper.

\subsection{Resolution of the deformation problem}

\begin{thm} \label{thm_existence}
Let $\psi$ be an admissible $I$-colouring of $\BB$ with values in $R$. We assume that $R$ is a principal ideal domain. For every representation $V$ in the category $\Ointc$, there exists a \hbox{$\psi$-integrable} representation of $\UF$ whose character is equal to the character of $V$.
\end{thm}

\proof
Let $V$ be a representation of $\glie$ in $\Ointc$. Denote by $\KK'$ the fraction \hbox{field $K(R)$} of $R$ and denote by $R'$ the $\KK'$-subalgebra of $\KK' \cch$ formed by the rational functions without pole at $0$ and $1$:
$$ R' \ := \ \Big \{ f/g \, ; \ f,g \in \KK'[h], \, \, g(0) \neq 0, \, \, g(1) \neq 0 \Big \} \, . $$
We know from \hbox{proposition \ref{prop_hadmissible}} that there exists a $h$-admissible \hbox{$I$-colouring} $\psi_h$ of $\BB$ with values in $\KK'[h]$ such that the evaluation of $\psi_h$ at $h=1$ is equal \hbox{to $\psi$}, viewed as an element in $\Colo[I]{\, \KK'}$. Remark that $\psi$ is admissible \hbox{over $R'$}. Besides, since the statement of the GQE conjecture holds by assumption, there exists an integrable \hbox{representation $\Vh$} of $\Uhat[\KK' \cch]{\glie,\psi_h}$ whose character is equal to the character \hbox{of $V$}. Since the \hbox{$I$-colouring} $\psi_h$ is admissible over $R'$ and since the ring extension $R' \subset \KK' \cch$ is admissible (see lemma \ref{lem_admissiblepair}), proposition \ref{prop_form} then implies that $\Vh$ admits a $\psi_h$-integrable $R'$-form $\Vhp$, whose character is equal to the character \hbox{of $V$}. \\

There exists, by specialisation of the representation $\Vhp$ at ${h \! = \! 1}$, a $\psi$-integrable representation $V'$ of $\UF[\KK']$ which has the same character than $V$. Since $\psi$ is by assumption admissible over $R$, and since the inclusion ring $R \subset \KK'$ is admissible (see remark \ref{rem_admissiblepair}), proposition \ref{prop_form} therefore implies \hbox{that $V'$} admits a $\psi$-integrable $R$-form $V$, whose character is equal to the character \hbox{of $V$}.
\qed

\subsection{The category $\Oint{\glie,\psi}$}
Let $\psi$ be an $I$-colouring of $\BB$ with values in $R$. We denote by $\Oint{\glie,\psi}$ the full subcategory of $\Cint{\glie,\psi}$ whose objects $V$ satisfy the following two conditions.
\begin{enumerate}
\item[\textbullet] The weight spaces of $V$ are of finite-rank over $R$.
\item[\textbullet] There exist finitely many elements $\lambda_1, \dots, \lambda_k \in X$ such that $\wt (V)$ are contained in $\bigcup_{s=1}^k (\lambda_s - \sum_{i \in I} \NN \: \! \alpha_i)$.
\end{enumerate}

\vsp
\vsp

\begin{thm} \label{thm_Uhatrepresentations}
Let $\psi$ be an admissible $I$-colouring of $\BB$ with values \hbox{in $R$}. We assume that $R$ is a principal ideal domain.
\begin{enumerate}
\item[1)] For every $\lambda \in \Xdom$, there is a unique, up to isomorphism, $\psi$-integrable representation $\Lrep{\lambda,\psi}$ whose character is equal to the character of $\Lc{\lambda}$.
\item[2)] Let $\lambda \in X$. The representation $\Lrep{\lambda,\psi}$ of $\Uhat{\glie,\psi}$ is topologically generated by a vector $v_{\lambda}$ and subject to the following relations:
$$ \check \mu . v_{\lambda} \, = \, \langle \lambda, \check \mu \rangle \; \! v_{\lambda} \, , \quad X^+_i . v_{\lambda} = 0 \, , \quad (X^-_i)^{\langle \lambda, \check \alpha_i \rangle +1} . v_{\lambda} \, = \, 0  \quad \ (i \in I, \ \check \mu \in Y) \, . $$
\item[3)] The representations $\Lrep{\lambda,\psi}$ ($\lambda \in \Xdom$) are, up to isomorphism, pairwise distinct and the unique indecomposable representations in $\Oint{\glie,\psi}$.
\item[4)] Let $\lambda,\lambda' \in \Xdom$. Every $\UF$-morphism from $\Lrep{\lambda,\psi}$ to $\Lrep{\lambda',\psi}$ is equal to a scalar multiple of the identity map if $\lambda = \lambda'$ and is zero else.
\item[5)] Every representation in the category $\Oint{\glie,\psi}$ is isomorphic to a direct sum of representations $\Lrep{\lambda,\psi}$ with $\lambda \in \Xdom$.
\end{enumerate}
\end{thm}

\proof
The proof is analog for some points to the one given in rank one (see theorem \ref{thm_intrephat}) and makes use of theorem \ref{thm_existence}. Besides, one has to prove that characters of $\psi$-integrable representations admit Weyl symmetries. One also has to define  (thanks in part to the Cartan involution $\omegahat$: see \hbox{definition \ref{defi-prop_Cartanhat}}) a contravariant endofunctor of the category $\Oint{\glie,\psi}$ whose action at the level of $R$-modules is given by (weighted) $R$-linear duals
\qed

\subsection{Non $\glieF$-triviality}
\begin{thm}
Let $\psi,\psi'$ be two admissible \hbox{$I$-colourings} of the crystal $\BB$ with values \hbox{in $R$}. The following two assertions are equivalent.
\begin{enumerate}
\item[(i)] The colourings $\psi_i$ and $\psi'_i$ are congruent for every $i \in I$.
\item[(ii)] The $\glieF$-pointed algebras $\Uck{\glie,\psi}$ and $\Uck{\glie,\psi'}$ are isomorphic.
\end{enumerate}
\end{thm}

\proof
Consequence of theorems \ref{thm_congruentUonehat} and \ref{thm_existence}.
\qed

\subsection{An application to GQE algebras}
\begin{thm} \label{thm_GQEalg}
Let $\psi$ be an $I$-colouring of $\BB$ with values in $\Kh$. The following two assertions are equivalent.
\begin{enumerate}
\item[(i)] The $I$-colouring $\psi$ is $h$-admissible.
\item[(ii)] The $\glieF$-pointed $\Kh$-algebra $\Uh$ is a formal deformation of $\Uc$.
\end{enumerate}
\end{thm}

\begin{proof}
We already have proved in part I that assertion (i) implies assertion (ii). Concerning the converse implication, one can prove, thanks to theorem \ref{thm_existence}, that assertion (ii) implies that the $\sltF$-pointed $\Kh$-algebra $\Uh[\slt,\psi_i]$ is a formal deformation of $\Uc[\slt]$ for every $i \in I$, and therefore that assertion (ii) implies assertion (i) in view of theorem I.\ref{thm_Uhonednd}.
\end{proof}

\section{Isogenic and Langlands Interpolations}
\subsection{Isogeny} \label{subs_isogeny}
Let $\Xdat'$ be another root datum whose associated Dynkin diagram is $I$. We denote by $\glie', X',Y', \alpha'_i, \check \alpha'_i$ ($i \in I$) the associated Kac-Moody algebra, root lattice, coroot lattice, simple roots and simple coroots, respectively.

\begin{defi}
An isogeny of root data from $\Xdat'$ to $\Xdat$ is a $\ZZ$-linear map $\xi$ from the root lattice $X$ \hbox{to the root lattice $X'$} which satisfies the following conditions.
\begin{enumerate}
\item[a)] The map $\xi$ and its transpose $\check \xi$ from $Y$ to $Y'$ are both injective.
\item[b)] There exists positive integers $\xi_i$ ($i \in I$) such that $\xi(\alpha'_i) = {\xi_i \: \! \alpha_i}$ and such that $\check \xi(\check \alpha_i) = {\xi_i \: \! \check \alpha'_i}$, for all $i \in I$.
\end{enumerate}
\end{defi}

Let $V$ be a weight representation of $\UF$. We denote by $\xiV$ the $R$-submodule of $V$ defined by the following:
$$ \xiV \ := \, \bigoplus_{\lambda \in \xi(X')} V_\lambda \, . $$
We denote by ${(X'_i)}^\pm$ ($i \in I$) the Chevalley generators of $\glie'$ and we denote again by $\langle \, \cdot \, , \cdot \, \rangle$ the perfect pairing between $Y'$ and $X'$.

\begin{lem} \label{lem_isogeny}
Let $V$ be a weight representation of $\UF$.
\begin{enumerate}
\item[1)] The subset $\xiV$ is stable under the action of ${(X^\pm_i)}^{\xi_i}$ for every $i \in I$.
\item[2)] There exists a unique weight action of $\UFp$ on $\xiV$ which satisfies the following for every $i \in I$ and for every $\check \mu' \in Y'$:
$$ {(X'_i)}^\pm . \! \: v \ = \ {(X^\pm_i)}^{\xi_i} . \! \: v \, , \quad \ \check \mu' . \! \: v \ = \ \langle \check \mu', \lambda' \rangle \! \; v \, , \quad \ \forall \, \lambda' \in X', \ \forall \, v \in {{(V)}_{\! \! \; \xi(\lambda')}} \, . $$
\end{enumerate}
\end{lem}

We regard from now $\xiV$ as a representation of $\UFp$, thanks to lemma \ref{lem_isogeny}.

\begin{rem} \label{rem_isogeny}
If $V$ admits a character, then $\xiV$ admits a character, which is equal to $\chi(V) \circ \xi$.
\end{rem}

Let $\psi \in \Colo[I]{R}$, let $i \in I$ and let $d \in \NN_{< \xi_i}$. We denote by $\xipsi$ the colouring \hbox{of $\BB$} with values in $R$, defined by the following:
$$ {\big[ \xipsi \big]}^{\! \! \; \pm} (n,k) \ := \ \prod_{k'=1}^{\xi_i} \! \; \psi_i^\pm \big( \xi_i \! \: n + 2 \! \: d \! \; , \, \xi_i (k-1) + k' + d \big) \, . $$

\begin{lem} \label{lem_isogenyint}
Let $\psi$ be an $I$-colouring of the crystal $\BB$ with values in $R$ and \hbox{let $V$} be a $\psi$-integrable representation of $\UF$. For every $i \in I$, the pull-back of $\xiV$ via $\DeltaFp$ is isomorphic to a direct sum of representations $\Lrep{n,\xipsi}$ \hbox{with $n \in \NN, d \in \NN_{< \xi_i}$}.
\end{lem}

\proof
See figure \eqref{Louloufig} in the introduction of this thesis.
\qed

\subsection{Isogenic specialisation} \label{subs_isogenyinter}
We designate in this subsection by $R$ the $\KK$-algebra formed by the rational functions without pole at $0$ and $1$:
$$ R \ := \ \Big \{ f/g \, ; \ f,g \in \KK[u], \, \, g(0) \neq 0, \, \, g(1) \neq 0 \Big \} \, . $$

Let $\psi$ and $\psi'$ be admissible $I$-colouring and $I'$-colouring of $\BB$ with values in $\KK$. We denote by $\psi_u$ the $I$-colouring of $\BB$ with values in $R$, defined by the following:
$$ {(\psi_u)}_i^\pm (n,k) \ := \ \left\{ {\renewcommand{\arraystretch}{1.4} \begin{array}{ll}
u \, \psi^\pm_i (n,k) \, + \, (1-u) \! \; {(\psi')}_i^\pm (n,k/ \xi_i) & \quad \ \text{if $k = 0$ mod $\xi$} \, , \\
u \, \psi^\pm_i (n,k) \, + \, (1-u) & \quad \ \text{else.}
\end{array}} \right. $$

\begin{lem} \label{lem_isogenyspecialisation}
Let $V$ be a $\psi_u$-integrable representation of $\UF$.
\begin{enumerate}
\item[1)] The specialisation of $V$ at ${u \! = \! 0}$ is a \hbox{$\psi$-integrable} representation of $\glieF$.
\item[2)] The specialisation of $\xiV$ at ${u \! = \! 1}$ is a $\psi'$-integrable representation of $\gliepF$.
\item[3)] The $I$-colouring $\psi_u$ is admissible over $R$.
\end{enumerate}
\end{lem}

\proof
Consequence of remark \ref{rem_changeofrings} and of lemma \ref{lem_isogenyint}.
\qed

\subsection{Coloured isogenic interpolation}
We assume in this subsection that the statement of the GQE conjecture holds. \\

We designate again in this subsection by $R$ the $\KK$-algebra formed by the rational functions without pole at $0$ and $1$:
$$ R \ := \ \Big \{ f/g \, ; \ f,g \in \KK[u], \, \, g(0) \neq 0, \, \, g(1) \neq 0 \Big \} \, . $$

\begin{thm} \label{thm_coloutedinter}
Let $\psi, \psi'$ be admissible $I$-colourings of $\BB$ with values \hbox{in $\KK$}. There exists an admissible $I$-colouring $\psi_u$ of $\BB$ with values \hbox{in $R$} such that, for every representation $V$ in $\Oint[\KK]{\glie,\psi}$, there exists a \hbox{representation $\VR$} in $\Oint{\glie,\psi_u}$ which satisfies the following two points.
\begin{enumerate}
\item[\textbullet] The specialisation of $\VR$ at ${u \! = \! 0}$ is a representation of $\Uhat[\KK]{\glie,\psi}$, which is isomorphic to $V$.
\item[\textbullet] The specialisation of $\xiVR$ at ${u \! = \! 1}$ is a representation in $\Oint[\KK]{\glie',\psi'}$, whose character is equal to $\chi(V) \circ \xi$.
\end{enumerate}
\end{thm}

\proof
Consequence of theorem \ref{thm_Uhatrepresentations}, of remark \ref{rem_isogeny} and of lemma \ref{lem_isogenyspecialisation}.
\qed

\subsection{Classical and quantum isogenic interpolations}
We assume in this subsection that the statement of the GQE conjecture holds. \\

We designate again in this subsection by $R$ the $\KK$-algebra formed by the rational functions without pole at $0$ and $1$:
$$ R \ := \ \Big \{ f/g \, ; \ f,g \in \KK[u], \, \, g(0) \neq 0, \, \, g(1) \neq 0 \Big \} \, . $$

\begin{thm} \label{thm_classicalinter}
There exists an admissible $I$-colouring $\psi_u$ of $\BB$ with values \hbox{in $R$} such that, for every representation $V$ in $\Ointc$, there exists a \hbox{representation $\VR$} in $\Oint{\glie,\psi_u}$ which satisfies the following two points.
\begin{enumerate}
\item[\textbullet] The specialisation of $\VR$ at ${u \! = \! 0}$ is a representation of $\glie$, isomorphic \hbox{to $V$}.
\item[\textbullet] The specialisation of $\xiVR$ at ${u \! = \! 1}$ is a representation in $\Ointc[(\glie')]$, whose character is equal to $\chi(V) \circ \xi$.
\end{enumerate}
\end{thm}

\proof
Consequence of theorems \ref{thm_exhat} and \ref{thm_coloutedinter}
\qed \\

We designate by $R_q$ the $\KK(q)$-algebra formed by the rational functions over $\KK(q)$ without pole at $0$ and $1$:
$$ R_q \ := \ \Big \{ f/g \, ; \ f,g \in \KK(q)[u], \, \, g(0) \neq 0, \, \, g(1) \neq 0 \Big \} \, . $$

\begin{thm} \label{thm_quantuminter}
There exists an admissible $I$-colouring $\psi_u$ of $\BB$ with values \hbox{in $R_q$} such that, for every representation $\Vq$ in $\Ointq$, there exists a \hbox{representation $\VR$} in $\Oint[\! R_q]{\glie,\psi_u}$ which satisfies the following two points.
\begin{enumerate}
\item[\textbullet] The specialisation of $\VR$ at ${u \! = \! 0}$ is a representation of $\Uq(\glie)$, which is isomorphic \hbox{to $\Vq$}.
\item[\textbullet] The specialisation of $\xiVR$ at ${u \! = \! 1}$ is a representation in $\Ointq[\glie']$, whose character is equal to $\chi(V) \circ \xi$.
\end{enumerate}
\end{thm}

\proof
Consequence of theorems \ref{thm_exhat} and \ref{thm_coloutedinter}
\qed

\subsection[Isogenic duality between representations of Lie algebras]{Isogenic duality between representations \\ of Lie algebras}
We assume in this subsection that the statement of the GQE conjecture holds.


\begin{thm} \label{thm_isogenyduality}
For every representation $V$ in $\Ointc$, there exists a unique, up to isomorphism, representation in $\Ointc[(\glie')]$ whose character is equal to $\chi(V) \circ \xi$, and which contains $\Lc{\lambda'}$ as a subrepresentation when $V$ is equal to $\Lc{\xi(\lambda')}$ with $\lambda' \in X'$.
\end{thm}

\proof
Consequence of theorem \ref{thm_classicalinter}.
\qed

\vsp

\begin{rem}
Let $\psi$ be an admissible $I$-colouring of $\BB$ with values in $R$ and let us assume that $R$ is a principal ideal domain. Recall that the category $\Oint{\glie,\psi}$ is parallel to the category $\Ointc$, in such a way that representations of the first category have the same characters than representations of the second one (see theorem \ref{thm_Uhatrepresentations}). The same observation holds for $\Oint{\glie',\psi}$ and $\Ointc[(\glie')]$. Therefore, theorem \ref{thm_isogenyduality} implies that for every representation $V$ in $\Oint{\glie,\psi}$, there exists a unique, up to isomorphism, representation in $\Oint{\glie',\psi}$ whose character is equal to $\chi(V) \circ \xi$, and which contains $\Lrep{\lambda'}$ as a subrepresentation when $V$ is equal to $\Lrep{\xi(\lambda')}$ with $\lambda' \in X'$.
\end{rem}

\subsection{The Langlands isogeny}
We denote by $\LC$ the transpose of the generalised Cartan matrix $C$ (remark \hbox{that $\LC$} is a generalised Cartan matrix). The \emph{Langlands dual root datum} $\LXdat$ is the root datum associated to $\LC$ and defined by the following.
\begin{enumerate}
\item[\textbullet] The root lattice $\LX$ of $\LXdat$ is the $\ZZ$-module $\{ \lambda \in X \; \! ; \, \langle \check \alpha_i, \lambda \rangle \in {d_i \: \! \ZZ} \, , \ \forall \, i \in I \}$.
\item[\textbullet] The coroot lattice of $\LXdat$ is the $\ZZ$-module $\mathrm{Hom}_{\ZZ}(\LX,\ZZ)$.
\item[\textbullet] The pairing of $\LXdat$ is the natural pairing between $\mathrm{Hom}_{\ZZ}(\LX,\ZZ)$ and $\LX$.
\item[\textbullet] The simple roots of $\LXdat$ are the elements $d_i \: \! \alpha_i$ ($i \in I$) in $\LX$.
\item[\textbullet] The simple coroots of $\LXdat$ are the elements $\check \alpha_i / d_i$ ($i \in I$).
\end{enumerate}

\vsp
\vsp
\vsp

The \emph{Langlands dual Kac-Moody algebra} $\Lglie$ is the Kac-Moody algebra associated to the Langlands dual root datum $\LXdat$. \\

Let $\lambda \in \LX$, we denote by $\Lc{\Llambda}$ the irreducible highest weight representation of $\Lglie$ of highest weight $\lambda$, in order to avoid any confusion with the irreducible representation $\Lc{\lambda}$ of $\glie$. \\

The \emph{Langlands isogeny} is the isogeny of root data from $\LXdat$ to $\Xdat$ defined as the inclusion map from $\LX$ to $X$. \\

Let $V$ be a weight representation of $\glie$, the \emph{Langlands dual character} $\Lchi(V)$ \hbox{of $V$} is defined as the composition of $\chi(V)$ with the Langlands isogeny. \\

Let $V$ be a weight representation of $\UF$, we denote by $\LV$ the $R$-submodule of $V$ defined by the following:
$$ \LV \ := \, \bigoplus_{\lambda \in \LX} V_\lambda \, . $$
We regard $\LV$ as a representation of $\UFL$, thanks to lemma \ref{lem_isogeny}.


\subsection[Langlands duality between representations of Lie algebras]{Langlands duality between representations \\ of Lie algebras}
We assume in this subsection that the statement of the GQE conjecture holds. \\

We designate in this subsection by $R$ the $\KK$-algebra formed by the rational functions without pole at $0$ and $1$:
$$ R \ := \ \Big \{ f/g \, ; \ f,g \in \KK[u], \, \, g(0) \neq 0, \, \, g(1) \neq 0 \Big \} \, . $$

\begin{thm} \label{thm_Langlandsinter}
There exists an admissible $I$-colouring $\psi_u$ of $\BB$ with values \hbox{in $R$} such that, for every representation $V$ in $\Ointc$, there exists a \hbox{representation $\VR$} in $\Oint{\glie,\psi_u}$ which satisfies the following two points.
\begin{enumerate}
\item[\textbullet] The specialisation of $\VR$ at ${u \! = \! 0}$ is a representation of $\glie$, isomorphic \hbox{to $V$}.
\item[\textbullet] The specialisation of $\LVR$ at ${u \! = \! 1}$ is a representation in $\Ointc[(\Lglie)]$, whose character is equal to $\Lchi(V)$.
\end{enumerate}
\end{thm}

\proof
Corollary of theorem \ref{thm_classicalinter}.
\qed \\

We obtain as an application of theorem \ref{thm_Langlandsinter} the following theorem, whose first proofs are due to Littelmann and McGerty (see the introduction of this thesis). In other words, representations of coloured Kac-Moody algebras yield an explanation of the Langlands duality between representations of Kac-Moody algebras, in terms of interpolation.

\begin{thm} \label{thm_Langlandsduality}
For every representation $V$ in $\Ointc$, there exists a unique, up to isomorphism, representation in $\Oint{\Lglie}$ whose character is equal to $\Lchi(V)$, and which contains $\Lc{\Llambda}$ as a subrepresentation when $V$ is equal to $\Lc{\lambda}$ \hbox{with $\lambda \in \LX$}.
\end{thm}

\subsection{Contribution to the programme of Frenkel-Hernandez}
We assume in this subsection that the statement of the GQE conjecture holds. \\

We designate by $R_q$ the $\KK(q)$-algebra formed by the rational functions over $\KK(q)$ without pole at $0$ and $1$:
$$ R_q \ := \ \Big \{ f/g \, ; \ f,g \in \KK(q)[u], \, \, g(0) \neq 0, \, \, g(1) \neq 0 \Big \} \, . $$

\begin{thm}
There exists an admissible $I$-colouring $\psi_u$ of the crystal $\BB$ with values \hbox{in $R_q$} such that, for every representation $\Vq$ in $\Ointq$, there exists a representation $\VR$ in $\Oint[\! R_q]{\glie,\psi_u}$ which satisfies the following two points.
\begin{enumerate}
\item[\textbullet] The specialisation of $\VR$ at ${u \! = \! 0}$ is a representation of $\Uq(\glie)$, which is isomorphic to $\Vq$.
\item[\textbullet] The specialisation of $\LVR$ at ${u \! = \! 1}$ is a representation in $\Ointq[\Lglie]$, whose character is equal to $\Lchi(\Vq)$.
\end{enumerate}
\end{thm}

\proof
Corollary of theorem \ref{thm_quantuminter}
\qed

\subsection{Langlands interpolation via quantum algebras}
The following result is established without reference to the GQE conjecture. Note that we obtain, as a consequence, a new proof of \hbox{theorem \ref{thm_Langlandsduality}}. In other words, the following result gives an explanation of the Langlands duality between representations of Lie algebras, in terms of interpolation, where interpolating representations are here representations of quantum algebras. \\

We designate here by $\KK$ the complex field $\CC$ and we denote by $\varepsilon$ a primitive root of unity in $\CC$ of order $2d$, where $d$ designates the least common multiple of the integers $d_i$ $(i \in I)$.


\begin{thm} \label{thm_Langlandsinterviaq}
For every $\lambda \in \Xdom$, the representation $\Lq (\lambda)$ of $\Uq(\glie)$ admits a $\CC[q^{\pm 1}]$-form $\Vqpm$ which satisfies the following two points.
\begin{enumerate}
\item[\textbullet] The specialisation of $\Vqpm$ at ${q \! = \! 1}$ is a representation of $\glie$, isomorphic to the irreducible representation $\Lc{\lambda}$.
\item[\textbullet] The specialisation of $\Vqpm$ at ${q \! = \! \varepsilon}$ is a representation in $\Ointc[(\Lglie)]$, whose character is equal to $\Lchi(\Lc{\lambda})$.
\end{enumerate}
\end{thm}

\proof
It is possible to define a notion of integral forms for coloured Kac-Moody algebras. All the constructions and all the results presented in this second part of the thesis can be proved to be compatible with integral forms of coloured Kac-Moody algebras, in such a way that the present theorem can be established thanks to theorem \ref{thm_exhat} and thanks to the following remarks.
\vsp
\vsp
\begin{itemize}
\item[\textbullet] The colouring $\psiqI$ interpolates between two copies of the colouring $\psiclI$.
\item[\textbullet] Let $i \in I$ and let $n \in \NN$. The following relations hold in $\Uq(\glie)$ for all $k,k' \in \NNI$ and for all $i \in I$ (see for example \cite[formula 9.3.13]{ChPr}):
\begin{equation*} \label{eq_chari}
(X_i^+)^{(k)} \, (X_i^-)^{(k')} \ = \ \sum_{k''=0}^{\max (k,k')} (X_i^-)^{(k'-k'')} \, {K_i \, ; \, 2 \! \: k'' - k' - k \brack k''}_{\! \: \! q_i} \, (X_i^+)^{(k-k'')} \, ,
\end{equation*}
$$ \text{where } \ {K_i \! \; ; \! \; b \brack a}_{\! \! \: q_i} \ := \ \prod_{c=1}^a \frac{K_i \! \; q_i^{b+1-c} - K_i^{-1} q_i^{c-1-b}}{q_i^s - q_i^{-s}} \quad \quad (a,b \in \NN) \, . $$
\end{itemize}
Note that we make use of the (quantum) divided powers ${(X_i^\pm)}^{(d/d_i)}$, rather the usual powers ${(X_i^\pm)}^{d/d_i}$, when defining the action \hbox{of $\LglieF$} on $\LVqpm$.
\qed

\begin{appendices}
\section{The equations ${\widehat{\mathscr E}_{\! \! \: R}}$}
\subsection{Notations}
We denote by $\Tinf[R]$ the $R$-algebra formed by the lower triangular matrices of infinite order with coefficients in the ring $\EndR(R^\ZZ)$ (remark that the matrix multiplication for $\Tinf[R]$ is well defined since matrices are lower triangular):
$$ \Tinf[R] \ := \ \left\{ \mathbf T \in \left( \EndR(R^\ZZ) \right)^{\! \: \! \NN^2} ; \, \ {\mathbf T_{\! \! \: p,a}} = 0 \, , \ \forall \, p,a \in \NN \, \text{ s.t. } p < a \right\} \, . $$

We denote by $\Sinf[R]$ the \hbox{$R$-module} formed by the column matrices of infinite order with coefficients in $R^\ZZ$:
$$ \Sinf[R] \ := \ \left( R^\ZZ \right)^{\! \! \; \NN} \, . $$
 We regard $\Sinf[R]$ as a left $\Tinf[R]$-module, where the action of a matrix $\mathbf T$ \hbox{in $\Tinf[R]$} on a column matrix $M$ in $\Sinf[R]$ is defined by the following:
$$ {\big( \mathbf T \! \; . \! \; M \big)}_{\! \! \: p} \ := \ \sum_{0 \leq a \leq p} {\mathbf T_{\! \! \: p,a}} \! \; . \! \; M_a \quad \quad (p \in \NN) \, . $$

Let $M \in \Sinf[R]$ and let $d \in \ZZ$. We denote by $M[d]$ the column matrix \hbox{in $\Sinf[R]$} defined by the following:
$$ {\big( M[d] \big)}_{\! \! \; p} \ := \ \left\{ {\renewcommand{\arraystretch}{1.2} \begin{array}{ccl}
M_{p-d} & \quad & \text{if } p \geq d \, , \\
0 && \text{else} \end{array}} \right. \quad \quad \quad (p \in \NN) \, . $$

\subsection{Definition}
Let $\psi \in \Colo{R}$. We denote by $\mathbf T(\psi)$ the lower triangular matrix in $\Tinf[R]$, defined by the following, \hbox{for $p,a \in \NN$}, $f \in R^\ZZ$ and $n \in \ZZ$:
$$ \bigg[ {\big( \mathbf T(\psi) \big)}_{\! \! \: p,a} \, . \, f \bigg] (n) \ := \ \left\{
{\renewcommand{\arraystretch}{1.6} \begin{array}{cl}
\displaystyle \frac{[\psi] (n,p) !}{[\psi] (n,p-a) !} \ f(n-2p+2a) & \quad \text{if } a \leq p \leq n \, , \\
0 & \quad \text{else.}
\end{array}} \right. $$
We denote by $N(\psi)$ and the column matrix in $\Sinf[R]$, defined by the following, \hbox{for $p \in \NN$} and $n \in \ZZ$:
$$ {\big( N(\psi) \big)}_{\! \! \: p} \! \; (n) \ := \ \left\{
{\renewcommand{\arraystretch}{1.3} \begin{array}{cl}
[\psi] (n,p) & \quad \ \text{if } 0 < p \leq n \, , \\
0 & \quad \ \text{else.}
\end{array}} \right. $$

\begin{defi} \label{defi_eqhat}
Let $\psi_1, \psi_2 \in \Colo{R}$ and let $d \in \ZZ$. We denote \hbox{by $\eqhat{\psi_1,\psi_2}{d}$} the following linear equation:
\begin{equation*} \label{eq_hat}
\mathbf T(\psi_1) \, . \, M \ = \ N(\psi_2)[d] \, ,
\end{equation*}
where the unknown $M$ is a column matrix in $\Sinf[R]$.
\end{defi}

\subsection{First interpretation}
\begin{prop} \label{prop_eqhatspan}
Let $\psi \in \Colo{R}$ and let $M \in \Sinf[R]$. If $M$ is a solution of the equation $\eqhat{\psi,\psi}{-1}$, then the following equality holds in $\Uhat{\slt,\psi}$:
$$ X^+ X^- \ = \ \sum_{a \in \NN} \, {(X^-)}^a \! \; M_a \, {(X^+)}^a \, . $$
\end{prop}

Remark that the series $\sum_{a \in \NN} {(X^-)}^a M_a {(X^+)}^a$ makes sense and converges in the topological $R$-algebra $\Uhat{\slt,\psi}$ (see proposition \ref{prop_Uhatconv}).

\proof
See the proof of proposition I.\ref{prop_GQEeqspan}.
\qed



\subsection{Existence of solutions}
\begin{prop} \label{prop_eqhatexist}
Let $\psi_1, \psi_2 \in \Colo{R}$ and let $d \in \ZZ$. If $\psi_1$ is admissible \hbox{over $R$}, then the equation $\eqhat{\psi_1, \psi_2}{d}$ has a solution.
\end{prop}

\proof
Remark that the evaluation of the equation $\eqhat{\psi_1, \psi_2}{d}$ at every $n \in \NN$ is naturally equivalent a to a finite-rank $R$-linear equation of triangular form, whose diagonal coefficients are invertible (in $R$) if $\psi_1$ is admissible (over $R$).
\qed

\end{appendices}

\end{otherlanguage}

\cleardoublepage
\part[Groupes quantiques d'interpolation \\ \hspace*{2.1pc} de Langlands de rang 1]{Groupes quantiques d'interpolation de Langlands de rang 1}
\label{part_IMRN}
\setcounter{section}{0}
\righthyphenmin=62
\lefthyphenmin=62

\vspace*{0.5cm}
Publié dans \emph{International Mathematics Research Notices}, 2012
\vspace*{0.5cm}

\begin{abstract}
\noindent
On étudie une certaine famille, paramétrée par un entier $g \in \NN_{\geq 1}$, de doubles déformations $\Uhhp$ de l'algèbre enveloppante $\Uc[\slt]$, dans l'esprit de Frenkel-Hernandez \cite{FH1}. On prouve que $\Uhhp$ déforme simultanément les groupes quantiques $\Uh[\slt]$ et $\Ughp$. On montre que cette propriété d'interpolation explique la dualité de Langlands entre les représentations des groupes quantiques en rang 1. On résout ainsi une conjecture de Frenkel-Hernandez \cite{FH1} dans cette situation : on prouve pour \hbox{tout $g$} l'existence de représentations de $\Uhhp$ qui déforment simultanément deux représentations Langlands duales. On étudie aussi plus généralement la théorie des représentations de rang fini de $\Uhhp$.
\end{abstract}

\vspace*{0.2cm}

\begin{otherlanguage}{english}
\righthyphenmin=62
\lefthyphenmin=62
\begin{abstract}
\noindent
We study a certain family, parameterised by an integer $g \in \NN_{\geq 1}$, of double deformations $\Uhhp$ of the enveloping algebra $\Uc[\slt]$, in the spirit of Frenkel-Hernandez \cite{FH1}. We prove that $\Uhhp$ deforms simultaneously the quantum groups $\Uh[\slt]$ and $\Ughp$. We prove that this interpolating property explains the Langlands duality between representations of quantum groups in rank 1. We thus resolve a conjecture of Frenkel-Hernandez \cite{FH1} in this situation : we prove for all $g \in \NN_{\geq 1}$ the existence of representations of $\Uhhp$ which deform simultaneously two Langlands dual representations. We also study more generally the finite-rank representation theory of $\Uhhp$.
\end{abstract}
\end{otherlanguage}

\addcontentsline{toc}{section}{Introduction}
\section*{Introduction}
\righthyphenmin=62
\lefthyphenmin=62
Soit $\glie$ une algèbre de Kac-Moody symétrisable et $\Lglie$ sa duale de Langlands, dont la matrice de Cartan est la transposée de celle de $\glie$. Il existe une dualité entre les représentations de $\glie$ et $\Lglie$ : voir \cite{FH2, FH1, Kas, Lit2, McG}. \\

La dualité de Langlands pour les représentations peut être décrite au niveau des caractères. Supposons pour simplifier que $\glie$ soit de type fini et simple ($\Lglie$ est alors aussi de type fini et simple). Soit $I$ l'ensemble des sommets du diagramme de Dynkin de $\glie$ et $r_i$ ($i \in I$) les étiquettes correspondantes. On désigne par $r$ le nombre maximal parmi les nombres $r_i$. Ce nombre $r$ est égal à $1$ \hbox{lorsque $\glie$} est simplement lacée, à $2$ pour $B_l$, $C_l$ et $F_4$, et à $3$ pour $G_2$. Supposons que le plus haut poids $\lambda$ de la représentation irréductible de dimension finie $L(\lambda)$ s'écrit sous la forme suivante :
$$ \lambda \ = \ \sum_{i \in I} \, (1-r + r_i) \: \! m_i \: \! \omega_i \, , \quad m_i \in \NN \, , $$
où $\omega_i$ sont les poids fondamentaux de $\glie$. En d'autres mots, notant $P$ le réseau des poids de $\glie$, supposons que $\lambda$ est dominant et appartient au sous-réseau $P' \subset P$ engendré par ${(1 - r + r_i) \: \! \omega_i}$ ($i \in I$). Le caractère de $L(\lambda)$ s'écrit de la manière suivante :
$$ \chi (L(\lambda)) \ = \ \sum_{\nu \in P} d(\lambda, \nu) \; \! e^{\nu} \, , \quad d(\lambda, \nu) \in \NN \, . $$
Soit $\nu = \sum_i (1-r+r_i) n_i \: \! \omega_i \in P'$, on note $\nu' := \sum_i n_i \: \! {}^L \omega_i$, où les ${}^L \omega_i$ ($i \in I$) désignent les poids fondamentaux de $\Lglie$. En remplaçant dans $\chi (L(\lambda))$ chaque $e^{\nu}$ par $e^{\nu'}$ lorsque $\nu \in P'$ et par $0$ sinon, il est montré dans \cite{FH1} que l'on obtient le caractère ${}^L \chi ( {}^L \! L(\lambda))$ d'une représentation virtuelle de $\Lglie$, qui contient la représentation irréductible $L(\lambda')$ de $\Lglie$, i.e. $\chi(L\! (\lambda')) \leq {}^L \chi ( {}^L L(\lambda))$. Pour une description de la dualité de Langlands au niveau des caractères dans le cas général d'une algèbre de Kac-Moody, voir \cite{McG}. \\

Dans le cas d'une algèbre $\glie$ de type fini Littelmann \cite{Lit2} a prouvé que la représentation virtuelle de $\Lglie$ est en fait une représentation ${}^L \! L(\lambda)$ réelle. Dans le cas général d'une algèbre de Kac-Moody, voir McGerty \cite{McG}. \\
La stratégie consiste à $q$-déformer la situation. Les théories des représentations de $\glie$ et du groupe quantique $U_q(\glie)$ (avec $q$ générique) sont semblables, en particulier il existe une représentation $L^q(\lambda)$ de $U_q(\glie)$, analogue de la représentation $L(\lambda)$ de $\glie$. Lorsque $q$ est spécialisé à une certaine racine de l'unité $\varepsilon$, l'homomorphisme de Frobenius quantique défini par Lusztig \cite[35]{Lus} peut être étendu afin de construire une action du groupe quantique modifiée $\dot{U}_{\varepsilon} ( \Lglie)$ sur $L^{\varepsilon} (\lambda)$, qui fournit alors l'action de $\Lglie$ sur ${}^L \! L(\lambda)$. \\

La dualité de Langlands décrite précédemment a des liens avec la correspondance géométrique de Langlands (voir \cite{Fre2} pour une introduction générale au programme de Langlands et \cite{FH1} pour d'autres références). On suit ici l'introduction de \cite{FH1}. \\
Soit $\glie$ une algèbre de Lie semi-simple de dimension finie. Un résultat clé dans la correspondance géométrique de Langlands est un isomorphisme entre le centre $Z(\hat{\glie})$ de l'algèbre enveloppante complétée de $\hat \glie$ au niveau critique et la $\mathcal{W}$-algèbre classique $\mathcal{W}(\Lglie)$. Afin de mieux comprendre cet isomorphisme, on peut $q$-déformer la situation et considérer le centre $Z_q(\hat{\glie})$ de l'algèbre affine quantique $U_q(\hat{\glie})$ au niveau critique, ce qui est le point de départ dans \cite{FR2}. Le centre $Z_q(\hat{\glie})$ apparaît alors relié à l'anneau de Grothendieck $\text{Rep} \, U_q(\hat{\glie})$ des représentations de dimension finie de $U_q(\hat{\glie})$. On espère donc obtenir un nouvel éclairage sur l'isomorphisme $Z(\hat{\glie}) \simeq \mathcal{W}(\Lglie)$ en étudiant les connections entre $Z_q(\hat{\glie})$, $\text{Rep} \, U_q(\hat{\glie})$ et la $\mathcal{W}$-algèbre classique $q$-déformée. \\
L'idée de Frenkel et Reshetikhin \cite{FR3} était de déformer un nouvelle fois la situation, en introduisant une double déformation $\mathcal{W}_{q,t}$ de $Z(\hat{\glie})$. La spécialisation $\mathcal{W}_{q,1}$ à $t=1$ est le centre $Z_q(\hat{\glie})$ et ils ont suggéré que la spécialisation $\mathcal{W}_{\varepsilon,t}$ à $q=\varepsilon$, où $\varepsilon$ désigne une certaine racine de l'unité, contient le centre $Z_t({}^L \hat{\glie})$ de l'algèbre affine quantique $U_t({}^L \hat{\glie})$ au niveau critique. $Z_t({}^L \hat{\glie})$ étant lié à l'anneau de Grothendieck $\text{Rep} \, U_t({}^L \hat{\glie})$ des représentations de dimension finie de $U_t({}^L \hat{\glie})$, la $\mathcal{W}$-algèbre $\mathcal{W}_{q,t}$ apparaît ainsi interpoler les anneaux de Grothendieck des représentations de dimension finie des algèbres affines quantiques associées à $\hat \glie$ et ${}^L \hat{\glie}$. En particulier cela suggère que ces représentations doivent être d'une certaines manières liées. On peut trouver dans \cite{FR3} des exemples de telles relations. \\

C'est, motivés par les questions précédentes, que Frenkel et Hernandez \cite{FH1} introduisent la dualité de Langlands pour les représentations de dimension finie de $\glie$ et $\Lglie$, dans la situation où $\glie$ est de type finie. \\
Cette dualité, inspirée par la correspondance de Langlands géométrique, se trouve être exactement la dualité (sans rapport a priori avec le programme de Langlands) établie dix ans plus tôt par Littelmann \cite{Lit2}.  L'approche dans \cite{FH1} est toutefois différente de celle de Littelmann. Dans l'esprit de \cite{FR3}, ils ont introduit une double déformation $\widetilde{U}_{q,t} (\glie)$ de $\widetilde{U}(\glie)$, l'algèbre enveloppante de $\glie$ sans les relations de Serre, et conjecturé que cette double déformation permettait d'expliquer la dualité de Langlands pour les représentations de $\glie$ et ${}^L \! \glie$. \\

Historiquement, les groupes quantiques ont été définis par Drinfeld \cite{Dri, DrP} et Jimbo \cite{Jim}. Drinfeld les définit comme des déformations formelles $U_h(\glie)$ des algèbres enveloppantes $U(\glie)$, tandis que Jimbo définit leurs versions rationnelles $U_q(\glie)$. L'algèbre $U_h(\mathfrak{sl}_2)$ avait été construite par Kulish et Reshetikhin \cite{reshetikhin}. \\
Le terme déformation formelle a dans ce cas la signification suivante : à une algèbre $A$ sur $\CC$, on associe une $\CC[[h]]$-algèbre que l'on note $A[[h]]$. En tant que $\CC[[h]]$-module, $A[[h]]$ est l'ensemble des séries formelles à coefficients dans $A$ :
$$ \sum_{n \in \NN} a_n \: \! h^n \, , \quad \text{avec } a_n \in A \, . $$
C'est alors le produit de $A[[h]]$ qui porte la déformation, dans le sens où l'on doit retrouver celui de $A$ quand $h = 0$. On pourra consulter à ce sujet \cite[III]{Gui}, ainsi que \cite[XVI]{Kas} et \cite[6.1]{ChPr}. \\

Les groupes quantiques d'interpolation $\widetilde{U}_{q,t} (\glie)$ introduits par Frenkel et Hernandez dépendent de deux paramètres $q$ et $t$. Cet article propose de prendre le point de vue des déformations formelles. Pour cela, il nous faudra considérer des doubles déformations d'algèbres, selon des paramètres formels $h$ et $h'$. À une algèbre $A$, on associe donc une algèbre $A[[h,h']]$, dont les éléments sont maintenant des séries formelles à deux indéterminées et à coefficients dans $A$
$$ \sum_{n,m \in \NN} a_{n,m} \, h^n (h')^m \, , \quad \text{avec } a_{n,m} \in A \, . $$
On construit dans cet article une classe, paramétrée par un entier $g \in \NN_{\geq 1}$, de déformations formelles $\Uhhp$ de $\Uc[\slt]$ selon les paramètres $h$ et $h'$. Ces déformations sont les groupes quantiques d'interpolation de Langlands de rang $1$ (i.e. pour $\glie = \mathfrak{sl}_2$). Il s'agit d'une première étape dans la construction de groupes quantiques d'interpolation de Langlands $\widetilde{U}_{h,h'}(\glie,g)$ associés à une algèbre de Kac-Moody (symétrisable), qui apparaîtront dans un prochain article. \\

L'un des objectifs des groupes quantiques d'interpolation de Langlands est de donner une explication de la dualité de Langlands pour les représentations d'une algèbre de Kac-Moody $\glie$ et de sa duale $\Lglie$. \\
On attend que, pour certaines valeurs de $g$, $\widetilde{U}_{h,h'}(\glie,g)$ interpolent les groupes quantiques $\widetilde{U}_h(\glie)$ et $\widetilde{U}_{g' h'}(^{L} \glie)$, où $g'$ est un entier qui dépend de $g$ et $\glie$ : la limite $h' \to 0$ de $\widetilde{U}_{h, h'}(\glie,g)$ est $\widetilde{U}_h(\glie)$ et la spécialisation à $Q := \exp h = \varepsilon$, avec $\varepsilon$ une racine $(2g)$-ième primitive de l'unité, admet comme sous-quotient $\widetilde{U}_{g' h'}(^{L} \glie)$.
$$ \xymatrix{
&& \widetilde{U}_{h,h'}(\glie,g) \ar[lld]_-{\lim_{h' \to 0 \ }} \ar[rrd]^-{\lim_{Q \to \varepsilon}} && \\
\widetilde{U}_h(\glie) && && \widetilde{U}_{g' h'}(\Lglie)
} $$
Dans cet article, cette conjecture est résolue au rang 1 : voir le théorème \ref{thm_uhhinterpolation}. \\

L'interpolation ci-dessus induit une correspondance $(\glie,g) \longleftrightarrow (\Lglie, g')$, qui rappelle la dualité $S$ pour les théories de jauge. Le paramètre $g$ y jouerait le rôle de la constante de couplage. On pourra consulter \cite{FreB} à propos des théories de jauge et de la dualité de Langlands. \\

Pour $\lambda \in P'$, on conjecture (voir aussi \cite[conjecture 1]{FH1}) l'existence d'une double déformation de $L(\lambda)$ où les actions de $\glie$ et $\Lglie$ seraient toutes deux contenues. Plus précisément on conjecture l'existence d'une représentation $L^{h,h'}(\lambda,g)$ de $\widetilde{U}_{h,h'}(\glie)$, qui interpole la représentation $L^h(\lambda)$ et une représentation Langlands $g$-duale ${}^L \! L^{h'}(\lambda,g)$ : la limite $h' \to 0$ de $L^{h,h'}(\lambda,g)$ est la représentation $L^h(\lambda)$, et la spécialisation à $Q = \varepsilon$ contient la représentation ${}^L \! L^{g'h'}(\lambda,g)$.
\begin{equation} \label{eq_reprinter}
\xymatrix{
&& L^{h,h'}(\lambda,g) \ar[lld]_-{\lim_{h' \to 0 \ }} \ar[rrd]^-{\lim_{Q \to \varepsilon}} && \\
L^h(\lambda) && && \quad {}^L \! L^{h'}(\lambda,g) \ \supset \ L^{g'h'}(\lambda')
} \end{equation}
Cette conjecture est dans cet article résolue en rang 1 : voir le théorème \ref{thm_reprinter}. \\

Pour établir un parallèle avec les groupes quantiques usuels, on peut penser, dans le cas $\glie$ de type fini et simple, à $\widetilde{U}_{q,t}(\glie)$ comme une version spécialisée de $\widetilde{U}_{h,h'}(\glie,r)$ (on rappelle que $r$ désigne la le nombre maximal parmi les étiquettes $r_i$ attachées au diagramme de Dynkin $I$ de $\glie$). Toutefois alors qu'on peut retrouver $U_q(\glie)$ comme sous-algèbre de $U_h(\glie)$, il n'est pas vrai que $\widetilde{U}_{q,t}(\glie)$ apparaît comme sous-algèbre de $\widetilde{U}_{h,h'}(\glie,r)$. Les groupes quantiques d'interpolation proposés par Frenkel et Hernandez sont en fait notablement différents : ils nécessitent notamment beaucoup plus de générateurs pour être définis, voir par exemple la remarque \ref{rem_hernandezgen}. \\
On trouvera dans \cite{KT} un autre exemple de doubles déformations de structure algébrique provenant des algèbres de Lie. On peut également citer \cite{multi} où des groupes quantiques à plusieurs paramètres ont été définis. \\

On donnera tout au long de l'article des comparaisons entre les groupes quantiques d'interpolation définis ici et ceux définis dans \cite{FH1}. Notons déjà que le principal avantage de cette nouvelle définition réside dans le fait qu'ils sont des déformations formelles de $\widetilde{U}(\glie)$. Par ailleurs :
\begin{itemize}
\item[\textbullet] on obtient dans cet article une définition plus claire et plus maniable des groupes quantiques d'interpolation de Langlands;
\item[\textbullet] ce travail permettra, en toute généralité, de définir un groupe quantique d'interpolation associé à une algèbre de Kac-Moody $\glie$ symétrisable et à un paramètre $g \in \NN_{\geq 1}$ (dans un prochain article);
\item[\textbullet] les outils et résultats de la théorie des déformations formelles sont, dans notre cadre, utilisables pour l'étude des groupes quantiques d'interpolation.
\end{itemize}

\addcontentsline{toc}{section}{Organisation}
\section*{Organisation}
On décrit ici l'organisation de l'article. Notons que tous les résultats présentés à partir de la section \ref{section_def} sont nouveaux.

\vsp
\vsp
\vsp

\begin{itemize}
\item[\textbullet] Dans la section \ref{section_deformations}, on présente certains éléments de la théorie des déformations formelles dans le cadre de plusieurs paramètres de déformation. Les propositions \ref{prop_hoch} et \ref{prop_hochrep} concluent cette section.

\item[\textbullet] Dans la section \ref{section_def}, on donne et on explique la définition des groupes quantiques d'interpolation $\Uhhp$.

\item[\textbullet] Dans la section \ref{section_prop}, on établit différentes propriétés des groupes quantiques d'interpolation de $\Uhhp$. On prouve notamment qu'ils sont des déformations formelles triviales de $\Uc[\slt]$ selon $h$ et $h'$, de $\Uh[\slt]$ selon $h'$ et de $\Urh{\: \! h'}$ selon $h$ : voir le théorème \ref{thm_uhhdefor}. \\
Ces différentes déformations peuvent être représentées par le diagramme commutatif
$$ \xymatrix{
&& \Uhhp \ar[dll]_-{\lim_{h' \to 0} \ } \ar[drr]^-{\lim_{h \to 0}} \ar[dd]^-{\lim_{h,h' \to 0}} && \\
\Uh[\slt] \ar[drr]_-{\lim_{h \to 0}} && && \Urh{\: \! h'} \ar[dll]^-{\ \lim_{h' \to 0}} \\
&& \Uc[\slt] &&
} $$
Deux des points techniques clés de cet article sont la proposition \ref{prop_uconst}, ainsi que le lemme \ref{lem_fond}, démontrés respectivement au tout début et à la fin de la section \ref{section_prop}.

\item[\textbullet] Dans la section \ref{section_repr} on étudie la théorie des représentations de rang fini de $\Uhhp$, qui s'avère être similaire à celle de $\Uc[\slt]$ et $\Uh[\slt]$. On prouve l'existence de doubles déformations pour tout $g \in \NN_{\geq 1}$ des représentations irréductibles de dimension finie et des modules de Verma de $\mathfrak{sl}_2$ : voir les propositions \ref{prop_vermadescr} et \ref{prop_indecdescr}. \\
Le théorème \ref{thm_rephh} établit plusieurs propriétés de la catégorie des représentations de rang fini de $\Uhhp$. Notamment, on prouve qu'elle possède la propriété de Krull-Schmidt et on classifie ces objets indécomposables.

\item[\textbullet] Dans la section \ref{section_langlands}, on établit l'existence d'une double déformation $L^{h,h'}(n,g)$ possédant la propriété d'interpolation \eqref{eq_reprinter}, résolvant ainsi la conjecture \cite[conjecture 1]{FH1} en rang 1 pour tout $g \in \NN_{\geq 1}$ : voir le théorème \ref{thm_reprinter}. \\
Le théorème \ref{thm_reprinter2} donne un résultat analogue pour les modules de Verma. \\
Enfin on prouve que $\Uhhp$ interpole les groupes quantiques $\Uh[\slt]$ et $\Urh{g'h'}$ : voir le théorème \ref{thm_uhhinterpolation}.
$$ \xymatrix{
&& \Uhhp \ar[lld]_-{\lim_{h' \to 0 \ }} \ar[rrd]^-{\lim_{Q \to \varepsilon}} && \\
\Uh[\slt] && && \Urh{g'h'}
} $$
\end{itemize}

\addcontentsline{toc}{section}{Notations et conventions}
\section*{Notations et conventions}
Toutes les algèbres et anneaux considérés sont supposés associatifs et unitaires. Un morphisme d'algèbre (ou d'anneau) envoie l'unité sur l'unité. Sans précision supplémentaire, on désigne par $\KK$ un anneau commutatif. \\

Soit $\Lambda$ un ensemble, on désigne par $\langle \Lambda \rangle$ le monoïde libre engendré par $\Lambda$, i.e. l'ensemble des mots (y compris le mot vide), dont les lettres sont des éléments de $\Lambda$, muni de la concaténation comme opération. On désigne alors par $\KK \langle \Lambda \rangle$ la $\KK$-algèbre du monoïde $\langle \Lambda \rangle$: c'est l'algèbre libre engendrée par $\Lambda$. \\

Une somme indexée par un ensemble vide est nulle, un produit indexé par un ensemble vide vaut $1$.

\section{Déformations formelles} \label{section_deformations}
La théorie des déformations formelles d'algèbres a été introduite par Gerstenhaber dans \cite{Ger}. La théorie est construite principalement pour un seul paramètre de déformation, bien que dans \cite{Ger} il est indiqué de quelle manière certains résultats peuvent être généralisés. On se propose de donner dans cette section une présentation d'une partie de la théorie dans le cadre de plusieurs paramètres de déformation. L'exposition qui suit est une généralisation naturelle de celles que l'on peut trouver, pour un seul paramètre, dans par exemple \cite[III]{Gui}, ainsi que \cite[XVI]{Kas} et \cite[6.1]{ChPr}. On prendra soin de détailler les arguments spécifiques à la généralisation, les autres seront rappelés.

\subsection{Premières définitions}
Soit $s \in \NNI$. On note $\KK[[h_1,h_2, \dots, h_s]]$, ou encore $\KK[[\bar{h}]]$ la $\KK$-algèbre commutative des séries formelles à $s$ indéterminées $h_1, \dots, h_s$. On note $\KK((h_1, h_2, \dots, h_s))$, ou encore $\KK((\bar{h}))$, le corps de fractions de $\KK[[\bar{h}]]$. \\

On considère sur la $\KK$-algèbre $\KK[[\bar{h}]]$ la graduation donnée par le degrés totaux des monômes :
$$ \text{deg} \big( h_1^{n_1} h_2^{n_2} \cdots h_s^{n_s} \big) \ := \ n_1 + n_2 + \cdots + n_s \, . $$
On note par ailleurs $(\bar h$) l'idéal de $\KK[[\bar h]]$ engendré par les éléments $h_1, \dots, h_s$. \\
La filtration associée à la graduation est également celle induite par l'idéal $(\bar h)$ et fait de $\KK[[\bar{h}]]$ une algèbre complète et séparée. On l'appelle la filtration $\bar h$-adique (la topologie associée est appelée la topologie $\bar h$-adique). \\
La filtration sur un $\KK[[\bar h]]$-module $M$, induite par celle de $\KK[[\bar{h}]]$, est aussi appelée la filtration $\bar h$-adique (la topologie associée sur $M$ est appelée la topologie $\bar h$-adique). \\

Soit $V$ un $\KK$-module, on note $V[[h_1, h_2, \dots, h_s]]$, ou encore $V[[\bar{h}]]$, l'espace des séries formelles à coefficients dans $V$ :
$$ \sum_{n_1, n_2, \dots, n_s \, \geq \, 0} v_{n_1,n_2, \dots, n_s} \, h_1^{n_1} h_2^{n_2} \cdots h_s^{n_s} \, , \quad \text{avec } \ v_{n_1,n_2, \dots, n_s} \in V . $$
Pour alléger l'écriture, introduisons les notations suivantes : $\bar n = (n_1, n_2, \dots, n_s)$ désignera un $s$-uplet de $\NN^s$, et $h^{\bar{n}}$ le monôme $h_1^{n_1} h_2^{n_2} \cdots h_s^{n_s}$. \\
Un élément de $V[[\bar{h}]]$ peut alors s'écrire sous la forme suivante :
$$ \sum_{\bar{n} \in \NN^s} v_{\bar n} \: \! h^{\bar{n}} \, , \quad \text{avec } \ v_{\bar{n}} \in V . $$

L'espace vectoriel $V[[\bar{h}]]$ est naturellement muni d'une structure de $\KK[[\bar{h}]]$-module :
$$ \Big( \sum_{\bar{p} \in \NN^s} x_{\bar p} \, h^{\bar{p}} \Big) . \Big( \sum_{\bar{q} \in \NN^s} v_{\bar q} \, h^{\bar{q}} \Big) \ := \ \sum_{\bar{n} \in \NN^s} \Bigg( \sum_{\bar{p} + \bar{q} = \bar{n}} x_{\bar{p} \, } .v_{\bar q} \Bigg) h^{\bar{n}} \, , $$
avec $x_{\bar{p}} \in \KK$ et $v_{\bar{q}} \in V$. \\
On définit une structure de $\KK[[\bar h]]$-module gradué sur $V[[\bar h]]$ en posant
$$ \text{deg} \big( h_1^{n_1} h_2^{n_2} \cdots h_s^{n_s} \big) \ := \ n_1 + n_2 + \cdots + n_s \, . $$
La filtration associée à cette graduation est la filtration $\bar h$-adique et fait de $V[[\bar{h}]]$ un $\KK[[\bar{h}]]$-module complet et séparé. \\

Le localisé du $\KK[[\bar{h}]]$-module $V[[\bar{h}]]$ par rapport aux éléments $h_1, h_2, \dots, h_s$ est un espace vectoriel sur $\KK((\bar{h}))$, que l'on notera $V((h_1, h_2, \dots, h_s))$, ou encore $V((\bar{h}))$. Il s'agit de l'espace des séries de Laurent à coefficients dans $V$
$$ \sum_{\bar{n} \, \succeq - \overline{N}} v_{\bar n} \: \! h^{\bar{n}} \, , \quad \text{avec } \ v_{\bar{n}} \in V , \text{ et } \overline{N} \in \NN^s . $$
L'ordre partiel $\preceq$ sur $\ZZ^s$ que l'on utilise est défini par $\bar{n} \preceq \bar{m}$ si $\bar{m} - \bar{n} \in \NN^s$. \\

Si $A$ est une $\KK$-algèbre (pas nécessairement commutative), on peut munir $A[[\bar{h}]]$ d'une structure de $\KK[[\bar{h}]]$-algèbre :
$$ \Big( \sum_{\bar{p} \in \NN^s} a_{\bar p} \, h^{\bar{p}} \Big) \cdot \Big( \sum_{\bar{q} \in \NN^s} b_{\bar q} \, h^{\bar{q}} \Big) \ := \ \sum_{\bar{n} \in \NN^s} \Bigg( \sum_{\bar{p} + \bar{q} = \bar{n}} a_{\bar{p} \, } b_{\bar q} \Bigg) h^{\bar{n}} \, , $$
où $a_{\bar{p}}$ et $b_{\bar{q}}$ sont des éléments de $A$. \\

On rappelle que $\KK \langle X \rangle$ désigne la $\KK$-algèbre du monoïde libre $\langle X \rangle$.

\begin{definition}
Soit $X$ un ensemble et $\mathcal{R}$ une partie de $\KK \langle X \rangle[[\bar{h}]]$. La $\KK[[\bar{h}]]$-algèbre topologiquement engendrée par $X$ et définie par les relations $\mathcal{R}$, que l'on note
$$ \KK \langle X \rangle[[\bar{h}]] \, / \, \overline{(\mathcal{R})}^{\: \! \bar h} \, , $$
est le quotient de $\KK \langle X \rangle[[\bar{h}]]$ par l'adhérence (topologique) de l'idéal bilatère engendré par $\mathcal{R}$ (ou encore le plus petit idéal bilatère fermé contenant $\mathcal{R}$).
\end{definition}

La filtration induite sur le quotient $\KK \langle X \rangle[[\bar{h}]] \, / \, \overline{(\mathcal{R})}^{\: \! \bar h}$ par celle de $\KK \langle X \rangle[[\bar{h}]]$ est la filtration $\bar h$-adique, et fait de $\KK \langle X \rangle[[\bar{h}]] \, / \, \overline{(\mathcal{R})}^{\: \! \bar h}$ une algèbre complète et séparée. La topologie $\bar h$-adique de $\KK \langle X \rangle[[\bar{h}]] \, / \, \overline{(\mathcal{R})}^{\: \! \bar h}$ est la topologie quotient. \\

Remarquons les propriétés suivantes. Leur démonstration est immédiate.

\begin{proposition} \label{prop_props}
\begin{itemize}
\item[1)] Si $V$ et $W$ sont deux $\KK$-modules, on a un isomorphisme canonique de $\KK[[\bar{h}]]$-module entre $\text{Hom}_{\KK[[\bar{h}]]} \big( V[[\bar{h}]],W[[\bar{h}]] \big)$ et $\text{Hom}_{\KK}(V,W)[[\bar{h}]]$.

\item[2)] Dans le cas où $V = W$, l'isomorphisme précédent entre $\text{End}_{\KK[[\bar{h}]]} \big( V[[\bar{h}]] \big)$ et $\left( \text{End}_{\KK} \: \! V \right) [[\bar{h}]]$ est un isomorphisme de $\KK[[\bar{h}]]$-algèbre.

\item[3)] Si $f : X \to B$ est une application de $X$ dans une $\KK[[\bar{h}]]$-algèbre $B$ complète et séparée (pour la filtration $\bar{h}$-adique), alors il existe un unique morphisme de $\KK[[\bar{h}]]$-algèbre $\tilde{f} : \KK \langle X \rangle[[\bar{h}]] \to B$ tel que $\tilde{f}(x) = f(x)$ pour tout $x \in X$.

\item[4)] Si $\mathcal{R}$ est une partie de $\KK \langle X \rangle[[\bar{h}]]$, l'ensemble des applications $f : X \to B$ telles que $\tilde f$ s'annule sur $\mathcal R$ est en bijection avec l'ensemble des morphismes de $\KK[[\bar{h}]]$-algèbre de $\KK \langle X \rangle[[\bar{h}]] \, / \, \overline{(\mathcal{R})}^{\: \! \bar h}$ vers $B$.
\end{itemize}
\end{proposition}

\subsection{Déformations d'algèbres}
Soit $A$ est une $\KK$-algèbre. Le produit de $A[[\bar h]]$ étend celui de $A$ au sens suivant : l'isomorphisme $\KK$-linéaire canonique entre $A[[\bar{h}]] / (\bar{h} = 0)$ (on quotiente par l'idéal bilatère de $A[[\bar h]]$ engendré par les éléments $h_1, h_2, \dots, h_s$) et $A$ est un isomorphisme de $\KK$-algèbre. \\
Il peut exister d'autres produits sur $A[[\bar{h}]]$ satisfaisant cette propriété, ce qui motive la définition suivante.

\begin{definition}
On appelle déformation formelle (à $s$ paramètres) d'une $\KK$-algèbre $A$ une structure de $\KK[[\bar{h}]]$-algèbre sur le $\KK[[\bar{h}]]$-module $A[[\bar{h}]]$ telle que $1 \in A$ est l'élément neutre de $A[[\bar{h}]]$, et telle que l'isomorphisme $\KK$-linéaire canonique entre $A[[\bar{h}]] / (\bar{h} = 0)$ et $A$ est un isomorphisme de $\KK$-algèbre.
\end{definition}

Soit $V$ un $\KK$-module. La donnée d'une application $\KK[[\bar{h}]]$-bilinéaire $u$ de $V[[\bar{h}]] \times V[[\bar{h}]]$ vers $V[[\bar{h}]]$ est équivalente à la donnée d'une famille $(u_{\bar{n}})_{\bar{n} \in \NN^s}$ d'applications $\KK$-bilinéaires de $V \times V$ dans $V$. Plus explicitement, on a :
$$ u \bigg( \sum_{\bar{p} \in \NN^s} v_{\bar{p}} \, h^{\bar{p}} \, , \sum_{\bar{q} \in \NN^s} w_{\bar{q}} \, h^{\bar{q}} \bigg) \ = \ \sum_{\bar{n} \in \NN^s} \Bigg( \sum_{\bar{p} + \bar{q} + \overline{r} = \bar{n}} u_{\overline{r}} \big( v_{\bar{p}}, w_{\bar{q}} \big) \Bigg) h^{\bar{n}} \, . $$

Notons $\mu$ le produit d'une $\KK$-algèbre $A$ et $\mu_{\bar h}$ une application $\KK[[\bar{h}]]$-bilinéaire sur $A[[\bar{h}]]$. En notant $\mu_{\bar n}$ les applications $\KK$-bilinéaires associées à $\mu_{\bar h}$, l'associativité de $\mu_{\bar h}$ est équivalente à la condition suivante :
\begin{equation} \label{eq_assoc}
\sum_{\bar{p} + \bar{q} = \bar{n}} \mu_{\bar{p}} \Big( \mu_{\bar{q}} (a,b), c \Big) - \, \mu_{\bar{p}} \Big( a, \mu_{\bar{q}}(b,c) \Big) \ = \ 0 \, , \quad \forall \, \bar{n} \in \NN^s, \ \forall \, a,b,c \in A \, .
\end{equation}
Supposons que $\mu_{\bar h}$ munisse $A[[\bar{h}]]$ d'une structure de $\KK[[\bar{h}]]$-algèbre (on suppose donc l'associativité de $\mu_{\bar h}$ et que $1 \in A$ est un élément neutre de $\mu_{\bar h}$), alors cette structure est une déformation formelle de $A$ si et seulement si $\mu_0 = \mu$. \\

On appelle déformation formelle constante la structure de $\KK[[\bar{h}]]$-algèbre canonique sur $A[[\bar{h}]]$, caractérisée par $\mu_{\bar{n}} = 0$ pour $\bar{n} \succ 0$.

\begin{definition}
Deux déformations formelles d'une $\KK$-algèbre $A$ sont dites équivalentes si il existe un isomorphisme de $\KK[[\bar{h}]]$-algèbre $\phi : A[[\bar{h}]] \to A[[\bar{h}]]$ entre les deux structures, qui induit l'identité sur $A[[\bar{h}]] / (\bar{h} = 0)$.
\end{definition}

On dira qu'une déformation formelle est triviale si elle est équivalente à la déformation formelle constante.

\subsection{Déformations de représentations}
\begin{definition}
Soit $A$ une $\KK$-algèbre. On appelle représentation d'une déformation formelle de $A$, que l'on note $(V,\pi_{\bar h})$, la donnée d'un $\KK$-module $V$ et d'un morphisme de $\KK[[\bar{h}]]$-algèbre $\pi_{\bar h} : A[[\bar h]] \to \text{End}_{\KK[[\bar{h}]]} (V[[\bar h]])$, où $A[[\bar{h}]]$ est considéré avec sa déformation formelle. \\
 $(V,\pi_{\bar h})$ est dit de rang fini (resp. de rang $r \in \NN$) si $V$ est un module libre de rang fini (resp. de rang $r$) sur $\KK$.
\end{definition}

On appellera sous-représentation d'une représentation $V[[\bar h]]$ une représentation de la forme $W[[\bar h]]$, avec $W$ un sous-$\KK$-module de $V$. On définit le quotient de $V[[\bar h]]$ par $W[[\bar h]]$ comme étant la représentation $(V/W)[[\bar h]]$. \\

Notons $\mu_{\bar h}$ le produit d'une déformation formelle de $A$. En utilisant le point 1) de la proposition \ref{prop_props}, on voit que la donnée d'une représentation $(V,\pi_{\bar h})$ est équivalente à la donnée d'une famille $\big( \pi_{\bar{n}} \big)_{\bar{n} \in \NN^s}$ d'applications $\KK$-linéaires de $A$ dans $\text{End}_{\KK}(V)$ qui vérifient :
\begin{equation} \label{eq_reprcg}
\sum_{\bar{p} + \bar{q} = \bar{n}} \Big[ \pi_{\bar{p}} \big( \mu_{\bar{q}}(a,b) \big)  \, - \ \pi_{\bar{p}}(a) \circ \pi_{\bar{q}}(b) \Big] \ = \ 0 \, , \quad \ \forall \, a, b \in A \, , \ \forall \, \bar{n} \in \NN^s \, .
\end{equation}
On voit que la représentation $(V,\pi_{\bar h})$ induit une représentation $(V,\pi_0)$ de $A$. On dira que $(V,\pi_{\bar h})$ est une déformation formelle d'une représentation $(V,\pi)$ de $A$ si $\pi_0 = \pi$.  \\

On fera la distinction entre représentation d'une déformation formelle et module d'une déformation formelle : un module d'une déformation formelle $A[[\bar h]]$ de $A$ désignera un $A[[\bar{h}]]$-module. En d'autres mots, une représentation est la donnée d'une structure de $A[[\bar{h}]]$-module sur un espace de la forme $V[[\bar{h}]]$. \\

En utilisant le point 1) encore de la proposition \ref{prop_props}, on voit que deux représentations $(V,\pi_{\bar h})$ et $(W,\rho_{\bar h})$ sont isomorphes en tant que $A[[\bar h]]$-modules, si et seulement si il existe une famille $\big( u_{\bar{n}} \big)_{\bar{n} \in \NN^s}$ d'applications $\KK$-linéaires de $V$ vers $W$ telle que $u_0$ est bijective et telle que
\begin{equation}
\sum_{\bar{p} + \bar{q} = \bar{n}} \Big[ u_{\bar{p}} \circ \pi_{\bar{q}}(a) \, - \ \rho_{\bar{p}}(a) \circ u_{\bar{q}} \Big] \ = \ 0 \, , \quad \ \forall \, a \in A \, , \ \forall \, \bar{n} \in \NN^s \, .
\end{equation}
On voit que si les deux représentations sont isomorphes, $u_0$ est un isomorphisme entre les représentations de $A$ induites $(V,\pi_0)$ et $(W,\rho_0)$. \\
On dira que deux représentations $(V,\pi_{\bar h})$ et $(W,\rho_{\bar h})$ sont équivalentes si l'on ajoute $u_0 = \text{id}$ aux conditions précédentes. \\

Considérons la déformation formelle constante de $A$. Dans ce cas, d'après \eqref{eq_reprcg}, la donnée d'une représentation $(V,\pi_{\bar h})$ est équivalente à la donnée d'une famille $\big( \pi_{\bar{n}} \big)_{\bar{n} \in \NN^s}$ d'applications $\KK$-linéaires de $A$ dans $\text{End}_{\KK}(V)$ qui vérifient
\begin{equation} \label{eq_repr}
\pi_{\bar{n}} (ab) \ = \ \sum_{\bar{p} + \bar{q} = \bar{n}} \, \pi_{\bar{p}}(a) \circ \pi_{\bar{q}}(b) \, , \quad \ \forall \, a, b \in A \, , \ \forall \, \bar{n} \in \NN^s \, .
\end{equation}
La déformation formelle constante d'une représentation $(V,\pi)$ de $A$ est définie par $\pi_0 = \pi$ et $\pi_{\bar{n}} = 0$ pour $\bar{n} \succ 0$. \\
Une déformation formelle sera dite triviale si elle est équivalente à une déformation formelle constante. \\

On peut définir de manière similaire la notion de bimodule $M[[\bar h]]$ d'une déformation formelle $A[[\bar h]]$. Toutes les définitions précédentes s'adaptent naturellement. \\

Notons que si $V$ est une représentation d'une $\KK$-algèbre $A$. Alors $\text{End}_{\KK} V$ est naturellement muni d'une structure de $A$-bimodule :
$$ (a.f) (x) \ := \ a. \big( f(x) \big) \, , \quad \ (f.a) (x) \ := \ f(a.x) \, , \quad \text{ avec } \ f \in \text{End}_{\KK} V, \ a \in A, \ x \in V \, . $$

La démonstration du lemme suivant est immédiate.

\begin{lemme} \label{lem_defor_end}
Soient $A$ une $\KK$-algèbre et $V$ une représentation de $A$. Si $V[[\bar h]$ est une déformation formelle triviale de $V$, alors $\text{End}_{\KK[[\bar h]]} (V[[\bar h]])$ est une déformation formelle triviale du $A$-bimodule $\text{End}_{\KK} V$.
\end{lemme}

Soit $A$ une $\KK$-algèbre et considérons une déformation formelle $A[[\bar h]]$ de $A$. \\
On note $\mathcal{C}(A)$ la catégorie abélienne des représentations de $A$ et $\mathcal{C}^{\bar{h}}(A)$ la catégorie des représentations de la déformation formelle $A[[\bar{h}]]$. \\

On vérifie que $\mathcal{C}^{\bar{h}}(A)$ est une catégorie additive $\KK[[\bar h]]$-linéaire. \\
La somme directe de deux représentations $V[[\bar h]], W[[\bar h]] \in \mathcal{C}^{\bar{h}}(A)$ est le $\CC[[\bar h]]$-module $(V \oplus W)[[\bar h]]$, qu'on munit naturellement d'une structure de représentation de $A[[\bar h]]$. \\
Une représentation indécomposable est un objet indécomposable de la catégorie $\mathcal{C}^{\bar{h}}(A)$, i.e. une représentation qui n'est pas isomorphe à la somme directe de deux représentations. \\

L'existence d'un conoyau d'un morphisme $f : V[[\bar{h}]] \to W[[\bar{h}]]$ dans $\mathcal{C}^{\bar{h}}(A)$ n'est en général pas vérifiée : par exemple si $V = W = \KK$ et $f$ est la multiplication par $h$. Par conséquent $\mathcal{C}^{\bar{h}}(A)$ n'est pas une catégorie abélienne. \\

La catégorie dont les objets sont les $\KK[[\bar h]]$-modules de la forme $V[[\bar h]]$, et les flèches les morphismes de $\KK[[\bar h]]$-module, est d'après ce qui précède une catégorie additive $\KK[[\bar h]]$-linéaire (c'est le cas particulier où $A = \KK$ et $\KK[[\bar h]]$ est muni de la déformation formelle constante). On montre que cette catégorie est une catégorie tensorielle, où le produit tensoriel est défini pour $V, W$ deux $\KK$-modules, par
$$ V[[\bar h]] \ \tilde{\otimes} \ W[[\bar h]] \ := \ (V \otimes_{\KK} W) [[\bar h]] \, . $$
$V[[\bar h]] \ \tilde{\otimes} \ W[[\bar h]]$ est un complété $\bar h$-adique de $V[[\bar h]] \otimes_{\KK[[\bar h]]} W[[\bar h]]$. \\

Notons $\mathcal{A}_{\KK}$ la catégorie abélienne des $\KK$-algèbres et $\mathcal{A}_{\KK}^{\bar{h}}$ la catégorie $\KK[[\bar h]]$-linéaire additive, dont les objets sont les déformations formelles de $\KK$-algèbres, et les flèches les morphismes de $\KK[[\bar h]]$-algèbre.  \\
On définit deux foncteurs $\KK[[\bar h]]$-linéaires
$$ \mathcal{Q}^{\bar{h}} \, : \ \mathcal{A}_{\KK} \ \longrightarrow \ \mathcal{A}_{\KK}^{\bar{h}} \quad \text{ et } \quad \lim_{\bar{h} \to 0} \, : \ \mathcal{A}_{\KK}^{\bar{h}} \longrightarrow \ \mathcal{A}_{\KK} \, , $$
qui, respectivement, à une $\KK$-algèbre $A$ associe sa déformation formelle constante, et à une déformation formelle de $A$ associe $A$ (les définitions de $\mathcal{Q}^{\bar{h}}$ et $\lim_{\bar{h} \to 0}$ pour les morphismes sont les définitions naturelles). Remarquons que
$$ \lim_{\bar{h} \to 0} \, \circ \ \mathcal{Q}^{\bar{h}} \ = \ \text{id}_{\! \mathcal{A}_{\KK}} \, . $$

On définit de façon similaire un foncteur $\KK[[\bar h]]$-linéaire de la catégorie des représentations d'une déformation formelle de $A$ vers la catégorie des représentations de $A$ :
$$ \lim_{\bar{h} \to 0} : \ \mathcal{C}^{\bar{h}}(A) \, \longrightarrow \ \mathcal{C}(A) \, . $$


Le foncteur $\lim_{\bar h \to 0}$, que ce soit dans les cadre des algèbres ou celui des représentations, sera appelé limite classique.

\subsection{Cohomologie de Hochschild}
\begin{definition}
Soient $A$ une $\KK$-algèbre et $M$ un $A$-bimodule. \\
Pour tout $n \in \NN$, on note $C_{\KK}^n(A,M)$, et on appelle l'espace des $n$-cochaînes, le $\KK$-module des applications $\KK$-multilinéaires de $A^n$ dans $M$ (on a par convention $C_{\KK}^0(A,M) = M$). \\
On définit une application $\KK$-linéaire $d^{\: \! n} : C_{\KK}^n(A,M) \to C_{\KK}^{n+1}(A,M)$ par
\begin{eqnarray*}
(d^n f) (a_1, \dots, a_{n+1}) & := & \ \ \quad a_1 . f(a_2, \dots, a_{n+1}) \\
&& + \ \sum_{i=1}^n (-1)^i \, f(a_1, \dots, a_i \: \! a_{j+1}, \dots, a_{n+1}) \\
&& + \ (-1)^{n+1} f(a_1, \dots, a_n) . a_{n+1} \,
\end{eqnarray*}
(pour $n=0$ et $x \in M$, on a $(d^{\: \! 0} x) \: \! a = a.x - x.a$). \\
On appelle, et on note
$$ \Big( C_{\KK}^{\bullet} (A,M),d \Big) := \Bigg( \bigoplus_{n \in \NN} \, C_{\KK}^n(A,M) \, , \, \sum_{n \in \NN} \, d^{\: \! n} \Bigg) \, , $$
le complexe de Hochschild de $A$ à coefficients dans $M$ (on vérifie en effet que  $d \circ d = 0$). \\
La cohomologie de $C_{\KK}^{\bullet}(A,M)$ est notée $H_{\KK}^{\bullet}(A,M)$ et appelée la cohomologie de Hochschild de $A$ à coefficients dans $M$.
\end{definition}

\begin{remarque} \label{rem_hoch}
Fixons la $\KK$-algèbre $A$. La cohomologie de Hochschild $H^{\bullet}_{\KK}(A, \cdot)$ définit un foncteur de la catégorie des $A$-bimodules vers la catégorie des $\KK$-modules $\NN$-gradués.
\end{remarque}

Le lemme suivant exprime le fait que la cohomologie de Hochschild commute en quelque sorte avec le processus de déformation formelle constante. Sa démonstration est immédiate.

\begin{lemme} \label{lem_hoch}
Soient $A$ une $\KK$-algèbre et $M$ un $A$-bimodule. En considérant $A[[\bar h]]$ et $M[[\bar h]]$ comme déformations formelles constantes, on a pour tout $n \in \NN$ un isomorphisme canonique $\KK[[\bar h]]$-linéaire :
$$ \Big( H_{\KK}^n(A,M) \Big) [[\bar h]] \ \simeq \ H_{\KK[[h]]}^n \big( A[[\bar h]],M[[\bar h]] \big) \, . $$
\end{lemme}

\begin{corollaire} \label{cor_hoch}
Soient $A$ une $\KK$-algèbre et $A[[\bar h]]$ une déformation formelle triviale de $A$. Il existe pour tout $n \in \NN$ un isomorphisme $\KK[[\bar h]]$-linéaire
$$ \Big( H_{\KK}^n(A,A) \Big) [[\bar h]] \ \simeq \ H_{\KK[[h]]}^n \big( A[[\bar h]],A[[\bar h]] \big) \, . $$
\end{corollaire}

\begin{proof}
$A[[\bar h]]$ étant isomorphe en tant que $\KK[[\bar h]]$-algèbre à la déformation formelle constante de $A$, leurs cohomologies de Hochschild sont isomorphes. On conclut grâce au lemme \ref{lem_hoch}.
\end{proof}

\begin{corollaire} \label{cor_hoch2}
Soient $A$ une $\KK$-algèbre et $M[[\bar h]]$ une déformation formelle triviale d'un $A$-bimodule $M$. Il existe pour tout $n \in \NN$ un isomorphisme $\KK[[\bar h]]$-linéaire
$$ \Big( H_{\KK}^n(A,M) \Big) [[\bar h]] \ \simeq \ H_{\KK[[h]]}^n \big( A[[\bar h]],M[[\bar h]] \big) \, , $$
où $A[[\bar h]]$ désigne la déformation formelle constante de $A$.
\end{corollaire}

\begin{proof}
Conséquence du lemme \ref{lem_hoch} et de la remarque \ref{rem_hoch}.
\end{proof}

L'introduction du langage cohomologique s'avère intéressante comme en témoigne le théorème suivant, généralisation de \cite[2.5.3]{Gui}. On vérifie que la démonstration dans \cite{Gui} fonctionne encore quand on travaille sur un anneau commutatif $\KK$ plutôt que sur $\CC$, puis on donne les arguments qui permettent de passer du cas d'un paramètre de déformation au cas de plusieurs paramètres.

\begin{proposition} \label{prop_hoch}
Soit $A$ une $\KK$-algèbre et $s \geq 1$ un entier. Si $H_{\KK}^2(A,A) = 0$, alors toute déformation formelle à $s$ paramètres de $A$ est triviale.
\end{proposition}

\begin{proof}
Nous allons procéder par récurrence sur $s$. Le cas $s = 1$ est l'énoncé de	\cite[2.5.3]{Gui}, la seule différence étant que l'on travaille ici sur un anneau commutatif $\KK$ (ce qui sera nécessaire pour la récurrence) et non sur $\CC$. On commence par montrer, sans supposer $H_{\KK}^2(A,A) = 0$, que si $\mu_1 = \mu_2 = \cdots = \mu_{n-1} = 0$ (avec $n \in \NN_{\geq 1}$), et si $\mu_n$, qui est un $2$-cocycle en vertu de \eqref{eq_assoc}, est un cobord, alors il existe $\mu_h'$ équivalent à $\mu_h$ et vérifiant $\mu'_1 = \mu'_2 = \cdots = \mu'_n = 0$. \\
Il existe en effet par hypothèse $f \in \text{End}_{\KK} A$ tel que
$$ \mu_n (a,b) \ = \ a \: \! f(b) \ - \ f(ab) \ + \ f(a) \: \! b \, . $$
On vérifie que l'automorphisme $\KK[[h]]$-linéaire $\text{id}_A - h^n f$ conjugue $\mu_h$ avec une déformation ${\mu}'_h$ telle que
$$ \mu_h' (a,b) - ab \ \equiv \ 0 \quad \text{mod } h^{n+1} \, . $$
En supposant à présent $H_{\KK}^2(A,A) = 0$ et $\mu_h$ quelconque, d'après ce qui précède on voit qu'il existe pour tout $p \in \NN_{\geq 1}$ un automorphisme $f_p$, tels que que le produit
\begin{equation*}
v_p \ := \ (\text{id}_A - h^p f_p) \circ \cdots \circ (\text{id}_A - h^2 f_2) \circ (\text{id}_A - h f_1)
\end{equation*}
conjugue $\mu_h$ avec une déformation $\mu_h^{(p)}$ telle que $\mu^{(p)}_1 = \mu^{(p)}_2 = \cdots = \mu^{(p)}_p = 0$. On conclut en remarquant que $v_p$ converge lorsque $p$ tend vers l'infini. \\

Supposons le théorème démontré pour $s \geq 1$ et considérons une déformation formelle à $s+1$ paramètres $A[[h_1, h_2 \dots, h_{s+1}]]$ de $A$. On note $\mu_{\bar h}$ le produit de cette déformation. \\
En considérant l'isomorphisme d'anneau canonique
$$ \KK[[h_1, h_2 \dots, h_{s+1}]] \ \simeq \ \Big( \KK[[h_1, h_2 \dots, h_s]] \Big) [[h_{s+1}]] $$
et l'isomorphisme linéaire canonique
$$ A[[h_1, h_2 \dots, h_{s+1}]] \ \simeq \ \Big( A[[h_1, h_2 \dots, h_s]] \Big) [[h_{s+1}]] $$
on voit que le produit $\mu_{\bar h}$ de $A[[h_1, h_2 \dots, h_{s+1}]]$ induit
\begin{itemize}
\item[\textbullet] sur $A[[h_1, h_2 \dots, h_s]]$ une structure de $\KK[[h_1,h_2 \dots, h_s]]$-algèbre qui est une déformation formelle à $s$ paramètres de $A$;
\item[\textbullet] sur $\Big( A[[h_1, h_2 \dots, h_s]] \Big) [[h_{s+1}]]$ une structure de déformation à un paramètre de la $\KK[[h_1, h_2, \dots, h_s]]$-algèbre $A[[h_1, h_2 \dots, h_s]]$ définie ci-dessus.
\end{itemize}
L'hypothèse de récurrence et le corollaire \ref{cor_hoch} permettent alors de conclure.
\end{proof}

Le théorème suivant généralise celui présenté dans \cite[III, 2.5.6]{Gui}. De la même façon que pour la proposition \ref{prop_hoch}, la démonstration est une généralisation naturelle dans le cas $s=1$ et nécessite un argument supplémentaire pour passer au cas de plusieurs paramètres de déformation.

\begin{proposition} \label{prop_hochrep}
Soient $A$ une $\KK$-algèbre et $(V, \pi_{\bar h})$ une représentation de la déformation formelle constante $A[[\bar h]]$ de $A$. Si $H_{\KK}^1(A,\text{End}_{\KK} V) = 0$, alors $(V, \pi_{\bar h})$ est une déformation formelle triviale de $\pi_0$.
\end{proposition}

\begin{proof}
Le cas $s=1$ est l'énoncé du théorème dans \cite{Gui}. Encore une fois, on vérifie que la démonstration se généralise à un anneau commutatif $\KK$ : en vertu de \eqref{eq_repr} et en supposant $H_{\KK}^1(A,\text{End}_{\KK} V) = 0$, il existe $f_1 \in \text{End}_{\KK} V$ tel que $\pi_1 = d (f_1)$. On conjugue alors $\pi_{\bar h}$ par l'endomorphisme $\text{id}_V + h f_1$. On continue le procédé de la même manière que dans la preuve de la proposition \ref{prop_hoch}. \\
Le lemme \ref{lem_defor_end} et le corollaire \ref{cor_hoch2} permettent de faire une induction sur $s$, de la même façon que dans la démonstration de la proposition \ref{prop_hoch}.
\end{proof}

\section{Définition} \label{section_def}
Dans cette section on définit, dans l'esprit de Frenkel-Hernandez \cite{FH1}, les groupes quantiques d'interpolation de Langlands de rang 1. Si le principe d'interpolation est repris ici, les formules qui l'expriment ne sont toutefois pas les mêmes. Plus précisément, la différence entre les définitions données ici et celles présentes dans \cite{FH1} ne se situe pas seulement au niveau du langage utilisé. Les formules que l'on découvrira ci-dessous sont de prime abord moins intuitives, mais permettent toutefois de donner une présentation plus concise et peut-être moins compliquée. L'intérêt de cette nouvelle définition n'est pas seulement esthétique, dans la mesure où elle permettra de démontrer dans les sections suivantes plusieurs propriétés nouvelles des groupes quantiques d'interpolation et de leurs théorie des représentations. \\

On sait classifier et décrire explicitement toutes les représentations irréductibles sur $\CC$ de dimension finie $L(n)$ ($n \in \NN$) de $\mathfrak{sl}_2$ : voir par exemple \cite[II.7]{humphreys}. Il existe une base de $L(n)$ dont on note les vecteurs $v_0, v_1, \dots, v_n$, pour laquelle l'action de $\mathfrak{sl}_2$ est donnée par les formules ($0 \leq i \leq n$):
$$ H. v_i \ = \ (n- 2i) \: \! v_i \, , \quad X^+.v_i \ = \ (n-i+1) \: \! v_{i-1} \, , \quad X^-.v_i \ = \ (i+1) \: \! v_{i+1} \, , $$
(on pose $v_{-1} := v_{n+1} := 0$).

$L(n)$ se déforme en une représentation $L^h(n)$ de $U_h(\mathfrak{sl}_2)$ (en tant que $\CC[[h]]$-module $L^h(n) = L(n) [[h]]$) :
$$ H. v_i \ = \ (n- 2i) v_i \, , \quad X^+.v_i \ = \ [n-i+1]_Q \, v_{i-1} \, , \quad X^-.v_i \ = \ [i+1]_Q \, v_{i+1} \, , $$
où $[a]_Q := \frac{Q^a - Q^{-a}}{Q - Q^{-1}}$ pour $a \in \ZZ$ et $Q := \exp h \in \CC[[h]]$ (la série formelle $[a]_Q$ est souvent appelée nombre quantique). \\

Le point de départ dans la définition des groupes quantiques d'interpolation de $\mathfrak{sl}_2$ est de déformer une nouvelle fois, par rapport à un paramètre $h'$, l'action de $H$ et $X^{\pm}$ sur $L^h(n)$, en une action sur le $\CC[[h,h']]$-module $L(n) [[h,h']]$. Contrairement à la précédente déformation par rapport à $h$, ces déformations dépendent d'un paramètre $g \in \NN_{\geq 1}$, de telle sorte que l'on a non pas une déformation mais une famille de déformations. Pour $g$ fixé, les actions déformées des générateurs $H$ et $X^{\pm}$ vont vérifier de nouvelles relations (indépendantes de $n$), déformations des relations de $U_h(\mathfrak{sl}_2)$. C'est ces nouvelles relations qui seront alors posées dans la définition du groupe quantique d'interpolation $U_{h,h'}(\mathfrak{sl}_2,g)$. On notera $L^{h,h'}(n,g)$ l'espace $L(n) [[h,h']]$ considéré comme représentation de l'algèbre $U_{h,h'}(\mathfrak{sl}_2,g)$. \\

On peut remarquer la chose suivante : lors de la quantification de $L(n)$, on a en quelque sorte remplacé les nombres $a$ par les nombres quantiques $[a]_Q$. De manière similaire, lors de la déformation par rapport à $h'$, on va remplacer les nombres quantiques $[a]_Q$ par de nouveaux nombres quantiques $[a]_{QT^{\{a\}}}$ :
\begin{equation}
[a]_{QT^{\{a\}}} \ := \ \frac{\left( Q \, T^{\{a\}} \right)^a \ - \ \left( Q \, T^{\{a\}} \right)^{-a}}{Q \, T^{\{a\}} - Q^{-1} \, T^{-\{a\}}} \, ,
\end{equation}
où $T := \exp h' \in \CC[[h']]$ et $\{a\}$ est un polynôme de Laurent en $Q^a$ que l'on décrira un peu plus loin. \\

De manière plus précise, l'action déformée par rapport à $h$ et $h'$ est décrite par les formules suivantes ($0 \leq i \leq n$):
$$ H. v_i \ = \ (n- 2i) v_i \, , $$
$$ X^+.v_i \ = \ \left( [n-i+1]_{QT^{\{n-i+1\}}} \right) v_{i-1} \, , \quad X^-.v_i \ = \ \left( [i+1]_{QT^{\{i+1\}}} \right) v_{i+1} \, . $$
En particulier, on voit que $X^{\pm} X^{\mp} . v_i = \pi_i^{\pm} \: \! v_i$ avec $\pi_i^{\pm}$ produit de deux nombres quantiques. Avant d'aller plus loin dans les explications, donnons la définition des groupes quantiques d'interpolation $U_{h,h'}(\mathfrak{sl}_2,g)$.

\begin{definition} \label{def_uhh}
Soit $g \geq 1$ un entier. \\
$U_{h,h'}(\mathfrak{sl}_2,g)$ est la $\CC[[h,h']]$-algèbre topologiquement engendrée par $X^+$, $X^-$, $H$, $C$ et définie par les relations
$$ [C, X^{\pm}] \ = \ [C, H] \ = \ 0 \, , \quad \quad [H, X^{\pm}] \ = \ \pm 2 \, X^{\pm} \, , $$

\begin{equation} \label{eq_defx}
X^\pm X^\mp \ = \ \xxp{\pm}{} \, ,
\end{equation}

avec
$$ Q \ := \ \text{exp}(h) \, , \quad  T \ := \ \text{exp} (h') \, , $$
$$ H_e^{\pm} \ := \ \frac{1}{2} \, \Big( \sqrt{C} + e \: \! H \mp e \Big) \, , \ \quad \quad {\{ H_e^{\pm} \}}_{\! \! \: Q} \ := \ P \! \left( Q^{H_e^{\pm}} \right) , \quad \quad e \in \{ -1,1 \} \, , $$
\begin{equation} \label{eq_poly}
P(u) \ := \ \frac{1}{2} \left( u^{g-1} + u^{1-g} \right) \ \prod_{k=1}^{g-1} \, \frac{\varepsilon^k \: \! u - \varepsilon^{-k} \: \! u^{-1}}{\varepsilon^k - \varepsilon^{-k}} \, .
\end{equation}
\end{definition}

Soit $P_g$ le polynôme interpolateur de Lagrange de degré $g-1$ vérifiant
$$ P_g(1) \ = \ 1 \, , \quad P_g(\varepsilon^2) \ = \ \cdots \ = \ P_g(\varepsilon^{2g-2}) \ = \ 0 \, , $$
où $\varepsilon := \exp (i\pi /g) \in \CC$ est une racine $(2g)$-ième primitive de l'unité. \\
Le polynôme de Laurent $P$ défini dans \eqref{eq_poly} peut s'obtenir à partir de $P_g$ par symétrisation (remarquons que $P$ est invariant par les deux transformations $u \mapsto -u$ et $u \mapsto u^{-1}$) :
$$ P(u) \ = \ \frac{1}{2} \Big( P_g(u^2) \ + \ P_g(u^{-2}) \Big) \, . $$
Par conséquent le polynôme $P$ vérifie
\begin{equation} \label{eq_valp}
P(\varepsilon^l) \ = \ 1 \quad \text{si $g$ divise $l$} \, , \quad \ P(\varepsilon^l) \ = \ 0 \quad \text{sinon} \, .
\end{equation}

L'élément $C$ dans $U_{h,h'}(\mathfrak{sl}_2,g)$ doit en quelque sorte être compris comme l'opérateur de multiplication par $(n+1)^2$ sur $L^{h,h'} (n,g)$.
Dans la relation \eqref{eq_defx}, on peut en quelque sorte reconnaître à droite de l'égalité le produit $\pi_i^{\pm}$. En effet l'opérateur $H_e^{\pm}$ sur $L^{h,h'} (n,g)$ est décrit par ($0 \leq i \leq n$) :
$$ H_1^+ . v_i \ = \ n-i \, , \quad H_1^- . v_i \ = \ n-i +1 \, , \quad H_{-1}^+ . v_i \ = \ i+1 \, , \quad H_{-1}^- . v_i \ = \ i \, . $$

La relation \eqref{eq_defx} contient pour ainsi dire la définition des polynômes $\{a\}$ ($a \in \ZZ$) :
\begin{equation}
\{a\} \ := \ P \! \left( Q^a \right) \ = \ \frac{1}{2} \left( Q^{a(g-1)} + Q^{a(1-g)} \right) \ \prod_{k=1}^{g-1} \, \frac{\varepsilon^k \: \! Q^a - \varepsilon^{-k} \: \! Q^{-a}}{\varepsilon^k - \varepsilon^{-k}} \, .
\end{equation}
Ces polynômes $\{a\}$ possèdent la propriété suivante : en spécialisant $Q$ à $\varepsilon$ on obtient d'après \eqref{eq_valp}
$$ \{a\} = 1 \quad \text{si $g$ divise $a$} \quad \text{ et } \quad \{a\}= 0 \quad \text{sinon}. $$

C'est cette propriété des polynômes $\{a \}$ qui permettra de démontrer que $L^{h,h'}(n,g)$ interpole la représentation $L^h(n)$ du groupe quantique $\Uh[\slt]$ et la représentation Langlands $g$-duale ${}^L \! L^{h'}(n,g)$ de $\Ughp$ : voir la section \ref{section_langlands} et le théorème \ref{thm_reprinter}. \\

Il faut préciser deux points dans la définition de $\Uhhp$.
\begin{itemize}
\item[\textbullet] On peut considérer
$$ \xxn{e}{\pm}{}{} \quad \text{ et } \quad \xxd{e}{\pm}{}{} $$
comme des séries formelles dans $\CC \left[H,\sqrt{C} \right] [[h,h']]$, où $H$ et $\sqrt C$ désignent ici les variables de l'algèbre de polynômes $\CC \left[H,\sqrt{C} \right]$. Les deux séries formelles sont divisibles par $h + h' {\{ H_e^{\pm} \}}_{\! Q}$. Le quotient
$$ \xxf{e}{\pm}{}{} $$
est alors défini dans $\CC \left[H,\sqrt{C} \right] [[h,h']]$ comme le quotient des deux séries formelles précédentes chacune divisée par $h + h' {\{ H_e^{\pm} \}}_{\! Q}$ (de telle sorte que la série au dénominateur devient alors inversible).
\item[\textbullet] En considérant le produit
$$ \Pi \ := \ \xxp{\pm}{} $$
comme un élément de l'algèbre $\CC \left[H,\sqrt{C} \right] [[h,h']]$, on vérifie que $\Pi$ est invariant par l'unique automorphisme de la $\CC[[h,h']]$-algèbre $\CC \left[H,\sqrt{C} \right] [[h,h']]$ qui envoie $H$ sur $H$ et $\sqrt C$ sur $- \sqrt C$ (car le polynôme $P$ défini dans \eqref{eq_poly} est invariant par la transformation $u \mapsto u^{-1}$). Cela prouve que $\Pi$ appartient en fait à la sous-algèbre $\CC \left[H,\left( \sqrt{C} \right)^2 \right] [[h,h']]$.
\end{itemize}

Dans la relation \eqref{eq_defx}, le terme à droite de l'égalité est l'image de $\Pi$ par le morphisme de $\CC[[h,h']]$-algèbre qui à $H$ associe $H$ et à $\left( \sqrt{C} \right)^2$ l'élément $C$. \\
$\sqrt C$ ne désigne donc pas un élément de l'algèbre $U_{h,h'}(\mathfrak{sl}_2,g)$. On utilisera souvent cette écriture dans la suite en omettant toutefois de faire les vérifications semblables à celles faites précédemment.

\begin{remarque} \label{rem_hernandezgen}
Dans \cite{FH1}, les groupes quantiques d'interpolation nécessitent d'autres générateurs que $X^{\pm}$, $H$ et $C$ (on verra dans la section suivante que l'on peut en fait se passer de $C$ pour définir $\Uhhp$). De plus, il est nécessaire de considérer des puissances de certains des générateurs par d'autres générateurs, ce à quoi on peut donner un sens quitte à rajouter au moins autant de nouveaux générateurs qu'il n'y a de telles puissances. Ces algèbres sont en quelque sorte beaucoup plus grosses que les algèbres $\Uhhp$, donc moins maniables, et plusieurs des résultats qui suivront dans les sections suivantes n'ont pas d'analogues dans \cite{FH1}.
\end{remarque}

$\Uc[\slt]$ désignera l'algèbre enveloppante de $\mathfrak{sl}_2$ : c'est la $\CC$-algèbre engendrée par $\bar{X}^{\pm}$, $\bar{H}$, et définie par les relations
$$ [\bar{H}, \bar{X}^{\pm}] \ = \ \pm 2 \, \bar{X}^{\pm} \, , \quad [\bar{X}^+, \bar{X}^-] \ = \ \bar{H} \, . $$
Le centre de $\Uc[\slt]$ est l'algèbre des polynômes en l'élément de Casimir
\begin{equation} \label{eq_casimir}
\bar{C} \ := \ (\bar{H} + 1)^2 \ + \ 4 \, \bar{X}^- \bar{X}^+ \ = \ (\bar{H} - 1)^2 \ + \ 4 \, \bar{X}^+ \bar{X}^- \, ,
\end{equation}
voir par exemple \cite[p. 304]{knapp}. \\

Rappelons par ailleurs la définition du groupe quantique $U_{\: \! h}(\mathfrak{sl}_2)$ (qui est une déformation formelle triviale de $\Uc[\slt]$ : voir par exemple \cite[6.4]{ChPr}). \\
Il s'agit de la $\CC[[h]]$-algèbre topologiquement engendrée par les éléments $\tilde{X}^{\pm}$, $\tilde H$, et définie par les relations
$$ [\tilde{H}, \tilde{X}^{\pm}] \ = \ \pm 2 \, \tilde{X}^{\pm} \, , \quad [\tilde{X}^+, \tilde{X}^-] \ = \ \frac{\text{sinh} (h \: \! \tilde{H})}{\text{sinh}(h)} \, . $$

On notera par ailleurs $U_{r h}(\mathfrak{sl}_2)$ ($r \in \NN_{\geq 1}$) la $\CC[[h]]$-algèbre topologiquement engendrée par les éléments $\tilde{X}^{\pm}$, $\tilde{H}$, et définie par les relations
$$ [\tilde{H}, \tilde{X}^{\pm}] \ = \ \pm 2 \, \tilde{X}^{\pm} \, , \quad [\tilde{X}^+, \tilde{X}^-] \ = \ \frac{\text{sinh} (r h \: \! \tilde{H})}{\text{sinh}(r h)} \, . $$
En d'autres mots, $U_{r h}(\mathfrak{sl}_2)$ est simplement obtenue à partir de $U_{\: \! h}(\mathfrak{sl}_2)$ par renormalisation du paramètre $h$.

\begin{exemple} \label{ex_g1}
Dans la définition \ref{def_uhh}, considérons le cas où $g = 1$. On a $P = 1$ pour le polynôme d'interpolation définie dans \eqref{eq_poly}. Par suite $U_{h,h'}(\mathfrak{sl}_2,1)$ est la $\CC[[h,h']]$-algèbre topologiquement engendrée par $X^+$, $X^-$, $H$, $C$ et définie par les relations
$$ [C, X^{\pm}] \ = \ [C, H] \ = \ 0 \, , $$
$$ [H, X^{\pm}] \ = \ \pm 2 \, X^{\pm} \, , $$
$$ X^\pm X^\mp \ = \ \prod_{e = 1,-1} \frac{\left( Q \, T \right)^{H_e^{\pm}} - \ \left( Q \, T \right)^{-H_e^{\pm}}}{Q \, T \ - \ Q^{-1} \, T^{-1} } \, , $$
avec
$$ H_e^{\pm} \ := \ \frac{1}{2} \, \Big( \sqrt{C} + e \: \! H \mp e \Big) \, , \ \quad  \text{où } \ e \in \{ -1,1 \} \, . $$
La relation suivante est vérifiée dans $U_{h,h'}(\mathfrak{sl}_2,1)$ :
$$ [X^+, X^-] \ = \ \frac{\left( QT \right)^H - \left( QT \right)^{-H}}{QT - Q^{-1} T^{-1}} \, . $$
Par conséquent, $U_{h,h'}(\mathfrak{sl}_2,1)$ contient un quotient de la $\CC[[h'']]$-algèbre $U_{h''}(\mathfrak{sl}_2)$ (toute $\CC[[h,h']]$-algèbre, et en particulier $U_{h,h'}(\mathfrak{sl}_2,1)$, est naturellement munie d'une structure de $\CC[[h'']]$-algèbre où $h''$ agit par $h + h'$). D'après la proposition \ref{prop_basehh} énoncée un peu plus loin, on voit que $U_{h''}(\mathfrak{sl}_2)$ est en fait une sous-algèbre de $U_{h,h'}(\mathfrak{sl}_2,1)$. En notant $h^{(3)} = h - h'$, $U_{h,h'}(\mathfrak{sl}_2,1)$ apparaît comme la déformation constante de $U_{h''}(\mathfrak{sl}_2)$ selon $h^{(3)}$. \\
Cette situation est particulière à $g = 1$. Pour $g \geq 2$, $U_{h,h'}(\mathfrak{sl}_2,g)$ n'est plus une déformation formelle constante selon $h^{(3)}$ du groupe quantique usuel. En particulier, on remarque que la symétrie entre $h$ et $h'$, observée pour $g=1$, est brisée pour $g \geq 2$.
\end{exemple}

\begin{remarque}
Les groupes quantiques d'interpolation $\Uhhp$ ($g = 1,2,3$) peuvent être comparés aux groupes quantiques d'interpolation élémentaires de Frenkel-Hernandez \cite{FH1}. Ces derniers permettent de définir un groupe quantique d'interpolation de Langlands associé à une algèbre de Lie simple de dimension finie. Dans cet article, la définition donnée est valable quelque soit $g$, ce qui permettra d'associer des groupes d'interpolation de Langlands à toute algèbre de Kac-Moody (symétrisable).
\end{remarque}

\section{Propriétés} \label{section_prop}
On étudie ici les groupes quantiques d'interpolation de Langlands $\Uhhp$ définis précédemment. La section est divisée en cinq sous-sections. Dans la première, on montre que les groupes quantiques d'interpolation sont des doubles déformations formelles triviales de $\Uc[\slt]$, et plus encore, des déformations triviales selon $h$ (ou $h'$) du groupe quantique usuel associé : voir le théorème \ref{thm_uhhdefor}. Il s'agit là d'un fait nouveau et d'un point clé de l'article. Les deux sous-sections qui suivent donnent différentes propriétés des groupes quantiques d'interpolation, obtenues grâce aux structures de double déformation de ceux-ci. La plupart, comme expliqué dans la remarque \ref{rem_hernandezgen}, n'ont pas d'analogues dans \cite{FH1}. On prouve notamment l'existence d'une décomposition triangulaire de $\Uhhp$ (voir la proposition \ref{prop_triang}), qui sera un outil efficace, par exemple pour l'étude des représentations. Dans la quatrième sous-section, on donne quelques propriétés de symétrie de $\Uhhp$. Enfin, on définit la spécialisation de $\Uhhp$ à $Q = \varepsilon$, et on établit un lemme crucial pour notre étude : voir le lemme \ref{lem_fond} et la remarque \ref{rem_fond}.

\subsection{Doubles déformations}
La proposition un peu technique qui suit est fondamentale, elle contient en quelque sorte le noyau dur de la preuve que $\Uhhp$ admet des structures de déformations de $\Uc[\slt]$ et du groupe quantique usuel : voir le théorème \ref{thm_uhhdefor}.

\begin{proposition} \label{prop_uconst}
Il existe un isomorphisme de $\CC[[h,h']]$-algèbre
$$ \psi \, : \ U_{h,h'}(\mathfrak{sl}_2,g) \ \stackrel{\sim}{\longrightarrow} \, \ \Uc[\slt][[h,h']] \, , $$
où $\Uc[\slt][[h,h']]$ désigne la déformation formelle constante de $\Uc[\slt]$, et tel que
$$ \psi(H) \ = \ \bar{H} \, , \quad \psi(C) \ = \ \bar{C} \, , \quad \psi(X^-) \ = \ \bar{X}^- \,  , $$
\begin{equation} \label{eq_uconst}
\psi(X^+) \ = \ \left( \prod_{e = 1, -1} \ \frac{\Big( Q \, T^{{\{ \bar{H}_e^+ \}}_Q} \Big)^{\bar{H}_e^+} - \ \Big( Q \, T^{{\{ \bar{H}_e^+ \}}_Q} \Big)^{-\bar{H}_e^+}}{\bar{H}_e^+ \Big( Q \, T^{{\{ \bar{H}_e^+ \}}_Q} \ - \ Q^{-1} \, T^{-{\{ \bar{H}_e^+ \}}_Q} \Big)} \right) \bar{X}^+ \, ,
\end{equation}
\end{proposition}

Avant de donner la démonstration de cette proposition, commençons par expliquer le terme à droite de l'égalité \eqref{eq_uconst}. \\
$\bar{H}_e^{\pm}$ est défini comme dans la définition \ref{def_uhh} :
$$ \bar{H}_e^{\pm} \ := \ \frac{1}{2} \, \Big( \sqrt{\bar{C}} + e \: \! \bar{H} \mp e \Big) \, , \ \quad  \ e \in \{ -1,1 \} \, . $$
Dans $\CC[H,\sqrt{\bar C}][[h,h']]$, la série formelle
\begin{equation*}
\Big( Q \, T^{{\{ \bar{H}_e^+ \}}_Q} \Big)^{\bar{H}_e^+} - \ \Big( Q \, T^{{\{ \bar{H}_e^+ \}}_Q} \Big)^{-\bar{H}_e^+} \ = \ 2 \: \! \sinh \! \Bigg[ \! \left( h + h' { \{ \bar{H}^+_e \} }_{\! Q} \right) \! \bar{H}^+_e \Bigg]
\end{equation*}
est divisible par $\left( h + h' { \{ \bar{H}^+_e \} }_{\! Q} \right) \! \bar{H}^+_e$. La série
$$ Q \, T^{{\{ \bar{H}_e^+ \}}_Q} \ - \ Q^{-1} \, T^{-{\{ \bar{H}_e^+ \}}_Q} \ = \ 2 \: \! \sinh \! \left( h + h' { \{ \bar{H}^+_e \} }_{\! Q} \right) $$
est quant à elle divisible par $h + h' { \{ \bar{H}^+_e \} }_{\! Q}$ dans $\CC[H,\sqrt{\bar C}][[h,h']]$. Le quotient
$$ \bar{F} \ := \ \left( \prod_{e = 1, -1} \ \frac{\Big( Q \, T^{{\{ \bar{H}_e^+ \}}_Q} \Big)^{\bar{H}_e^+} - \ \Big( Q \, T^{{\{ \bar{H}_e^+ \}}_Q} \Big)^{-\bar{H}_e^+}}{\bar{H}_e^+ \Big( Q \, T^{{\{ \bar{H}_e^+ \}}_Q} \ - \ Q^{-1} \, T^{-{\{ \bar{H}_e^+ \}}_Q} \Big)} \right) $$
a donc un sens dans $\CC[H,\sqrt{\bar C}][[h,h']]$ une fois les divisions faites. De la même manière que dans la définition \ref{def_uhh}, on vérifie que $\bar F$ appartient en fait à la sous-algèbre $\CC[H,(\sqrt{\bar C})^2][[h,h']]$. \\
Le terme à droite de l'égalité dans \eqref{eq_uconst} est défini comme l'image de $\bar F$ par l'unique morphisme de $\CC[[h,h']]$-algèbre qui envoie $\bar H$ sur $\bar H$ et $(\sqrt{\bar C})^2$ sur $\bar C$.

\begin{proof}
Pour prouver l'existence du morphisme $\psi$, il s'agit, d'après le point 4) de la proposition \ref{prop_props}, de voir que
$$ \Big( \psi(X^+), \psi(X^-), \psi(H), \psi(C) \Big) $$
vérifient les relations définissant $U_{ \: \! h,h'} (\mathfrak{sl}_2)$.

\begin{itemize}
\item[\textbullet] Remarquons que $\psi$ envoie $C$ sur l'élément central $\bar C$. Les relations suivantes sont par suite vérifiées :
$$ \big[ \psi(C), \psi(X^{\pm}) \big] \ = \ \big[ \psi(C), \psi(H) \big] \ = \ 0 \, . $$

\item[\textbullet] $$ \big[ \psi(H), \psi(X^-) \big] \ = \ \big[ \bar{H}, \bar{X}^- \big] \ = \ - \: \! 2 \: \! \bar{X}^- \ = \ \ - \: \! 2 \: \! \psi(X^-) \, . $$

\item[\textbullet] L'élément
$$ \bar \Pi \ := \ \prod_{e = 1, -1} \ \frac{\Big( Q \, T^{{\{ \bar{H}_e^+ \}}_Q} \Big)^{\bar{H}_e^+} - \ \Big( Q \, T^{{\{ \bar{H}_e^+ \}}_Q} \Big)^{-\bar{H}_e^+}}{\bar{H}_e^+ \Big( Q \, T^{{\{ \bar{H}_e^+ \}}_Q} \ - \ Q^{-1} \, T^{-{\{ \bar{H}_e^+ \}}_Q} \Big)} $$
appartient à la sous-$\CC[[h,h']]$-algèbre topologiquement engendrée par $\bar H$ et $\bar C$, donc commute avec $\bar H$. Par suite,
$$ \big[ \psi(H), \psi(X^+) \big] \ = \ \big[ \bar{H}, \bar \Pi \: \! \bar{X}^+ \big] \ = \ \bar \Pi \big[ \bar{H},\bar{X}^+ \big] \ = \ 2 \: \! \bar \Pi \: \! \bar{X}^+ \ = \ 2 \: \! \psi(X^+) \, . $$

\item[\textbullet] Dans $U_h(\mathfrak{sl}_2)$, on a $\bar{X}^+ \bar{X}^- = \frac{1}{4} \bar{C} - \frac{1}{4} (\bar{H} -1)^2 = \bar{H}_1^+ \: \! \bar{H}_{-1}^+$, donc
\begin{eqnarray*} \psi(X^+) \: \! \psi(X^-) & =  & \prod_{e = 1, -1} \ \frac{\Big( Q \, T^{{\{ \bar{H}_e^+ \}}_Q} \Big)^{\bar{H}_e^+} - \ \Big( Q \, T^{{\{ \bar{H}_e^+ \}}_Q} \Big)^{-\bar{H}_e^+}}{Q \, T^{{\{ \bar{H}_e^+ \}}_Q} \ - \ Q^{-1} \, T^{-{\{ \bar{H}_e^+ \}}_Q} \Big)} \\
& = & \psi \left( \xxp{+}{} \right) \, .
\end{eqnarray*}

\item[\textbullet] Toujours dans $U_h(\mathfrak{sl}_2)$, on a $\bar{X}^- \bar{H}_e^+ \ = \ \bar{H}_e^- \bar{X}^-$, donc
\begin{multline*}
\bar{X}^- \left( \prod_{e = 1, -1} \ \frac{\Big( Q \, T^{{\{ \bar{H}_e^+ \}}_Q} \Big)^{\bar{H}_e^+} - \ \Big( Q \, T^{{\{ \bar{H}_e^+ \}}_Q} \Big)^{-\bar{H}_e^+}}{\bar{H}_e^+ \Big( Q \, T^{{\{ \bar{H}_e^+ \}}_Q} \ - \ Q^{-1} \, T^{-{\{ \bar{H}_e^+ \}}_Q} \Big)} \right) \\
= \ \left( \prod_{e = 1, -1} \ \frac{\Big( Q \, T^{{\{ \bar{H}_e^- \}}_Q} \Big)^{\bar{H}_e^-} - \ \Big( Q \, T^{{\{ \bar{H}_e^- \}}_Q} \Big)^{-\bar{H}_e^-}}{\bar{H}_e^- \Big( Q \, T^{{\{ \bar{H}_e^- \}}_Q} \ - \ Q^{-1} \, T^{-{\{ \bar{H}_e^- \}}_Q} \Big)} \right) \bar{X}^- \, .
\end{multline*}
Par ailleurs, $\bar{X}^- \bar{X}^+ = \frac{1}{4} \bar{C} - \frac{1}{4} (\bar{H} +1)^2 = \bar{H}_1^- \: \! \bar{H}_{-1}^-$, donc
\begin{eqnarray*}
\psi(X^-) \: \! \psi(X^+) & = & \prod_{e = 1, -1} \ \frac{\Big( Q \, T^{{\{ \bar{H}_e^- \}}_Q} \Big)^{\bar{H}_e^-} - \ \Big( Q \, T^{{\{ \bar{H}_e^- \}}_Q} \Big)^{-\bar{H}_e^-}}{Q \, T^{{\{ \bar{H}_e^- \}}_Q} \ - \ Q^{-1} \, T^{-{\{ \bar{H}_e^- \}}_Q}} \\
& = & \psi \left( \xxp{-}{} \right) \, .
\end{eqnarray*}
\end{itemize}

On peut définir un inverse $\psi_1$ de $\psi$ en posant
$$ \psi_1(\bar{H}) \ = \ H \, , \quad \psi_1(\bar{X}^-) \ = \ X^- \, , $$
\begin{equation} \label{eq_uconst2}
\psi_1(\bar{X}^+) \ = \ \left( \prod_{e = 1, -1} \ \frac{H_e^+ \Big( Q \, T^{{\{ H_e^+ \}}_Q} \ - \ Q^{-1} \, T^{-{\{ H_e^+ \}}_Q} \Big)}{\Big( Q \, T^{{\{ H_e^+ \}}_Q} \Big)^{H_e^+} - \ \Big( Q \, T^{{\{ H_e^+ \}}_Q} \Big)^{-H_e^+}} \right) X^+ \, .
\end{equation}
En reprenant la discussion précédant cette démonstration, on voit que le terme à droite de l'égalité \eqref{eq_uconst2} est bien défini : il faut remarquer que la série formelle
$$ \xxn{e}{+}{}{} \ = \ 2 \: \! \sinh \! \Bigg[ \! \left( h + h' { \{ H^+_e \} }_{\! Q} \right) \! H^+_e \Bigg] \, , $$
une fois divisée par $\left( h + h' { \{ H^+_e \} }_{\! Q} \right) \! H^+_e$, est inversible dans $\CC[H, \sqrt C][[h,h']]$. \\
On prouve grâce à des arguments similaires à ceux donnés pour $\psi$ que $\psi_1$ est un morphisme de $\CC[[h,h']]$-algèbre. En particulier on obtient que
$$ \psi_1(\bar{X}^{\pm}) \: \! \psi_1(\bar{X}^{\mp}) \ = \ \frac{1}{4} C \ - \ \frac{1}{4} (H \mp 1)^2 \, , $$
ce qui implique que $\psi_1(\bar{C}) = C$ et par suite que $(\psi_1 \circ \psi)(C) = C$ et $(\psi \circ \psi_1) (\bar C) = C$. Les autres identités à vérifier, pour montrer que $\psi$ et $\psi_1$ sont inverses l'un de l'autre, sont évidentes.
\end{proof}

La proposition \ref{prop_uconst} montre en particulier que la $\CC[[h,h']]$-algèbre $U_{h,h'}(\mathfrak{sl}_2,g)$ est isomorphe à la déformation formelle constante de $\Uc[\slt]$ selon $h$ et $h'$.


\begin{proposition} \label{prop_basehh}
Soit $g \in \NN_{\geq 1}$. \\
$U_{\: \! h,h'}(\mathfrak{sl}_2,g)$ admet pour base topologique la famille $\left( (X^-)^a H^b (X^+)^c \right)_{a, b, c \in \NN} $.
\end{proposition}

\begin{proof}
On utilise l'isomorphisme $\psi$ de la proposition \ref{prop_uconst} et le théorème de Poincaré-Birkhoff-Witt pour $\mathfrak{sl}_2$, en remarquant que
\begin{equation} \label{eq_psimod}
\psi(H) \ \equiv \ \bar{H} \ \text{ mod }(h,h') \, , \quad \ \psi(X^\pm) \ \equiv \ \bar{X}^{\pm} \ \text{ mod }(h,h') \, .
\end{equation}
\end{proof}

L'existence de cette base topologique fournit un isomorphisme $\CC[[h,h']]$-linéaire naturel $B^{h,h'} : \Uc[\slt][[h,h']] \stackrel{\sim}{\to} \Uhhp$, puis, par transport, une structure de $\CC[[h,h']]$-algèbre sur $\Uc[\slt][[h,h']]$. En considérant ci-dessous $\Uc[\slt][[h,h']]$ à gauche muni de cette structure et $\Uc[\slt][[h,h']]$ à droite comme la déformation constante de $\Uc[\slt]$,
$$ \psi \circ B^{h,h'} \, : \ \Uc[\slt][[h,h']] \ \to \ \Uc[\slt][[h,h']] $$
est alors un isomorphisme de $\CC[[h,h']]$-algèbre. En utilisant \eqref{eq_psimod} on en conclut que la structure transportée par $B^{h,h'}$ sur $\Uc[\slt][[h,h']]$ est une déformation formelle de $\Uc[\slt]$ par rapport à $h$ et $h'$, qu'on notera par abus de langage $\Uhhp$. \\

$U_h(\mathfrak{sl}_2)$ admet comme base topologique la famille de monômes $(\tilde{X}^-)^a \tilde{H}^b (\tilde{X}^+)^c$ avec $a,b,c \in \NN$ : voir par exemple \cite[6.4]{ChPr}. De même que précédemment, par transport on définit sur $\Uc[\slt][[h]]$ une structure de déformation formelle de $\Uc[\slt]$ par rapport à $h$, qu'on notera par abus de langage $\Uh[\slt]$. \\

Dans le quotient $U_{h,h'}(\mathfrak{sl}_2,g) / (h'=0)$ les relations suivantes sont vérifiées : $[H, X^{\pm}] = \pm 2 \: \! X^{\pm}$ et
\begin{eqnarray*}
[X^+, X^-] & = &  X^+ X^- \ - \ X^- X^+ \\
& = & \frac{Q^{\sqrt{C}} + Q^{-\sqrt{C}}}{(Q - Q^{-1})^2} \ - \ \frac{Q^{H - 1} + Q^{-H+1}}{(Q - Q^{-1})^2} \\
&& \quad \quad - \ \frac{Q^{\sqrt{C}} + Q^{-\sqrt{C}}}{(Q - Q^{-1})^2} \ + \ \frac{Q^{H + 1} + Q^{-H-1}}{(Q - Q^{-1})^2} \\
& = & \frac{Q^H - Q^{-H}}{Q - Q^{-1}} \, .
\end{eqnarray*}

Cela prouve l'existence d'un morphisme de $\CC[[h]]$-algèbre
\begin{IEEEeqnarray*}{lcll}
\Uh[\slt] = \Uc[\slt][[h]] & \ \longrightarrow \ & \Uhhp / (h'=0) \, & = \Uc[\slt][[h,h']] / (h'=0) \\
&&& = \Uc[\slt][[h]]
\end{IEEEeqnarray*}
égal à l'idendité. En d'autres mots $\Uhhp = \Uc[\slt][[h,h']] = \Uc[\slt][[h]][[h']]$ est une déformation formelle selon $h'$ de $\Uh[\slt] = \Uc[\slt][[h]]$. \\

De manière similaire, on voit que $\Uhhp$ est une déformation formelle de $\Urh{h'}$ selon $h$ : il faut remarquer que dans le quotient $U_{h,h'}(\mathfrak{sl}_2,g) / (h=0)$ on a $P(Q^{H_e^{\pm}}) = 1$. \\

L'isomorphisme $\psi$ de la proposition \ref{prop_uconst} induit un isomorphisme de $\CC[[h]]$-algèbre
$$ \varphi \ := \ \lim_{h' \to 0} \psi \, : \ U_h(\mathfrak{sl}_2) \ \stackrel{\sim}{\longrightarrow} \ \Uc[\slt][[h]] \, , $$
qui n'est autre que l'isomorphisme de \cite[prop. 6.4.6]{ChPr}. Pour le voir, il suffit de calculer $\varphi$ sur les générateurs $\tilde{X}^{\pm}$ et $\tilde H$, faisons-le pour $\tilde{X}^+$ (c'est évident pour les autres) :
\begin{eqnarray*}
\varphi(\tilde{X}^+) & = & \left( \prod_{e=1,-1} \frac{Q^{\bar{H}_e^+} - Q^{- \bar{H}_e^+}}{\bar{H}_e^+ \left( Q - Q^{-1} \right)} \right) \bar{X}^+ \\
& = & 4 \left( \frac{Q^{\sqrt{\bar{C}}} + Q^{-\sqrt{\bar{C}}} - Q^{\bar{H} - 1} - Q^{-\bar{H} +1}}{\left( \bar{C} - (\bar{H}-1)^2 \right) \left( Q - Q^{-1} \right)^2} \right) \bar{X}^+ \\
& = & 2 \left( \frac{\cosh h(\bar{H} -1) - \: \! \cosh h \sqrt{\bar{C}}}{\left( (\bar{H}-1)^2 - \bar{C} \right) \sinh^2 h} \right) \bar{X}^+ \, .
\end{eqnarray*}

On pose $\phi$ l'unique isomorphisme de $\CC[[h,h']]$-algèbre tel que le diagramme suivant est commutatif :
$$ \xymatrix{
&& U_h(\mathfrak{sl}_2) [[h']] \ar[rrd]_-{\sim}^-{\mathcal{Q}^{h'} (\varphi)} && \\
U_{h,h'} (\mathfrak{sl}_2,g) \ar[rru]^-{\phi}_-{\sim} \ar[rrrr]^-{\psi}_-{\sim} && && \Uc[\slt] [[h,h']]
} $$

De la même manière, on a l'existence d'un isomorphisme de $\CC[[h']]$-algèbre
$$ \varphi' \ := \ \lim_{h \to 0} \psi \, : \ U_{h'}(\mathfrak{sl}_2) \ \stackrel{\sim}{\longrightarrow} \ \Uc[\slt][[h']] \, , $$
et d'un unique isomorphisme $\phi'$ de $\CC[[h,h']]$-algèbre, de telle sorte que le diagramme précédent se complète par symétrie en le diagramme commutatif suivant :
\begin{eqnarray} \label{eq_diag1}
\xymatrix{
&& U_h(\mathfrak{sl}_2) [[h']] \ar[rrd]_-{\sim}^-{\mathcal{Q}^{h'} (\varphi)} && \\
U_{h,h'} (\mathfrak{sl}_2,g) \ar[rru]^-{\phi}_-{\sim} \ar[rrd]_-{\phi'}^-{\sim} \ar[rrrr]^-{\psi}_-{\sim} && && \Uc[\slt] [[h,h']] \\
&& U_{h'}(\mathfrak{sl}_2) [[h]] \ar[rru]^-{\sim}_-{\mathcal{Q}^{h} (\varphi')} &&
} \end{eqnarray}

On résume la discussion précédente dans le théorème \ref{thm_uhhdefor} suivant, qu'on illustre par le diagramme commutatif
$$ \xymatrix{
&& \Uhhp \ar[dll]_-{\lim_{h' \to 0} \ } \ar[drr]^-{\lim_{h \to 0}} \ar[dd]^-{\lim_{h,h' \to 0}} && \\
\Uh[\slt] \ar[drr]_-{\lim_{h \to 0}} && && \Urh{\: \! h'} \ar[dll]^-{\ \lim_{h' \to 0}} \\
&& \Uc[\slt] &&
} $$

\begin{theoreme} \phantomsection \label{thm_uhhdefor}
\begin{itemize}
\item[1)] $U_{\: \! h,h'}(\mathfrak{sl}_2,g)$ est une déformation formelle triviale de $\Uc[\slt]$ selon $h, h'$.
\item[2)] $U_{\: \! h,h'}(\mathfrak{sl}_2,g)$ est une déformation formelle triviale de $U_h(\mathfrak{sl}_2)$ selon $h'$.
\item[3)] $U_{\: \! h,h'}(\mathfrak{sl}_2,g)$ est une déformation formelle triviale de $U_{h'}(\mathfrak{sl}_2)$ selon $h$.
\end{itemize}
\end{theoreme}

\begin{remarque}
La cohomologie de Hochschild de $\Uc[\slt]$ est nulle en degré 2 : voir par exemple \cite[II.11]{Gui2}. La proposition \ref{prop_hoch} et le lemme \ref{lem_hoch} donnent donc une preuve que les déformations formelles précédentes sont triviales. La proposition \ref{prop_uconst} a toutefois l'avantage de fournir explicitement des isomorphismes avec les déformations constantes.
\end{remarque}

\begin{remarque}
Le théorème \ref{thm_uhhdefor} signifie en particulier que les groupes quantiques d'interpolation $\Uhhp$ ne sont en fait pas plus gros que l'algèbre enveloppante $\Uc[\slt]$ (seul l'espace des scalaires est grossi). Par ailleurs, il permet de formuler le fait qu'ils sont des objets comparables à $\Uc[\slt]$.
\end{remarque}

\subsection{Corollaires}
Ici on donne des corollaires des résultats de la section précédente. Ces résultats sont nouveaux. \\
Le premier est une conséquence immédiate de la proposition \ref{prop_basehh}. Il précise que $\Uhhp$ est pour ainsi dire défini à partir de seulement trois éléments $X^-$, $H$ et $X^+$.

\begin{corollaire} \label{cor_xheng}
$U_{h,h'}(\mathfrak{sl}_2,g)$ est topologiquement engendrée par les éléments $X^{\pm}$ et $H$.
\end{corollaire}

Comme conséquences immédiates de la proposition \ref{prop_uconst}, on a les deux corollaires suivants.
\begin{corollaire}
$U_{\: \! h,h'}(\mathfrak{sl}_2,g)$ est intègre.
\end{corollaire}

\begin{corollaire} \label{cor_eng}
Le centre de $U_{h,h'}(\mathfrak{sl}_2,g)$ est topologiquement engendré par $C$ et isomorphe à $\CC[C] [[h,h']]$.
\end{corollaire}

On donne maintenant un résultat de décomposition triangulaire, analogue à ceux que l'on connaît pour $\Uc[\slt]$ et $\Uh[\slt]$. \\
On note $U^-_{h,h'}(\mathfrak{sl}_2,g)$, $U^0_{h,h'}(\mathfrak{sl}_2,g)$ et $U^+_{h,h'}(\mathfrak{sl}_2,g)$ les sous-$\CC[[h,h']]$-algèbres de $U_{h,h'}(\mathfrak{sl}_2,g)$, topologiquement engendrées respectivement par $X^-$, $H$ et $X^+$. \\
D'après la proposition \ref{prop_basehh}, elles sont, en tant que $\CC[[h,h']]$-algèbres, canoniquement isomorphes respectivement à $\CC[X^-][[h,h']]$, $\CC[H][[h,h']]$ et $\CC[X^+][[h,h']]$.

\begin{proposition} \label{prop_triang}
$U_{\: \! h,h'}(\mathfrak{sl}_2,g)$ admet une décomposition triangulaire :
\begin{equation} \label{eq_triang}
U_{h,h'}(\mathfrak{sl}_2,g) \ \simeq \ U^-_{h,h'}(\mathfrak{sl}_2,g) \ \tilde{\otimes}_{\CC[[h,h']]} \, \ U^0_{h,h'}(\mathfrak{sl}_2,g) \ \tilde{\otimes}_{\CC[[h,h']]} \ \, U^+_{h,h'}(\mathfrak{sl}_2,g) \, ,
\end{equation}
où l'isomorphisme de $\CC[[h,h']]$-modules est défini par la multiplication de $U_{h,h'}(\mathfrak{sl}_2,g)$ :
$$ x^- \otimes x^0 \otimes x^+ \ \mapsto \ x^- \cdot x^0 \cdot x^+ \, , \quad \text{ avec } \ x^s \in U^s_{h,h'}(\mathfrak{sl}_2,g), \ s \in \{-,0,+ \} \, . $$
\end{proposition}

\begin{proof}
C'est une conséquence de la proposition \ref{prop_basehh}.
\end{proof}

\subsection{Autre présentation}
On donne ici une autre présentation des algèbres $\Uhhp$. Elle présente l'avantage de ne faire appel qu'aux générateurs $X^{\pm}$ et $H$. Les résultats ici sont nouveaux. \\

$U_{\: \! h,h'}(\mathfrak{sl}_2,g)$ est engendrée par $X^{\pm}$, $H$ et $C$. Le corollaire \ref{cor_xheng} signifie donc que $C$ peut s'écrire en fonction de $X^{\pm}$ et $H$. Plus précisément, on a le résultat suivant.

\begin{defiprop} \label{defiprop_cxx}
Soit $\CC \langle X^{\pm}, H \rangle^0$ le sous-$\CC$-espace vectoriel de $\CC \langle X^{\pm}, H \rangle$ engendré par les monômes de la forme
$$ (X^-)^a H^b (X^+)^a \, , \quad a, b \in \NN \, . $$

\begin{itemize}
\item[1)] Il existe une unique famille d'éléments $C_{n,m} \in \CC \langle X^{\pm}, H \rangle^0$, avec $n,m \in \NN$, telle que dans $U_{\: \! h,h'}(\mathfrak{sl}_2,g)$
$$ C \ = \ \sum_{n,m \, \in \, \NN} C_{n,m} \, h^n (h')^m \, . $$

\item[2)] Il existe une unique famille d'éléments $\alpha_{n,m} \in \CC \langle X^{\pm}, H \rangle^0$, avec $n,m \in \NN$, telle que dans $U_{\: \! h,h'}(\mathfrak{sl}_2,g)$
$$ [X^+, X^-] \ = \ \sum_{n,m \, \in \, \NN} \alpha_{n,m} \, h^n (h')^m \, . $$
\end{itemize}
\end{defiprop}

\begin{proof}
La proposition \ref{prop_basehh} implique que le sous-$\CC[[h,h']]$-module de $\Uhhp$ topologiquement engendré par les monômes de la forme
$$ (X^-)^a H^b (X^+)^a \, , \quad a, b \in \NN $$
est égal à la sous-$\CC[[h,h']]$-algèbre de $\Uhhp$ commutant avec $H$. En effet on a
$$ [H, (X^-)^a H^b (X^+)^c] \ = \ 2 \: \! (c-a) \: \! (X^-)^a H^b (X^+)^c \, , \quad \text{ pour } a,b,c \in \NN \, . $$
Puisque $C$ appartient au centre de $\Uhhp$ et puisque $[X^+, X^-]$ commute avec $H$, on obtient les deux points de la proposition.
\end{proof}

\begin{proposition} \label{prop_autrepres}
Soit $g \in \NN_{\geq 1}$. On note $V_{h,h'}(\mathfrak{sl}_2,g)$ la $\CC[[h,h']]$-algèbre topologiquement engendrée par $X^\pm$, $H$ et définie par les relations
$$ [H, X^{\pm}] \ = \ \pm 2 \, X^{\pm} \, , \quad \quad [X^+, X^-] \ = \ \sum_{n,m \, \in \, \NN} \alpha_{n,m} \, h^n (h')^m \, . $$
Il existe un isomorphisme de $\CC[[h,h']]$-algèbre
$$ \varphi \, : \ V_{h,h'}(\mathfrak{sl}_2,g) \ \stackrel{\sim}{\longrightarrow} \ \Uhhp \, , $$
uniquement déterminé par $\varphi(X^{\pm}) = X^{\pm}$ et $\varphi(H) = H$. En outre, $\varphi$ vérifie :
$$ \varphi^{-1}(C) \ = \ \sum_{n,m \in \NN} C_{n,m} \, h^n (h')^m \, .$$
\end{proposition}

\begin{proof}
D'après la définition des $\alpha_{n,m}$, il existe un unique morphisme d'algèbre
$$ \varphi \, : \ V_{h,h'}(\mathfrak{sl}_2,g) \ \to \ \Uhhp \, , $$
vérifiant $\varphi(X^{\pm}) = X^{\pm}$ et $\varphi(H) = H$. \\
D'après les relations définissant l'algèbre $V_{h,h'}(\mathfrak{sl}_2,g)$, on voit que cette dernière est topologiquement engendrée par les monômes de la forme
$$ (X^-)^a H^b (X^+)^c \, , \quad \text{ avec } a,b,c \in \NN \, . $$
La proposition \ref{prop_basehh} implique alors que $\varphi$ est un isomorphisme. \\
On a $\varphi^{-1} (X^{\pm}) = X^{\pm}$ et $\varphi^{-1} (H) = H$, donc, d'après la définition des $C_{n,m}$, on voit que $\varphi^{-1}(C) = \sum_{n,m \in \NN} C_{n,m} \, h^n (h')^m$.
\end{proof}

\begin{remarque} \label{rem_coeff}
\begin{itemize}
\item[1)] $\Uhhp$ est une déformation de $\Uc[\slt]$, où la limite classique de $C$ vérifie la relation \eqref{eq_casimir} : on le vérifie en calculant la limite classique des relations \eqref{eq_defx} de $\Uhhp$. On a donc
\begin{equation} \label{eq_casimircoeff}
C \ = \ 4 \: \! X^+ X^- + (H-1)^2 \ + \sum_{n+m \, \geq \, 1} C_{n,m} \, h^n (h')^m \, .
\end{equation}

\item[2)] $\Uhhp$ est une déformation de $\Uh[\slt]$, par suite :
$$ [X^+, X^-] \ = \ \frac{\exp (h \: \! H) - \exp (-h \: \! H)}{\exp (h) - \exp (-h)} \ + \ h' \! \sum_{n,m \, \in \, \NN} \alpha_{n,m+1} \, h^n (h')^m \, . $$
\end{itemize}
\end{remarque}

\subsection{Symétries}
On donne ici des propriétés de symétrie de l'algèbre $U_{\: \! h,h'}(\mathfrak{sl}_2,g)$, analogues naturels de propriétés qui existent déjà pour $\Uc[\slt]$ et $\Uh[\slt]$. Les résultats ici sont nouveaux.

\begin{lemme}
\begin{itemize}
\item[1)] Il existe un unique automorphisme $\omega$ de la $\CC[[h,h']]$-algèbre $U_{\: \! h,h'}(\mathfrak{sl}_2,g)$ tel que $\omega(X^{\pm}) = X^{\mp}$, $\omega(H) = -H$. Il est involutif et envoie $C$ sur $C$.

\item[2)] Il existe un unique anti-automorphisme $\tau$ de la $\CC[[h,h']]$-algèbre $U_{\: \! h,h'}(\mathfrak{sl}_2,g)$ tel que $\tau(X^{\pm}) = X^{\pm}$, $\tau(H) = -H$. Il est involutif et envoie $C$ sur $C$.
\end{itemize}
\end{lemme}

\begin{proof}
Dans les deux points, le fait que $U_{h,h'}(\mathfrak{sl}_2,g)$ est topologiquement engendrée par $X^-$, $X^+$, et $H$ implique l'unicité. Reste donc l'existence.
\begin{itemize}
\item[1)] Il suffit de vérifier que
$$ \big( \omega(X^{+}), \omega(X^-), \omega(H), \omega(C) \big) \ = \ \big( X^-, X^+, -H, C \big) $$
vérifient les relations de la définition \ref{def_uhh} dans l'algèbre $U_{h,h'}(\mathfrak{sl}_2,g)$. Faisons-le pour les relations \eqref{eq_defx}, les autres étant évidentes. $\omega(H_e^{\pm}) \ = \ H_{-e}^{\mp}$, par suite
\begin{eqnarray*}
\omega \left( X^{\pm} \right) \omega \left( X^{\mp} \right) & = & X^{\mp} X^{\pm} \\
& = & \xxp{\mp}{} \\
& = & \omega \left( \xxp{\pm}{} \right) .
\end{eqnarray*}

\item[2)] Ici, il faut vérifier que
$$ \big( \tau(X^{+}), \tau(X^-), \tau(H), \tau(C) \big) \ = \ \big( X^+, X^-, -H, C \big) $$
vérifient les relations de la définition \ref{def_uhh} dans l'algèbre $U_{h,h'}^{opp}(\mathfrak{sl}_2,g)$, l'algèbre opposée de $U_{h,h'}(\mathfrak{sl}_2,g)$. Faisons-le pour les relations \eqref{eq_defx}, les autres étant évidentes. \\
$\tau (H_e^{\pm}) \ = \ H_{-e}^{\mp}$, par suite
\begin{eqnarray*}
\tau \left( X^{\pm} \right) \cdot_{opp} \left( X^{\mp} \right) & = & X^{\mp} X^{\pm} \\
& = & \xxp{\mp}{} \\
& = & \tau \left( \xxp{\pm}{} \right) .
\end{eqnarray*}
\end{itemize}
\end{proof}

\subsection{Interpolation}
On définit ici rigoureusement ce que signifie spécialiser $\Uhhp$ en $Q = \varepsilon$. Une fois les définitions posées, on prouve le lemme \ref{lem_fond}. Ce résultat est crucial dans notre étude, il exprime que toutes les relations de $\Ughp$ peuvent être obtenues à partir de la spécialisation à $Q = \varepsilon$ de $\Uhhp$. \\

On note $\mathcal{A}$ la sous-$\CC$-algèbre de $\CC[[h]]$ engendrée par $Q^{\pm 1}$, canoniquement isomorphe aux polynômes de Laurent en $Q$ :
$$ \Aq \ \simeq \ \CC \! \left[ Q^{\pm 1} \right] . $$
On rappelle que $\varepsilon $ désigne la racine $(2g)$-ième primitive de l'unité $\text{e}^{i \pi / g} \in \CC$ et on pose
$$ {[g-1]}^!_{\varepsilon} \ := \ \prod_{k=1}^{g-1} \, {[k]}_{\varepsilon} \ = \ \prod_{k=1}^{g-1} \, \frac{\varepsilon^k - \varepsilon^{-k}}{\varepsilon - \varepsilon^{-1}} \ \in \ \CC^{\ast} \, . $$

\begin{definition}
Soit $g \in \NN_{\geq 1}$. \\
$U_{\mathcal{A},h'} (\mathfrak{sl}_2,g)$ est la sous $\mathcal{A}[[h']]$-algèbre de $U_{h,h'} (\mathfrak{sl}_2,g)$ topologiquement engendrée (pour la topologie ($h'$)-adique) par
$$ C \, , \ H \, , \ X^{\pm} \quad \text{et} \quad Q^{\pm H} \, , \ Q^{\sqrt{C}} + Q^{-\sqrt{C}} \, . $$
\end{definition}

La démonstration du lemme suivant est le passage le plus technique de l'article. Le lemme \ref{lem_fond} assure essentiellement que la propriété d'interpolation de $\Uhhp$ en $Q = \varepsilon$ est vérifiée. Il est le noyau dur qui permettra de montrer plus loin dans l'article que la spécialisation à $Q = \varepsilon$ de $\Uhhp$ contient $\Urh{gh'}$ comme sous-quotient. Par ailleurs, il généralise à tout $g \in \NN_{\geq 1}$ un résultat analogue dans \cite{FH1}. Sa preuve n'est en toutefois pas une simple généralisation/adaptation. \\
Dans \cite{FH1} l'algèbre quantique duale de Langlands est réalisée de plusieurs manières différentes comme sous-quotients de la spécialisation à $Q = \varepsilon$ du groupe quantique d'interpolation de Langlands. Ici, la réalisation est plus universelle puisqu'elle ne fait intervenir qu'un seul sous-quotient. De manière plus précise, on obtient l'algèbre quantique duale en imposant à $\Uhhp$ moins de relations qu'il n'est fait dans \cite{FH1} pour chacune des réalisations.

\begin{lemme} \label{lem_fond}
Dans la $\CC[[h']]$-algèbre
$$ U_{\mathcal{A}, h'} (\mathfrak{sl}_2, g) \ / \ \overline{ \big( Q = \varepsilon, \, Q^{2H} = 1, \, Q^{2 \sqrt{C}} + Q^{-2 \sqrt{C}} = \varepsilon^{2} + \varepsilon^{-2} \big) }^{\, h'} $$
où $\overline{ \ \cdot \ }^{\: \! h'}$ désigne l'adhérence pour la topologie $(h')$-adique, on a l'égalité :
\begin{equation} \label{eqlemfond}
\big( \varepsilon \, T - \varepsilon^{-1} \: \! T^{-1} \big)^2 \, \Big[ (X^+)^g , (X^-)^g \Big] \ = \ ({[g-1]}_{\varepsilon} ^!)^2 \, \big( T^g - T^{-g} \big) \big( T^H - T^{-H} \big) \, .
\end{equation}
\end{lemme}

\begin{proof}
Dans $\Uhhp$, la relation \eqref{eq_defx} implique l'égalité
\begin{equation*}
(X^\pm)^g (X^\mp)^g \ = \ \prod_{k=0}^{g-1} \ \prod_{e=1,-1} \frac{\xxn{e}{\pm}{\mp ek}{}}{\xxd{e}{\pm}{\mp ek}{}} \, .
\end{equation*}
Pour $0 \leq k \leq g-1$, on peut considérer les éléments
\begin{multline} \label{eq_defxg}
\check{\Pi}_k \ := \ \prod_{e=1,-1} \ \Bigg[ \xxd{e}{\pm}{\mp ek}{} \Bigg] \\
\text{et } \quad \hat{\Pi}_k \ := \ \prod_{e=1,-1} \ \Bigg[ \xxn{e}{\pm}{\mp ek}{} \Bigg]
\end{multline}
dans l'algèbre $\mathcal{A}[H,\sqrt{C}, Q^{\pm H}, Q^{\pm \sqrt{C}}] [[h']]$. On remarque alors qu'ils invariants par la transformation $Q^{\sqrt{C}} \mapsto Q^{- \sqrt{C}}$. Par conséquent Ils appartiennent en fait à la sous-algèbre $\mathcal{A}[H,\sqrt{C}, Q^{\pm H}, Q^{\sqrt{C}} + Q^{-\sqrt{C}}] [[h']]$. En particulier, leurs images dans $\Uhhp$ sont des éléments de $U_{\mathcal{A},h'} (\mathfrak{sl}_2,g)$. \\

La relation \eqref{eq_defxg} implique dans $U_{\mathcal{A},h'} (\mathfrak{sl}_2,g)$ la relation
\begin{multline} \label{eq_defax}
\left( \prod_{k=0}^{g-1} \ \prod_{e = 1, -1} \left[ \xxd{e}{\pm}{\mp ek}{} \right] \right) (X^\pm)^g (X^\mp)^g \\
= \ \prod_{k=0}^{g-1} \ \prod_{e = 1, -1} \left[ \xxn{e}{\pm}{\mp ek}{} \right] \, . \quad \quad
\end{multline}

Pour $0 \leq k \leq g-1$, les éléments $\check{\Pi}_k$ et $\hat{\Pi}_k$ sont invariants par la transformation $Q^H \mapsto -Q^H$, ainsi que par la transformation $Q^{\sqrt{C}} \mapsto -Q^{\sqrt{C}}$. Ils appartiennent donc à la sous-algèbre $\mathcal{A}[H,C,Q^{\pm 2H}, Q^{2\sqrt{C}} + Q^{-2\sqrt{C}}] [[h']]$.

Par ailleurs le morphisme de $\mathcal{A}$-algèbre
$$ \mathcal{A}[Q^{\pm 2H}, Q^{2\sqrt{C}} + Q^{-2\sqrt{C}}] \ \longrightarrow \ \mathcal{A}[Q^{\pm H_e^a}]_{e \in \{ -1, 1 \}, \, a \in \{-, + \} } $$
défini par
$$ Q^{\pm 2H} \ \mapsto \ Q^{2 H_1^+} Q^{-2H_{-1}^-} \quad \text{ et } \quad Q^{2 \sqrt{C}} + Q^{-2 \sqrt{C}} \ \mapsto \ Q^{2H_1^+} Q^{2 H_{-1}^+} + Q^{-2H_1^+} Q^{-2 H_{-1}^+} $$
induit un isomorphisme
\begin{multline*} \mathcal{A} \ \simeq \ \mathcal{A}[Q^{\pm 2H}, Q^{2\sqrt{C}} + Q^{-2\sqrt{C}}] \ / \ \big(Q^{2H} = 1, Q^{2\sqrt{C}} + Q^{-2\sqrt{C}} = \varepsilon^2 + \varepsilon^{-2} \big) \\
\stackrel{\sim}{\longrightarrow} \quad \mathcal{A}[Q^{\pm H_e^a}]_{e \in \{ -1, 1 \}, \, a \in \{-, + \} } \ / \ \big( Q^{H_e^a} = \varepsilon^{(1 - a \: \! e) / 2} \big) \ \simeq \ \mathcal{A} \, ,
\end{multline*}
et par suite un isomorphisme de $\CC[H,C][[h']]$-algèbre
\begin{multline*}
\mathcal{M} \ := \ \left( \frac{\mathcal{A} \Big[ H,C,Q^{\pm 2H}, Q^{2\sqrt{C}} + Q^{-2\sqrt{C}} \Big]}{\Big( Q=\varepsilon, \, Q^{2H} = 1, \, Q^{2\sqrt{C}} + Q^{-2\sqrt{C}} = \varepsilon^2 + \varepsilon^{-2} \Big)} \right) [[h']] \\
\stackrel{\sim}{\longrightarrow} \ \left( \frac{\mathcal{A} \Big[ H,C,Q^{\pm H_e^a} \Big]_{e \in \{ -1, 1 \}, \, a \in \{-, + \} }}{\Big( Q = \varepsilon, \, Q^{H_e^a} = \varepsilon^{(1-a \: \! e)/2} \Big)} \right) [[h']] \, .
\end{multline*}

Par conséquent, dans $\mathcal{M}$, on a
\begin{equation} \label{eq_final1}
\prod_{k=0}^{g-1} \ \prod_{e=1,-1} \ \Bigg[ \xxd{e}{\pm}{\mp ek}{} \Bigg] \ = \ (\varepsilon - \varepsilon^{-1})^{2g - 2} \, \big( \varepsilon \, T - \varepsilon^{-1} \: \! T^{-1} \big)^2
\end{equation}

\begin{multline} \label{eq_final2}
\text{et} \quad \prod_{k=0}^{g-1} \ \prod_{e=1,-1} \Bigg[ \xxn{e}{\pm}{\mp ek}{} \Bigg] \\
= \ \Bigg[ \prod_{k=1}^{g-1} \left( \varepsilon^k - \varepsilon^{-k} \right)^2 \Bigg] \prod_{e=1,-1} \Big( T^{\: \! H_e^{\pm} + (g -1 \mp eg \pm e)/2} - T^{\: \! - H_e^{\pm} - (g -1 \mp eg \pm e)/2} \Big) \\
= \ \Bigg[ \prod_{k=1}^{g-1} \left( \varepsilon^k - \varepsilon^{-k} \right)^2 \Bigg]  \Big( T^{\: \! \sqrt{C} +g -1} \ + \ T^{\: \! - \sqrt{C} -g +1} \ - \ T^{\: \!  \mp g} \, T^{\: \! H} \ - \ T^{\: \! \pm g} \, T^{\: \! -H} \Big) \, .
\end{multline}

Dans le quotient $U_{\! A, h'} (\mathfrak{sl}_2, g) \ / \ \overline{ \big( Q = \varepsilon, \, Q^{2H} = 1, \, Q^{2 \sqrt{C}} + Q^{-2 \sqrt{C}} = \varepsilon^{2} + \varepsilon^{-2} \big) }^{\, h'}$, les relations \eqref{eq_final1}  et \eqref{eq_final2} impliquent l'égalité
\begin{multline} \tag{$\ast_{\pm}$}
(\varepsilon - \varepsilon^{-1})^{2g - 2} \, \big( \varepsilon \, T - \varepsilon^{-1} \: \! T^{-1} \big)^2 \, (X^{\pm})^g (X^{\mp})^g \\
= \ \Bigg[ \prod_{k=1}^{g-1} \left( \varepsilon^k - \varepsilon^{-k} \right)^2 \Bigg]  \Big( T^{\: \! \sqrt{C} +g -1} \ + \ T^{\: \! - \sqrt{C} -g +1} \ - \ T^{\: \!  \mp g} \, T^{\: \! H} \ - \ T^{\: \! \pm g} \, T^{\: \! -H} \Big) \, . \quad \quad \quad
\end{multline}
On obtient \eqref{eqlemfond} en faisant la différence de $(\ast_+)$ et $(\ast_-)$.
\end{proof}

\begin{remarque} \label{rem_fond}
D'après le lemme précédent, dans le localisé par rapport à $h'$
$$ \Bigg( \, U_{\mathcal{A}, h'} (\mathfrak{sl}_2, g) \ / \ \overline{ \big( Q = \varepsilon, \, Q^{2H} = 1, \, Q^{2 \sqrt{C}} + Q^{-2 \sqrt{C}} = \varepsilon^{2} + \varepsilon^{-2} \big) }^{\, h'} \, \Bigg)_{\! \! h'} \, , $$
les éléments
$$ {}^L \! X^+ \ := \ ({[g-1]}_{\varepsilon}^!)^{-2} \, \big( \varepsilon \, T -  \varepsilon^{-1} T^{-1} \big)^2 \,  \big( T^g - T^{-g} \big)^{-2} \, \, (X^+)^g \, , $$
$$  {}^L \! X^- \ := (X^-)^g \, , \quad \ \text{ et } \quad {}^L \! H \ := \ \frac{H}{g} $$
vérifient
$$ [ {}^L \! H, {}^L \! X^{\pm}] \ = \ \pm 2 \,  {}^L \! X^{\pm} \, , \quad  [  {}^L \! X^+ ,  {}^L \! X^- ] \ = \ \frac{\text{sinh} (g h' \: \!  {}^L \! H)}{\text{sinh} (g h')} \, . $$
Autrement dit, la sous-$\CC[[h']]$-algèbre engendrée par ${}^L \! X^{\pm}$ et ${}^L \! H$, que nous noterons $\langle {}^L \! X^{\pm}, {}^L \! H \rangle$, est un quotient de $U_{gh'}(\mathfrak{sl}_2)$. \\
Nous verrons, après avoir établi la dualité de Langlands pour les représentations des groupes quantiques de rang 1, que $\langle {}^L \! X^{\pm}, {}^L \! H \rangle$ est en fait isomorphe à $U_{gh'}(\mathfrak{sl}_2)$ : voir le théorème \ref{thm_uhhinterpolation}.
\end{remarque}

\section{Représentations} \label{section_repr}
Dans cette section, on étudie la théorie des représentations de $\Uhhp$. On commence, en utilisant la décomposition triangulaire de $\Uhhp$ (proposition \ref{prop_triang}), par construire des modules de Verma de $\Uhhp$, déformations selon $h'$ pour chaque $g \in \NN_{\geq 1}$, des modules de Verma de $\Uh[\slt]$ correspondants. Comme pour $\mathfrak{sl}_2$ et $\Uh[\slt]$, on peut en donner une description explicite. Voir la proposition \ref{prop_vermadescr}. \\
Pour chaque $g \in \NN_{\geq 1}$, on définit ensuite une déformation selon $h'$ des représentations de $\Uh[\slt]$ de rang fini $L^h(n)$ ($n \in \NN$). Voir la proposition \ref{prop_indecdescr}. Comme pour $\mathfrak{sl}_2$ et $\Uh[\slt]$, cette déformation peut-être réalisée comme quotient du module de Verma de $\Uhhp$. \\
On utilise ensuite les outils de la section \ref{section_deformations} pour décrire précisément la catégorie des représentations de rang fini de $\Uhhp$ : voir le théorème \ref{thm_rephh}. On obtient notamment que toute représentation de $\mathfrak{sl}_2$ de dimension finie peut-être déformée en une représentation de $\Uhhp$, que la catégorie possède la propriété de Krull-Schmidt (toute représentation est, d'une manière unique, somme directe de représentations indécomposables), et que ses seules représentations indécomposables sont les déformations de $L^h(n)$ ($n \in \NN$) mentionnées précédemment. \\
Ces résultats sur la catégorie des représentations de rang fini des groupes quantiques d'interpolation de Langlands de rang 1 sont nouveaux. Dans \cite{FH1}, pour $g=2,3$ des modules de Verma sont définis, ainsi que des déformations des représentations irréductibles de dimension finie $L^q(n)$ de $U_q(\mathfrak{sl}_2)$. Toutefois une étude systématique de la théorie des représentations de dimension finie des groupes quantiques d'interpolation de rang 1 n'est pas faite. Celle-ci semble difficilement réalisable du point de vue de \cite{FH1}, du fait notamment du nombre trop important de générateurs des groupes quantiques d'interpolation. \\

Soit $g \in \NN_{\geq 1}$. Le groupe quantique d'interpolation $\Uhhp$ est, d'après le théorème \ref{thm_uhhdefor}, une déformation de $\Uc[\slt]$. En reprenant les définitions de la section \ref{section_deformations}, on rappelle qu'une représentation $V^{h,h'}$ de $\Uhhp$ est donnée par un $\CC$-espace vectoriel $V$ et par une action de $\Uhhp$ sur $V[[h,h']]$. \\

La représentation $V^{h,h'}$ est une représentation de poids si le $\CC[[h,h']]$-module $V^{h,h'}$ admet une décomposition en somme directe :
\begin{equation}
V^{h,h'} \ = \ \bigoplus_{n \in \ZZ} \, V^{h,h'}_n \, , \quad \text{ avec } \ V^{h,h'}_n \ := \ \{ f \in V^{h,h'} : \, H.f = n \: \! f \} \, .
\end{equation}
Remarquons que $V^{h,h'}_n$ ($n \in \ZZ$) est un sous-$\CC[[h,h']]$-module fermé de $V^{h,h'}$ (en effet, $H$ est $\CC[[h,h']]$-linéaire, donc continu). \\

Un élément $f \in V^{h,h'}_n$ est appelé un élément de poids $n$, et si $X^+.f = 0$, $f$ est appelé un élément de plus haut poids $n$. \\
Si $V^{h,h'}_n \neq (0)$, alors $n$ est appelé un poids de $V^{h,h'}$ et $V^{h,h'}_n$ un espace de poids de $V^{h,h'}$. On notera $\text{wt} \: \! V^{h,h'}$ l'ensemble des poids de $V^{h,h'}$. \\

La catégorie additive $\CC[[h,h']]$-linéaire des représentations de $\Uhhp$ de rang fini sera notée $\cmodhh$. \\

On pose $U := \Uc[\slt]$. On rappelle que le module de Verma $M(n)$ de $\Uc[\slt]$ ($n \in \ZZ$) est donné par :
$$ M(n) \ := \ U \ / \, \Big( U \: \! \bar{X}^+ + \ U \: \! (\bar{H} - n \: \! 1) \Big) \, . $$
En notant $m_0 \in M(n)$ l'image de $1 \in U$, et $m_j := (\bar{X}^-)^j. m_0$ ($j \in \NN)$, on obtient une base $(m_j)_{j \in \NN}$ de $M(n)$ telle que
\begin{eqnarray*}
H.m_j & = & (n - 2j) \, m_j \, , \\
X^- . m_j & = & m_{j+1} \, , \\
X^+ . m_j & = & j \: \! (n -j+1) \, m_{j-1} \, , \quad \quad \text{(on pose $m_{-1} := 0$).} \\
\end{eqnarray*}

On pose $U^h := \Uh[\slt]$. Le module de Verma $M^h(n)$ de $\Uh[\slt]$ ($n \in \ZZ$) est donné par :
$$ M^h(n) \ := \ U^h \ / \, \Big( U^h \: \! \tilde{X}^+ + \ U^h \: \! (\tilde{H} - n \: \! 1) \Big) \, . $$

On pose $U^{h,h'} := \Uhhp$ et on définit le module de Verma $M^{h,h'} (n,g)$ de $\Uhhp$ ($n \in \ZZ$) :
\begin{equation} \label{eq_verma}
M^{h,h'}(n,g) \ := \ U^{h,h'} \ / \, \Big( U^{h,h'} \: \! X^+ + \ U^{h,h'} \: \! (H - n \: \! 1) \Big) \, .
\end{equation}

$\Uhhp = U(\mathfrak{sl_2})[[h,h']]$ (on rappelle que cette égalité doit être interprétée en considérant la base donnée dans la proposition \ref{prop_basehh}) est une représentation, appelée régulière, de $\Uhhp$ en considérant l'action par multiplication à gauche. C'est une déformation de la représentation régulière de $\Uc[\slt]$. \\

Puisque $X^+ H = H X^+ - 2 \: \! X^+$, on a l'égalité de $\CC[[h,h']]$-module
$$ U^{h,h'} \: \! X^+ + \ U^{h,h'} \: \! (H - n \: \! 1) \ = \ \Big( U \: \! \bar{X}^+ + \ U \: \! (\bar{H} - n \: \! 1) \Big)[[h,h']] \, . $$
Le sous-$\Uhhp$-module $U^{h,h'} \: \! X^+ + \ U^{h,h'} \: \! (H - n \: \! 1)$ de la représentation régulière est donc une sous-représentation, et par suite une déformation formelle de la représentation $U \: \! \bar{X}^+ + \ U \: \! (\bar{H} - n \: \! 1)$ de $\Uc[\slt]$. \\
On en déduit que $M^{h,h'}(n,g) = M(n)[[h,h']]$ est une déformation formelle du module de Verma $M(n)$ de $\Uc[\slt]$ selon $h$ et $h'$ (remarquons que cela prouve en particulier que $M^{h,h'}(n,g)$ est non trivial). \\

De la même façon, $M^h(n) = M(n)[[h]]$ ($n \in \ZZ$) est une déformation du module de Verma $M(n)$ de $\Uc[\slt]$ selon $h$ et $M^{h,h'}(n,g)$ est une déformation de $M^h(n)$ selon $h'$. \\

On résume ce qui précède, et on donne une description explicite du module $M^{h,h'}(n,g)$, dans la proposition \ref{prop_vermadescr} suivante.

\begin{proposition} \label{prop_vermadescr}
Soient $n \in \ZZ$ et $g \in \NN_{\geq 1}$.
\begin{itemize}
\item[1)] $M^{h,h'}(n,g)$ est une représentation de rang infini de $\Uhhp$.
\item[2)] $M^{h,h'}(n,g)$ est une déformation formelle du module de Verma $M(n)$ de $\Uc[\slt]$ selon $h$ et $h'$.
\item[3)] $M^{h,h'}(n,g)$ est une déformation formelle du module de Verma $M^h(n)$ de $\Uh[\slt]$ selon $h'$.
\item[4)] Le module de Verma $M^{h,h'}(n,g)$ est une représentation de poids.
\item[5)] L'action de $\Uhhp$ sur $M^{h,h'}(n,g)$ est donnée dans la base topologique $(m_j)_{j \in \NN}$ par :
\begin{eqnarray*}
H.m_j & = & (n - 2j) \, m_j \, , \\
C. m_j & = & (n + 1)^2 \, m_j \, , \\
X^- . m_j & = & m_{j+1} \, , \\
X^+ . m_j & = & \big[ j \big]_{\! Q \: \! T^{ {\{ j \}}}} \, \big[ n -j+1 \big]_{\! Q \: \! T^{ {\{ n -j+1 \}}}} \, m_{j-1} \, ,
\end{eqnarray*}
\end{itemize}
\end{proposition}

\begin{proof}
Les trois premiers points ont déjà été prouvés. \\
On utilise les relations de la définition \ref{def_uhh} de $\Uhhp$ et le fait que $m_j = (X^-)^j . m_0$ ($j \in \NN$) pour montrer le dernier point. \\
En examinant l'action de $H$, on obtient que $M^{h,h'}(n,g)$ est une représentation de poids avec $M^{h,h'}_{n-2j}(n,g) = \CC[[h,h']]. m_j$ ($j \in \NN$) et
$$ M^{h,h'}(n,g) \ = \ \bigoplus_{j \in \NN} \, M^{h,h'}_{n - 2j}(n,g) \, . $$
\end{proof}

L'action de $\Uh[\slt]$ sur $M^h(n)$ ($n \in \ZZ$) est donnée dans la base topologique $(m_j)_{j \in \NN}$ par :
\begin{eqnarray*}
H.m_j & = & (n - 2j) \, m_j \, , \\
X^- . m_j & = & m_{j+1} \, , \\
X^+ . m_j & = & {[j]}_Q \, {[n-j+1]}_Q \, m_{j-1} \, ,
\end{eqnarray*}

Formellement donc, le module de Verma $M^{h,h'}(n,g)$ de $\Uhhp$ s'obtient à partir du module de Verma $M^h(n)$ de $U_h(\mathfrak{sl}_2)$, en remplaçant les boîtes quantiques $[a]_Q$ par les boîtes ${[a]}_{Q \: \! T^{\{a \}}}$ ($a \in \ZZ)$.  \\

Les modules de Verma de $\Uhhp$ vérifient la propriété universelle suivante. Sa démonstration s'obtient immédiatement à partir de \eqref{eq_verma}.

\begin{proposition} \label{prop_verma}
Si $V^{h,h'}$ est une représentation de $\Uhhp$ et $f \in V^{h,h'}$ un élément de plus haut poids $n$, alors il existe un unique morphisme
$$ \varphi \, : \ M^{h,h'}(n,g) \ \to \ V^{h,h'} $$
qui vérifie $\varphi \! \left( m_0 \right) = f$.
\end{proposition}

On rappelle que la représentation irréductible de dimension finie $L(n)$ de $\Uc[\slt]$ est obtenue comme quotient de $M(n)$ par la sous-représentation $\bigoplus_{j \geq n+1} \CC \, m_j$. \\
$L(n)$ admet alors comme base la famille des vecteurs $m_j$ ($0 \leq j \leq n$) et l'action de $\Uc[\slt]$ sur $L(n)$ est donnée par
\begin{eqnarray*}
H.m_j & = & (n - 2j) \, m_j \, , \\
X^- . m_j & = & m_{j+1} \, , \\
X^+ . m_j & = & j \: \! (n -j+1) \, m_{j-1} \, ,
\end{eqnarray*}
(on pose $m_{-1} = m_{n+1} := 0$). \\

Une sous-représentation $V[[h,h']]$ de $M^{h,h'}(n,g) = M(n)[[h,h']]$ vérifie en particulier $H.V \subset V$. L'action de $H$ sur $M(n)$ étant localement finie et semi-simple, elle l'est aussi sur $V$. Par suite $V$ admet comme base une sous-famille de $(m_j)_{j \in \NN}$. \\

D'après le point 5) de la proposition \ref{prop_vermadescr}, $M^{h,h'}(n,g)$ admet une sous-représentation non triviale si et seulement si $n \in \NN$. Dans ce dernier cas, il existe une unique sous-représentation non triviale : $\left( \bigoplus_{j \geq n+1} \CC \, m_j \right) [[h,h']]$ qui est isomorphe à $M^{h,h'}(-n-2)$ d'après la proposition \ref{prop_verma}. \\
On note $L^{h,h'}(n,g)$ la représentation quotient, de rang fini égal à $n+1$. \\
On définit de la même manière la représentation de rang fini $L^h(n)$ de $\Uh[\slt]$. \\

Par des arguments similaires à ceux donnés dans le cas des modules de Verma, on vérifie que $L^{h,h'}(n,g)$ est une déformation selon $h$ et $h'$ de la représentation irréductible $L(n)$ de $\Uc[\slt]$, et une déformation selon $h'$ de la représentation $L^h(n)$ de $\Uh[\slt]$. \\

On résume ce qui précède, et on donne une description explicite de $L^{h,h'}(n,g)$, dans la proposition suivante.

\begin{proposition} \label{prop_indecdescr}
Soient $n \in \NN$ et $g \in \NN_{\geq 1}$.
\begin{itemize}
\item[1)] $L^{h,h'}(n,g)$ est une représentation de $\Uhhp$ de rang fini égal à $n+1$.
\item[2)] $L^{h,h'}(n,g)$ est une déformation formelle de la représentation irréductible $L(n)$ de $\Uc[\slt]$ selon $h$ et $h'$.
\item[3)] $L^{h,h'}(n,g)$ est une déformation formelle de la représentation $L^h(n)$ de $\Uh[\slt]$ selon $h'$.
\item[4)] $L^{h,h'}(n,g)$ est une représentation de poids.
\item[5)] L'action de $\Uhhp$ sur $L^{h,h'}(n,g)$ est donnée dans la base topologique $(m_j)_{0 \leq j \leq n}$ par :
\begin{eqnarray*}
H.m_j & = & (n - 2j) \, m_j \, , \\
C. m_j & = & (n + 1)^2 \, m_j \, , \\
X^- . m_j & = & m_{j+1} \, , \\
X^+ . m_j & = & \big[ j \big]_{\! Q \: \! T^{ {\{ j \}}}} \, \big[ n -j+1 \big]_{\! Q \: \! T^{ {\{ n -j+1 \}}}} \, m_{j-1} \, ,
\end{eqnarray*}
\end{itemize}
\end{proposition}

\begin{proof}
Les trois premiers points ont déjà été prouvés. \\
Le dernier point est une conséquence du point 5) de la proposition \ref{prop_vermadescr}. \\
En examinant l'action de $H$, on obtient que $L^{h,h'}(n,g)$ est une représentation de poids avec $L^{h,h'}_{n-2j}(n,g) = \CC[[h,h']]. m_j$ ($0 \leq j \leq n$) et
$$ L^{h,h'}(n,g) \ = \ \bigoplus_{0 \leq j \leq n} \, L^{h,h'}_{n - 2j}(n,g) \, . $$
\end{proof}

L'action de $\Uh[\slt]$ sur la représentation $L^h(n)$ ($n \in \NN$) est donnée par :
\begin{eqnarray*}
H.m_j & = & (n-2j) \, m_j \, , \\
X^- . m_j & = & m_{j+1} \, , \\
X^+ . m_j & = & [j \big]_Q \: \!  [n -j+1]_Q \, m_{j-1} \, .
\end{eqnarray*}

Ici encore, il suffit de remplacer, comme expliqué plus haut, les boîtes quantiques de la représentation $L^h(n)$ de $U_h(\mathfrak{sl}_2)$ pour obtenir la représentation $L^{h,h'}(n,g)$ de $U_{h,h'}(\mathfrak{sl}_2)$.

\begin{remarque}
Pour $g= 1,2,3$, des analogues des représentations $M^{h,h'}(n,g)$ et $L^{h,h'}(n,g)$ existent dans \cite{FH1}. Toutefois celui de $M^{h,h'}(n,g)$ n'apparaît pas comme un module de Verma (les groupes quantiques d'interpolation de Langlands dans \cite{FH1} n'admettent a priori pas de décomposition triangulaire), quotient de la représentation régulière.
\end{remarque}

On note $\cmod$ la catégorie abélienne des représentations de dimension finie de l'algèbre de Lie $\mathfrak{sl}_2$. On rappelle que $\cmod$ est une catégorie semi-simple : voir par exemple \cite[II.6]{humphreys}. Semi-simple ici signifie que toute représentation de dimension finie de $\mathfrak{sl}_2$ est une somme directe finie de représentations irréductibles. \\

D'après la section \ref{section_deformations} sur les déformations formelles, on dispose du foncteur additif :
$$ \lim_{h,h' \to 0} : \ \cmodhh \, \longrightarrow \ \cmod \, . $$

On note $\grot$ et $\grothh$ les groupes de Grothendieck des catégories additives respectives $\cmod$ et $\cmodhh$. \\
On dira qu'une catégorie additive $\mathcal C$ possède la propriété de Krull-Schmidt si chaque objet dans $\mathcal C$ est une somme directe d'objets indécomposables, les facteurs composant la somme étant uniques à isomorphisme et ordre près. \\

Le théorème suivant donne plusieurs résultats nouveaux sur la catégorie des représentations de rang fini de $\Uhhp$. Ceux-ci n'ont pas d'analogues dans \cite{FH1}.

\begin{theoreme} \phantomsection \label{thm_rephh}
\begin{itemize}
\item[1)] Toute représentation $V \in \cmodh$ admet une déformation $V^{h,h'}$ dans $\cmodhh$.
\item[2)] Deux représentations dans $\cmodhh$ sont isomorphes si et seulement si leurs limites classiques le sont.
\item[3)] Le foncteur $\lim_{h,h' \to 0} : \cmodhh \to \cmod$ induit un isomorphisme additif
$$ \left[ \lim_{h,h' \to 0} \right] : \ \grothh \, \stackrel{\sim}{\longrightarrow} \ \grot $$
du groupe de Grothendieck $\grothh$ sur le groupe de Grothendieck $\grot$.
\item[4)] La catégorie $\cmodhh$ possède la propriété de Krull-Schmidt.
\item[5)] Pour tout $n \in \NN$, la représentation de rang fini $L^{h,h'}(n,g)$ est indécomposable. Pour toute représentation $V^{h,h'}$ de $\Uhhp$ de rang fini indécomposable, il existe un unique $n \in \NN$ tel que $V^{h,h'}$ est isomorphe à $L^{h,h'}(n,g)$.
\item[6)] Toute représentation dans $\cmodhh$ est une représentation de poids.
\end{itemize}
\end{theoreme}

\begin{proof}
On note $\mathcal C$ la catégorie des représentations de rang fini de la déformation formelle constante $\Uc[\slt][[h,h']]$. On sait d'après le théorème \ref{thm_uhhdefor} que $\Uhhp$ est une déformation formelle triviale de $\Uc[\slt]$. Il existe alors une équivalence de catégorie additive $\CC[[h,h']]$-linéaire entre $\cmodhh$ et $\mathcal C$, de telle sorte que le diagramme suivant commute :
\begin{equation} \label{eq_commtriv} \xymatrix{
\mathcal C \ar[rr]^-{\simeq} \ar[dr]_-{\lim_{h,h' \to 0}} && \cmodhh \ar[dl]^-{\quad \lim_{h,h' \to 0}} \\
& \cmod &
} \end{equation}
Soit $V \in \cmod$, d'après la proposition \ref{prop_hochrep} et puisque $H^1 \! \left( U(\mathfrak{sl_2}),\text{End}_{\CC} V \right) = 0$ (voir par exemple \cite[II.11]{Gui2}), toute déformation $V^{h,h'}$ dans $\mathcal C$ de $V$ est équivalente à la déformation formelle constante de $V$. \\

Après ces remarques, démontrons dans l'ordre les différents points du théorème.
\begin{itemize}
\item[1)] Toute représentation $V \in \cmod$ admet une déformation dans $\mathcal C$ : la déformation constante. On conclut grâce à \eqref{eq_commtriv}.
\item[2)] D'après \eqref{eq_commtriv} il suffit de montrer le point 2) pour $\mathcal C$. \\
Le foncteur $\lim_{h,h' \to 0}$ induit une application, que l'on note $\left[ \lim_{h,h' \to 0} \right]$, des classes d'isomorphisme de $\mathcal C$ vers les classes d'isomorphisme de $\cmod$. Le foncteur
$$ \mathcal{Q}^{h,h'} \, : \ \cmod \ \to \ \mathcal C \, , $$
qui à une représentation $V \in \cmod$ associe sa déformation constante $V[[h,h']]$, induit clairement un inverse à gauche $\left[ \mathcal{Q}^{h,h'} \right]$ de $\left[ \lim_{h,h' \to 0} \right]$. $\left[ \mathcal{Q}^{h,h'} \right]$ est aussi un inverse à droite, puisque toute représentation dans $\mathcal C$ est isomorphe à la déformation constante de sa limite classique.
\item[3)] Les deux foncteurs précédents étant additifs, $\left[ \lim_{h,h' \to 0} \right]$ et $\left[ \mathcal{Q}^{h,h'} \right]$ induisent, sur les groupes de Grothendieck de $\mathcal C$ et $\cmod$, des isomorphismes inverses l'un de l'autre. On conclut grâce à \eqref{eq_commtriv} encore une fois.
\item[4)] $\cmod$ étant semi-simple, elle possède en particulier la propriété de Krull-Schmidt. D'après le point 2) et puisque $\lim_{h,h' \to 0}$ est additif, on en déduit que $\cmodhh$ la possède aussi.
\item[5)] D'après la proposition \ref{prop_indecdescr}, la limite classique de la représentation $L^{h,h'}(n,g)$ est la représentation irréductible (donc indécomposable) $L(n)$. Toute représentation indécomposable (donc irréductible) de $\mathfrak{sl}_2$ est isomorphe à un $L(n)$ ($n \in \NN$). On conclut en remarquant que le point 2) implique qu'une représentation de $\cmodh$ est indécomposable si et seulement si sa limite classique l'est.
\item[6)] D'après ce qui précède, toute représentation de $\Uhhp$ est isomorphe à une somme directe de représentations $L^{h,h'}(n,g)$. On conclut grâce à la proposition \ref{prop_indecdescr}.
\end{itemize}
\end{proof}

On note, en identifiant les $\CC[[h,h']]$-algèbres $U^0_{h,h'}(\mathfrak{sl}_2,g)$ et $\CC[H][[h,h']]$,
$$ \Pi^0 \, : \ \Uhhp \ \to \ \CC[H][[h,h']] $$
la projection $\CC[[h,h']]$-linéaire sur $U^0_{h,h'}(\mathfrak{sl}_2,g)$ définie naturellement selon la décomposition triangulaire \eqref{eq_triang}. \\
En composant $\Pi^0$ avec l'unique automorphisme de la $\CC[[h,h']]$-algèbre $\CC[H][[h,h']$ qui à $H$ associe $H+1$, on obtient un nouveau morphisme $\CC[[h,h']]$-linéaire $\delta : \Uhhp \to \CC[H][[h,h']]$. \\
En utilisant la remarque \ref{rem_coeff} et la description explicite de l'action de $C$ sur  les représentations $L^{h,h'}(n,g)$ donnée dans la proposition \ref{prop_indecdescr}, on voit que $\delta(C) = H^2$. Le corollaire \ref{cor_eng} implique alors la proposition suivante, qui donne un équivalent pour $\Uhhp$ de l'isomorphisme d'Harish-Chandra (voir \cite{harish}, \cite{Tan} et \cite[VI.4]{Kas}).

\begin{proposition}
La restriction de $\delta$ sur le centre de $\Uhhp$ définit un isomorphisme de $\CC[[h,h']]$-algèbre
$$ \delta \, : \ Z \Big( \Uhhp \Big) \ \stackrel{\sim}{\longrightarrow} \, \ \CC[H^2][[h,h']] \, , $$
qui à $C$ associe $H^2$.
\end{proposition}

\section{Dualité de Langlands pour les représentations \\ des groupes quantiques de rang 1} \label{section_langlands}
Dans cette section, on étudie la spécialisation à $Q = \varepsilon$ de la représentation indécomposable $L^{h,h'}(n,g)$ et du module de Verma $M^{h,h'}(n,g)$. \\
Il est montré que ces spécialisations contiennent de manière naturelle des représentations du groupe quantique $\Urh{gh'}$ : les représentations Langlands $g$-duales de $L^{h,h'}(n,g)$ et $M^{h,h'}(n,g)$. Voir les théorèmes \ref{thm_reprinter} et \ref{thm_reprinter2}. \\
En particulier, la conjecture \cite[conjecture 1]{FH1} concernant l'existence de représentations qui déforment simultanément deux représentations Langlands duales, est résolue ici dans le cas du rang 1 et pour tout $g \in \NN_{\geq 1}$. \\
La propriété d'interpolation/dualité pour les représentations permet de terminer la démonstration, entamée par le lemme \ref{lem_fond}, de la propriété d'interpolation de $\Uhhp$ : voir le théorème \ref{thm_uhhinterpolation}. \\

Soit $g \in \NN_{\geq 1}$. Rappelons que $\varepsilon$ désigne la racine $(2g)$-ième primitive de l'unité $\text{e}^{i \pi/ g} \in \CC$.

\begin{definition}
La spécialisation à $Q = \varepsilon$ de $\Uhhp$ est la $\CC[[h']]$-algèbre
$$ U_{\varepsilon,h'}(\mathfrak{sl_2},g) \ := \ U_{\Aq,h'}(\mathfrak{sl}_2,g) \, / \, \overline{(Q = \varepsilon)}^{\: \! h'} \, . $$
\end{definition}

On définit aussi une notion de poids pour les $U_{\varepsilon,h'}(\mathfrak{sl_2},g)$-modules. \\
Un module $U_{\varepsilon,h'}(\mathfrak{sl_2},g)$-module $V^{\varepsilon,h'}$ est un module de poids si il existe une décomposition en espace de poids :
\begin{equation}
V^{\varepsilon,h'} \ = \ \bigoplus_{n \in \ZZ} \, V^{\varepsilon,h'}_n \, , \quad \text{ avec } \ V^{\varepsilon,h'}_n \ := \ \{ f \in V^{\varepsilon,h'} : \, H.f = n \: \! f \} \, .
\end{equation}
Un vecteur $f \in V^{\varepsilon,h'}_n$ est appelé un vecteur de poids $n$. Si $V^{\varepsilon,h'}_n \neq (0)$, alors $n$ est appelé un poids de $V^{\varepsilon,h'}$ et $V^{\varepsilon,h'}_n$ un espace de poids de $V$. \\

On note $M^{\Aq,h'}(n,g)$ ($n \in \ZZ$) et $L^{\Aq,h'}(n,g)$ ($n \in \NN$) les sous-$\Aq[[h']]$-modules de respectivement $M^{h,h'}(n,g)$ et $L^{h,h'}(n,g)$, topologiquement engendrés par les bases $(m_j)_{j \in \NN}$ et $(m_j)_{0 \leq j \leq n}$ (pour la topologie $(h')$-adique). On a
$$ M^{\Aq,h'}(n,g) \ \simeq \ \Big( \bigoplus_{j \in \NN} \Aq \, m_j \Big) [[h']] \, , \quad \ L^{\Aq,h'}(n,g) \ \simeq \ \Big( \bigoplus_{0 \leq j \leq n} \Aq \, m_j \Big) [[h']] \, . $$
D'après les propositions \ref{prop_vermadescr} et \ref{prop_indecdescr}, $X^{\pm}$, $H$, $C$, $Q^{\pm H}$ et $Q^{\sqrt{C}} + Q^{-\sqrt{C}}$ stabilisent $M^{\Aq,h'}(n,g)$ et $L^{\Aq,h'}(n,g)$. Par suite les actions de $U_{h,h'} (\mathfrak{sl}_2,g)$ sur $M^{h,h'}(n,g)$ et $L^{h,h'}(n,g)$ induisent des actions de $U_{\Aq,h'} (\mathfrak{sl}_2,g)$ sur $M^{\Aq,h'}(n,g)$ et $L^{\Aq,h'}(n,g)$. \\

On note $M^{\varepsilon,h'}(n,g)$ et $L^{\varepsilon,h'}(n,g)$ les spécialisations à $Q = \varepsilon$ de $M^{\Aq,h'}(n,g)$ et $L^{\Aq,h'}(n,g)$ :
$$ M^{\varepsilon,h'}(n,g) \ := \ M^{\Aq,h'}(n,g) \, / \, \overline{(Q -\varepsilon). M^{\Aq,h'}(n,g)}^{\: \! h'} \, , $$
$$ L^{\varepsilon,h'}(n,g) \ := \ L^{\Aq,h'}(n,g) \, / \, \overline{(Q -\varepsilon). L^{\Aq,h'}(n,g)}^{\: \! h'} \, , $$
où $\overline{\ \cdot \ }^{\: \! h'}$ désignent l'adhérence pour la topologie $(h')$-adique. $M^{\varepsilon,h'}(n,g)$ et $L^{\varepsilon,h'}(n,g)$ sont en particulier des $\CC[[h']]$-modules et on a
\begin{equation} \label{eq_vermaindecspe}
M^{\varepsilon,h'}(n,g) \ \simeq \ \Big( \bigoplus_{j \in \NN} \CC \, m_j \Big) [[h']] \, , \quad \ L^{\varepsilon,h'}(n,g) \ \simeq \ \Big( \bigoplus_{0 \leq j \leq n} \CC \, m_j \Big) [[h']] \, .
\end{equation}
L'action de $U_{\Aq,h'} (\mathfrak{sl}_2,g)$ sur $M^{\Aq,h'}(n,g)$ induit une action de $U_{\varepsilon,h'}(\mathfrak{sl_2},g)$ sur $M^{\varepsilon,h'}(n,g)$. De même l'action de $U_{\Aq,h'} (\mathfrak{sl}_2,g)$ sur $L^{\Aq,h'}(n,g)$ induit une action de $U_{\varepsilon,h'}(\mathfrak{sl_2},g)$ sur $L^{\varepsilon,h'}(n,g)$.

\begin{definition}
Soit $g \in \NN_{\geq 1}$.
\begin{itemize}
\item[1)] Soit $n \in \ZZ$. Le $U_{\varepsilon,h'}(\mathfrak{sl_2},g)$-module $M^{\varepsilon,h'}(n,g)$ est appelé la spécialisation à $Q = \varepsilon$ du module de Verma $M^{h,h'}(n,g)$.
\item[2)] Soit $n \in \NN$. Le $U_{\varepsilon,h'}(\mathfrak{sl_2},g)$-module $L^{\varepsilon,h'}(n,g)$ est appelé la spécialisation à $Q = \varepsilon$ de la représentation indécomposable $L^{h,h'}(n,g)$.
\end{itemize}
\end{definition}

On vérifie immédiatement que les $U_{\varepsilon,h'}(\mathfrak{sl_2},g)$-modules $M^{\varepsilon,h'}(n,g)$ et $L^{\varepsilon,h'}(n,g)$ sont des modules de poids. \\

Supposons à présent que $n \in g \ZZ$ et explicitons les modules $M^{\varepsilon,h'}(n,g)$ et $L^{\varepsilon,h'}(n,g)$ :
\begin{eqnarray*}
H.m_j & = & (n-2j) \, m_j \, , \\
Q^H . m_j & = & (-1)^{n/g} \, \varepsilon^{-2j} \, m_j \, , \\
C. m_j & = & (n + 1)^2 \, m_j \, , \\
\big( Q^{\sqrt{C}} + Q^{-\sqrt{C}} \big) . m_j & = & (-1)^{n/g} \, (\varepsilon + \varepsilon^{-1} )  \, m_j \, , \\
X^- . m_j & = & m_{j+1} \, , \\
X^+ . m_j & = & \begin{cases}
\big[ j \big]_{\varepsilon \: \! T} \, \big[n -j+1 \big]_{\varepsilon} \, m_{j-1} \, , & \text{si } j \equiv 0 \ [g] \, , \\
\big[ j \big]_{\varepsilon} \, \big[n -j+1 \big]_{\varepsilon \: \! T} \, m_{j-1} \, , & \text{si } j \equiv 1 \ [g] \, , \\
\big[ j \big]_{\varepsilon} \, \big[n -j+1 \big]_{\varepsilon} \, m_{j-1} \, , & \text{sinon.}
\end{cases}
\end{eqnarray*}

Formellement on obtient l'action de $X^{\pm}$ sur $M^{\varepsilon,h'}(n,g)$ et $L^{\varepsilon,h'}(n,g)$ à partir de celle de $X^{\pm}$ sur les représentations de $U_h(\mathfrak{sl}_2)$ correspondantes, en remplaçant les boîtes quantiques $[j]_Q$ par les boîtes $[j]_{\varepsilon \: \! T}$ quand $g$ divise $j$, et par $[j]_{\varepsilon}$ sinon. \\
Dans le cas où $g = 2$ ou $3$, on retrouve alors formellement les actions de $X^{\pm}$ sur les représentations décrites dans \cite{FH1}, avec $q=\varepsilon$ et $t = T$. \\

On suppose toujours $n \in g \ZZ$. \\
Notons ${}^L \! M^{h'}(n,g)$ et ${}^L \! L^{h'}(n,g)$ les sommes des espaces de poids divisibles par $g$ des $U_{\varepsilon,h'}(\mathfrak{sl_2},g)$-modules $ M^{\varepsilon,h'}(n,g)$ et $L^{\varepsilon,h'}(n,g)$ respectivement. \\
Posons $g' := g/2$ et $n' := 2 \: \!n/g$ si $g$ est pair, $g' := g$ et $n' := n/g$ si $g$ est impair. On a
$$ {}^L \! M^{h'}(n,g) \ \simeq \ \Big( \bigoplus_{j \in \NN} \CC \, m_{g' j} \Big) [[h']] \, , \quad \ {}^L \! L^{h'}(n,g) \ \simeq \ \Big( \bigoplus_{0 \leq j \leq n'} \CC \, m_{g' j} \Big) [[h']] \, . $$
D'après ce qui précède, d'une part $(X^{\pm})^g$, $H$, $C$, $Q^{\pm H}$ et $Q^{\sqrt{C}} + Q^{-\sqrt{C}}$ stabilisent ${}^L \! M^{h'}(n,g)$ et ${}^L \! L^{h'}(n,g)$, d'autre part $Q^{2H}$ agit par $1$ et $Q^{2 \sqrt{C}} + Q^{- 2\sqrt{C}}$ par $\varepsilon^2 + \varepsilon^{-2}$. \\

On note ${}^L \! M^{(h')}(n,g)$ et ${}^L \! L^{(h')}(n,g)$ les localisés par rapport à $h'$ des $\CC[[h']]$-modules ${}^L \! M^{h'}(n,g)$ et ${}^L \! L^{h'}(n,g)$. On a
$$ {}^L \! M^{(h')}(n,g) \ \simeq \ \Big( \bigoplus_{j \in \NN} \CC \, m_{g' j} \Big) ((h')) \quad \text{ et } \quad {}^L \! L^{(h')}(n,g) \simeq \Big( \bigoplus_{0 \leq j \leq n'} \CC \, m_{g' j} \Big) ((h')) \, . $$
Les actions de $U_{\varepsilon,h'}(\mathfrak{sl_2},g)$ induisent alors d'après la remarque \ref{rem_fond}, des actions de $U_{\! gh'}(\mathfrak{sl}_2)$ sur ${}^L \! M^{(h')}(n,g)$ et ${}^L \! L^{(h')}(n,g)$. \\

Grâce à la description précédente des modules $M^{\varepsilon,h'}(n,g)$ et $L^{\varepsilon,h'}(n,g)$, on vérifie que ${}^L \! X^{\pm}$ et ${}^L \! H$ stabilisent ${}^L \! M^{h'}(n,g)$ et ${}^L \! L^{h'}(n,g)$, considérés comme sous-$\CC[[h']]$-modules respectivement de ${}^L \! M^{(h')}(n,g)$ et ${}^L \! L^{(h')}(n,g)$. \\
${}^L \! M^{h'}(n,g)$ et ${}^L \! L^{h'}(n,g)$ sont alors des représentations de $U_{\! gh'}(\mathfrak{sl}_2)$

\begin{definition}
\begin{itemize}
\item[1)] Soit $n \in g \ZZ$. La représentation ${}^L \! M^{h'}(n,g)$ de $\Ughp$ est appellé la représentation Langlands $g$-duale du module de Verma $M^{h}(n)$ de $\Uh[\slt]$.
\item[2)] Soit $n \in g \NN$. La représentation ${}^L \! L^{h'}(n,g)$ de $\Ughp$ est appellé la représentation Langlands $g$-duale de la représentation indécomposable $L^h(n)$ de $\Uh[\slt]$.
\end{itemize}
\end{definition}

Dans la suite $\cmodh$ désignera la catégorie additive $\CC[[h]]$-linéaire des représentations de rang fini de $\Uh[\slt]$ et $\groth$ son groupe de Grothendieck. \\
$\mathcal{C}^{gh'} \! (\mathfrak{sl}_2)$ désignera la catégorie additive $\CC[[h']]$-linéaire des représentation de rang fini de $\Ughp$ et $\text{Rep}^{\: \! gh'} \! (\mathfrak{sl}_2)$ son groupe de Grothendieck. \\

On dispose des foncteurs additifs
$$ \lim_{h' \to 0} : \ \cmodhh \, \longrightarrow \ \cmodh \, , \ \quad \lim_{h \to 0} : \ \cmodh \, \longrightarrow \ \cmod \, , $$
$$ \text{ et } \quad \lim_{h' \to 0} : \ \mathcal{C}^{gh'} \! (\mathfrak{sl}_2) \, \longrightarrow \ \cmod \, . $$

On note $P = \ZZ \, \! \omega$ le réseau des poids de $\mathfrak{sl}_2$ et $P^+ = \NN \: \! \omega \subset P$ le sous-ensemble des poids dominants. On définit une application $\Pi_g : P \to P$ par
$$ \Pi_g (n \: \! \omega) \ := \ \frac{n}{g} \: \! \omega \quad \text{si } g \, / \, n \, , \quad \Pi_g (n \: \! \omega) \ := 0 \quad \text{ sinon,} $$
qui s'étend linéairement en une application $\Pi_g : \ZZ[P] \to \ZZ[P]$ ($\Pi_g$ conserve l'addition mais pas le produit). \\

On note
$$ \chi \, : \ \grot \ \to \ \ZZ[P] \, , \quad \ \chi^h \, : \ \groth \ \to \ \ZZ[P] \, , $$
$$ \text{ et } \ \quad \chi^{g h'} \, : \ \text{Rep}^{gh'} \! (\mathfrak{sl}_2)\ \to \ \ZZ[P] $$
les morphismes de caractères (qui sont on le rappelle des morphismes d'anneau). \\

On peut, de la même façon, définir grâce au point 6 du théorème \ref{thm_rephh} un morphisme de caractère
$$ \chi_g^{h,h'} \, : \ \grothh \ \to \ \ZZ[P] \, . $$
La différence ici est que $\chi_g^{h,h'}$ est un morphisme seulement additif. \\

On vérifie sans difficulté que le diagramme suivant est commutatif :
$$ \xymatrix{
\grothh \ar[ddrr]_-{\chi_g^{h,h'}} \ \ar[rr]^-{\left[ \lim_{h' \to 0} \right]} && \ \groth \ar[dd]_-{\chi^h} \ \ar[rr]^-{\left[ \lim_{h \to 0} \right]} && \ \grot \ar[ddll]^-{\chi} \\
&&&& \\
&& \ZZ[P] &&
} $$

Les deux théorèmes qui suivent résument les constructions faites au début de cette section. Ils donnent également une description des représentations Langlands $g$-duales, qu'on obtient immédiatement d'après ce qui précède.

\begin{theoreme} \label{thm_reprinter}
Soient $g \in \NN_{\geq 1}$ et $n \in g \: \! \NN$.
\begin{itemize}
\item[1)] La représentation $L^{h,h'}(n,g)$ de $\Uhhp$ est une déformation selon $h'$ de la représentation indécomposable $L^h(n)$ de $\Uh[\slt]$ :
$$ \lim_{h' \to 0} L^{h,h'}(n,g) \  = \ L^h(n) \, . $$

\item[2)] La spécialisation $L^{\varepsilon,h'}(n,g)$ à $Q = \varepsilon$ de $L^{h,h'}(n,g)$ est un $U_{\varepsilon,h'}(\mathfrak{sl}_2,g)$-module de poids.

\item[3)] $L^{\varepsilon,h'}(n,g)$ contient la représentation ${}^L \! L^{h'}(n,g)$ de $\Urh{gh'}$, Langlands $g$-duale de $L^h(n)$, comme somme des espaces de poids divisibles par $g$.

\item[4)] Les caractères de $L^h(n)$ et ${}^L \! L^{h'}(n,g)$ sont reliés par
$$ \left( \Pi_g \circ \chi^h \right) \left( L^h(n) \right) \ = \ \chi^{g h'} \left( {}^L \! L^{h'}(n,g) \right) . $$

\item[5)] Si $g$ est paire :
\begin{eqnarray*}
{}^L \! L^{h'}(n,g) & \simeq & L^{gh'} \left( \frac{n}{g} \right) \oplus L^{gh'} \left( \frac{n}{g} - 1 \right) \quad \text{ pour } n > 0 \, , \\
{}^L \! L^{h'}(0,g) & \simeq & L^{gh'}(0) .
\end{eqnarray*}
\item[6)] si $g$ est impaire :
$$ {}^L \! L^{h'}(n,g) \ \simeq \ L^{gh'} \left( \frac{n}{g} \right) . $$
\end{itemize}
\end{theoreme}

\begin{definition}
La limite $h' \to 0$ de la représentation Langlands $g$-duale ($n \in g \NN$) ${}^L \! L^{h'}(n,g)$ est une représentation de $\mathfrak{sl}_2$ que l'on note ${}^L \! L(n,g)$ et qu'on appelle la représentation Langlands $g$-duale de la représentation irréductible $L(n)$ de $\mathfrak{sl}_2$.
\end{definition}

Au niveau des caractères on a :
$$ \left( \Pi_g \circ \chi \right) \left( L(n) \right) \ = \ \chi \left( {}^L \! L(n,g) \right) . $$

On peut illustrer le théorème \ref{thm_reprinter} et la remarque précédente par le diagramme suivant ($g \in \NN_{\geq 1}$, $n \in g \: \! \NN$) :
$$ \xymatrix{
&& L^{h,h'}(n,g) \ar[lld]_-{\lim_{h' \to 0 \ }} \ar[rrd]^-{\lim_{Q \to \varepsilon}} && \\
L^h(n) \ar@<-2.5pt>[d]_-{\lim_{h \to 0}} &&&& \quad {}^L \! L^{h'}(n,g) \ar@<2.5pt>[d]^-{\lim_{h' \to 0}} \ \supset \ L^{gh'}(n/g) \\
L(n) &&&& \quad {}^L \! L(n,g) \ \supset \ L(n/g)
} $$

En d'autres mots, la représentation irréductible $L(n)$ de $\mathfrak{sl}_2$, qui peut être déformée une première fois selon $h$ en une représentation $L^h(n)$ de $\Uh[\slt]$, peut une seconde fois être déformée selon $h'$ en une représentation $L^{h,h'}(n,g)$ de $\Uhhp$, les rangs des espaces de poids restant invariants sous chacune des déformations. \\
Par ailleurs, la spécialisation à $Q = \varepsilon$ de cette double déformation $L^{h,h'}(n,g)$ contient la représentation ${}^L \! L^{h'}(n,g)$ de $\Urh{gh'}$, Langlands $g$-duale de $L^h(n)$. ${}^L \! L^{h'}(n,g)$ est en outre la déformation selon $h'$ de la représentation ${}^L \! L(n,g)$ de $\mathfrak{sl}_2$, Langlands $g$-duale de la représentation $L(n)$ de $\mathfrak{sl}_2$.

\begin{theoreme} \label{thm_reprinter2}
Soit $n \in g \: \! \ZZ$.
\begin{itemize}
\item[1)] Le module de Verma $M^{h,h'}(n,g)$ de $\Uhhp$ est une déformation selon $h'$ du module de Verma $M^h(n)$ de $\Uh[\slt]$ :
$$ \lim_{h' \to 0} M^{h,h'}(n,g) \  = \ M^h(n) \, . $$
\item[2)] La spécialisation $M^{\varepsilon,h'}(n,g)$ à $Q = \varepsilon$ de $M^{h,h'}(n,g)$ est un $U_{\varepsilon,h'}(\mathfrak{sl}_2,g)$-module de poids.
\item[3)] $M^{\varepsilon,h'}(n,g)$ contient la représentation ${}^L \! M^{h'}(n,g)$ de $\Urh{gh'}$, Langlands $g$-duale de $M^h(n)$, comme somme des espaces de poids divisibles par $g$.
\item[4)] Si $g$ est paire :
$$ {}^L \! M^{h'}(n,g) \ \simeq \ M^{gh'} \left( \frac{n}{g} \right) \oplus M^{gh'} \left( \frac{n}{g} - 1 \right) . $$
\item[5)] si $g$ est impaire :
$$ {}^L \! M^{h'}(n,g) \ \simeq \ M^{gh'} \left( \frac{n}{g} \right) . $$
\end{itemize}
\end{theoreme}

\begin{definition}
Soit $n \in g \ZZ$. La limite quand $h'$ tend vers $0$ de la représentation Langlands $g$-duale ${}^L \! M^{h'}(n,g)$ est une représentation de $\mathfrak{sl}_2$ que l'on note ${}^L \! M(n,g)$ et qu'on appelle la représentation Langlands $g$-duale du module de Verma $M(n)$ de $\Uc[\slt]$.
\end{definition}

On peut illustrer le théorème \ref{thm_reprinter2} et la remarque précédente par le diagramme suivant ($g \in \NN_{\geq 1}$, $n \in g \: \! \ZZ$) :
$$ \xymatrix{
&& M^{h,h'}(n,g) \ar[lld]_-{\lim_{h' \to 0 \ }} \ar[rrd]^-{\lim_{Q \to \varepsilon}} && \\
M^h(n) \ar@<-2.5pt>[d]_-{\lim_{h \to 0}} &&&& \quad {}^L \! M^{h'}(n,g) \ar@<2.5pt>[d]^-{\lim_{h' \to 0}} \ \supset \ M^{gh'}(n/g) \\
\quad M(n) &&&& {}^L \! M(n,g) \ \supset \ M(n/g)
} $$

En d'autres mots, le module de Verma $M(n)$ de $\mathfrak{sl}_2$, qui peut être déformée une première fois selon $h$ en une représentation $M^h(n)$ de $\Uh[\slt]$, peut une seconde fois être déformée selon $h'$ en une représentation $M^{h,h'}(n,g)$ de $\Uhhp$. \\
Par ailleurs, la spécialisation à $Q = \varepsilon$ de cette double déformation $M^{h,h'}(n,g)$ contient la représentation ${}^L \! M^{h'}(n,g)$ de $\Urh{gh'}$, Langlands $g$-duale de $M^h(n)$. ${}^L M^{h'}(n,g)$ est la déformation selon $h'$ de la représentation ${}^L \! M(n,g)$ de $\mathfrak{sl}_2$, Langlands $g$-duale du module de Verma $M(n)$ de $\mathfrak{sl}_2$. \\

L'étude précédente des représentations $L^{h,h'}(n,g)$ de $U_{\: \! h,h'}(\mathfrak{sl}_2,g)$ et de leurs spécialisations à $Q = \varepsilon$ permet de terminer la preuve des propriétés d'interpolation de $\Uhhp$ entre les groupes quantiques $\Uh[\slt]$ et $\Ughp$.

\begin{theoreme} \phantomsection \label{thm_uhhinterpolation}
\begin{itemize}
\item[1)] $\Uhhp$ est une déformation formelle de $\Uh[\slt]$ selon le paramètre $h'$. \\
En particulier, $\Uhhp / (h' = 0)$ est isomorphe, en tant que $\CC[[h]]$-algèbre, à $\Uh[\slt]$, via l'identification $\tilde{X}^{\pm} = X^{\pm}$ et $\tilde{H} = H$.

\item[2)] La sous-$\CC[[h']]$-algèbre $\langle {}^L \! X^{\pm}, {}^L H \rangle$ de
$$ \Bigg( \, U_{\! \: \mathcal{A}, h'} (\mathfrak{sl}_2, g) \ / \ \overline{ \big( Q = \varepsilon, \, Q^{2H} = 1, \, Q^{2 \sqrt{C}} + Q^{-2 \sqrt{C}} = \varepsilon^{2} + \varepsilon^{-2} \big) }^{ \, h'} \, \Bigg)_{\! \! h'} \, , $$
 topologiquement engendrée par ${}^L \! X^{\pm}$ et ${}^L H$, est isomorphe à $\Ughp$.
\end{itemize}
\end{theoreme}

\begin{proof}
Le premier point a déjà été traité : voir le théorème \ref{thm_uhhdefor}. \\
Quant au second, d'après le théorème \ref{thm_reprinter}, quelque soit $m \in \NN$, l'action de $\Ughp$ sur la représentation $L^{gh'} (m)$ se factorise à travers le morphisme surjectif
$$ U_{\! g h'}(\mathfrak{sl}_2) \ \longrightarrow \ \langle {}^L \! X^{\pm}, {}^L \! H \rangle $$
donné dans la remarque \ref{rem_fond}. Par suite ce morphisme est injectif. En effet un élément de $U_{\! g h'}(\mathfrak{sl}_2)$ appartenant au noyau de la surjection précédente agit alors par zéro sur toutes ses représentations de dimension finie, et par suite est nul. Ce dernier fait repose sur l'existence de la base PBW de $U_{\! g h'}(\mathfrak{sl}_2)$ et sur la description de ses représentations indécomposables de rang finie (on pourra consulter \cite[5.11]{jantzen} pour un résultat analogue dans le cas de $U_q(\glie)$, où $\glie$ est une algèbre de Lie semi-simple de dimension finie; notons toutefois que la démonstration dans le de $U_q(\mathfrak{sl}_2)$ ou $U_h(\mathfrak{sl}_2)$ est plus simple).
\end{proof}

On peut illustrer le théorème précédent par le diagramme suivant :
$$ \xymatrix{
&& \Uhhp \ar[lld]_-{\lim_{h' \to 0 \ }} \ar[rrd]^-{\lim_{Q \to \varepsilon}} && \\
\Uh[\slt] \ar@<-2.5pt>[d]_-{\lim_{h \to 0}} &&&& \Urh{gh'} \ar@<2.5pt>[d]^-{\lim_{h' \to 0}} \\
\Uc[\slt] &&&& \Uc[\slt]
} $$

En d'autres mots, l'algèbre enveloppante $\Uc[\slt]$, qui admet le groupe quantique $\Uh[\slt]$ comme première déformation selon $h$, est déformée une seconde fois selon $h'$ pour chaque $g \in \NN_{\geq 1} $ : on obtient alors les groupes quantiques d'interpolation de Langlands. \\
Ceux-ci, quand ils sont spécialisés à $Q = \exp h = \varepsilon$, permettent de retrouver les groupes quantiques $\Ughp$, déformations selon $h'$ de $\Uc[\slt]$.

\newpage
\thispagestyle{empty}
\mbox{}
\newpage
\cleardoublepage
\addcontentsline{toc}{part}{Bibliographie -- Bibliography}
\part*{Bibliographie -- Bibliography}
\bibliographystyle{plain-fr}
\renewcommand{\refname}{}
\bibliography{phd}

\end{document}